\documentclass[11pt,a4paper]{amsart}
\usepackage[margin=2.5cm]{geometry}
\usepackage{longtable,booktabs}
\usepackage{mathtools}
\usepackage[hidelinks]{hyperref}
\usepackage[nameinlink]{cleveref}
\usepackage{stackrel}
\usepackage{todonotes}

\usepackage{multirow}
\usepackage[normalem]{ulem}

\usepackage{amsthm}

\usepackage{graphicx}
\usepackage{braket}
\usepackage{enumitem}
\usepackage[all,cmtip]{xy}
\usepackage{rotating}

\usepackage[cmyk,table]{xcolor}
\usepackage{tikz}
\usetikzlibrary{positioning}
\usetikzlibrary{matrix,arrows,decorations.pathreplacing,decorations.pathmorphing,cd,calc,patterns,patterns.meta}
\usepackage{comment}

\usepackage{accents}

\usepackage{amssymb}
\usepackage{amsmath}
\usepackage{mathrsfs}
\usepackage{mathbbol}

\usepackage{bbm}

\definecolor{refkey}{gray}{.3}
\definecolor{labelkey}{gray}{.3}

\newcommand{\Sym}{\mathrm{Sym}}

\newcommand{\rO}{\mathrm{O}}

\newcommand{\pr}{\mathrm{pr}}

\newcommand{\Rad}{\mathrm{Rad}}
\newcommand{\Res}{\mathrm{Res}}

\newcommand{\Aut}{\mathrm{Aut}}

\newcommand{\KZ}{K_0}

\DeclareMathOperator{\Mod}{Mod}

\def\olF{\overline{F}}

\def\barbB{\overline{\bB}}
\def\barbH{\overline{\bH}}

\def\Stab{{\rm Stab}}

\def\id{{\rm id}}

\def\sgn{{\rm sgn}}

\makeatletter
\def\inn#1#2{\left\langle 
\def\ta{#1}\def\tb{#2}
\ifx\ta\@empty{\;} \else {\ta}\fi ,
\ifx\tb\@empty{\;} \else {\tb}\fi
\right\rangle} 
\makeatother

 

\def\innvv#1#2{\inn{#1}{#2}_{\bV}}

\def\Sp{{\rm Sp}}
\def\Rep{\mathcal{R}}
\def\Rep{\mathrm{Rep}}

\def\SL{{\rm SL}}
\def\sl{{\frak{sl}}}
\def\O{{\rm O}}

\def\SO{{\rm SO}}
\def\det{{\rm det}}

\def\NV{{\cN_{\bV}}}
\def\tNV{{\widetilde{\cN}_{\bV}}}

\usepackage{xparse}
\def\usecsname#1{\csname #1\endcsname}
\def\useLetter#1{#1}
\def\usedbletter#1{#1#1}

\usepackage{stackengine}
\stackMath
\newcommand\tsup[2][2]{%
 \def\useanchorwidth{T}%
  \ifnum#1>1%
    \stackon[-1.3ex]{\tsup[\numexpr#1-1\relax]{#2}}{\mathchar"307E}%
  \else%
    \stackon[-1ex]{#2}{\mathchar"307E}%
  \fi%
}

\ExplSyntaxOn 

\def\mydefol#1#2#3{\expandafter\def\csname ol#3{#1}\endcsname{\overline{#2{#1}}}}
\def\mydefbar#1#2#3{\expandafter\def\csname bar#3{#1}\endcsname{\bar{#2{#1}}}}
\def\mydefhat#1#2#3{\expandafter\def\csname hat#3{#1}\endcsname{\hat{#2{#1}}}}
\def\mydefwh#1#2#3{\expandafter\def\csname wh#3{#1}\endcsname{\widehat{#2{#1}}}}
\def\mydeft#1#2#3{\expandafter\def\csname t#3{#1}\endcsname{\tilde{#2{#1}}}}
\def\mydefu#1#2#3{\expandafter\def\csname u#3{#1}\endcsname{\underline{#2{#1}}}}
\def\mydefr#1#2#3{\expandafter\def\csname r#3{#1}\endcsname{\mathrm{#2{#1}}}}
\def\mydefb#1#2#3{\expandafter\def\csname b#3{#1}\endcsname{\mathbb{#2{#1}}}}
\def\mydefwt#1#2#3{\expandafter\def\csname wt#3{#1}\endcsname{{\widetilde{#2{#1}}}}}
\def\mydeff#1#2#3{\expandafter\def\csname f#3{#1}\endcsname{\mathfrak{#2{#1}}}}
\def\mydefbf#1#2#3{\expandafter\def\csname bf#3{#1}\endcsname{\mathbf{#2{#1}}}}
\def\mydefc#1#2#3{\expandafter\def\csname c#3{#1}\endcsname{\mathcal{#2{#1}}}}
\def\mydefsf#1#2#3{\expandafter\def\csname sf#3{#1}\endcsname{\mathsf{#2{#1}}}}
\def\mydefs#1#2#3{\expandafter\def\csname s#3{#1}\endcsname{\mathscr{#2{#1}}}}
\def\mydefcks#1#2#3{\expandafter\def\csname cks#3{#1}\endcsname{{\check{
\csname s#2{#1}\endcsname}}}}
\def\mydefckc#1#2#3{\expandafter\def\csname ckc#3{#1}\endcsname{{\check{
\csname c#2{#1}\endcsname}}}}
\def\mydefck#1#2#3{\expandafter\def\csname ck#3{#1}\endcsname{{\check{#2{#1}}}}}

\NewDocumentCommand{\doGreek}{m}
{
\clist_map_inline:nn {alpha,beta,gamma,Gamma,delta,Delta,epsilon,varepsilon,zeta,eta,theta,vartheta,Theta,iota,kappa,lambda,Lambda,mu,nu,xi,Xi,pi,Pi,rho,sigma,varsigma,Sigma,tau,upsilon,Upsilon,phi,varphi,Phi,chi,psi,Psi,omega,Omega,tG} {#1{##1}{\usecsname}{\useLetter}} 
}

\NewDocumentCommand{\doAtZ}{m}
{
\clist_map_inline:nn {A,B,C,D,E,F,G,H,I,J,K,L,M,N,O,P,Q,R,S,T,U,V,W,X,Y,Z} {#1{##1}{\useLetter}{\useLetter}} 
}
\NewDocumentCommand{\doatz}{m}
{
\clist_map_inline:nn {a,b,c,d,e,f,g,h,i,j,k,l,m,n,o,p,q,r,s,t,u,v,w,x,y,z} {#1{##1}{\useLetter}{\usedbletter}} 
}

\NewDocumentCommand{\doallAtZ}{}
{
\clist_map_inline:nn {mydefol,mydefsf,mydeft,mydefu,mydefwh,mydefhat,mydefr,mydefwt,mydeff,mydefb,mydefbf,mydefc,mydefs,mydefck,mydefcks,mydefckc,mydefbar} {\doAtZ{\csname ##1\endcsname}}
}
\NewDocumentCommand{\doallatz}{}
{
\clist_map_inline:nn {mydefol,mydefsf,mydeft,mydefu,mydefwh,mydefhat,mydefr,mydefwt,mydeff,mydefb,mydefbf,mydefc,mydefs,mydefck,mydefbar} {\doatz{\csname ##1\endcsname}}
}

\NewDocumentCommand{\doallGreek}{}
{
\clist_map_inline:nn {mydefol,mydefck,mydefwt,mydeft,mydefwh,mydefbar,mydefu} {\doGreek{\csname ##1\endcsname}}
}

\NewDocumentCommand{\doGroups}{m}
{
\clist_map_inline:nn {GL,Sp,rO,rU,fgl,fsp,foo,fuu,fkk,fuu,ufkk,uK} {#1{##1}{\usecsname}{\useLetter}} 
}
\NewDocumentCommand{\doallGroups}{}
{
\clist_map_inline:nn {mydefol,mydeft,mydefu,mydefwh,mydefhat,mydefwt,mydefck,mydefbar} {\doGroups{\csname ##1\endcsname}}
}

\doallAtZ
\doallatz
\doallGreek
\doallGroups
\ExplSyntaxOff

\def\GL{\mathrm{GL}}

\def\Id{\mathrm{Id}}

\def\Ker{{\rm Ker}\,}
\def\Lie{{\rm Lie}}
\def\Im{{\rm Im\,}}
\def\Mat{{\rm Mat}}
\def\Ind{{\rm Ind}}

\DeclareMathOperator{\Ad}{Ad}
\DeclareMathOperator{\Spec}{Spec}
\DeclareMathOperator{\Hom}{Hom}
\DeclareMathOperator{\End}{End}
\DeclareMathOperator{\Irr}{Irr}
\DeclareMathOperator{\Herm}{Herm}

\DeclareMathOperator{\IC}{IC}

\def\Supp{\mathrm{Supp}}
\def\Sym{\mathrm{Sym}}

\def\Gr{\mathrm{Gr\,}}

\DeclareMathOperator{\rank}{rank}

\DeclareMathOperator{\Gal}{Gal}

\DeclareMathOperator{\St}{St}

\long\def\mjj#1{{{\color{purple}MJJ: #1}}}
\long\def\qcl#1{{{\color{green}QCL: #1}}}
\long\def\Yun#1{{{\color{brown}Yun: #1}}}
\def\LSG{\widehat{\mathcal{N}}_G}
\def\LSGO{\widehat{\mathcal{N}}_{G^\circ}}
\def\LSGS{\widehat{\mathcal{N}}_{G^{\star}}}
\def\LSGi{\widehat{\mathcal{N}}_{G_i}}

\newcommand{\zjl}[2][]{\if\relax\detokenize{#1}\relax{\color{blue}\vspace{0em}{ZJL:}#2}\else\ifx#1h\relax\else{\color{blue}\vspace{0em}{ZJL}#2}\fi\fi
}

\long\def\delete#1{}

\newcommand{\trivial}[2][]{\if\relax\detokenize{#1}\relax {\color{red} \vspace{0em} {[} #2 {]}} \else\ifx#1h\ifcsname showtrivial\endcsname{\color{orange} \vspace{0em} {[}  #2 {]}}\fi\else {\red Wrong argument!} \fi\fi}

\def\Hom{{\mathrm{Hom}}}

\def\Stab{{\mathrm{Stab}}}
\def\Sp{{\mathrm{Sp}}}

\def\half{{\frac{1}{2}}}

\NewDocumentCommand\cent{m o}{ 
\IfValueTF{#2}{\mathrm{Z}_{#1}({#2})
}{\mathrm{Z}({#1})}
}

\def\bcB{\overline{\cB}}
\def\bcH{\overline{\cH}}
\def\bcM{\overline{\cM}}

\def\Dtate{D^{\mathrm{Tate}}}

\def\NL{\cN_{\bL}}

\def\Ql{\overline{\bQ_\ell}}

\def\mydefdb#1#2#3{\expandafter\def\csname db#3{#1}\endcsname{\breve{#2{#1}}}}
\doGreek{\mydefdb}
\doAtZ{\mydefdb}
\doatz{\mydefdb}

\def\TTidx#1#2{\,{}^{#1}\hspace{-0.1em}#2} 

\def\any{\smblksquare}
\def\mydefTT#1#2#3{
\expandafter\def\csname ii#3{#1}\endcsname{\TTidx{i}{#2{#1}}}
\expandafter\def\csname jj#3{#1}\endcsname{\TTidx{j}{#2{#1}}}
\expandafter\def\csname zz#3{#1}\endcsname{\TTidx{0}{#2{#1}}}
\expandafter\def\csname ll#3{#1}\endcsname{\TTidx{l}{#2{#1}}}
\expandafter\def\csname aa#3{#1}\endcsname{\TTidx{a}{#2{#1}}}
\expandafter\def\csname bb#3{#1}\endcsname{\TTidx{b}{#2{#1}}}
\expandafter\def\csname oo#3{#1}\endcsname{\TTidx{1}{#2{#1}}}
\expandafter\def\csname ss#3{#1}\endcsname{\TTidx{\boxslash}{#2{#1}}}
\expandafter\def\csname dg#3{#1}\endcsname{\TTidx{\boxbackslash}{#2{#1}}}
\expandafter\def\csname any#3{#1}\endcsname{\TTidx{\any}{#2{#1}}}
}
\def\usecsname#1{\csname #1\endcsname}
\def\useLetter#1{#1}
\def\usedbletter#1{#1#1}

 \def\ol{\overline}
\doGreek{\mydefTT}
\doAtZ{\mydefTT}
\doatz{\mydefTT}

\mydefTT{cTp}{\usecsname}{\useLetter}
\mydefTT{sB}{\usecsname}{\useLetter}
\mydefTT{sL}{\usecsname}{\useLetter}
\mydefTT{tD}{\usecsname}{\useLetter}
\mydefTT{tSigma}{\usecsname}{\useLetter}
\mydefTT{Sigmap}{\usecsname}{\useLetter}
\mydefTT{ckGamma}{\usecsname}{\useLetter}
\mydefTT{Gammap}{\usecsname}{\useLetter}
\mydefTT{OmegaK}{\usecsname}{\useLetter}
\mydefTT{etanD}{\usecsname}{\useLetter}
\mydefTT{teta}{\usecsname}{\useLetter}
\mydefTT{dbkappa}{\usecsname}{\useLetter}
\mydefTT{bomega}{\usecsname}{\useLetter}
\mydefTT{biota}{\usecsname}{\useLetter}
\mydefTT{nrho}{\usecsname}{\useLetter}
\mydefTT{nnrho}{\usecsname}{\useLetter}
\mydefTT{dbrho}{\usecsname}{\useLetter}
\mydefTT{bfbb}{\usecsname}{\useLetter}
\mydefTT{sL}{\usecsname}{\useLetter}
\mydefTT{sLp}{\usecsname}{\useLetter}
\mydefTT{lD}{\usecsname}{\useLetter}
\mydefTT{nD}{\usecsname}{\useLetter}
\mydefTT{tnD}{\usecsname}{\useLetter}
\mydefTT{Vp}{\usecsname}{\useLetter}
\mydefTT{bfW}{\usecsname}{\useLetter}
\mydefTT{dbK}{\usecsname}{\useLetter}
\mydefTT{dbeta}{\usecsname}{\useLetter}
\mydefTT{dbG}{\usecsname}{\useLetter}
\mydefTT{dbJ}{\usecsname}{\useLetter}
\mydefTT{SK}{\usecsname}{\useLetter}
\mydefTT{ckG}{\usecsname}{\useLetter}
\mydefTT{End}{\usecsname}{\useLetter}
\mydefTT{fgg}{\usecsname}{\useLetter}
\mydefTT{fggp}{\usecsname}{\useLetter}
\mydefTT{sH}{\usecsname}{\useLetter}

\newcommand{\BTB}[2][]{\if\relax\detokenize{#1}\relax
\cB(#2)
\else 
\cB(#2,#1)
\fi
}
\newcommand{\rBTB}[2][]{\if\relax\detokenize{#1}\relax
\cB_{\rm red}(#2)
\else 
\cB_{\rm red}(#2,#1)
\fi
}

\newcommand{\remove}[1]{\relax}

\newcommand{\tr}{\mathrm{tr}}
\newcommand{\diag}{\mathrm{diag}}
\newcommand{\red}{\color{red}}

\crefformat{equation}{(#2#1#3)}

\crefformat{enumi}{(#2#1#3)}
\crefformat{enumi}{(#2#1#3)}
\Crefformat{enumi}{part (#2#1#3)}
\crefmultiformat{enumi}{(#2#1#3)}{ and~(#2#1#3)}{, (#2#1#3)}{ and~(#2#1#3)}
\Crefmultiformat{enumi}{part~(#2#1#3)}{ and~(#2#1#3)}{, (#2#1#3)}{ and~(#2#1#3)}
\crefrangeformat{enumi}{(#3#1#4) to~(#5#2#6)}
\Crefrangeformat{enumi}{part~(#3#1#4) to~(#5#2#6)}
\crefformat{rmk}{remark~#2#1#3}

\crefformat{thmM}{main theorem}
\Crefformat{thmM}{Main Theorem}
\newtheorem{thm}{Theorem}[section]

\newtheorem{conv}[thm]{Convention}

\newtheorem{lemma}[thm]{Lemma}
\newtheorem{example}[thm]{Example}
\newtheorem{ass}[thm]{Assumption}
\newtheorem{remark}[thm]{Remark}

\newtheorem*{lemma*}{Lemma}
\newtheorem{prop}[thm]{Proposition}

\newtheorem{cor}[thm]{Corollary}
\newtheorem{Exa}[thm]{Example}
\newtheorem{Cor}[thm]{Corollary}

\newtheorem{claim}{Claim}
\newtheorem*{claim*}{Claim}

\theoremstyle{definition}

\newtheorem{defn}[thm]{Definition}

\newtheorem{warning}[thm]{Warning}

\crefformat{del}{{\mbox{\red Deleted reference!}}}
\Crefformat{del}{{\mbox{\red Deleted reference!}}}

\newlist{enumR}{enumerate}{1} 
\setlist[enumR]{wide,label=\arabic*.}

\newlist{enumC}{enumerate}{3} 
\setlist[enumC]{label=(\alph*)}
\newlist{enumP}{enumerate}{3} 
\setlist[enumP]{label=(\roman*)}
\setlist[enumP,1]{label=(\roman*)}
\setlist[enumP,2]{label=(\alph*)}
\setlist[enumP,3]{label=(\arabic*)}
\newlist{enumPF}{enumerate}{3}
\setlist[enumPF]{label=(\roman*),wide}
\setlist[enumPF,1]{label=(\roman*),wide}
\setlist[enumPF,2]{label=(\alph*)}
\setlist[enumPF,3]{label=\arabic*).}
\newlist{enumS}{enumerate}{3} 
\setlist[enumS]{label=\roman*)}
\setlist[enumS,1]{label=\roman*)}
\setlist[enumS,2]{label=\alph*)}
\setlist[enumS,3]{label=\arabic*.}
\newlist{enumIL}{enumerate*}{1} 
\setlist*[enumIL]{label=\roman*)}
\newlist{enumI}{itemize}{1} 
\setlist*[enumI]{label=\textbullet}

\newlist{enumST}{enumerate}{2} 
\setlist*[enumST,1]{wide,font=\bfseries,label=Step~\arabic*,leftmargin=2em}
\setlist*[enumST,2]{wide,font=\bfseries, label*=.\arabic*,leftmargin=2em}

\def\CC{\sfC}

\newcommand{\VV}{\mathbb{V}}

\newcommand{\DD}{\mathbb{D}}
\newcommand{\ZZ}{\mathbb{Z}}
 
\newcommand{\frb}{\mathfrak{b}}
\newcommand{\frg}{\mathfrak{g}}
\newcommand{\frh}{\mathfrak{h}}
\newcommand{\frk}{\mathfrak{k}}
\newcommand{\frl}{\mathfrak{l}}
\newcommand{\frn}{\mathfrak{n}}
\newcommand{\fru}{\mathfrak{u}}
\newcommand{\frp}{\mathfrak{p}}

\newcommand{\Four}{\mathrm{Four}}

\newcommand{\ov}{\overline}
\newcommand{\ot}{\otimes}
\def\wt#1{{\widetilde{#1}}}
\newcommand{\bt}{\boxtimes}
\newcommand{\Qlbar}{\overline{\mathbb{Q}}_{\ell}}
\newcommand{\bij}{\longleftrightarrow}
\newcommand{\Gm}{\mathbb{G}_{m}}

\newtheorem{exam}[thm]{Example}
\newtheorem{cons}[thm]{Construction}

\renewcommand\a\alpha
\renewcommand\b\beta
\newcommand\G\Gamma
\newcommand\g\gamma
\renewcommand\d\delta
\newcommand\D{\Delta}
\newcommand{\e}{\epsilon}
\newcommand{\io}{\iota}

\newcommand{\ph}{\varphi}
\renewcommand{\r}{\rho}
\newcommand{\s}{\sigma}
\renewcommand{\t}{\tau}

\newcommand{\z}{\zeta}
\newcommand{\ep}{\epsilon}

\renewcommand{\l}{\lambda}
\renewcommand{\L}{\Lambda}
\newcommand{\om}{\omega}

\newcommand{\Sig}{\Sigma}
\newcommand{\Fr}{\mathrm{Fr}}
\newcommand{\AS}{\mathcal{AS}}

\def\NLC{\cN_{\bL}^{\bC}}
\def\PM{{\mathrm{PM}}}
\def\SPM{{\mathrm{SPM}}}
\def\add{{\mathrm{add}}}
\newcommand\chM[2]{\mathrm{H}_{c}^{#1}(#2)}
\newcommand\hBM[2]{\mathrm{H}^{\mathrm{BM}}_{#1}(#2)}
\newcommand\xr{\xrightarrow}
\newcommand\cohog[2]{\mathrm{H}^{#1}(#2)}

\newcommand{\Corr}{\mathrm{Corr}}
\newcommand{\Ext}{\mathrm{Ext}}
\newcommand{\rs}{\mathrm{rs}}

\author{Jia-jun Ma}
\address{
School of Mathematical Sciences, Xiamen University,
Xiamen China
}

\email{hoxide@xmu.edu.cn}

\author{Congling Qiu}
\address{
Department of Mathematics, 77 Massachusetts Avenue, Cambridge, MA 02139, USA 
}
\email{qiuc@mit.edu}

\author{Zhiwei Yun}
\address{
Department of Mathematics, 77 Massachusetts Avenue, Cambridge, MA 02139, USA
}
\email{zyun@mit.edu}

\author{Jialiang Zou}
\address{
Department of Mathematics, 77 Massachusetts Avenue, Cambridge, MA 02139, USA
}
\email{jlzou@mit.edu}

\def\CC{\mathrm{CC}}

\subjclass{22E46, 22E47}

\numberwithin{equation}{section}

\title[]{Theta correspondence and Springer correspondence}

\def\ch{\mathrm{ch}}
\def\chA{\mathrm{ch^A}}

\def\AND{\quad\text{and}\quad}
\providecommand{\Span}{\mathrm{Span}}

\def\Ak'{A'_{k-1}}
\def\cAk'{\cA'_{k-1}}

\def\PGL{{\mathrm{PGL}}}

\def\RS{\mathrm{RS}}
\def\tcN{{\widetilde{\cN}}}
\def\tcG{{\widetilde{\mathfrak g}}}
\def\tcR{{\widehat{\cR}}}
\def\btcN{{\widetilde{\ov \cN}}}

\def\Herm{\mathrm{Herm}}
\def\pt{\mathrm{pt}}
\def\Supp{\mathrm{Supp}\,}
\def\For {\mathrm {For}}

\def\Znu{\bZ[\nu^{\half},\nu^{-\half}]}

\def\NLone{\cN_{\bL_1}}
\def\NLtwo{\cN_{\bL_2}}

\def\Dmb#1{{D_{m,#1}^b}}
\def\Db#1{{D_{#1}^b}}
\def\starbarB{\stackrel[\barB]{}{\star}}
\def\starbarbB{\stackrel[{\barbB}]{}{\star}}
\def\starB{\stackrel[B]{}{\star}}
\def\DB{D^b_{m,B}}

\def\Dpari{D^{pari}}
\def\Mpari{\cM^{pari}}
\def\Hpari{\cH^{pari}}
\def\bHpari{\overline{\cH}^{pari}}

\def\Perv{{\mathrm{Perv}}}

\def\tcF{\widetilde{\cF}}

\def\tsG{\widetilde{\sG}}

\def\Semi{\mathrm{Semi}}
\def\brH{\overline{\mathrm{H}}}
\def\NLone{\cN_{\bL_1}}
\def\NLtwo{\cN_{\bL_2}}
\def\CM{\mathrm{FT}}

\def\DbarBB{D_{m,\barB_1\times \barB_2}^b}
\def\DbarBB{D_{m,\barB}^b}
\def\GmBB{\Gm\times \barB_1\times \barB_2}
\def\GmBB{\Gm\times \barB}
\def\BB{\barB_1\times \barB_2}
\def\DGmBB{D_{m,\GmBB}^b}
\def\kk{\mathbb{k}}

\def\wtfll{{\widetilde{\fll}}}
\def\wtfgg{{\widetilde{\fgg}}}

\def\oddeven{\mathrm{OE}}
\def\evenodd{\mathrm{EO}}
\def\eveneven{\mathrm{EE}}
\def\oddodd{\mathrm{OO}}
\def\odd{\mathrm{O}}
\def\even{\mathrm{E}}

\def\bL{{\mathbb L}}

\tikzcdset{
  cells={font=\everymath\expandafter{\the\everymath\displaystyle}},
}
\def\WP{{W_P}}
\def\WL{{W_{L}}}

\def\quot#1#2{{\frac{#1}{#2}}}
\def\quot#1#2{{{#1}{}}}

\def\rhors{\rho^{rs}}

\def\barG{\overline{G}}
\def\barB{\overline{B}}

\def\PP{\mathrm{PP}}
\def\SPP{\mathrm{SPP}}

\begin{document}

\def\sfPp{\sfP_{V'^+}}
\def\xiVpp{\xi_{V'^+}}
\def\sfbb{\bW}
\def\disc{\mathrm{disc}}

\def\lVp{l'}
\def\lWp{l}
\def\ZV{Z'}
\def\ZW{Z}
\def\PV{P'}
\def\PW{P}
\def\LV{L'}
\def\LW{L}
\def\UV{U'}
\def\UW{U}

\def\cksigmaP{\cksigma_{P}}
\def\cksigmaPp{\cksigma'_{P'}}
\def\sigmaP{\sigma_{l}}
\def\sigmaL{\sigma_{L}}
\def\tsigmaL{\widetilde{\sigma}_{L}}
\def\sigmaPp{\sigma'_{l'}}

\def\tsgn{\widetilde{\sgn}}
\def\sigmaPW{\sigma_{\PW}}
\def\sigmaPV{\sigma'_{\PV}}
\def\eV{\varepsilon_V}

\def\Mp{\mathrm{Mp}}

\def\SPM{\mathrm{SPM}}
\def\Tate{\mathrm{Tate}}
\def\bcF{\overline{\cF}}

\def\barX{\overline{X}}

\def\Iso{\mathrm{Iso}}
\def\CIC{{\mathcal{IC}}}
\def\IC{{\mathrm{IC}}}
\def\ICone{{\mathrm{{}^1IC}}}
\def\ICtwo{{\mathrm{{}^2IC}}}
\def\barIC{\overline{\mathrm{IC}}}
\def\DBB{D_{\barB_1\times B_2}(\cN_{\bL})}
\def\piSpr{\pi^{\mathrm{Spr}}}
\def\NCC{\mathrm{cc}}
\def\cNC{{\cN^\bC}}
\def\cNF{{\cN^{\olF}}}
\def\HBM{{\textup{H}_{\textup{top}}^{\textup{BM}}}}

\def\deltaan{\delta}

\def\Weil#1{{{\omega}_{\ensuremath{#1}}}}
\def\tNM{\widetilde{\mathcal{N}}_{M}}
\def\tNV{\widetilde{\mathcal{N}}_{\mathbb V}}
\def\tNVs{\widetilde{\mathcal{N}}_{\mathbb V, s}}
\def\ttNVs{\accentset{\approx}{\mathcal{N}}_{\mathbb V, s}}
\def\NV{{\mathcal{N}}_{\mathbb V}}
\def\OSP{\mathcal {OSP}}
\def\ROSP{\mathcal {ROSP}}
\def\BP{{\mathcal {DP}}}
\def\RBP{{\mathcal {RDP}}}
\def\RT{{\mathcal {RT}}}
\def\RQ{{\mathcal {RQ}}}
\def\bsfH{\overline{\sfH}}
\def\res{{\mathrm{res}}}
\def\hS{{\mathscr{X}}}
\def\hSX{{\mathscr{X}}}
\def\hSXY{{\mathscr{X}}}
\def\bH{\mathbb H}
\def\LS{\mathrm{LS}}
\def\RS{\mathrm{RS}}
\def\LM{\mathrm{LM}}
\def\RM{\mathrm{RM}}

\def\LXY{\Lambda(X,Y)}
\def\cRX{\wt {\mathcal R}^{\diamond}}

\def\cNCL#1{\cN_{\bL_{#1}}^\bC}

\def\SRHom{\mathcal{R}\mathcal{H}om}

\def\cRV{\mathcal{R}^{\heartsuit}_{G_1\times G_2}(\mathbb V)}
\def\cRL{\mathcal{R}^{\heartsuit}_{\ov G_1\times \ov G_2}(\ov \bL)}

\def\BC{B^{\bC}}
\def\cBC{\cB^{\bC}}
\def\barBC{\overline{B}^{\bC}}
\def\cNC{\mathcal{N}^{\bC}}
\def\fnnC{\mathfrak{n}^{\bC}}
\def\act{\mathrm{act}}
\def\ckact{\check{\mathrm{act}}}
\def\FT{\mathrm{FT}}
\def\SemiBBNL{\Semi_{\barB_1\times B_2}(\cN_{\bL_1})}
\def\lsim{{\sim \hspace{-.8em} \raisebox{-0.4em}{\tiny{L}}\hspace{.2em}}}
\def\piCC{\pi^{\CC}}
\def\Mod#1{\mathrm{Mod}(#1)}
\def\AS{\mathrm{AS}}
\def\BS{\mathrm{BS}}
\def\*S{\mathrm{*Sym}}
\def\AJ{\mathrm{AJ}}
\def\BJ{\mathrm{BJ}}
\def\AM{\mathrm{AM}}
\def\BM{\mathrm{BM}}
\newcommand{\bsfM}{\overline{\mathsf{M}}}
\newcommand{\mix}{\mathrm{mix}}
\newcommand{\bs}{\backslash}
\newcommand{\Frob}{\mathrm{Frob}}
\newcommand{\un}{\underline}
\newcommand{\one}{\mathbf{1}}
\newcommand{\wh}{\widehat}
\newcommand{\isom}{\xrightarrow{\sim}}
\newcommand{\Spr}{\textup{Spr}}
\def\ptau{{}^p\tau}

\newcommand{\bigset}[1]{\left\{\,\begin{aligned} #1 \end{aligned}\,\right\}}

\newcommand{\Comp}{\mathrm{Comp}}

\newcommand{\hs}{\heartsuit}
\newcommand{\spd}{\spadesuit}

\def\cNL#1{\cN_{\bL_{#1}}}
\def\Orb#1#2{\underline{#1 \backslash #2}}

\def\TOP{\mathrm{top}}

\maketitle
\setcounter{tocdepth}{1}

\newcommand{\dashedbackslash}{%
  \tikz[baseline, x=1.2ex, y=1.2ex, line width=0.1ex, dash pattern=on 0.4ex off 0.3ex]{
  \draw (0.6,0) -- (0,1.5); 
  }\,%
}

\newcommand{\dashedslash}{%
  \tikz[baseline, x=1.2ex, y=1.2ex, line width=0.1ex, dash pattern=on 0.4ex off 0.3ex]{
    \draw (0,0) -- (0.6,1.5); 
  }%
}

\tikzset{
  wgraph/node/.style={draw, thick, circle, minimum size=14mm, align=center},
  wgraph/edge/.style={-Latex, line width=0.9pt}
}

\newcommand{\WNode}[5][]{%
  \node[wgraph/node] (#2) at #3 {#4};
  \node[#1=2pt of #2] {\scriptsize #5};
}

\newcommand{\WEdge}[3]{%
  \ifnum#3=0
    \draw[wgraph/edge] (#1) -- (#2);
  \else
    \draw[wgraph/edge] (#1) to[bend left=#3] (#2);
  \fi
}
\newcommand{\WEdgeBoth}[3]{%
  \ifnum#3=0
    \draw[<->,thick] (#1) -- (#2);
  \else
    \draw[<->,thick,bend left=#3] (#1) to (#2);
  \fi
}

\section*{Abstract}
In this paper we obtain an explicit formula for the theta correspondence of unipotent principal-series representations between an even orthogonal and a symplectic group or between general linear groups over a finite field. The formula is in terms of the Springer correspondence. Along the way we prove general results about module categories of Hecke categories arising from spherical varieties, and give a similar formula for the multiplicities of unipotent principal series representations in the function space of the spherical variety in terms of relative Springer theory.

 \tableofcontents 
\section{Introduction}
Let $\bF_q$ be a finite field. Let $(G_1,G_2)$ be a reductive dual pair over $\bF_q$ in either of the following cases:
\begin{itemize}
    \item $G_1=\rO(V_1)$ is the full orthogonal group on an even dimensional {\em split} orthogonal space $(V_1,\langle ,\rangle_{V_1})$, and $G_2=\Sp(V_2)$ is the symplectic group on $(V_2,\langle ,\rangle_{V_2})$. In this case, we require the characteristic of $\bF_q$ is not $2$;
    \item $G_1$ and $G_2$ are both general linear groups.
\end{itemize}
Let $W_i$ be the Weyl group of $G_i$.

The theta correspondence restricts to a correspondence between irreducible unipotent principal series representations of $G_1(\bF_q)$ and $G_2(\bF_q)$, which are known to be in natural bijection with $\Irr(W_1)$ and $\Irr(W_2)$ (irreducible complex representations of $W_1$ and $W_2$). The main result of this paper is a description of the theta correspondence for unipotent principal series representations of $G_1(\bF_q)$ and $G_2(\bF_q)$ as a correspondence between $\Irr(W_1)$ and $\Irr(W_2)$ via Springer theory, namely via nilpotent orbits of $G_1$ and $G_2$ and local systems on them.

For the introduction, we will focus on the case of the ortho-symplectic dual pair, where the main result is Theorem \ref{thm: Springer for theta}.




\subsection{Theta correspondence for unipotent principal series}\label{theta correspondence finite field}

Let $(G_1=\rO(V_1),G_2=\Sp(V_2))$ be an ortho-symplectic dual pair where $V_1$ is split with $\dim V_1=2m$ 
 and $\dim V_2=2n$. 

For any finite group $\Gamma$ or a finite dimensional $\bC$-algebra $A$, let $\Irr(\Gamma)$ or $\Irr(A)$ denote the set of isomorphism classes of irreducible $\bC$-representations of $\Gamma$ or $A$. 

Fix a Borel subgroup $B_i\subset G_i$ for $i=1,2$.
Let $\Irr^{pu}(G_i(\bF_q))\subset \Irr(G_i(\bF_q))$ be the classes of irreducible representations that have nonzero fixed vectors under $B_i(\bF_q)$. Equivalently, $\Irr^{pu}(G_i(\bF_q))$ consists of isomorphism classes of irreducible summands of the unipotent principal series $\Ind_{B_i(\bF_q)}^{G_i(\bF_q)}(1)$. We call such representations {\em irreducible unipotent principal series} of $G_i(\bF_q)$.

Recall the Iwahori-Matsumoto Hecke algebra 
\begin{equation} \label{eq:Hecke def}
\begin{split}
  H_i:=& \{f: G_i(\bF_q) \to \mathbb{C} \mid f(b_1gb_2)=f(g),\,\, \forall b_1,b_2\in B_i(\bF_q)\},
\end{split}
\end{equation}
with the product given by convolution. 
It is well-known that taking $B_i(\bF_q)$-fixed vectors gives a bijection
\begin{equation}\label{bij pu with Hecke mod}
    \cE_i: \Irr^{pu}(G_i(\bF_q))\bij \Irr(H_i).
\end{equation}


Let $\VV=V_1\ot V_2$ be equipped with the symplectic form $\langle,\rangle_{\VV}:=\langle{,\rangle}_{V_1}\ot \langle{,\rangle}_{V_2}$. Fix a non-trivial additive character $\psi$ of $\bF_q$ and let $\omega$ be the Weil representation (with $\bC$-coefficients) of $\Sp(\VV)(\bF_q)$ attached to $\psi$ (see \cite[\S 1.4]{MR4677077}). We view $\om$ as a representation of $G_1(\bF_q)\times G_2(\bF_q)$ via the canonical map $G_1\times G_2\to \Sp(\VV)$.  The main problem in the theta correspondence is to decompose \( \omega \) into irreducible representations of \( G_1(\mathbb{F}_q) \times G_2(\mathbb{F}_q) \). This problem has been studied in \cite{AM}, \cite{AMR}, \cite{MR4752682}, \cite{MR4214396}, and \cite{MR4677077}.

Let $\pi_i$ 
be an irreducible representation of $G_i(\bF_q)$, i=1,2. By \cite[\S3,IV.4]{MR1041060}, if $\pi_1\bt \pi_2$ appears in $\om$, and if one of $\{\pi_1,\pi_2\}$ is an irreducible unipotent principal series representation 
the same is true for the other one. Thus, to study how irreducible unipotent principal series representations behave under theta correspondence, we may focus on the summand $\om^{pu}\subset \om$ consisting of $\pi_1\bt\pi_2$ where  $\pi_i\in \Irr^{pu}(G_i(\bF_q))$ for both $i=1$ and $2$. The bijection \eqref{bij pu with Hecke mod} shows that the knowledge of $\om^{pu}$ as a $G_1(\bF_q)\times  G_2(\bF_q)$-representation is equivalent to the knowledge of 
\begin{equation}\label{eq:M}
M:=\omega^{B_1(\bF_q)\times B_2(\bF_q)}
\end{equation}
as a module for the tensor product Hecke algebra $H_1\otimes H_2$. We call $M$ the {\em oscillator bimodule} for the dual pair $(G_1,G_2)$.





Let $W_i$ be the abstract Weyl group of $G_i$. 
Lusztig \cite[Theorem 3.1]{Lu1981BC} has given a canonical isomorphism between $H_i$  and the group algebra $\bC[W_i]$. In particular, we may view $M$ as a $W_1\times W_2$-module, whose structure will be described in terms of Springer correspondence.



\subsection{The geometry of moment maps and Springer correspondence}\label{subsection moment map}  
In this section, we work over the base field $\bC$. 

We first recall the Springer correspondence for a connected reductive group $G$ over $\bC$. Let $\fgg$ be the  Lie algebras of $G$ and $\cN\subset \fgg$ be the nilpotent cone. Let
\begin{equation}\label{eq local system}
 \LSG\coloneqq \set{ (\cO,\cL)|\mbox{$\cO\subset \cN $ a nilpotent orbit, $\cL$ an irreducible $G$-equivariant local system on $\cO$}}.
\end{equation}
Here local systems have $\bC$-coefficients. If $e\in \cO$, and $A_e:=\pi_0(C_G(e))$, then the set of $\cL$ such that $(\cO,\cL)\in \LSG$ are in canonical bijection with irreducible characters of $A_e$. 
Let $W$ be the Weyl group of $G$.
The Springer correspondence is an injective map
\begin{equation}\label{Springer map}
    \Spr_G: \Irr(W)\hookrightarrow \LSG.
\end{equation}
For $(\cO,\cL)\in \LSG$ in the image of $\Spr_G$, let $E_{\cO,\cL}=\Spr_G^{-1}(\cO,\cL)\in \Irr(W)$ be the corresponding irreducible representation of $W$; otherwise define $E_{\cO,\cL}$ to be zero. 
Here we normalize the Springer correspondence such that the trivial local system on the $0$-orbit corresponds to the sign representation of $W$ (cf. \Cref{section relative Springer general}). 

In our paper, we will use the Springer correspondence for the full orthogonal group, which is not connected. The Springer correspondence for disconnected groups is worked out in \cite{MR3845761} and \cite{dillery2023stacky}. For the reader’s convenience, we include a self-contained description of the Springer correspondence for the full orthogonal group in \Cref{sec:orthogonal Springer}.

We now return to the setup of ortho-symplectic dual pair but working over the base field $\bC$:  $(V_1, \langle, \rangle_{V_1})$ and $(V_2, \langle, \rangle_{V_2})$ are orthogonal and symplectic vector spaces over  $\bC$, and $G_1=\rO(V_1)$ and $G_2=\Sp(V_2)$ are also considered over $\bC$. Note that $G_1$ is not connected; let $G_1^\circ = \SO(V_1) \subset G_1$ be its neutral component. Let $\mathfrak{g}_1$ and $\mathfrak{g}_2$ denote the Lie algebras of $G_1^\circ$ and $G_2$, respectively. In order to state the theta correspondence between $\Irr(W_1)$ and $\Irr(W_2)$, we need the notion of a \textit{relevant quintuple}, which we now introduce.


The space $\VV:=  {V}_1\otimes  {V}_2$ is equipped with the symplectic form characterized by 
\begin{equation}\label{eq:vvform}
\innvv{v_1\otimes v_2}{v'_1 \otimes v'_2} =\inn{v_1}{v_2}_{V_1} \inn{v'_1}{v'_2}_{V_2} 
\end{equation}
for all $v_1,v'_1\in V_1$ and $v_2,v'_2\in V_2$.

The $G_1\times G_2$-action on the symplectic space 
has a 
moment map
\begin{equation}\label{eq:moment map}
\mu=(\mu_1,\mu_2): \VV\longrightarrow \mathfrak g_1^* \times \mathfrak g_2^* 
\end{equation}
characterized by 
\[
\mu_1(v)(X_1)=\half \langle X_1 v, v\rangle_{\VV}, 
\quad \mu_2(v)(X_2)=\half \langle X_2 v, v\rangle_{\VV}\quad \mbox{for}\,\, v\in \VV, X_1\in \fgg_1, X_2\in \fgg_2. 
\]
\trivial[h]{
 The factor of $\frac{1}{2}$ appears because when you differentiate a quadratic form, you get a factor of 2 which is then canceled
}
We identify $\mathfrak g_1^*$ and $\mathfrak g_2^*$ with $\mathfrak g_1$ and $\mathfrak g_2$ via the trace forms and view the moment map $\mu$ as a map from $\VV$ to $\fgg_1\times \fgg_2$.  
Let 
\begin{equation}\label{Def: nilcone}
    \cN_{\VV}\coloneqq \mu^{-1}(\cN_1\times \cN_2) 
\end{equation}
be the nilcone of $\VV$. By Kraft--Procesi \cite[\S 6]{MR694606}, we know that $\cN_{\VV}$ has finitely many $G_1\times G_2$-orbits.

\begin{defn}\label{def relevant orthsymplectic}
    \begin{enumerate}
        \item A $G_1\times G_2$-orbit $\cO\subset \cN_{\VV}$ is called \emph{relevant} if it satisfies 
\begin{equation}\label{eq:relevant}
    \dim\cO=\frac{1}{2}(\dim\VV+\dim\mu_1(\cO)+\dim\mu_2(\cO)).
\end{equation} 
We denote the set of relevant $G_1\times G_2$-orbits in $\cN_{\VV}$ by $\cR_{\VV}$. 
        \item Consider the quintuple $(\cO, \cO_{1},\cL_1,\cO_{2},\cL_2)$ such that 
        \begin{itemize}
        \item $\cO$ is a $G_1\times G_2$-orbit in $ \cN_{\VV}$,
        \item  $\cO_i=\mu_i(\cO)\subset \cN_i$ under the moment map for $i=1,2$ and 
        \item $(\cO_i,\cL_i)\in \LSGi$ for $i=1,2$.
        \end{itemize}
        It is called {\em relevant} if the following conditions are satisfied:
        \begin{enumerate}
            \item[(i)] $\cO$ is relevant in the sense of part (1);
            \item[(ii)] Choose an $e\in \cO$ and let $\mu_i(e)=e_i$ for $i=1,2$. Then $\cL_i$ corresponds to $\chi_i\in \Irr(A_{e_i})=\Hom(A_{e_i},\bC^\times)$. Let
            \begin{equation}\label{eq component group O}
           A_{e}\coloneqq \pi_{0}(\Stab_{G_{1}\times G_{2}}(e))
        \end{equation} 
        equipped with maps
        \begin{equation}\label{eq moment map component group}
            \D_i: A_e\to A_{e_i}
        \end{equation} induced by the projections $\Stab_{G_1\times G_2}(e)\to C_{G_i}(\mu_i(e))$. Then the restrictions $\chi_i|_{A_{e}}:=\chi_i\circ\D_i\in \Irr(A_e)$ satisfy
        \begin{equation}\label{eq relevant character}
        (\chi_1|_{A_{e}})\cdot (\chi_2|_{A_e})=1\in \Irr(A_{e}).
        \end{equation}
        \end{enumerate}

    \end{enumerate}
We denote the set of relevant quintuples by $\cRV$. 
\end{defn}




Following the discussion in the last paragraph of \Cref{theta correspondence finite field},  for each $E\in\Irr(W_i)$, let  $E(q)$ denote the corresponding $H_i$-module via Lusztig's canonical isomorphism $H_i \cong \mathbb{C}\left[W_i\right]$.

The following is our main result that determines the structure of the oscillator bimodule. 

\begin{thm}\label{thm: Springer for theta}
Let $M$ be as in \eqref{eq:M}. There is a canonical isomorphism of $H_1\otimes H_2$-modules 
\begin{equation}\label{deformation 1}
   M\ot \sgn_{H_1\ot H_2} \cong 
\bigoplus_{(\cO, \cO_{1},\cL_{1},\cO_{2},\cL_{2})\in \cRV}E_{\cO_{1}, \cL_{1}}(q)\bt E_{\cO_{2},\cL_{2}}(q).
\end{equation}
Here, $\sgn_{H_1\ot H_2}$ is the sign character of $H_1\ot H_2$. 
\end{thm} 
\begin{remark}
Since $G_1$ is disconnected, special care is required in defining the sign character of $H_1$. We use the convention established in \Cref{sgn orthogonal group} for $\sgn_{H_1}$.
\end{remark}
In \Cref{orthosymplectic orbits},
we will give a combinatorial description of the set $\cRV$ in terms of {\em ortho-symplectic partitions}, making the above result more concrete. We remark that, in a relevant quintuple $(\cO, \cO_1,\cL_1,\cO_2, \cL_2)$, $\cO$ is in fact uniquely determined by $(\cO_1,\cO_2)$ (see  \Cref{relevatbt quintuple counting}). The proof of \Cref{thm: Springer for theta} will be given in \Cref{section proof main thm}. 


\subsection{Spherical modules for the Hecke algebra}\label{subsection spherical hecke modules}
Our method for proving \Cref{thm: Springer for theta} applies more broadly to a large class of Hecke modules arising from spherical varieties. We describe the parallel results here.

Let $G$ be a split connected reductive group over $\bF_q$, and let $B \subset G$ be a Borel subgroup. Let $k=\ov\bF_q$. Let $X$ be a smooth spherical $G$-variety; that is, $B$ has finitely many {\em geometric} orbits on $X$ (i.e., orbits of $B_k=B\otimes_{\bF_q}k$ acting on $X_k=X\otimes_{\bF_q}k$). We impose the following simplifying assumption: 
\begin{equation}\label{intro stab}
    \text{Each geometric $B$-orbit in $X$ contains an $\bF_q$-point with connected stabilizer.}
\end{equation}
We remark that many interesting examples—such as symmetric varieties—do not satisfy this assumption. In those cases, the analogue of the result below becomes significantly more involved and will not be addressed in this paper.

The space of functions $M \coloneqq \bC[B(\bF_q) \backslash X(\bF_q)]$ is naturally a module over the Hecke algebra $H \coloneqq  \bC[B(\bF_q) \backslash G(\bF_q) / B(\bF_q)]$ via the convolution 
\begin{equation}\label{hecke function covolution}
    (f* h) (x)=|B(\bF_q)|^{-1} \sum_{g\in G(\bF_q)} f(g) h(g^{-1}x) \quad \forall f\in H, h\in M, x\in X(\bF_q).
\end{equation}
We refer to $M$ as a {\em spherical Hecke module}. Our goal is to describe the $H$-module $M$ in terms of the Springer correspondence, in a manner analogous to \Cref{thm: Springer for theta}.

\begin{ass}
We assume that $G$ and $X$ (and the action) admit integral models. We denote by $G_\bC$ and $X_{\bC}$ their base changes to $\bC$. Furthermore, we assume the existence of a well-behaved bijection (see \Cref{orbit assumption bFq bC} for a precise formulation) between the orbit sets $\underline{B(k)\backslash X(k)}$ and $\underline{B(\bC)\backslash X(\bC)}$.
\end{ass}

We now pass to the complex setting. Let $\frak g_\bC$ denote the Lie algebra of $G_\bC$, and let $\cN_\bC \subset \frak g_\bC$ be the nilpotent cone. The $G_\bC$-action on $X_\bC$ lifts to a Hamiltonian action on the cotangent bundle $T^*X_\bC$, with moment map $\mu: T^*X_\bC \to \frak g_\bC^*$. The following definitions is an analog of those in \Cref{def relevant orthsymplectic}.

\begin{defn}\label{defn intro relevant}
\leavevmode
\begin{enumerate}
    \item A $G_\bC$-orbit $\cO \subset \cN_\bC$ is said to be \emph{$X_\bC$-relevant} if it satisfies the dimension condition
    \[
    \dim \mu^{-1}(\cO) = \tfrac{1}{2} \dim \cO + \dim X_\bC.
    \]
    We denote the set of such orbits by $\cR_{G_\bC}(X_\bC)$.
    
    \item A pair $(\cO,\cL) \in \widehat {\cN}_{G_\bC}$ is \emph{$X_\bC$-relevant} if $\cO \in \cR_{G_\bC}(X_\bC)$, and for a representative $e \in \cO$, the fiber $(T^*X_\bC)_e := \mu^{-1}(e)$ satisfies
        \begin{equation}\label{Vechi}
                V_{e,\chi}(X_\bC) := \Hom_{A_e}(\chi, \HBM((T^*X_\bC)_e,\bC)) \neq 0,
        \end{equation}
        where $\chi \in \Irr(A_e)$ corresponds to the local system $\cL$ and  $\HBM(-,\bC)$ denotes top Borel-Moore homology with $\bC$-coefficients.
\end{enumerate}
We write $\widehat \cR_{G_\bC}(X_\bC)$ for the set of $X_\bC$-relevant pairs.
\end{defn}
In the above definition, the vector space $V_{e,\chi}(X_\bC)$ depends only on $(\cO,\cL)\in\widehat {\cN}_{G_\bC}$; we henceforth denote it by $V_{\cO,\cL}(X_\bC)$. 

\begin{thm}\label{thm:intro spherical}
Let $X$ be a smooth spherical $G$-variety over $\bF_q$ satisfying \Cref{intro stab} and admitting an integral model satisfying \Cref{orbit assumption bFq bC}. Then there is a canonical isomorphism of $H$-modules:
\[
M\ot \sgn_H\cong  
\bigoplus_{(\cO,\cL) \in \widehat \cR_{G_\bC}(X_\bC)} E_{\cO,\cL}(q) \otimes V_{\cO,\cL}(X_\bC).
\]
\end{thm}
The proof of \Cref{thm:intro spherical} will be given in \Cref{subsection spherical proof}. We list three examples covered by \Cref{thm:intro spherical}:
\begin{exam}\label{example spherical}
\begin{enumerate}
    \item \textbf{Type II theta correspondence.} Let $L_1$ and $L_2$ be finite-dimensional $\bF_q$-vector spaces of dimensions $m$ and $n$, respectively. Let $G = \GL(L_1) \times \GL(L_2)$ act naturally on the affine space $X = \Hom(L_1, L_2)$. Then $X$ is a spherical $G$-variety. We study this example in detail in \Cref{ex: GL theta}.
\item \textbf{Rankin-Selberg case.} Let $L_2$ be a $n$-dimensional vector space over $\bF_q$, and $L_1=L_2\oplus L$ be the direct sum with a one-dimensional space $L$.
    Let $\GL(L_1)$ and $\GL(L_2)$ be the corresponding general linear groups, and set $G = \GL(L_1) \times \GL(L_2)$, with $H = \GL(L_2)$ diagonally embedded in $G$. Then $X = G/H$ is a spherical $G$-variety. We will explore this example in a forthcoming work.
    \item \textbf{Orthogonal Gan–Gross–Prasad case.} Let $(W_2, \langle \cdot, \cdot \rangle_{W_2})$ be a $2n$-dimensional split orthogonal space over $\bF_q$, and $W_1=W_2\oplus L$ be the orthogonal direct sum with a one-dimensional orthogonal space $L$.
    Let $\SO(W_1)$ and $\SO(W_2)$ be the corresponding special orthogonal groups, and set $G = \SO(W_1) \times \SO(W_2)$, with $H = \SO(W_2)$ diagonally embedded in $G$. Then $X = G/H$ is a spherical $G$-variety. We will explore this example in a forthcoming work.
\end{enumerate}
\end{exam}

\subsection{Sketch of proofs}

Both \Cref{thm: Springer for theta} and \Cref{thm:intro spherical} are established by geometrizing the oscillator bimodule and spherical Hecke module and applying a relative version of Springer theory. We outline the main ideas of the proofs in the remainder of this introduction.

\subsubsection{Geometrization and sketch of proof of \Cref{thm:intro spherical}}\label{intro geom}

Fix a prime $\ell$ different from $\textup{char}(\bF_q)$ and an isomorphism $\Qlbar\cong \bC$. For a variety or stack $X$ over $\bF_q$, let $D^{\mix}(X)$ denote the bounded derived category of mixed $\Qlbar$-sheaves on $X$. The sheaf-to-function correspondence defines a linear map
\begin{equation*}
    \phi: K_0(D^{\mix}(X)) \to \Qlbar[X(\bF_q)] \cong \bC[X(\bF_q)].
\end{equation*}
We regard $D^{\mix}(X)$ as a geometrization (or categorification) of the function space $\bC[X(\bF_q)]$. In this framework, the Hecke algebra $H$ of a reductive group $G$ over $\bF_q$ geometrizes to the monoidal Hecke category $\cH^{\mix} = D^{\mix}_B(\cB)$, where $\cB = G/B$ is the flag variety. The $H$-module $\bC[B(\bF_q) \backslash X(\bF_q)]$ (for a $G$-variety $X$) geometrizes to $D^{\mix}_B(X)$, which carries an action of $\cH^{\mix}$.

The Hecke algebra $H$ fits into a one-parameter deformation known as the generic Hecke algebra $\sfH$: it is an $R = \bZ[v,v^{-1}]$-algebra generated by $\sfT_s$ ($s \in W$ simple reflections), subject to the quadratic relation $(\sfT_s+1)(\sfT_s - v^2) = 0$ and the braid relations. For $r \in \bC^\times$, we denote by $\sfH_{v=r,\bC}\coloneqq \sfH\otimes_{R} \bC$ the extension of scalars along the specialization map $\bZ[v,v^{-1}]\rightarrow \bC$ sending $v$ to $r\in \bC^\times$. In particular, the Hecke algebra $H$ of $G$ is the specialization $\sfH_{v=q^{1/2}, \bC}$, and $\sfH_{v=1,\bC}$ is the group ring $\bC[W]$. Similar specialization notation applies to any $R$-module $\sfN$.

In the situation of \Cref{thm:intro spherical}, the main players in the proof can be summarized as follows:
\begin{equation}\label{M diagram}
\xymatrix{
(\cH^{\mix} \curvearrowright \cM^{\mix}) \ar[d]^{\phi} & \ar[l] \ar[r] (\cH^{\Tate} \curvearrowright \cM^{\Tate}) \ar[d]^{\ch^{\Tate}} & (\cH^{k} \curvearrowright \cM^{k}) \ar[d]^{\chi} \ar@{--}[r] & (\cH^{\bC} \curvearrowright \cM^{\bC}) \ar[d]^{\CC} \\
(H \curvearrowright M) & \ar[l]_-{v=q^{1/2}} \ar[r]^-{v=1} (\sfH \curvearrowright \sfM) & (W \curvearrowright \sfM_{v=1,\bC}) \ar[r]^-{c}_-{\sim} & (W \curvearrowright \HBM(\St_{T^*X_{\bC}}, \bC))
}
\end{equation}

Here the top row consists of categories of sheaves on $B \backslash X$ with Hecke category actions. The vertical arrows map the Grothendieck groups of the first row to modules on the second row, compatible with the corresponding Hecke algebra actions.

The map $\phi$ is the sheaf-to-function map. The categories $\cM^\Tate$ and $\cH^\Tate$ are full subcategories of $\cH^\mix$ and $\cM^\mix$ by imposing conditions on the Frobenius eigenvalues on stalks. They are designed so that by passing to Grothendieck groups, the $\sfH\cong K_0(\cH^{\Tate})$-module $\sfM:=K_0(\cM^{\Tate})$ gives a one-parameter deformation of $M$ as an $H$-module. The categories $\cH^{k}$ and $\cM^{k}$ are $B_k$-equivariant sheaves on the base changes $\cB_k$ and $X_k$, where $k=\ov\bF_q$. Finally, using suitable integral models of $G$ and $X$, we may talk about $B_\bC$-equivariant sheaves on the complex fibers $\cB_\bC$ and $X_\bC$, so that the dotted arrow above means an isomorphism between Grothendieck groups $K_0(\cM^k)\cong K_0(\cM^{\bC})$. In the bottom right corner, we have
\begin{equation*}   
\St_{T^*X_\bC}:=(T^*X_\bC)\times_{\frg^*_\bC}\tcN_\bC
\end{equation*}
where $\tcN_\bC=T^*(\cB_\bC)$ with moment map $\pi:\tcN_\bC\to \frg^*_\bC$, so that $G_\bC\bs \St_{T^*X_\bC}$ is the Hamiltonian reduction of $T^*X_\bC$ by $B_\bC$. 

The map labeled by $\CC: K_0(\cM^{\bC})_{\bC}\to \HBM(\St_{T^*X_{\bC}},\bC)$ in \eqref{M diagram} is taking characteristic cycles, and it is an isomorphism. The first two horizontal maps on the second row are specialization maps. The last one $c: \sfM_{v=1,\bC}\cong\HBM(\St_{T^*X_{\bC}},\bC)$ uses the bijection between orbit sets $\un{B_k\bs X_k}$ and $\un{B_{\bC}\bs X_{\bC}}$, and the $\CC$ map.

Next we recall how the top Borel--Moore homology $\HBM(\St_{T^*X_{\bC}}, \bC)$
carries a $W$-action from relative Springer theory. Identifying $\frg_{\bC}^* \cong \frg_{\bC}$ via a non-degenerate invariant form, $\pi$ factors through the Springer resolution $\pi_{\cN_{\bC}} : \tcN_{\bC} \to \cN_{\bC}$. The Springer sheaf is defined by
\begin{equation}\label{SpringerSheaf}
    \cS_{\cN_{\bC}} := (\pi_{\cN_{\bC}})_* \bC_{\tcN_{\bC}} [\dim \tcN_{\bC}] \in \Perv_{G_\bC}(\cN_{\bC}).
\end{equation}
It is semisimple and carries a canonical $W$-action. To simplify notation, we denote by \( \hS \coloneqq T^*X_{\bC} \), equipped with the moment map \( \mu_{\hS} \colon \hS \to \mathfrak{g}^* \). Define the \emph{\( \hS \)-relevant Springer sheaf} $\cS_{\hS}$ and its perverse truncation $\cS^\diamond_{\hS}$:
\begin{equation}\label{ReleventSpringerSheaf}
    \cS_{\hS} := i_{\cN}^! \mu_{\hS*} \mathbb{D}_{\hS}[-\dim \hS], \quad \cS^\diamond_{\hS} := \ptau_{\leq 0} \cS_{\hS} \in \Perv_{G_{\bC}}(\cN_{\bC}).
\end{equation}
where $\mathbb{D}_{\hS}$ is the dualizing sheaf on $\hS$. A standard computation (see \Cref{W-action on relative hom} and \Cref{eq iso trucation}) yields a canonical isomorphism
\begin{equation}\label{eq W-action on relative hom}
    \HBM(\St_{\hS}, \bC) \cong \Hom_{D^b_{G_{\bC}}(\cN_{\bC})}(\cS_{\cN}, \cS_{\hS}) \cong \Hom_{\Perv_{G_{\bC}}(\cN_{\bC})}(\cS_{\cN}, \cS^\diamond_{\hS})
\end{equation}
and hence an induced action of \( W \) on \( \HBM(\St_{\hS}, \bQ) \) via the action on $\cS_{\cN}$. 
The decomposition of $\cS_{\cN}$ and $\cS^\diamond_{\hS}$ into simple perverse sheaves yields the following (see \Cref{sec Relevant Orbits}):
\begin{thm}\label{Thm decomposition of SthS intro}
There is a decomposition of $\HBM(\St_{\hS}, \bC)$ as a $W$-representation:
\[
    \HBM(\St_{\hS}, \bC) \cong \bigoplus_{(\cO,\cL) \in \widehat \cR_{G_{\bC}}(X_{\bC})} E_{\cO,\cL} \otimes V_{\cO,\cL}(X_{\bC}).
\]
\end{thm}

In \Cref{Prop. CC W eq}, we show that the $W$-action on $\HBM(\St_{\hS}, \bC)$ induced from the Hecke category coincides with the Springer action, up to tensoring with the sign character. \Cref{thm:intro spherical} follows from the bottom row of \eqref{M diagram} together with \Cref{Thm decomposition of SthS intro}.

\subsubsection{Geometrization of oscillator bimodule and sketch of proof of \Cref{thm: Springer for theta}}\label{Section Geometrization of oscillator bimodule}
The oscillator bimodule $M$ comes from the non-polarizable Hamiltonian $G_1\times G_2$-variety $\VV=V_1\ot V_2$, and its geometrization is more involved. As is common in theta correspondence, we use two polarizations of $\VV$ to get two diagrams as in \eqref{M diagram}, each only seeing the action of part of the Hecke algebra $\sfH_1\ot \sfH_2$. We then glue the two constructions by the partial Fourier transform.

Let $V_1=L_1\oplus L_1^\vee$ be a polarization of $V_1$ such that 
$L_{1}$ is stable under $B_{1}$, and 
\begin{equation}\label{def L1}
\bL_1:=\Hom(L_1, V_2).
\end{equation}
We write $\barG_1 := \GL(L_1)$, and denote by $\barB_1$ the image of $B_1$ in $\barG_1$, which is a Borel subgroup of $\barG_1$.  Let $\ov H_1$ denote the Hecke algebra of $\barG_1$, which is naturally a subalgebra of $H_1$. Define the {\em moment cone} 
\begin{equation}\label{moment cone 1}
      \cNL{1}\coloneqq \set{T\in \Hom(L_1,V_2)| T(L_1)\, \text{is isotropic}}.
\end{equation}
The group $B_2$ and $\ov B_1\subset \GL(L_1)$ acts on $\cNL{1}$. Using the Schr\"odinger model of $\om$ given by $\bL_1$, one checks that there is a canonical isomorphism of $\ov H_1\ot H_2$-modules (cf. \Cref{Lem. N_L invariant})
\begin{equation*}
    M\cong M_1:=\bC[(\ov B_1\times B_2)(\bF_q)\bs \cNL{1}(\bF_q)].
\end{equation*}
We then apply the geometrization procedure in \Cref{intro geom} to $M_1$ to obtained a certain category $\cM^\Tate_1$ of $\ov B_1\times B_2$-equivariant complexes of sheaves on $\cNL{1}$. This is a module category under $\ov \cH_1^\Tate \ot \cH_2^\Tate$. Passing to the Grothendieck group gives a $\ov\sfH_1\ot \sfH_2$-module $\sfM_1$, whose specialization at $v=q^{1/2}$ recovers $M_1$.

Using the polarization $V_2=L_2\oplus L_2^\vee$ where $L_2$ is stable under $B_2$ instead, we define another moment cone
\begin{equation}\label{moment cone 2}
      \cNL{2}\coloneqq \set{T\in \bL_2:=\Hom(L_2,V_1)| T(L_2)\, \text{is isotropic}}.
\end{equation}
We write $\barG_2 := \GL(L_2)$, the Borel subgroup $\barB_2$ of $\barG_2$ as image of $B_2$ and the Hecke algebra $\ov H_2$ as a subalgebra of $H_2$. Then $\cNL{2}$ carries an action of $B_1\times \ov B_2$. We similarly observes that $M$ as an $H_1\ot \ov H_2$-module is isomorphic to $M_2:=\bC[(B_1\times \ov B_2)(\bF_q)\bs \cNL{2}(\bF_q)]$, and the latter geometrizes to a category $\cM_2^{\Tate}$ of $B_1\times \ov B_2$-equivariant sheaves on $\cNL{2}$, giving an $\sfH_1\ot \ov\sfH_2$-module $\sfM_2$ whose specialization at $v=q^{1/2}$ recovers $M_2$.

The following theorem, whose proof is completed in \Cref{Sec Fourier}, allows us to ``glue" the two geometrizations.

\begin{thm}\label{thm:Fourier}
The partial Fourier-Deligne transform defines an equivalence of $\ov\cH^{\Tate}_1\ot \ov\cH^{\Tate}_2$-module categories 
\[
\FT : \cM^{\Tate}_1\isom \cM^{\Tate}_2.
\]

\end{thm}
In \Cref{subsection Fourier bijection}, we determine the bijection between simple perverse sheaves induced by the above partial Fourier-Deligne transform. Passing to Grothendieck groups, $\FT$ induces an isomorphism $\sfM_1\isom \sfM_2$, through which both acquire modules structure under $\sfH_1\ot \sfH_2$. We denote the resulting $(\sfH_1, \sfH_2)$-bimodule by $\sfM$, the generic oscillator bimodule.


Moreover, the specialization $\sfM_{v=q^{1/2}, \bC}$ is canonically isomorphic to the oscillator $H_1\times H_2$-module $M$.
The specialization $\sfM_{v=1,\bC}$ then carries an action of $W_1\times W_2$. We have now obtained the analog of the bottom row of \eqref{M diagram} for the oscillator bimodule
\begin{equation}\label{osc M diagram}
    \xymatrix{(H_1\ot H_2\curvearrowright M) & \ar[l]_-{v=q^{1/2}}\ar[r]^-{v=1}(\sfH_1\ot \sfH_2\curvearrowright \sfM) & (W_1\times W_2\curvearrowright\sfM_{v=1,\bC}) \ar[r]^-{c}_-{\sim}& (W_1\times W_2\curvearrowright\HBM(\St_{\VV},\bC))}
\end{equation}
Here $\St_{\VV}$ is the {\em ortho-symplectic Steinberg variety}, considered before by Braverman--Finkelberg--Travkin \cite{MR4503990} in a related context.
\begin{equation}\label{eq orthogonal-symplectic Steinberg variety}
    \St_{\VV}:= \VV\times_{\frg_1^*\times\frg_2^*}(\wt\cN_1\times\wt\cN_2)
    =\Set{(T,\mathsf{B}_1,\mathsf{B}_2)\in \cN_{\VV}\times \cB_1\times \cB_2 | \mu_i(T)\in (\Lie \mathsf{B}_i)^\perp, i=1,2}.
\end{equation}
Then $(G_1\times G_2)\bs \St_\VV$ is the Hamiltonian reduction of $\VV$ by $B_1\times B_2$.

Similar to the case of spherical varieties, the relative Springer theory provides a canonical action of \( W_1 \times W_2 \) on the top Borel–Moore homology $\HBM(\St_{\VV}, \mathbb{C})$.
The following theorem describes the decomposition of \( \HBM(\St_{\VV}, \mathbb{C}) \), which will be established in \Cref{section Ortho-symplectic}.
\begin{thm}\label{thm:Htop}
As a representation of \( W_1 \times W_2 \), there is a canonical isomorphism
\[
\HBM(\St_{\VV}, \bC) \cong \bigoplus_{(\cO, \cO_1, \cL_1, \cO_2, \cL_2) \in \cRV} E_{\cO_1, \cL_1} \boxtimes E_{\cO_2, \cL_2}.
\]
\end{thm}

In \Cref{prop:CC W12 eq}, we show that the isomorphism $c: \sfM_{v=1,\bC}\isom \HBM(\St_\VV,\bC)$ in the bottom row of \eqref{osc M diagram} intertwines the $W_1\times W_2$-action on $\sfM_{v=1,\bC}$ (as the specialization of the Hecke algebra action) and the Springer $W_1\times W_2$-action on $\HBM(\St_\VV,\bC)$, up to tensoring with the sign characters of $W_1\times W_2$. The proof of \Cref{thm: Springer for theta} follows from the bottom row of \eqref{osc M diagram}  together with \Cref{thm:Htop}.

\subsection{Future directions}
\subsubsection*{The oscillator $W_1\times W_2$-graph}
In the course of geometrizing the oscillator bimodule $M$, we introduce a Kazhdan-Lusztig type basis for it, and use it to define a cell filtration on $M$ and a $W_1\times W_2$-graph in the sense of \cite{MR0560412} that we call the {\em oscillator $W_1\times W_2$-graph} (see \Cref{Sec gen oscillator bimod} and \Cref{ss: type I W graph examples} for some examples). These combinatorial objects deserve further study.

\subsubsection*{More hyperspherical varieties}
In view of the relative Langlands duality proposed by Ben-Zvi--Sakellaridis--Venkatesh \cite{ben2024relative}, the geometric input of both \Cref{thm: Springer for theta} and \Cref{thm:intro spherical} are examples of hyperspherical $G$-varieties: it is the $G$-variety $T^*X$ in \Cref{thm:intro spherical}, and  the $G_1\times G_2$-variety $\VV$ in \Cref{thm: Springer for theta}. The latter is non-polarizable, and the proof of \Cref{thm: Springer for theta} is significantly more involved. 

Both theorems are proved by constructing a sheaf category that can be viewed as a (constructible) quantization of the $B$-Hamiltonian reduction of a hyperspherical $G$-variety $\mathscr{X}$.  We expect a suitable adaptation of our method to yield constructible quantizations and similar results on Hecke modules for a broader class of hyperspherical varieties.


\subsubsection*{The Koszul duality and relative Langlands duality}
In \cite{ben2024relative}, Ben-Zvi, Sakellaridis and Venkatesh proposed and partially defined a duality between hyperspherical varieties:
\[
G \circlearrowright \mathscr{X} \longleftrightarrow \mathscr{X}^{\vee} \circlearrowleft G^{\vee},
\]
Here, $G^\vee$ denotes the Langlands dual group of $G$. 
One striking instance of this duality is that the hyperspherical variety $\mathscr{X} =\VV$ for $G = \SO(V_1) \times \Sp(V_2)$ with 
$\dim V_1=\dim V_2=2n$ considered in \Cref{subsection moment map}, underlying the type I theta correspondence, is dual to the hyperspherical variety $\mathscr{X}^\vee=T^*(X)$ of $G^\vee=\SO(W_1)\times \SO(W_2)$ with $\dim W_2=2n,\dim W_1=2n+1$ in \Cref{example spherical} (3), underlying the branching problem occurring in the orthogonal Gan-Gross-Prasad conjecture. 

In  \cite{MR4865134}, Finkelberg, Ginzburg, and Travkin  provide a geometric conjecture expressing this duality as an isomorphism of Borel–Moore homology groups: 
\begin{equation}\label{FGT conjecture}
    \HBM(\St_{\mathscr{X}},\bC) \cong \HBM(\St_{\mathscr{X}^\vee},\bC)
\end{equation}
as $W = W^\vee$-modules, where $W$ and $W^\vee$ are the Weyl group of $G$ and $G^\vee$. Going further in the spirit of Koszul duality, they also conjecture an equivalence between suitably defined sheaf categories attached to the $B$-Hamiltonian reduction of $\mathscr{X}$ and the $B^\vee$-Hamiltonian reduction of $\mathscr{X}^\vee$. Moreover, such an equivalence should intertwine the actions of the Hecke categories of $G$ and $G^\vee$ under the monoidal Koszul duality in \cite{MR3003920}. We will attempt to prove their Koszul duality conjecture in the case of type I theta correspondence and the orthogonal GGP.

\subsubsection*{Relative Springer sheaves and generalized Springer correspondence}
In \Cref{intro geom}, we introduced the relative Springer sheaf $\cS_{\hS}$ and its perverse truncation $\cS^\diamond_{\hS}$ for $\hS = T^*(X_{\bC})$, where $X$ is a $G$-spherical variety. The key results—\Cref{eq W-action on relative hom}, \Cref{thm:intro spherical}, and \Cref{Thm decomposition of SthS intro}—show that the multiplicity of the Springer-type IC sheaves $\IC(\cO,\cL)$ (for $(\cO,\cL) \in \Im(\Spr_G)$; see \Cref{Springer map}) in $\cS^\diamond_{\hS}$ equals the multiplicity of the corresponding Springer representation $E_{\cO,\cL}(q)$ in
\[
    M \coloneqq \bC[B(\bF_q) \backslash X(\bF_q)].
\]
In short, the Springer summands of $\cS^\diamond_{\hS}$ can be viewed as an avatar of $M$.

The map in \Cref{Springer map} is generally not surjective.  In \cite{MR0732546}, Lusztig extended the Springer correspondence \Cref{Springer map} to all pairs in $\LSG$, now known as the \emph{generalized Springer correspondence}. In the generalized Springer correspondence, the set $\LSG$ admits a canonical partition into blocks, indexed by irreducible cuspidal local systems on nilpotent orbits in Levi subgroups of $G$. Each block is naturally parametrized by the irreducible representations of a relative Weyl group. 
It is an interesting problem to study other simple summands of the sheaf $\cS^\diamond_{\hS}$ in terms of the generalized Springer correspondence, and look for their representation-theoretic interpretation.


The construction of the relative Springer sheaf extends to $G$-hyperspherical varieties $\hS$ (see \Cref{section relative Springer general}). A particularly interesting case occurs when $\hS = \VV$ attached to a dual pair $G = G_1 \times G_2$ considered in \Cref{subsection moment map}:
\[
    \cS^\diamond_{\VV} \in \Perv_{G_1 \times G_2}(\cN_1 \times \cN_2).
\]
The decomposition of $\cS^\diamond_{\VV}$ into simple perverse sheaves gives rise to a correspondence between simple objects of $\Perv_{G_1}(\cN_1)$ and simple objects of $\Perv_{G_2}(\cN_2)$,
which may be viewed as a geometric avatar of the classical theta correspondence. We plan to study this version of the theta correspondence in  forthcoming work.

\subsection{Organization of the paper}\label{subsec:organization}
Section~\ref{Sec Springer} develops a relative Springer theory, studying the geometry obtained from the Springer resolution by base change. Section~\ref{Sec. geom gen} geometrizes Hecke modules attached to spherical varieties and proves \Cref{thm:intro spherical}. Section~\ref{ex: GL theta} treats type~II theta correspondence as a special case of Section~\ref{Sec. geom gen}, and \Cref{Theta type II Hecke module} refines \Cref{thm:intro spherical} in this setting. Section~\ref{sec Hecke type I} analyzes the Hecke modules arising from type~I theta correspondence: using two polarizations, we construct two models and glue them via a partial Fourier transform. The proof of \Cref{thm: Springer for theta} is then completed in Section~\ref{section proof main thm}.

The paper have five appendices. \Cref{sec:orthogonal Springer} records general results on the Springer correspondence for disconnected groups and makes the correspondence explicit for even orthogonal groups. \Cref{TitsLusztigsec} recalls Tits’ deformation and Lusztig’s homomorphism used in our specialization arguments. \Cref{subsection explicit relavant} gives a combinatorial parametrization of various nilpotent orbits appearing throughout the paper,  based on the work \cite{MR549399}. \Cref{subsection Fourier bijection} determines the bijection of simple perverse sheaves under (partial) Fourier transform arising in the theta correspondence, with input from \cite{MR4865134}. \Cref{Sec. Hotta} revisits Hotta’s local formula \cite{Hotta} for the Springer action of a simple reflection.

\subsection*{Acknowledgment}

The research of C.Q. was supported by the Simons Foundation through Wei~Zhang. 
The research of Z.Y. was partially supported by the Simons Foundation. The research of J.Z. was supported by the Simons Foundation. J.Z would like to thanks Tasho Kaletha and Yiannis Sakellaridis for the interest and encouragement on this work. He also thanks Rui Chen, Nhat Hoang Le, Hao Peng, Yicheng Qin, Zeyu Wang for helpful discussions. 

\section*{Notation and symbols}
 

\begin{longtable}{>{\raggedright\arraybackslash}p{2cm}>{\raggedright\arraybackslash}p{9cm} >{\arraybackslash}p{2cm} }

\toprule
\textbf{Notation} & \textbf{Meaning} & \textbf{Location} \\
\midrule
\endhead
 
\midrule
\multicolumn{3}{r}{\textit{Continued on next page}} \\
\endfoot
 
\bottomrule
\endlastfoot
 
$\cS_\cN$ & Springer sheaf on the nilpotent cone $\cN$ &\eqref{SpringerSheaf} \\
 $\IC(\cO,\cL)$ &  The (perverse) intersection complex associated to a local system $\cL$ on $\cO$ (middle extended to $\ov\cO$)& \eqref{Springer sheaf decomposition}\\
  $E_{\cO,\cL}$ &  The Springer representation associated to a local system $\cL$ on $\cO$ & \eqref{Springer map}\\
$\mathbb D$ & Dualizing sheaf & \\
$\St_\sX$ & Relative Steinberg scheme associated to $\sX$ & \eqref{Relative Steinberg}\\
$\cS_\sX$ & $\sX$-relevant Springer sheaf & \eqref{ReleventSpringerSheaf}\\
$\cS^\diamond_{\hS} $ & Perverse truncation of $\cS_\sX$ & \Cref{defn trancated Springer}\\
$\phi$ & Sheaf to function map & \eqref{phi}\\
$\chi$ & Euler characteristic map & \eqref{chi}\\
$\ch$ & Weight polynomial map & \eqref{ch}\\
$\CC$ & The characteristic cycle map & \Cref{CC gen}\\
$\tcR_{G}(\hS)$ & Relevant pairs associated to $\hS$ &\Cref{defn relevant pair} \\
$\cRL$ & Relevant triples for type II dual pair & \Cref{def relevant GL}\\
 $\cRV$ & Relevant quintuples for type I dual pair & \Cref{def relevant orthsymplectic}\\
$\cP(n)$ & The set of partitions of $n$ & \Cref{defn ab diagram}\\
$\BP(m,n)$& The set of decorated bipartitions of $(m,n)$& \Cref{defn ab diagram}\\
$\RBP(m,n)$& The set of relevant decorated bipartitions of $(m,n)$& \Cref{defn ab diagram}\\
$\cP_{\pm}(n)$ & The set of orthogonal and symplectic partitions of $n$ & \Cref{defn orthogonal symplectic partition}\\
$\OSP(M, N)$ & The set of ortho-symplectic partitions & \Cref{defn orth-symplectic partitions} \\
$\ROSP(M, N)$ & The set of relevant ortho-symplectic partitions & \Cref{defn orth-symplectic partitions} \\
$\PM(m, n)$ & The set of partial matchings & 
\Cref{def partial matching}\\

$\SPM(m, n)$ & The set of signed partial matchings & 
\Cref{def signed partial matching}\\

\end{longtable}

\section{Relative Springer theory}\label{Sec Springer}

Relative Springer theory refers to the study of geometry obtained from the Springer resolution by base change. After proving a general result (\Cref{Thm decomposition of SthS}), we consider three special cases in the remaining subsections, corresponding respectively to the geometry involved in \Cref{Sec. geom gen}, \Cref{ex: GL theta} and \Cref{sec Hecke type I}.

In this section, all schemes are considered over $\bC$ and let $G$ be a connected reductive group defined over $\bC$. All sheaves are constructible sheaves in the analytic topology with $\bC$-coefficients. In particular, $\hBM{*}{-,\bC}$ denotes Borel-Moore homology with $\bC$-coefficients.

\subsection{Generalities on relative Springer theory}\label{section relative Springer general}
We begin by recalling the classical Springer theory. We retain the notation introduced at the end of \Cref{intro geom}, omitting the subscript \( \bC \) for simplicity. In particular, the Springer sheaf \( \cS_{\cN} \) is defined as in \Cref{SpringerSheaf}.

There is a canonical isomorphism (see \cite[Theorems~2 and~3]{BM})
\begin{equation}\label{End Springer}
    \End_{D^b_{G}(\cN)}(\cS_{\cN}) \cong \bC[W],
\end{equation}
\trivial[h]{Am I right that this still true over $\bQ$ and the decomposition \Cref{Springer sheaf decomposition} also holds where all Springer sheaves appears. The only problem is the decomposition of $ \hBM{\TOP}{\cB_e,\bQ}$ may have some $\IC(\cO,\cL)$ glued together due to the problem that the representation of $A_e$ may not defined over $\bQ$. All the Springer sheaf are defined over $\bQ$, where means the correspondence representation of $A_e$ is defined over $\bQ$.
}
which endows \( \cS_{\cN} \) with an action of the Weyl group \( W \). We normalize this action so that the space of \( W \)-invariants in \( \cS_{\cN} \) is the shifted constant sheaf \( \bC_{\cN}[\dim \cN] \). The Springer sheaf \( \cS_{\cN} \) is semisimple and decomposes as
\begin{equation}\label{Springer sheaf decomposition}
    \cS_{\cN} \cong \bigoplus_{(\cO,\cL)\in \LSG} \IC(\cO,\cL) \boxtimes V_{\cO,\cL},
\end{equation}
where $\LSG$ is defined in \eqref{eq local system}, $\IC(\cO,\cL)$ is the associated simple perverse sheaf and $V_{\cO,\cL}$ is the multiplicity space. The isomorphism \eqref{End Springer} implies that each \( V_{\cO,\cL} \) is an irreducible representation of \( W \), and that every irreducible representation of \( W \) arises in this way. Let $E_{\cO,\cL}=V_{\cO,\cL}^\vee$ endowed with the dual $W$-action, the map $E_{\cO,\cL}\mapsto (\cO,\cL)$ establishes the Springer correspondence \( \Spr_G \) defined in \Cref{Springer map}.

We now explain how the representation \( V_{\cO,\cL}\cong E_{\cO,\cL}^\vee \) relates to the top Borel--Moore homology of the Springer fiber. Fix a point \( e \in \cO \), and let \( \cB_e \coloneqq \pi_{\cN}^{-1}(e)\) denote the Springer fiber over \( e \). Let \( A_e := \pi_0(C_G(e)) \) be the component group. Then an irreducible \( G \)-equivariant local system \( \cL \) on \( \cO \) corresponds to an irreducible representation \( \chi \in \Irr_{\bC}(A_e) \).

Taking the costalk of both sides of \eqref{Springer sheaf decomposition} at the point \( e \in \cO \), and then taking cohomology in the appropriate degree, yields an isomorphism of \( A_e \times W \)-modules:
\[
    \hBM{\TOP}{\cB_e,\bC}\cong   \bigoplus_{\chi \in A_e} \chi \boxtimes E^\vee_{\cO,\cL},
\]
where $(\cO,\cL)$ corresponds to $\chi$. Then the $W$-module $E^\vee_{\cO,\cL}$ is
\begin{equation}\label{eq Springer representation}
E^\vee_{\cO,\cL}\cong E^\vee_{e,\chi}:=\Hom_{A_{e}}(\chi, \HBM(\cB_{e},\bC)).
\end{equation}
This realizes the Springer representation \( E^\vee_{\cO,\cL} \) as a multiplicity space inside the top Borel--Moore homology of the Springer fiber \( \cB_e \).

\medskip 
Now we introduce the set up for the \emph{relative Springer theory}. Let $\hS$ be a $G$-scheme equipped with a $G$-equivariant map $\mu_{\hS}:\hS\rightarrow \frak g^*$. 
\begin{defn}\label{defn Steiberg}
Define the \em{$\hS$-Steinberg scheme} $\St_{\hS}\coloneqq \hS\times_{\frak g^*} \wt\cN$. More precisely, we have 
\begin{equation}\label{Relative Steinberg}
\St_{\hS} =\set{(x,B)\in \hS\times \cB|\mu_{\hS}(x)\in \Lie(B)^\bot}.
\end{equation}
and following diagram with the square Cartesian.  
 \begin{equation}\label{St hS Cart}
\begin{tikzcd}
\St_\sX \arrow[d, "\Pi"'] \arrow[r, "\nu_{\sX}"] & \tcN \arrow[d, "\pi"] \ar[rd,"\pi_\cN"]\\
\sX \arrow[r, "\mu_{\sX}"] & \fgg^* & \ar[l,"i_\cN"'] \cN \\
\end{tikzcd} 
    \end{equation}
\end{defn}

From this point forward, we fix a Borel subgroup $B$ with Lie algebra $\frb$ and define 
\begin{equation}\label{Lambdafrb}
    \L_{\hS}\coloneqq \mu_{\hS}^{-1}(\frb^\bot).
\end{equation}
We identify $\cB$ with $G/B$.
Then the map $(g,x) \mapsto (gB,g x)$ gives an identification 
\begin{equation}\label{eqSteinberg}
G\times^{B} \L_{\hS} = \St_{\hS}
\end{equation} 
This identification induces a natural isomorphism 
\[
\HBM(\L_{\hS},\bC)\cong \HBM(\St_{\hS},\bC).\]

Define the \emph{$\hS$-relevant Springer sheaf} 
\begin{equation}\label{ReleventSpringerSheaf}
\cS_{\hS}\coloneqq i_\cN^!\,  \mu_{\hS *}\mathbb D_{\hS}[-2\dim \St_{\hS}+\dim \wt \cN].
\end{equation}

\begin{lemma}\label{W-action on relative hom}
    There is a natural isomorphism 
    \begin{equation}\label{eq W-action on relative hom}
          \Hom_{D^b_G(\cN)}(\cS_{\cN},\cS_{\hS}) \cong \HBM(\St_{\hS},\bC) .
    \end{equation}
\end{lemma}
\begin{proof}
This ia a standard computation using the formalism of the six functors as below: 
   \[
   \begin{split}
    \Hom_{D^b_G(\cN)}(\cS_{\cN},\cS_{\hS})
    & \cong  
    \Hom_{D^b_G(\hS)}({\mu_{\hS }^* i_\cN}_!\cS_{\cN},  \mathbb D_{\hS}[-2\dim \St_{\hS}+\dim \wt \cN]) \\
    & \cong  
    \Hom_{D^b_G(\hS)}(\Pi_*\bC_{\St_\hS}[\dim \tcN],  \mathbb D_{\hS}[-2\dim \St_{\hS}+\dim \tcN]) \\
    &\cong  \Hom_{D^b_G(\St_\hS)}(\bC_{\St_\hS}, \bD_{\St_\hS}[-2\dim \St_{\hS}]) \\
    & = \HBM(\St_{\hS},\bC).
    \end{split}
   \]
\end{proof}   
Combining \Cref{End Springer} and \Cref{eq W-action on relative hom}, we obtain a natural \( W \)-action on \( \HBM(\St_{\hS}, \bC) \). Alternatively, one may describe the \( W \)-action on \( \hBM{*}{\St_{\hS}, \bC} \) as follows. Transport the \( W \)-action from the Springer sheaf \( \cS_{\cN} \) via base change using the Cartesian diagram \eqref{St hS Cart}. This yields a natural \( W \)-action on 
\begin{equation}\label{cons:W on hBM}
    \Pi_{*} \bC_{\St_{\hS}}[\dim \widetilde{\cN}] \cong \mu_{\hS}^{*} {i_\cN}_! \cS_{\cN}.
\end{equation}
Since \( \Pi \) is proper, this induces a natural \( W \)-action on the compactly supported cohomology \( \chM{*}{\St_{\hS}, \bC} \), and hence a dual \( W \)-action on the Borel–Moore homology \( \hBM{*}{\St_{\hS}, \bC} \).

We now explain how the decomposition of \( \HBM(\St_{\hS}, \bC) \) as a \( W \)-representation relates to the Borel--Moore homology of certain fibers of the moment map, which we call \emph{hyperspherical fibers}.

Let \( e \in \cN \) be a nilpotent element. Define the hyperspherical fiber at \( e \) to be
\begin{equation}\label{Hyperspherical fiber}
    \sX_{e} \coloneqq \mu_{\hS}^{-1}(e).
\end{equation}
The top Borel--Moore homology \( \HBM(\sX_{e}, \bC) \) carries a natural action of the component group \( A_e := \pi_0(C_G(e)) \). For each \( \chi \in \Irr(A_e) \), define the \emph{multiplicity space}
\begin{equation}\label{eq hyperspherical fibre multiplicty}
    V_{e,\chi}(\hS) \coloneqq \Hom_{A_e}(\chi, \HBM(\sX_{e}, \bC)).
\end{equation}
In the above definition, the vector space $V_{e,\chi}(\hS)$ only depends only on the corresponding  $(\cO,\cL)\in\widehat {\cN}_{G}$; we henceforth denote it by $V_{\cO,\cL}(\hS)$.
\trivial[h]{
Suppose $ge = e'$, 
then $\Hom_{A_e}(\cL_e,\HBM(\hS_e,\bQ))\rightarrow \Hom_{A_{e'}}(\cL_{e'},\HBM(\hS_{e'},\bQ))$
is given by $T\mapsto  g T g^{-1}$. 
Note that if $e'=e$, i.e. $g\in C_G(e)$ then $gTg^{-1}=T$, i.e. the map is the identity map. 
}
\begin{lemma}\label{lem dimension relevant}\label{Defi relevant orbit}
    Let $\cO\subset \cN$ be a nilpotent orbit. Then
    \begin{equation}
        \label{leq relevant dim}
        \dim \mu_{\hS}^{-1}(\cO)\leq  \dim \St_{\hS}+\half \dim \cO -\half \dim \wt \cN.
    \end{equation}
\end{lemma}
\begin{proof}
 For $e\in \cO$, let $\cB_e:= \pi^{-1}(e)$ denote the Springer fiber. By Steinberg's theorem \cite{St}, we have $\dim \cB_e=\frac{1}{2} (\dim \wt \cN-\dim \cO)$. By the Cartesian diagram \eqref{St hS Cart}, for $e\in \cO$, we have
    \begin{equation}\label{muhS fiber dim}
        \dim \mu_{\hS}^{-1}(\cO)+\dim \cB_e=\dim \Pi^{-1}\mu_{\hS}^{-1}(\cO)\leq \dim \St_{\hS} 
    \end{equation}
    and the lemma follows. 
\end{proof}

\begin{defn}\label{defn relevant pair}
\begin{enumerate}
    \item   A nilpotent orbit $\cO$ in $\cN$ is called \emph{$\hS$-relevant} if the following equality holds    \begin{equation}\label{eq hS relevant}
        \dim \mu_{\hS}^{-1}(\cO) =  \dim \St_{\hS}+\half \dim \cO -\half \dim \wt \cN.
    \end{equation}
Let $\cR_{G}(\sX)$ be the set of relevant nilpotent $G$-orbits.
\item  A pair $(\cO,\cL)\in \LSG$ is called a \em{$\hS$-relevant pair} if $\cO\in \cR_{G}(\sX)$ and 
   $V_{\cO,\cL}(\hS)\neq 0$. We denote the set of $\hS$-relevant pairs be $\tcR_{G}(\hS)$. 
\end{enumerate}
\end{defn}

\begin{thm}\label{Thm decomposition of SthS}
  We have a decomposition of $\HBM(\St_{\hS},\bC)$ as a $W$-module 
  \[
 \HBM(\St_{\hS},\bC)\cong \bigoplus_{(\cO,\cL)\in \tcR_{G}(\hS)} E_{\cO,\cL}\otimes V_{\cO,\cL}(\hS).
  \]
\end{thm}
The remainder of this subsection is devoted to the proof of \Cref{Thm decomposition of SthS}.

\begin{lemma}\label{lem Half perversity}
    We have 
    \begin{equation*}\label{Half perversity}
           \cS_{\hS}\in {}^p D^b_G(\cN)^{\geq 0}. 
    \end{equation*}
\end{lemma}
\begin{proof}

  Since $\cS_{\hS}$ is $G$-equivariant, 
    it suffices to prove that for each nilpotent orbit $\cO$ in $\cN$ and any $e\in \cO$ , the cohomology $H^k(i_e^* i_{\cO}^!\cS_{\hS})$ is non-zero only if the $k \geq - \dim \cO$ (see \cite[Exercises 3.1.5]{MR4337423}).
 We have the following Cartesian diagram:
\[
\label{Base change cO}
  \begin{tikzcd}
\sX_{e}\ar[r,"i'_e"]\ar[d,"\mu_e"] &\mu_\hS^{-1}(\cO) \ar[rr,"i'"] \ar[d,"\mu'"] &  &\hS \ar[d,"\mu_\hS"]\\
\set{e} \ar[r,"i_e"]& \cO \ar[r,"i_\cO"] &  \cN \ar[r,"i_\cN"]&   \frak g^* \\ 
 \end{tikzcd} 
\]   
Since $\cO$ is smooth, applying base change yields
   \begin{equation}\label{eq cShScO}
   \begin{split}
     i_e^* i_{\cO}^!\cS_{\hS} & = i_e^![2\dim \cO] (i_{\cO}^! i_{\cN}^! {\mu_\sX}_* \bD_{\hS}[-2\dim \St_{\hS}+\dim \wt \cN])\\ 
    & = {\mu_e}_* \mathbb D_{\hS_{e}}[2\dim \cO -2\dim \St_{\hS}+\dim \wt \cN]. 
    \end{split}
  \end{equation}
  
Applying \eqref{leq relevant dim},    we see that 
    $H^k(i^*_ei^!_\cO \cS_\hS) \neq 0$ only if
    \[
    \begin{split}
        k &\geq - 2\dim \hS_e - 2 \dim \cO  + 2\dim \St_\hS - \dim \tcN 
        \geq - \dim \cO
    \end{split} 
    \] 
This finishes the proof. 

\end{proof}
\begin{defn}\label{defn trancated Springer}
    Define the \emph{truncated Springer sheaf} by
\[
    \cS^\diamond_{\hS} \coloneqq \ptau_{\leq 0} \cS_{\hS} \in \Perv_G(\cN).
\]
\end{defn}
Decompose  $ \cS^\diamond_{\hS}$ into simple perverse sheaves gives  
    \begin{equation}\label{perverse dec}
  \cS^\diamond_{\hS}= \bigoplus_{(\cO,\cL)\in \LSG} \IC(\cO,\cL)\otimes M_{\cO,\cL}
    \end{equation}
    where $M_{\cO,\cL}$ is the multiplicity space. Hence
\begin{equation}\label{eq iso trucation}
        \begin{split}
  \Hom_{D^b_G(\cN)}(\cS_{\cN},\cS_{\hS}) &\cong   \Hom_{\Perv_G(\cN)}(\cS_{\cN},\cS^\diamond_{\hS})
  \cong \bigoplus_{(\cO,\cL)\in \LSG} E_{\cO,\cL} \otimes M_{\cO,\cL}.
    \end{split}
\end{equation}
\trivial[h]{
From the distinguished triangle 
\begin{equation}\label{distinguished triangle}
  \ptau_{\leq 0} \cS_{\hS} \longrightarrow  \cS_{\hS} \longrightarrow \ptau_{\geq 1} \cS_{\hS},  
\end{equation}
We get 
\[
\begin{split}
&\cdots \to \Hom_{D^b_G(\cN)}(\cS_{\cN}, \tau_{\geq 1}\cS_{\hS}[-1]) 
\longrightarrow 
\Hom_{D^b_G(\cN)}\big(\cS_{\cN}, \tau_{\leq 0}\cS_{\hS}\big)\\
\longrightarrow 
&\Hom_{D^b_G(\cN)}\big(\cS_{\cN}, \cS_{\hS}\big)
\to \Hom_{D^b_G(\cN)}(\cS_{\cN}, \tau_{\geq 1}\cS_{\hS,\cN}) 
\longrightarrow \cdots
\end{split}
\]
Note that $\Hom(A, B)=0$  for $A \in D^{\leq n}$ and $B\in D^{\geq n+1}$.
$\tau_{\geq n}[-k] = \tau_{\geq 0}[-n-k]=\tau_{\geq n+k}$.
So that the last term and first terms are zero.  
}
\begin{lemma}\label{Decomposition relevant springer sheaf}
    The space $M_{\cO,\cL}\neq 0$ if and only if $(\cO,\cL)\in \tcR_G(\hS)$ and in that case, 
    \[
    M_{\cO,\cL} = V_{\cO,\cL}(\hS). 
    \]
\end{lemma}
\begin{proof}
 Take a nilpotent orbit $\cO$ and an element $e\in \cO$.  Applying $H^{\dim \cO}(i_e^! i_\cO^!(\underline{\ \ \ } ) )$ to the right hand side of  
\eqref{perverse dec} yields $\cL_e \otimes M_{\cO,\cL}$.
\trivial[h]{
   Note $H^{\dim \cO}(i_e^! \IC(\cO,\cL)) = H^{\dim \cO}(i_e^* [-2\dim \cO] \cL[\dim \cO]) = \cL_e$.   

   If $e \notin \overline{\cO'}$, 
   then $H^{\dim \cO}(i_e^! \IC(\cO',\cL')) = 0$. 
   
   If $e \notin \overline{\cO'} - \cO'$,
   then $i_{\cO}^! \IC(\cO',\cL')\in D^b_{locf}(\cO)^{\geq - \dim \cO +1}$, 
    then $i_e^! i_{\cO}^! \IC(\cO',\cL')
    =  i_e^* [-2\dim \cO] i_{\cO}^! \IC(\cO',\cL')\in D^b(\set{e})^{\geq \dim\cO +1}$.
    Therefore,  
   $H^{\dim \cO}(i_e^! i_{\cO}^! \IC(\cO',\cL')) = 0$. 

Apply $i_\cO^!$ to the distinguished triangle 
we have long exact sequence 
\[
i^!_\cO \ptau_{\geq 2} S \to i^!_\cO \ptau_{\leq 0} S  \to i^!_\cO S \to i^!_\cO \ptau_{\geq 1} S 
\]
Now $i^!_\cO \ptau_{\geq k} S \in D(\cO)^{\geq - \dim \cO +k}$. 
Apply $i^!_e = i^*_e[-2\dim \cO]$, 
one have  $i^!_e i^!_\cO \ptau_{\geq k} S  \in D^b(\set{e})^{\geq \dim \cO +k}$.
So taking $H^{\dim \cO}$ yields the isomorphism $H^{\dim \cO}(i_e^! i_\cO^! \ptau_{\leq 0}\cS_{\hS} )\cong    H^{\dim \cO}(i_e^!  i_{\cO}^! \cS_{\hS} )$. 
} 
\begin{equation}\label{HBMXe}
\begin{split}
H^{\dim \cO}(i_e^! i_\cO^! \cS^\diamond_{\hS} )\cong    H^{\dim \cO}(i_e^!  i_{\cO}^! \cS_{\hS} )\cong \hBM{2\dim \St_{\hS}-\dim \wt \cN-\dim \cO}{\hS_e,\bQ}.
\end{split}
\end{equation}
By \Cref{lem dimension relevant},  
\[
\dim \hS_e\leq \dim  \St_{\hS}-\half(\dim \wt \cN+\dim \cO).
\]
and the equality holds if and only if $\cO$ is $\hS$-relevant. Compare the $\cL_e$-isotypic component finish the proof.  
\end{proof} 

\begin{proof}[Proof of \Cref{Thm decomposition of SthS}]
    This follows by combining \Cref{W-action on relative hom}, \Cref{eq iso trucation} and \Cref{Decomposition relevant springer sheaf}.
\end{proof}

\subsection{Relative Springer theory for spherical varieties}\label{sec Relevant Orbits}
The goal of this subsection is to prove \Cref{Thm decomposition of SthS intro} by applying \Cref{Thm decomposition of SthS}. We shall work in a slightly more general setting than that of \Cref{Thm decomposition of SthS intro}. Namely, we assume that \( X \) is a \( G \)-variety with finitely many \( B \)-orbits. We do not require \( X \) to be smooth, but we assume that \( X \) admits a \( G \)-equivariant embedding into a smooth \( G \)-variety \( Y \). This generality will be needed in \Cref{sec Hecke type I}.

Let $\hSXY\coloneqq (T^*Y)|_X$  be the restriction of $T^*Y$ to $X$, and 
$\mu_\hSXY: \hSXY \to \frg^*$ be the restriction of the moment map $T^*Y\to \fgg^*$ for the $G$-action on $T^*Y$.
By the construction in \Cref{section relative Springer general}, we have 
\begin{equation}\label{LhSXY}
   \L_{\hSXY}=\bigsqcup_{Z\in \underline{B\backslash X}} T_Z^*(Y)\subset T^*(Y). 
\end{equation}
Since  $\underline{B\backslash X}$ is finite, $\Lambda_\hSXY$ is a finite union of Lagrangian sub-varieties of $T^*Y$ and by \eqref{eqSteinberg}  
\begin{equation}\label{dim SthSX}
        \dim \St_{\hSXY}=\dim Y+ \half \dim \wt \cN.
\end{equation}
Applying \Cref{defn relevant pair} in this case, we see that a nilpotent orbit is relevant if and only if 
\begin{equation}\label{eq dim relevant orbit spherical}
\dim \mu_\hSXY^{-1}(\cO)= \dim Y + \half \dim \cO.
\end{equation}
%
\begin{example}
Suppose  $H$ be a spherical subgroup of $G$ and $Y = X = G/H$. We have 
 $\hSX\coloneqq T^*(X)\cong G\times^H \fhh^{\perp}$ and the moment map $\mu_{\hS}: \hSX=G\times^H \frak h^{\perp} \to \frg^*$ is given by 
\[
\mu_{\hS}(g,f)=\Ad^*(g) f,\quad \mbox{for all}\, g\in G, f\in \frak h^\perp
\]
where $\Ad^*$ is the adjoint action of $G$ on $\fgg^*$. 
One compute that $\mu_{\hS}^{-1}(\cO)\cong G\times^H(\cO\cap \fhh^\perp)$, hence deduce that  
$\cO$ is relevant if and only if
\begin{equation}\label{relevant dimen homogenious}
    \dim (\cO\cap \frak h^\perp) =\half \dim \cO. 
\end{equation}
\end{example}
\trivial[h]{
we deduce that \(\cO\in \cR_G(\hSX)\) (see \Cref{def Relevant orbit}) if and only if 
\[
\dim \mu^{-1}(\cO)=\dim X+\half\dim\cO.
\]
}

Let $\{X_i\}$ be the set of $G$-orbits of $X$.
We define 
\[
\cR_G(X)=\bigcup_{i} \cR_G(T^*X_i)  \quad \text{and}\quad  \widehat \cR_G (X)=\bigcup_{i} \widehat \cR_G(T^*X_i) \subset\LSG.
\]
where $\cR_G(T^*X_i)$ and $ \widehat \cR_G(T^*X_i)$ are defined by applying  \Cref{defn relevant pair} to the Hamiltonian $G$ action on $T^*X_i$. For each $(\cO,\cL)\in \widehat \cR (X)$, we also define 
\begin{equation}\label{VcOcLX}
    V_{\cO,\cL}(X)=\bigoplus_{i}V_{\cO,\cL}(T^*X_i). 
\end{equation}
where $V_{\cO,\cL}(T^*X_i)$ is defined according to \Cref{eq hyperspherical fibre multiplicty}. Recall that we also defined 
$\cR_G(\hSXY), \widehat \cR_G(\hSXY),  V_{\cO,\cL}(\hS)$ in \Cref{eq hyperspherical fibre multiplicty} and \Cref{defn relevant pair}.

\begin{lemma}\label{lem relevant decomoposition into pieces}
    We have 
    $    \cR_G(X)= \cR_G(\hSXY)$,$\widehat \cR_G (X)=\widehat \cR_G(\hSXY)$ and 
    $V_{\cO,\cL}(X) \cong V_{\cO,\cL}(\hS)$ for each $(\cO,\cL)\in \widehat \cR_G (X)=\widehat \cR_G(\hSXY)$. 
\end{lemma}
\begin{proof}
    Write $\hS_i := (T^*Y)|_{X_i}$ for each $i$. 
For a nilpotent orbit $\cO$ and $e\in \cO$, we have 
\begin{equation}\label{muYX decomposition}
  (\mu_{\hS})^{-1}(e)=\bigsqcup_{i} (\mu_{\hS_i})^{-1}(e). 
\end{equation}
So \Cref{eq hS relevant} holds if and only if 
\begin{equation}\label{eq dim relevant fibre i}
    \dim (\mu_{\hS_i})^{-1}(e)= \dim Y-\half \dim \cO. 
\end{equation}
for some $i$. 

Since $X_i$ is a $G$-orbit, the natural map $\pi_{X_i}: (T^*Y)|_{X_i}\rightarrow T^*X_i $ is the fibration whose fibers are 
affine spaces modeled by $(T_xY/T_xX_i)^*$ with $x\in X_i$.
On the other hand, $\mu_{\hS_i}= \mu_{T^*X_i}\circ \pi_{X_i}$.
Pulling back via the natural inclusion $\mu_{T^*X_i}^{-1}(e)\to T^*X_i$ induces the fibration 
\[
p_{X_i,Y}: (\mu_{\hS_i})^{-1}(e) \longrightarrow (\mu_{T^*X_i})^{-1}(e).
\]
Since $(T_xY/T_xX_i)^*$ has dimension $\dim Y -\dim X_i$, it implies that \Cref{eq dim relevant fibre i} is equivalent to  
\begin{equation}\label{dim relevant single orbit}
    \dim \mu_{T^*X_i}^{-1} (e)=\dim X_i-\half \dim \cO,
\end{equation}
 and so $\cO$ being relevant is independent of $Y$. Since the fibers of $p_{X_i,Y}$ are affine spaces, $p_{X_i,Y}$ induces the following $A_e$-equivariant isomorphism 
 \[p_{X_i,Y}^*: \hBM{2\dim X_i - \dim \cO }{\mu_{T^*X_i}^{-1}(e),\bC}\cong \hBM{2\dim Y - \dim \cO}{(\mu_{\hS_i})^{-1}(e),\bC}.
 \]
 In view of \Cref{muYX decomposition}, the lemma follows.
\end{proof}
The above lemma shows that the definition of relevant pairs $\widehat \cR_G(\hSXY)$ is independent of the embedding into $Y$.  

  \begin{thm}\label{Thm decomposition of SthS intro more}
There is a decomposition of $\HBM(\St_{\hS}, \bC)$ as a $W$-representation:
\[
    \HBM(\St_{\hS}, \bQ) \cong \bigoplus_{(\cO,\cL) \in \widehat \cR_{G}(X)} E_{\cO,\cL} \otimes V_{\cO,\cL}(X).
\]
\end{thm}
\begin{proof}
    The theorem follows by applying \Cref{Thm decomposition of SthS} to the case \( \hSXY= (T^*Y)|_X \), together with the result of \Cref{lem relevant decomoposition into pieces}.
\end{proof}

\begin{proof}[Proof of \Cref{Thm decomposition of SthS intro}]
This is a special case of \Cref{Thm decomposition of SthS intro more}, applied to the situation where \( X \) is smooth and we take $Y=X$.
\end{proof}

\subsection{Relative Springer theory for the Steinberg variety arising from type II theta correspondence}\label{section type II}

Let $L_1$ and $L_2$ be finite-dimensional vector spaces over $\bC$ of dimensions $m$ and $n$, respectively. Set $\ov G_1 = \GL(L_1)$ and $\ov G_2 = \GL(L_2)$, with Lie algebras $\ov{\mathfrak g}_1$ and $\ov{\mathfrak g}_2$ and nilcones $\ov {\cN}_1$ and $\ov \cN_2$. Let $\ov W_1$ and $\ov W_2$ be the Weyl groups of $\ov G_1$ and $\ov G_2$, respectively. Define $\ov \bL_1 \coloneqq \Hom(L_1, L_2)$, $\ov \bL_2 \coloneqq \Hom(L_2, L_1)$, and $\ov \bL \coloneqq \ov \bL_1 \oplus \ov \bL_2$.

The trace pairing
\begin{equation}\label{bL1 dual complex}
\begin{split}
    \langle\,,\,\rangle: \ov \bL_1 \otimes \ov \bL_2 &\longrightarrow \bC \\
    (T_1, T_2) &\longmapsto \tr(T_1 T_2)
\end{split}
\end{equation}
is a perfect pairing between $\ov \bL_1$ and $\ov \bL_2$, inducing identifications of $\ov \bL$ with $T^*\ov \bL_1$ and $T^*\ov \bL_2$. The group $\ov G_1 \times \ov G_2$ acts naturally on both $\ov \bL_1$ and $\ov \bL_2$, making them spherical $\ov G_1 \times \ov G_2$-varieties (cf. \Cref{PM vs orbits}). This action naturally extends to $\ov \bL$ and is Hamiltonian, with moment map
\begin{equation}\label{moment map type II}
\mu_{\ov\bL} = (\mu_1, \mu_2): \ov \bL = \ov \bL_1 \oplus \ov \bL_2 \longrightarrow \ov{\mathfrak g}_1^* \times \ov{\mathfrak g}_2^*
\end{equation}
given by
\[
\mu_1(T_1, T_2)(A_1) = \inn{A_1 T_1}{ T_2}, \quad \mu_2(T_1, T_2)(A_2) = \inn{T_1}{A_2 T_2} \quad \forall A_1\in \fgg_1, A_2\in \fgg_2.
\]
Identifying $\ov{\fgg}_i^*$ with $\ov{\fgg}_i$ via the trace form, we regard $\mu$ as a map from $\ov \bL$ to $\ov{\fgg}_1 \times \ov{\fgg}_2$ with
\[
\mu_1(T_1, T_2) = T_2 T_1 \in \ov{\fgg}_1, \quad \mu_2(T_1, T_2) = T_1 T_2 \in \ov{\fgg}_2.
\]

View $\ov \bL$ as $T^*\ov\bL_1$. It fits into the framework of \Cref{sec Relevant Orbits}, and \eqref{dim SthSX} reads:
\begin{equation}\label{Lagrangian Steiberg bL dimension}
    \dim \St_{\ov \bL} = \tfrac{1}{2}(\dim \btcN_1 + \dim \btcN_2 + \dim \ov \bL).
\end{equation}
Note that the component groups $A_{e_1}$ and $A_{e_2}$ are trivial for every nilpotent orbit $\cO_1 \times \cO_2$ and $(e_1,e_2)\in \cO_1\times \cO_2$. Hence, each such orbit admits a unique irreducible local system, namely the trivial one. Consequently, the natural projection gives an identification $\widehat \cR_{\ov G_1 \times \ov G_2}(\ov \bL) = \cR_{\ov G_1 \times \ov G_2}(\ov \bL)$. For convenience, we write:
\[
V_{\cO_1 \times \cO_2}(\ov \bL) := V_{\cO_1 \times \cO_2, \cL_1 \otimes \cL_2}(\ov \bL) \quad \text{and} \quad E_{\cO_i} := E_{\cO_i, \cL_i}
\]
where $\cL_i$ denotes the trivial local system on $\cO_i$ for $i = 1, 2$. Then \Cref{Thm decomposition of SthS} gives the decomposition:
\begin{equation}\label{eq hBMV type II first}
    \hBM{\TOP}{\St_{\ov \bL}, \bC} \cong \bigoplus_{(\cO_1, \cO_2) \in \cR_{\ov G_1 \times \ov G_2}(\ov \bL)} E_{\cO_1} \otimes E_{\cO_2} \otimes V_{\cO_1\times \cO_2}(\ov \bL).
\end{equation}
Next we refine this decomposition by computing each multiplicity space $V_{\cO_1, \cO_2}(\ov \bL)$ explicitly. 

Define the nilcone of $\ov \bL$ as follows:
\begin{equation}\label{Def: nilcone type II}
    \cN_{\ov \bL} \coloneqq \mu^{-1}(\ov \cN_1 \times \ov \cN_2).
\end{equation}
It is well known that $\cN_{\ov \bL} = \mu_i^{-1}(\ov \cN_i)$ for $i=1$ or $2$, i.e., an element $(T_1, T_2) \in \ov \bL$ lies in $\cN_{\ov \bL}$ if and only if $T_1 T_2$ (or equivalently $T_2 T_1$) is nilpotent. We refer to such elements as \emph{nilpotent pairs}. The number of $\ov G_1 \times \ov G_2$-orbits in $\cN_{\ov \bL}$ are classified by the ``ab'' diagrams, which are finite in number (see \Cref{prop GL orbits}). 
\begin{defn}\label{def relevant GL}
   An $\ov G_1\times \ov G_2$-orbit $\cO \subset \cN_{\ov \bL}$ is called \emph{relevant} if it satisfies the condition:
\begin{equation}\label{relevant in cnbL}
    \dim \cO = \tfrac{1}{2} \left( \dim \mu_1(\cO) + \dim \mu_2(\cO) + \dim \ov \bL \right).
\end{equation}
We refer to $(\cO, \mu_1(\cO), \mu_2(\cO))$ as a \emph{relevant triple} if $\cO$ is relevant.
Let $\cR_{\ov \bL}$ denote the set of relevant orbits in $\ov \bL$, and let $\cRL$ denote the set of relevant triples.
\end{defn}
By \Cref{eq dim relevant orbit spherical}, a pair $(\cO_1, \cO_2)$ lies in $\cR_{\ov G_1 \times \ov G_2}(\ov \bL)$ if and only if
\begin{equation}\label{Dimension relevant}
    \dim \mu^{-1}(\cO_1 \times \cO_2) = \tfrac{1}{2}(\dim \ov \bL + \dim \cO_1 + \dim \cO_2).
\end{equation}
In view of \Cref{relevant in cnbL} and \Cref{Dimension relevant}, a triple $(\cO,\cO_1, \cO_2)$ is relevant if and only if $\cO$ attains the maximal possible dimension $\mu^{-1}(\cO_1\times \cO_2)$. 

\begin{thm}\label{thm:Htop type II}
As a $\ov W_1 \times \ov W_2$-module, there is an isomorphism:
\[
\HBM(\St_{\ov \bL}, \bC) \cong \bigoplus_{(\cO,\cO_1, \cO_2) \in \cRL} E_{\cO_1} \boxtimes E_{\cO_2}.
\]
\end{thm}
\begin{proof}
Let $\cO_1 \times \cO_2 \in \cR_{\ov G_1 \times \ov G_2}(\ov \bL)$. Then $\mu$ restricts to a fibration $\mu^{-1}(\cO_1 \times \cO_2) \to \cO_1 \times \cO_2$. Fix a point $(e_1, e_2) \in \cO_1 \times \cO_2$. The evaluation of \eqref{eq hyperspherical fibre multiplicty} yields:
\begin{equation}\label{eq VOO type II}
\begin{split}
    V_{\cO_1 \times \cO_2}(\ov\bL) &\cong \HBM(\mu^{-1}(e_1, e_2), \bC) 
    \cong \HBM(\mu^{-1}(\cO_1 \times \cO_2), \bC) \\
    &\cong \bC\left[\left\{\ov \cO \;\middle|\; (\cO, \cO_1, \cO_2) \in \cRL\right\}\right].
\end{split}
\end{equation}
The last isomorphism holds because $\underline{(\ov G_1 \times \ov G_2)\backslash \cN_{\ov \bL}}$ is finite. Hence, the union $\bigsqcup_{(\cO, \cO_1, \cO_2) \in \cRL} \cO$ is open and dense in $\mu^{-1}(\cO_1 \times \cO_2)$. Therefore, the cycles generated by $\left\{\ov \cO \mid (\cO, \cO_1, \cO_2) \in \cRL\right\}$ form a basis of $\HBM(\mu^{-1}(\cO_1 \times \cO_2), \bC)$. Substituting \eqref{eq VOO type II} into \Cref{eq hBMV type II first} completes the proof.
\end{proof}
Next, we provide a combinatorial description of $\cRL$ based on results of Kraft-Procesi \cite[\S 5.3]{MR549399}.

\def\mudg{\mu^{\dagger}}

\begin{defn}\label{defn ab diagram}
For each non-negative integer $k$, let $\mathcal{P}(k)$ denote the set of all partitions of $k$. 
Practically, we view a partition as an equivalence class of multisets of non-negative integers under the relation $\gamma \sim \gamma \cup \set{0}$.    
\begin{enumerate}[wide]
    \item 
A \emph{decorated bipartition} of $(m,n)$ is a multiset of decorated pairs of non-negative integers:
    \begin{equation}\label{eq DBP}
        \gamma = \{ (x_1, y_1)^{\epsilon_1}, \ldots, (x_r, y_r)^{\epsilon_r} \}
    \end{equation}
    where each pair $(x_i, y_i)$ is either undecorated, in which case we set $\ep_i=\varnothing$, or decorated with a sign $\ep_i=+$ or $-$ and satisfies: 
    \begin{itemize}
        \item Each $(x_i, y_i)^{\epsilon_i}$ is one of the four types: $(k, k-1)$, $(k-1, k)$, $(k, k)^+$, or $(k, k)^-$ for some $k \geq 1$;
        \item $\sum_{i=1}^r x_i = m$ and $\sum_{i=1}^r y_i = n$.
    \end{itemize}
    Let $\BP(m,n)$ denote the set of all decorated bipartitions of $(m,n)$. 
\item 
    A decorated bipartition $\gamma \in \BP(m,n)$ is called \emph{relevant} if, for every $k \geq 1$, the types $(k,k-1)$ and $(k-1,k)$ do not simultaneously appear in $\gamma$. Let $\RBP(m,n) \subseteq \BP(m,n)$ denote the subset of relevant decorated bipartitions.
\item 
The combinatorial \emph{moment map}  
    \begin{equation}\label{moment nilcone II}
        \mu^{\dagger} := \mu_1^{\dagger} \times \mu_2^{\dagger} \colon \BP(m,n) \longrightarrow \cP(m) \times \cP(n).
    \end{equation}
    is defined by:
    \[
    \mudg_1(\gamma):= \set{x_i| 1\leq i \leq r} \quad \text{and} \quad
    \mudg_2(\gamma):= \set{y_i| 1\leq i \leq r}. 
    \]
    for a decorated bipartition as in \eqref{eq DBP}. 
\end{enumerate}
\end{defn}

The $\ov G_1$- and $\ov G_2$-orbits in $\ov \cN_1$ and $\ov \cN_2$ are parameterized by the sets $\cP(m)$ and $\cP(n)$ respectively by taking the lengths of Jordan blocks. In the following discussion, for $\gamma_1 \in \cP(m)$ and $\gamma_2 \in \cP(n)$, we denote the corresponding orbits in $\ov \cN_1$ and $\ov \cN_2$ by $\cO_{\gamma_1}$ and $\cO_{\gamma_2}$.
Moreover, there is a unique way of writing
\begin{equation}\label{std ga1 ga2}
    \gamma_1 = \set{x_1, x_2, \ldots, x_r} \quad  
    \gamma_2 = \set{y_1, y_2, \ldots, y_r},
\end{equation}
such that 
    $ x_1 \geq \cdots \geq x_r$, 
    $y_1 \geq \cdots \geq y_r$,
and $x_r$ and $y_r$ are not both zero. 

\begin{defn}\label{relevant pair type II}
\begin{enumerate}
    \item We call the pair $(\gamma_1,\gamma_2)\in \cP(m)\times \cP(n)$ \emph{relevant} 
if, under the unique expression \eqref{std ga1 ga2},  we have $|x_i - y_i| \leq 1$ for all $1 \leq i \leq r$. We denote the set of relevant pairs by $\cR(\cP(m)\times \cP(n))$. 
    \item For a relevant $(\g_1,\g_2)\in \cP(m)\times \cP(n)$, let $\lambda_k$ denotes the multiplicity of $(k,k)$ in the multiset of pairs
\begin{equation}\label{eq relevant DBP}
   \set{ (x_1, y_1), (x_2, y_2), \ldots, (x_r, y_r) },
\end{equation}
and define 
\begin{equation}\label{eq multi Type II}
    m_{\gamma_1,\gamma_2}:= \prod_{k=1}^\infty (\lambda_k + 1).
\end{equation}
\end{enumerate}
\end{defn}    

\begin{lemma}\label{relevant triple counting type II}
The pair $(\gamma_1,\gamma_2)\in \cP(m)\times \cP(n)$ is in the image of $\mudg$ if and only if the pair is relevant. 
In this case, ${\mudg}^{-1}(\gamma_1,\gamma_2)\cap \RBP(m,n)$ is a non-empty set containing $m_{\gamma_1,\gamma_2}$ elements.      
\end{lemma}
\begin{proof}
Suppose $(\gamma_1,\gamma_2)$ is in the image of $\mudg$, take $\gamma = \set{ (a_j,b_j)^{\epsilon_j} | 1\leq j \leq r }\in \BP(m,n)$ such that   $\mudg(\gamma) = (\gamma_1,\gamma_2)$ where the pairs $\set{(a_j,b_j) }$ is ordered in lexical order ($a_j \geq a_{j+1}$ and $(b_j \geq b_{j+1})$ if $a_j=a_{j+1}$) and $m_j$ is the multiplicity of the pair. Observe that 
\begin{equation}\label{obs relevant type II}
\text{
$\gamma$ is relevant if and only if $b_j \geq b_{j+1}$ for all $j$}
\end{equation}
In particular, this implies $\set{(a_j,b_j)| 1\leq j\leq r}$ equals to \eqref{eq relevant DBP} 
and $(\gamma_1,\gamma_2) \in \cR(\cP(m)\times \cP(n))$. 

If $\gamma$ is not relevant, then $b_j < b_{j+1}$ only happens when $(a_j,b_j) = (k,k-1)$ and $(a_{j+1},b_{j+1}) = (k-1,k)$ for some $k\geq 1$.
Replace all the pairs $\set{ (k, k-1), (k-1, k)}$ iteratively from the largest $k$ by 
$\set{(k,k)^+, (k-1,k-1)^+} $
we get a relevant decorated bipartition in ${\mudg}^{-1}(\gamma_1,\gamma_2)$.  This proves the first part of the lemma. 

    The second part of the lemma also follows from \eqref{obs relevant type II} directly.  
\end{proof}

\begin{exam}
   Let $\gamma_1=\set{4,3,1}$ and $\gamma_2=\set{4,3,2}$. Then $(\gamma_1,\gamma_2)$ is relevant with 
   \[
    {\mudg}^{-1}(\gamma_1,\gamma_2)=\bigset{ 
   \set{(4,4)^+, (3,3)^+,(1,2)}, \set{(4,4)^+, (3,3)^-,(1,2)},\\
    \set{(4,4)^-, (3,3)^+,(1,2)}, \set{(4,4)^-, (3,3)^-,(1,2)} \\
     \set{(4,3)^+, (3,4)^-,(1,2)}
    }   \subseteq \BP(8,9), 
   \]
   and 
    \[
    {\mudg}^{-1}(\gamma_1,\gamma_2)\cap \RBP(8,9)=\bigset{ 
   \set{(4,4)^+, (3,3)^+,(1,2)}, \set{(4,4)^+, (3,3)^-,(1,2)},\\
    \set{(4,4)^-, (3,3)^+,(1,2)}, \set{(4,4)^-, (3,3)^-,(1,2)}
    } . 
   \]
\end{exam} 
The following lemma is a rephrasing of the work by Kraft and Procesi \cite[\S6]{MR694606}. For the reader’s convenience, we include a complete proof in \Cref{Sec GL nilpotent}.

\begin{prop}\label{prop GL orbits}
There is a bijection
\[
\BP(m,n)\;\xrightarrow{\sim}\;\underline{(\ov G_1 \times \ov G_2)\backslash \cN_{\ov \bL}},\qquad 
\gamma \longmapsto \cO_\gamma,
\]
with the following properties:
\begin{enumerate}
    \item $\mu(\cO_\gamma) = (\cO_{\mudg_1(\gamma)}, \cO_{\mudg_2(\gamma)})$;
    \item $\gamma \in \RBP(m,n)$ if and only if $\cO_\gamma$ is in the set $\cR_{\ov \bL}$ defined in \Cref{def relevant GL}.
\end{enumerate}
\end{prop}

Consequently, \Cref{thm:Htop type II} can be rephrased as follows.
\begin{cor}\label{thm:Htop type II combi}
   As a $\ov W_1 \times \ov W_2$-module, there is an isomorphism:
\[
\HBM(\St_{\ov \bL}, \bC) \cong \bigoplus_{\gamma \in\RBP(m,n)}  E_{\cO_{ \mudg_1(\gamma)}} \boxtimes E_{\cO_{ \mudg_2(\gamma)}}\cong  \bigoplus_{(\gamma_1, \gamma_2)\in \cR(\cP(m)\times \cP(n))} m_{\gamma_1,\gamma_2} E_{\cO_{\gamma_1}} \boxtimes E_{\cO_{\gamma_2}}.
\]
\end{cor}

\subsection{Relative Springer theory for the ortho-symplectic Steinberg variety}\label{section Ortho-symplectic}

In this subsection, we apply \Cref{Thm decomposition of SthS} to the ortho-symplectic Steinberg variety $\St_{\VV}$ introduced in \Cref{eq orthogonal-symplectic Steinberg variety} to prove \Cref{thm:Htop}. 

We retain the notation from \Cref{Section Geometrization of oscillator bimodule}. The action of $G_1 \times G_2$ on $\VV$, together with the moment map $\mu=\mu_1\times \mu_2: \VV \to  \frak g_1^* \times \frak g_2^*$, fits into the framework of \Cref{section relative Springer general}. One checks that the variety $\St_{\VV}$ defined in \Cref{eq orthogonal-symplectic Steinberg variety}  fits into \Cref{defn Steiberg}. Applying \Cref{defn relevant pair} and \Cref{eq hyperspherical fibre multiplicty} yields the set $\tcR_{G_1 \times G_2}(\VV)$ and, for each element $(\cO_1\times \cO_2,\cL_1\otimes \cL_2) \in \tcR_{G_1 \times G_2}(\VV)$, a multiplicity space $V_{\cO_1\times \cO_2,\cL_1\otimes \cL_2}(\VV)$.

We thus obtain, via \Cref{Thm decomposition of SthS}, a $W_1\times W_2$-action on $\hBM{\TOP}{\St_{\VV},\bC}$ and the following decomposition as $W_1\times W_2$-module:
\begin{equation}\label{eq hBMV first}
    \hBM{\TOP}{\St_{\VV},\bC} \cong \bigoplus_{(\cO_1\times \cO_2,\cL_1\otimes \cL_2) \in \tcR_{G_1 \times G_2}(\VV)} E_{\cO_1,\cL_1} \otimes E_{\cO_2,\cL_2} \otimes V_{\cO_1\times \cO_2,\cL_1\otimes \cL_2}(\VV).
\end{equation}
Next we refine this decomposition to prove \Cref{thm:Htop}. 

Applying \Cref{Lambdafrb} with the fixed Borel subgroup $B_1\times B_2$ of $G_1\times G_2$ yields  
\begin{equation}\label{Eq Lambda}
  \Lambda_{\VV} := \mu^{-1}(\frak b_1^\perp \times \frak b_2^\perp). 
\end{equation}

Then \Cref{eqSteinberg} gives an isomorphism
\begin{equation}\label{isomorhis StVV Lambdavv}
    \St_{\VV}\cong (G_1\times G_2)\times^{B_1\times B_2} \Lambda_{\VV},
\end{equation}
which induces an isomorphism
\[
\hBM{\TOP}{\St_{\VV},\bC} \cong \hBM{\TOP}{\Lambda_{\VV},\bC}.  
\]

Recall that we fixed a polarization $V_1 = L_1 \oplus L_1^\vee$  stable under $B_1$ and defined the variety $\bL_1$ and the moment cone $\cN_{\bL_1}$ in \Cref{def L1} and \Cref{moment cone 1}. There is a natural isomorphism 
\begin{equation}\label{VV poloraization}
\VV \cong T^* \bL_1.  
\end{equation}
We may regard $\Lambda_{\VV}$ as a subscheme of $T^* \bL_1$ via the isomorphism above. We begin by describing its relationship with the microlocal geometry of the moment cone $\cN_{\bL_1}$  under the action of $\ov B_1 \times B_2$. 
For each such orbit $\cO\in \underline{\overline B_1\backslash \cN_{\bL_1}/B_2}$,  let
\begin{equation*}
    T^*_{\cO} \bL_1 \subset T^* \bL_1
\end{equation*}
denote the conormal bundle to $\cO$.

\begin{prop}\label{pro Lambda}
We have
\[
\Lambda_{\VV}(\bC) = \bigsqcup_{\cO \in \underline{\overline B_1\backslash \cN_{\bL_1}/B_2}} T^*_\cO \bL_1(\bC). 
\]
Since $\underline{\overline B_1\backslash \cN_{\bL_1}/B_2}$ is finite (cf. \Cref{SPM vs orbits}), the reduced scheme of $\Lambda_{\VV}$ is a Lagrangian subvariety of $\VV \cong T^* \bL_1$.
\end{prop}
\begin{proof}
Let $U_1$ be the unipotent radical of the Siegel parabolic subgroup $P_1$ of $G_1$ stabilizing $L_1$. We write $\frak u_{1}=\Lie(U_1), \ov{\frak b}_1=\Lie(\ov B_1)$. Then we have a semidirect decomposition:
\begin{equation}\label{unipotent semi-direct}
    \frak b_1 = \frak u_{1} \oplus \ov{\frak b}_1.
\end{equation}

It follows that
\begin{equation}\label{Eq lambda vv}
    \Lambda_{\VV} = \mu^{-1}(\frb_1^\perp \times \frb_2^\perp) = \mu_1^{-1}(\fru_{1}^\perp) \cap \mu^{-1}(\ov \frb_1^\perp \times \frb_2^\perp).
\end{equation}
Let $\Pi_1: \VV \to \bL_1$ denote the natural projection. One checks that
\begin{equation}\label{Moment map fru1}
    \mu_1^{-1}(\fru_{L_1}^\perp) = \Pi_1^{-1}(\cN_{\bL_1}).
\end{equation}

\trivial[h]{
For each $T\in \Hom(V_1,V_1)$, we define $T^*\in \Hom(V_1,V_1)$ by requiring that 
\[
\langle Tv, v'\rangle_{V_1}= \langle v, T^* v'\rangle_{V_1},\quad \forall v,v'\in V_1. 
\]
Then
\[
\frak g_1= \Herm(V_1,V_1) \coloneqq \{T\in \Hom(V_1,V_1)|T^*=-T\}.
\]
We identify $\Hom(V_1/L_1,L_1)$ as a subspace of $\Hom(V_1,V_1)$ via the natural embedding. Then 
\[
\frak u_{L_1}=\Herm(V_1/L_1, L_1) \coloneqq \Hom(V_1/L_1,L_1)\cap \Herm(V_1,V_1).
\]
Let $\pr: \Hom(V_1,V_1)\rightarrow \Hom(L_1,V_1/L_1) $ be the natural quotient map and define 
\[
\Herm(L_1,V_1/L_1)\coloneqq \pr(\Herm(V_1,V_1)). 
\]
We identify $\Herm(L_1,V_1/L_1)$ with $\frak u_{L_1}^*=\Herm(V_1/L_1, L_1)^*$ by using the perfect trace pairing 
    \[
    \begin{split}
    \Herm(L_1,V_1/L_1)\times \Herm(V_1/L_1,L_1)&\longrightarrow \bC\\
     (b, b')& \mapsto \tr(b'b).
         \end{split}
    \]
For $A\in \bL_1=\Hom(L_1, V_2)$, we define $A^*\in \Hom(V_2, V_1/L_1)$ characterized by  
\begin{equation}\label{T*}
    \langle A v_1 , v_2\rangle_{V_2}=\langle v_1, A^* v_2\rangle_{V_1},\quad \forall v_1\in L_1, v_2\in V_2. 
\end{equation}
Note that $A^*A \in \Herm(L_1,V_1/L_1)=\frak u_{L_1}^*$ and we define the quadratic map 
\begin{equation}\label{eq quadratic map}
\begin{split}
     \l_1: \bL_1=\Hom(L_1,V_2) &\longrightarrow  \frak u_{L_1}^*\\
           A& \mapsto \frac{1}{2} A^* A. 
\end{split}
\end{equation}
Recall the moment cone $\cN_{\bL_1}$ from \eqref{moment cone 1}, which is the same as $\l^{-1}_1(0)$. It is easy to check that we have the following commutative diagram 
\[
\begin{tikzcd}[ampersand replacement=\&]
    \Hom(V_1,V_2) \ar[r,"\mu_1"] \ar[d,"\Pi_1"] \&  \frak g_1^* \ar[d,"\res^{\frak g_1}_{\frak u_{L_1}}"] \\
     \Hom(L_1,V_2) \ar[r,"\l_1"] \& \frak u_{L_1}^*
\end{tikzcd}
\]
where $\Pi_1$ is the natural projection map and $\res^{\frak g_1}_{\frak u_{L_1}}$ is the restriction map. In particular, 
    \begin{equation}\label{Moment map fru1}
     \mu_1^{-1}(\frak u_{L_1}^{\perp})= \Pi_1^{-1}(\cN_{\bL_1}).
    \end{equation}
}
Furthermore, since the isomorphism \Cref{VV poloraization} is $\GL(L_1) \times G_2$-equivariant, we have
\begin{equation}\label{Eq conormal 1}
    \mu^{-1}(\ov\frb_1^\perp \times \frb_2^\perp)(\bC) = \bigsqcup_{\cO \in \underline{\ov B_1 \backslash \bL_1 /B_2}} T^*_\cO \bL_1(\bC).
\end{equation}
The result follows by combining \Cref{Eq lambda vv}, \Cref{Moment map fru1}, and \Cref{Eq conormal 1}.
\end{proof}
Now switch the roles of $V_1$ and $V_2$. Let $ V_2=L_2\oplus L_2^\vee$ be a polarization of $V_2$ stable under $B_2$ and define $\bL_2 = \Hom(L_2, V_1)$ and $\cN_{\bL_2}$ under the $B_1 \times \ov B_2$-action. 
The same argument as in \Cref{pro Lambda} applies here as well.

\begin{cor}\label{Lagrangian Steiberg VV}
The reduced scheme of $\St_{\VV}$ is a Lagrangian subvariety of $\VV \times \tcN_1 \times \tcN_2$. In particular,
\begin{equation}\label{Lagrangian Steiberg VV dimension}
    \dim \St_{\VV} = \tfrac{1}{2}(\dim \wt\cN_1 + \dim \wt\cN_2 + \dim \VV).
\end{equation}
\end{cor}

\begin{proof}
By construction, $\St_{\VV}$ is the zero fiber of the moment map $\VV \times \wt\cN_1 \times \wt\cN_2 \to \frg_1^* \times \frg_2^*$ associated with the diagonal Hamiltonian $G_1\times G_2$-action. Hence, $\St_{\VV}$ is coisotropic in $\VV \times \wt\cN_1 \times \wt\cN_2$. Since $\Lambda_{\VV}$ is Lagrangian in $\VV$, we have $\dim \Lambda_{\VV} = \tfrac{1}{2} \dim \VV$, and the stated formula for $\dim \St_{\VV}$ follows from \Cref{isomorhis StVV Lambdavv}. Thus, $\St_{\VV}$ is Lagrangian.
\end{proof}

Let $\cR_{G_1\times G_2}(\VV)$ be the set of relevant nilpotent orbits as in \Cref{defn relevant pair}. It follows directly from \Cref{defn relevant pair} and \Cref{Lagrangian Steiberg VV dimension} that a pair $\cO_1\times \cO_2$ belongs to $\cR_{G_1\times G_2}(\VV)$ if and only if
\begin{equation}\label{ortho-symplectic fibre dimension bound}
    \dim \mu^{-1}(\cO_1 \times \cO_2) = \tfrac{1}{2}(\dim \VV + \dim \cO_1 + \dim \cO_2).
\end{equation}
\begin{proof}[Proof of \Cref{thm:Htop}]
Let $\cO_1 \times \cO_2 \in \cR_{G_1 \times G_2}(\VV)$, and choose a representative $e_1 \times e_2 \in \cO_1 \times \cO_2$. Set $\VV_{e_1,e_2} := \mu^{-1}(e_1, e_2)$. For each $\cO \in \mu^{-1}(\cO_1 \times \cO_2) \cap \cR_{\VV}$, we also define $\VV^{\cO}_{e_1,e_2} := \VV_{e_1,e_2} \cap \cO$.

Since the quotient $\underline{(G_1 \times G_2) \backslash \cN_{\VV}}$ is finite, the union 
\[
\bigsqcup_{\cO \in \mu^{-1}(\cO_1 \times \cO_2) \cap \cR_{\VV}} \VV^\cO_{e_1,e_2}
\]
is open and dense in $\VV_{e_1,e_2}$. It follows that the restriction map
\begin{equation}\label{restri top Borel}
    \res: \hBM{\TOP}{\VV_{e_1,e_2},\bC} \longrightarrow \bigoplus_{\cO \in \mu^{-1}(\cO_1 \times \cO_2) \cap \cR_{\VV}} \hBM{\TOP}{\VV^{\cO}_{e_1,e_2},\bC}
\end{equation}
is an isomorphism. Moreover, this isomorphism intertwines the natural action of $A_{e_1} \times A_{e_2}$ on both sides.

Fix a $e \in \cO$ such that $\mu(e) = (e_1, e_2)$. The natural isomorphism
\[
\VV^{\cO}_{e_1,e_2}\cong C_{G_1}(e_1) \times C_{G_2}(e_2) / \Stab_{G_1 \times G_2}(e)
\]
induces an $A_{e_1} \times A_{e_2}$-equivariant isomorphism
\[
\hBM{\TOP}{\VV^{\cO}_{e_1,e_2},\bC} \cong \bC[A_{e_1} \times A_{e_2}/A_e] \cong \bigoplus_{(\chi_1,\chi_2) \in \widehat A_{e_1} \times \widehat A_{e_2}} (\chi_1 \otimes \chi_2)^{\dim (\chi_1 \otimes \chi_2)^{A_e}}.
\]
Since $A_{e_1}$ and $A_{e_2}$ are abelian, all characters $\chi_1$, $\chi_2$ are one-dimensional, and we have
\[
\dim (\chi_1 \otimes \chi_2)^{A_e} =
\begin{cases}
    1 & \text{if } (\chi_1|_{A_e}) \cdot (\chi_2|_{A_e}) = 1 \in \Irr(A_e), \\
    0 & \text{otherwise}.
\end{cases}
\]
Therefore, we obtain the decomposition
\begin{equation}\label{eq decom VVcO}
   \hBM{\TOP}{\VV^{\cO}_{e_1,e_2},\bC} \cong 
\bigoplus_{\substack{
  (\chi_1,\chi_2) \in \widehat A_{e_1} \times \widehat A_{e_2} \\
  (\chi_1|_{A_e}) \cdot (\chi_2|_{A_e}) = 1
}}
 \chi_1 \otimes \chi_2.
\end{equation}
Substituting \eqref{restri top Borel} and  \eqref{eq decom VVcO} into \Cref{eq hBMV first}  completes the proof.
\end{proof}
Next we provide a combinatorial description of $\cRV$ based on results of Kraft-Procesi \cite[\S 6]{MR694606}.

\begin{defn}\label{defn orthogonal symplectic partition}
\begin{enumerate}
    \item For a non-negative integer $k$, we define the \emph{orthogonal partitions} (resp. \emph{symplectic partitions}) $\cP_{+}(k)$ (resp. $\cP_{-}(k)$) of $k$ be the subset of $\cP(k)$  consisting of partitions in which even (resp. odd) parts occur with even multiplicity.   
    \item Let 
    \[
    \lambda = [\underbrace{t_1, \ldots, t_1}_{z_1}, \ldots, \underbrace{t_l, \ldots, t_l}_{z_l}] \in \cP_{\pm}(k),
    \]
    where $t_1 > \cdots > t_l > 0$. Define $\odd_{\lambda} := \{ t_i \mid t_i \text{ is an odd part of $\l$} \}$ and $\even_{\lambda} := \{ t_i \mid t_i \text{ is an even part of $\l$} \}$.

    \begin{enumerate}
        \item If $\lambda \in \cP_{+}(k)$, define the  group
        \[
        \cS_{\lambda} \coloneqq (\bZ/2\bZ)^{\odd_{\l}}=\bigoplus_{t \in \odd_{\lambda}} \bZ/2\bZ \cdot a_t.
        \]
        \item If $\lambda \in \cP_{-}(k)$, define the  group
        \[
        \cS_{\lambda} \coloneqq (\bZ/2\bZ)^{\even_{\l}}=\bigoplus_{t \in \even_{\lambda}} \bZ/2\bZ \cdot a_t.
        \]
    \end{enumerate}
\end{enumerate}
\end{defn}

It is well known (cf. \cite[\S 5.1]{MR1251060}) that the set of $G_1$-orbits on $\cN_1$ is parametrized by $\cP_{+}(2m)$, and the set of $G_2$-orbits on $\cN_2$ are parametrized by $\cP_{-}(2n)$, by recording the lengths of Jordan blocks. For each $i = 1, 2$, denote by $\cO_\lambda\subset \cN_i$ the nilpotent orbit corresponding to a partition $\lambda \in \cP_{+}(2m)$ or $\lambda \in \cP_{-}(2n)$, with a representative $e_i \in \cO_\lambda$. There is a canonical isomorphism (cf. \cite[\S 6.1]{MR1251060}) 
\[
A_{e_i} \cong \cS_{\lambda},
\]
and we identify these groups accordingly. For $\lambda \in \cP_{+}(2m)$ or $\cP^-(2n)$ and $\chi \in \Irr(\cS_{\lambda})$, we denote the corresponding Weyl group representation by  $E_{\cO_{\lambda}, \chi}$.
\begin{defn}\label{defn orth-symplectic partitions}
\begin{enumerate}
    \item For a pair of non-negative integers $(M, N)$, define the set of \emph{ortho-symplectic partitions} $\OSP(M, N) \subseteq \BP(M, N)$ (see \Cref{defn ab diagram}) to consist of all
    \[
    \{(x_1, y_1)^{\epsilon_1}, \ldots, (x_r, y_r)^{\epsilon_r}\} \in \BP(M, N)
    \]
    satisfying:
    \begin{itemize}
        \item For $k\ge1$, $(k,k)^+$-parts and $(k,k)^-$-parts occur with the same multiplicity;
        \item For even $k$, both $(k,k+1)$ and $(k,k-1)$ appear with even multiplicities.
    \end{itemize}
    \item  We define the set \emph{relevant ortho-symplectic partitions} $\ROSP(M, N)\coloneqq \OSP(M, N)\cap \RBP(M,N)$, i.e., $\gamma \in \OSP(M, N)$ is said to be \emph{relevant} if, for every $k \geq 0$, the types $(k+1,k)$ and $(k,k+1)$ do not simultaneously appear in $\gamma$.
    \item Let 
   \[
        \gamma = [\underbrace{(x_1, y_1)^{\epsilon_1}, \ldots, (x_1, y_1)^{\epsilon_1}}_{z_1}, \ldots, \underbrace{(x_l, y_l)^{\epsilon_l}, \ldots, (x_l, y_l)^{\epsilon_l}}_{z_l}] \in \OSP(M,N)
    \]
    with pairwise distinct $(x_i, y_i)^{\epsilon_i}$. Define:
    \begin{itemize}
        \item $\oddeven_\gamma$: the set of distinct pairs $(x_i,y_i)$ in $\gamma$ where $x_i$ is odd and $y_i$ is even;
        \item $\evenodd_\gamma$: the set of distinct pairs $(x_i,y_i)$ in $\gamma$ where $x_i$ is even and $y_i$ is odd;
        \item $\oddodd_\gamma$:  the set of distinct pairs $(x_i,x_i)^{\e_i}$ in $\gamma$ where both $x_i$ is odd;
        \item $\eveneven_\gamma$:  the set of distinct pairs $(x_i,x_i)^{\e_i}$ in $\gamma$ where $x_i$ is even.
    \end{itemize}
    Define the group
    \begin{equation}\label{eq component group orth-symplectic partitions}
        \cS_\gamma \coloneqq (\bZ/2\bZ)^{\oddeven_\gamma}=\bigoplus_{(x_i,y_i)\in \oddeven_\gamma} \bZ/2\bZ \cdot a_{(x_i,y_i)}.
    \end{equation}

    \item The combinatorial moment map $\mu^\dagger$ defined in \Cref{moment nilcone II} restricts to a map
    \begin{equation}\label{moment nilcone I}
        \mu^{\dagger} := \mu_1^{\dagger} \times \mu_2^{\dagger} \colon \OSP(M,N) \longrightarrow \cP^+(M) \times \cP^-(N).
    \end{equation}
    Given $\mu^\dagger(\gamma) = (\gamma_1, \gamma_2)$, we obtain a natural map 
    \begin{equation}\label{map component group}
        \Delta^\dagger := (\Delta_1^\dagger, \Delta_2^\dagger) \colon \cS_\gamma \longrightarrow \cS_{\gamma_1} \times \cS_{\gamma_2}
    \end{equation}
    defined by $\Delta_1^\dagger(a_{(x,y)}) = a_x$ and $\Delta_2^\dagger(a_{(x,y)}) = a_y$ for all $(x,y)\in \oddeven_\gamma$.

    \item A \emph{relevant quintuple} associated to $(M, N)$ is a tuple $(\gamma, \gamma_1, \chi_1, \gamma_2, \chi_2)$ such that:
    \begin{itemize}
        \item $\gamma \in \ROSP(M, N)$;
        \item $(\gamma_1, \gamma_2) = \mu^\dagger(\gamma)$;
        \item $(\chi_1 \otimes \chi_2) \circ \Delta^\dagger = 1 \in \Irr(\cS_\gamma)$.
    \end{itemize}
    Denote the set of relevant quintuples by $\RQ(M, N)$.
\end{enumerate}
\end{defn}

\begin{remark}
    In \cite[\S 6]{MR694606}, the set $\OSP(M,N)$ is referred to as the ortho-symplectic $ab$-diagrams.
\end{remark}

\begin{Exa}
    Let $\gamma=\set{(3,3)^+,(3,3)^-,(3,2),(1,2)}\in \ROSP(10,10)$.
\begin{enumerate}
    \item  We have $\gamma_1=\mudg_1(\gamma)=\set{3,3,3,1}$ and $\gamma_2=\mudg_1(\gamma)=\set{3,3,2,2}$. 
     \item We have 
    \[
    \cS_{\gamma}\cong \bZ/2\bZ\cdot  a_{(3,2)}\oplus \bZ/2\bZ \cdot a_{(1,2)},
    \]
    \[
    \cS_{\gamma_1}\cong \bZ/2\bZ\cdot  a_{3}\oplus \bZ/2\bZ \cdot a_{1},\quad \cS_{\gamma_2}\cong \bZ/2\bZ\cdot  a_2\]
    and the map $\Delta^\dagger=(\Delta^\dagger_1,\Delta^\dagger_2)$ is given by 
        \[
        \Delta_1^\dagger(a_{(3,2)})=a_3,\quad \Delta_1^\dagger(a_{(1,2)})=a_{1} 
        \] 
        and 
        \[
\Delta_2^\dagger(a_{(3,2)})=\Delta_2^\dagger(a_{(1,2)})=a_2.
        \]
\end{enumerate}    
\end{Exa}

\begin{lemma}\label{relevatbt quintuple counting}
Let $(\gamma_1, \gamma_2) \in \cP^+(M) \times \cP^-(N)$, written as
\[
    \gamma_1 = [x_1, \ldots, x_l], \quad x_1 \geq \cdots \geq x_l, \qquad \gamma_2 = [y_1, \ldots, y_l], \quad y_1 \geq \cdots \geq y_l,
\]
where zero parts are added as needed so that $\gamma_1$ and $\gamma_2$ have equal length. Then there exists $\gamma \in \ROSP(M,N)$ such that $\mu^{\dagger}(\gamma) = (\gamma_1, \gamma_2)$ if and only if $|x_i - y_i| \leq 1$ for all $1 \leq i \leq l$. In this case, there is a unique such $\gamma$ given by
\[
    \gamma = \{ (x_1, y_1)^{\epsilon_1}, \ldots, (x_l, y_l)^{\epsilon_l} \},
\]
where the signs $\epsilon_i$ are uniquely determined by the condition in \Cref{defn orth-symplectic partitions}(1).
\end{lemma}

\begin{proof}
This follows immediately from the same reasoning as in the proof of \Cref{relevant triple counting type II}.
\end{proof}

The following result is partially proved in \cite[\S6]{MR694606}.  For the reader’s convenience we give a complete proof in \Cref{Sec Orth-symplectic}.

\begin{prop}\label{orthosymplectic orbits}
There is a bijection
\[
\OSP(2m,2n)\;\xrightarrow{\sim}\;\underline{(G_1 \times G_2)\backslash \cN_{\VV}},\qquad 
\gamma \longmapsto \cO_\gamma,
\]
with the following properties:
\begin{enumerate}
\item For any $e \in \cO_\gamma$ there is an isomorphism
\[
A_e \;\cong\; \cS_\gamma,
\]
where $A_e$ is defined in \Cref{eq component group O} and $\cS_\gamma$ in \Cref{eq component group orth-symplectic partitions}.

\item The moment map satisfies
\[
\mu(\cO_\gamma)=\bigl(\cO_{\mudg_1(\gamma)},\cO_{\mudg_2(\gamma)}\bigr),
\]
and under the identification in (1) the component-group map $\Delta_i$ of \Cref{eq moment map component group} coincides with the map  $\Delta_i^\dagger$ from \Cref{map component group}.

\item One has $\gamma\in \ROSP(2m,2n)$ if and only if $\cO_\gamma\in\cR_{\VV}$, where $\cR_{\VV}$ is defined in \Cref{def relevant orthsymplectic}.  
In particular, the relevant quintuples 
\[
(\cO,\cO_1,\cL_1,\cO_2,\cL_2)\in\cRV
\]
are parametrized by the set $\RQ(2m,2n)$ of relevant quintuples $(\gamma,\gamma_1,\chi_1,\gamma_2,\chi_2)$.
\end{enumerate}
\end{prop}

Consequently, \Cref{thm:Htop} can be rephrased as follows.
\begin{cor}\label{thm:Htop com}
As a $W_1 \times W_2$-module, there is a canonical isomorphism:
\[
\HBM(\St_{\VV}, \bC) \cong \bigoplus_{(\gamma, \gamma_1, \chi_1, \gamma_2, \chi_2) \in \RQ(2m,2n)} E_{\cO_{\gamma_1}, \chi_1} \boxtimes E_{\cO_{\gamma_2}, \chi_2}.
\]
\end{cor}

\section{Geometrization of Hecke modules from spherical varieties}
\label{Sec. geom gen}

\subsection{Generalities on geometrization by mixed sheaves}\label{Sec. Cat gen}
In this subsection, we work over $\bF_q$ and sheaves are constructible sheaves in the \'etale topology with $\Qlbar$-coefficients. Let $k=\ov\bF_q$. Let $X$ be a scheme of finite type over $\bF_q$. Let $B$ be an algebraic group over $\bF_q$ acting on $X$ with finitely many {\em geometric} orbits $\{X_\a\}_{\a\in I}$ (i.e., orbits of $B_k=B\otimes_{\bF_q}k$ acting on $X_k=X\otimes_{\bF_q}k$). Define a partial order on $I$ so that $\a\le \b$ iff $X_\a\subset \ov X_\b$.

We make the following simplifying assumption:
\begin{equation}\label{orbit assumption}
    \mbox{Each geometric orbit $X_\a$ contains an $\bF_q$-point with connected stabilizer.}
\end{equation} 
This assumption implies that each $X_\a$ is defined over $\bF_q$, and that $B(\bF_q)$ acts transitively on $X_\a(\bF_q)$.  We choose a point $x_\a\in X_\a(\bF_q)$.

Let $D^b_{B}(X)$ be the $B$-equivariant bounded derived category of $\Qlbar$-complexes on $X$.  We also have the similarly defined derived category $D^b_{B_k}(X_k)$ for the base-changed situation over $k$. We have the pullback functor 
\begin{equation*}
    \b: D^b_{B}(X)\to D^b_{B_k}(X_k).
\end{equation*}
We choose a square root $\sqrt{q}$ of $q$ in $\Ql$. Let $\Ql(\half)$ be the corresponding square root of Tate sheaf on $\Spec(\bF_q)$. For ease of notation, given an object \( \sF \in D_{B}^b(X) \), we write
\[
\sF\braket{n} \coloneqq \sF[n]\left(\tfrac{n}{2}\right),
\]
where \( \sF[n] \) denotes the cohomological shift and \( \sF\!\left(\tfrac{n}{2}\right) =\sF \otimes \Ql\left(\tfrac{n}{2}\right) \) denotes the \( n/2 \)-th Tate twist.

For $\a\in I$, let $j_\a: X_\a\hookrightarrow X$ be the inclusion. Let $d_\a=\dim X_\a$. Let 
\begin{equation}\label{Define Deltaa}
    \Delta_\a=j_{\a,!}(\underline{\Ql} )\braket{d_\a}
\quad\mbox{and} \quad \IC_\a=j_{\a,!*}(\underline{\Ql})\braket{d_\a}
\end{equation}
be the standard perverse sheaf and the IC sheaf constructed from the shifted and twisted constant sheaves on $X_\a$.

For $\sF\in D^b_{B}(X)$, its stalk $\sF_{x_\a}$ carries an action of the geometric Frobenius $\Frob\in \Gal(\ov\bF_q/\bF_q)$. Fix an isomorphism of fields $\io: \Qlbar\cong\bC$. We say an element $a\in \Qlbar$ has $\io$-weight $m\in \ZZ$ if $|\io(a)|=q^{m/2}$.
\begin{defn} An object $\sF\in D^b_{B}(X)$ is called
    \begin{enumerate}
        \item {\em mixed} if the eigenvalues of $\Frob$ on $H^i\sF_{x_\a}$ have integer $\io$-weights, for all $\a\in I$ and $i\in \ZZ$.
        \item  {\em Tate} if the eigenvalues of $\Frob$ on $H^i\sF_{x_\a}$ are inside  $\{q^{m/2}; m\in \ZZ\}\subset \Qlbar$, for all $\a\in I$ and $i\in \ZZ$.

        \item  {\em $*$-pure of weight 0} if the the eigenvalues of $\Frob$ on $H^i\sF_{x_\a}$ have  $\io$-weight $i$, for all $\a\in I$ and $i\in \ZZ$.

        \item {\em $*$-even} if $H^{i}\sF_{x_\a}=0$ for all odd $i$ and $\a\in I$; {\em $*$-odd} if $H^{i}\sF_{x_\a}=0$ for all odd $i$ and $\a\in I$. We say $\sF$ $*$-parity with parity $\e\in \bZ/2\bZ$ to mean that $\sF$ is $*$-even if $\e=0\mod 2$ and $*$-odd if $\e=1\mod 2$.
    \end{enumerate}
\end{defn}
Because of the transitive $B(\bF_q)$-action on each $X_\a(\bF_q)$, the above definitions are independent of the choice of the point $x_\a\in X_\a(\bF_q)$.
Let $D^{\mix}_{B}(X)\subset D^b_{m, B}(X)$ be the full subcategory of mixed objects; let   $D^{\Tate}_{B}(X)\subset D^b_{m, B}(X)$ be the full subcategory of Tate objects.

For any commutative ring $A$, consider the free $A$-module $A[\un{B\bs X}]$ over the geometric orbit set $\un{B\bs X}$ with a basis $\{\one_\a\}_{\a\in I}$ .
By Assumption \eqref{orbit assumption},
We may identify  $I = \un{B\bs X}$ with $B(\bF_q)\bs X(\bF_q)$, and thus identify $A[\un{B\bs X}]$ with $B(\bF_q)$-invariant $A$-valued functions on $X(\bF_q)$ by identifying  $\one_\alpha$ with the characteristic function of $X_\alpha(\bF_q)$.

Let $R=\bZ[v,v^{-1}]$. 
We now construct a commutative diagram
\begin{equation}\label{diag cat}
\begin{tikzcd}[ampersand replacement=\&,column sep=large, row sep=large]
K_0(D^b_B(X))
  \arrow[r, "\phi"]
\& \overline{\mathbb{Q}}_\ell [\underline{B \backslash X}] \makebox[0em][l] {$\displaystyle\  =\overline{\mathbb{Q}}_\ell [X(\mathbb{F}_q)]^{B(\mathbb{F}_q)} $}\\
K_0(D^{\Tate}_B(X))
  \arrow[u]
  \arrow[d]
\\
K_0(D^{\mix}_B(X))
  \arrow[r, "\operatorname{ch}"]
  \arrow[d]
\& R[\underline{B\backslash X}]
  \arrow[uu, "v = \sqrt{q}"']
  \arrow[d, "v=1"]
\\
K_0(D^b_{B_k}(X_k))
  \arrow[r, "\chi"]
\& \mathbb{Z}[\underline{B\backslash X}]
\end{tikzcd}
\end{equation}
Here the maps on the left side and induced by the natural functors between sheaf categories (such as the pullback from $X$ to $X_k$); the arrows on the right are the extension of scalars along the specialization maps $R=\bZ[v,v^{-1}]\to \Qlbar$ (sending $v$ to $\sqrt{q}\in \Qlbar$) and  $R=\bZ[v,v^{-1}]\to \bZ$ (sending $v$ to $1$).

We explain the horizontal maps $\phi,\ch$ and $\chi$. 
\begin{enumerate}
    \item The map $\phi$ is the sheaf-to-function map. It sends the class of $\sF\in D^b_B(X)$ to
    \begin{equation}\label{phi}    \phi(\sF)=\sum_{\a\in I}\left(\sum_j (-1)^j \tr(\Frob, H^j(\sF_{x_\a}))\right)\one_{\a}.
    \end{equation}

    \item The map $\chi$ is the Euler characteristic map. It sends the class of $\sF\in D^b_{B_k}(X_k)$ to
    \begin{equation}\label{chi}     \chi(\sF)=\sum_{\a\in I}\left(\sum_{j}(-1)^j \dim H^j(\sF_{x_\a})\right)\one_{\a}.
    \end{equation}

    \item The map $\ch$ is the weight polynomial map. For $\sF\in D^{\mix}_{B}(X)$, $\a\in I$ and $j\in \ZZ$,  $H^j(\sF_{x_\a})$ has a weight grading
    \begin{equation*}       H^j(\sF_{x_\a})=\oplus_{i\in \ZZ}\Gr^{\mathrm{W}}_i(H^j(\sF_{x_\a}))
    \end{equation*}
    where $\Gr^{\mathrm{W}}_i(H^j(\sF_{x_\a}))$ is the direct sum of generalized eigenspaces of $H^j(\sF_{x_\a})$ under $\Frob$ with eigenvalues of $\io$-weight $i$. Then $\ch$ sends the class of $\sF$ to 
    \begin{equation}\label{ch}
        \ch(\sF)=\sum_{\alpha\in I} \left(\sum_{i}\left(\sum_{j}(-1)^j\dim  \Gr^{\mathrm{W}}_i H^j(\sF_{x_\a}) \right)v^{i}\right) \one_\alpha.
    \end{equation}    
\end{enumerate}
It is straightforward to check that the diagram \eqref{diag cat} is commutative.

The following properties of $\ch$ are also easy to check:
\begin{align*}
\ch(j_{\a!}\Qlbar) &= \one_\a,\\
\ch(\sF(1/2)) &= v^{-1} \ch(\sF), \\
\ch(\sF[1]) &= - \ch(\sF),\\
\ch(\sF\braket{1}) &= - v^{-1} \ch(\sF). 
\end{align*}
Therefore, we have $\ch(\Delta_\a)=(-v)^{-d_\alpha} \one_{\a}$. 
The Grothendieck groups $K_0(D^\Tate_B(X)), K_0(D^\mix_B(X))$ and $K_0(D^b_B(X))$ carry $R$-module structures: $v\in R$ acts by $\sF\mapsto \sF(-1/2)$. Then the map $\ch$ is $R$-linear. 
\begin{warning} 
Although the sheaf-to-function map is also defined for mixed sheaves, it only agrees with the composition of $\ch$ and specialization at $v = \sqrt{q}$ when restricted to Tate objects as \eqref{diag cat}  indicates. 
\end{warning}

\begin{lemma}\label{Lem. ch isom}
    \begin{enumerate}
        \item The map $\chi: K_0(D^b_{B_k}(X_k))\to \bZ[\un{B\bs X}]$ is an isomorphism.
        \item The restriction of the map $\ch$ to Tate objects
        \begin{eqnarray*}
            \ch^\Tate: K_0(D^\Tate_B(X))\to R[\un{B\bs X}]
        \end{eqnarray*}
        is an isomorphism.
    \end{enumerate}
\end{lemma}
\begin{proof}
    (1) Since $\chi(j_{\a!}\Qlbar)=\one_\a\in \bZ[\un{B\bs X}]$, $\chi$ is surjective. 
    
    Now if $\chi(\sF)=0$, we want to show $[\sF]=0\in K_0(D^b_{B_k}(X_k))$. Now $\sF$ is a successive extension of $j_{\a!}j_\a^*\sF$, and each  $j_\a^*\sF$ is a successive extension of its shifted cohomology sheaves (which are constant), we have in the Grothendieck group
    \begin{equation}\label{F in K}
        [\sF]=\sum_{\a\in I}\sum_{i\in \bZ}(-1)^i[j_{\a!}(H^i\sF_{x_\a})]=\sum_{\a\in I}\left(\sum_{i\in \bZ}(-1)^i\dim(H^i\sF_{x_\a}) \right)[j_{\a!}\Qlbar].
    \end{equation}
    The inner sum is the  coefficient of $\one_\a$ in $\chi(\sF)$, which is zero by assumption. Therefore $[\sF]=0\in K_0(D^b_{B_k}(X_k))$.

    (2) The argument is the same. The only modification is that \eqref{F in K} should be replaced by
    \begin{equation}\label{F in K Tate}
        [\sF]=\sum_{\a\in I}\left(\sum_{i\in \bZ}(-1)^i[H^i\sF_{x_\a}]\right)[j_{\a!}\Qlbar].
    \end{equation}
    Here we understand the inner sum as an element in $K_0(D^\Tate(\pt))$. It is easy to check that $\ch$ induces an isomorphism $K_0(D^\Tate(\pt))\cong R$, therefore the inner sum in \eqref{F in K Tate} is the coefficient of $\one_\a$ in $\ch(\sF)$. 
\end{proof}
For each \( \alpha \in \underline{B \backslash X} \), we define
\begin{equation}\label{define Kazhan-Lustig}
    C'_\alpha := (-1)^{-d_\a} \, \ch(\IC_\a) \in R[\un{B\bs X}].
\end{equation}
The element \( C'_\a \) serves as an analogue of the Kazhdan--Lusztig basis in \( R[\underline{B \backslash X}] \).

\subsection{Geometrization of the Hecke algebra}\label{Geo hecke algebra}

Let $G$ be a connected split reductive group over $\bF_q$ with flag variety $\cB$. Choose a Borel subgroup $B\subset G$ so that $\cB=G/B$. The left $B$-action on $\cB$ satisfies the assumption \eqref{orbit assumption}. We then apply the general discussion in \Cref{Sec. Cat gen} to the $B$-action on $X=\cB$ to obtain  various versions of the Hecke category
\begin{equation*}
    \xymatrix{ \cH^{\Tate}=D^{\Tate}_{B}(\cB) \ar@{^{(}->}[r] & \cH^{\mix}=D^{\mix}_{B}(\cB) \ar@{^{(}->}[r] &\cH=D^{b}_{B}(\cB) \ar@{->}[r] &\cH^k=D^b_{B_{k}}(\cB_{k}).}
\end{equation*}
Now $\cH, \cH^{\mix}$ and $\cH^k$ carry canonical monoidal structures by convolution, so that their Grothendieck groups carry canonical ring structures. 

The $G$-orbits on $\cB\times \cB$, or equivalently, the $B$-orbits on $\cB$, are indexed by the abstract Weyl group $W$ of $G$. The length function $l: W\to \ZZ_{\ge0}$ is defined so that $l(w)$ is the dimension of the $B$-orbit $\cO_w\subset \cB$. Let $S=l^{-1}(1)\subset W$. Then $(W,S)$ is a Coxeter system with $S$ as the set of simple reflections.

We have the standard and IC sheaves in $\cH$
\[
\Delta_w=j_{w,!}(\underline{\Ql} )\braket{l(w)}\quad\mbox{and} \quad \IC_w=j_{w,!*}(\underline{\Ql} )\braket{l(w)}.
\]

The following facts about the Hecke categories are well-known (cf. \cite[\S 4]{BGS} and \cite[\S 7.5]{MR4337423}).
\begin{prop}
\begin{enumerate}
    \item The full subcategory $\cH^{\Tate} \subset\cH$ is closed under convolution. In particular, it inherits a monoidal structure from $\cH$, and $K_0(\cH^\Tate)$ carries a canonical ring structure.
    \item For all $w\in W$, we have $\IC_w$ is Tate, $*$-pure of weight zero, and $*$-parity with parity $l(w)$. In particular, $\IC_w\in \cH^{\Tate}$.
\end{enumerate}
\end{prop}

Recall that the generic Hecke algebra $\sfH$ of the Coxeter system $(W,S)$ is defined 
to be the unique associative $R$-algebra with $R$-basis $\set{\sfT_{w}  |  w\in W}$
  such that
\begin{enumerate}[label=(\alph*),wide=0pt]
\item $(\sfT_{s}+1)(\sfT_{s}-v^2)=0$ for $s\in S$,
\item $\sfT_{w_1} \sfT_{w_2} = \sfT_{w_1w_2}$ if 
$l(w_1w_2) = l(w_1)+l(w_2)$. 
\end{enumerate}

We identify $R[\un{B\bs \cB}]$ with the underlying $R$-module of the generic Hecke algebra $\sfH$ so that $\one_w$ corresponds to the standard basis $\sfT_w$. Then we have 
\begin{equation}\label{p:ch Hecke}
        \ch(\D_w)=(-v)^{-l(w)}\sfT_w.
\end{equation}
We may view $\ch$ as a map
\begin{equation*}
    \ch: K_0(\cH^\mix)\to \sfH.
\end{equation*}
\trivial[h]{
If we following \Cref{ss:KLWCell} to define the Kazhdan-Lusztig polynomial $P_{y,w}$. Then we have 
\[
C'_w= v^{-l(w)}\sum_{y\le w}P_{y,w}(-v)\sfT_y.
\]
In this case, $\IC_w$ is $*$-parity of $l(w)$, hence $P_{y,w}(v)$ is an even polynomial. I think in Kazhdan-Lusztig's paper, the Kazhdan-Lusztig polynomial $P'_{y,w}$ (I use this for distinguish) is defined to be 
\[
P'_{y,w}(v)= P_{y,w}(\sqrt v)
\]
Therefore, the above equation rewrite as 
\[
C'_w= v^{-l(w)}\sum_{y\le w}P'_{y,w}(v^2)\sfT_y.
\]
Why the parity is important here. 
}

Similarly, we identify $\bZ[\un{B\bs \cB}]$ with $\bZ[W]$ and identify $\Qlbar[\un{B\bs \cB}]$ with the Iwahori-Hecke algebra $H$ (using also the isomorphism $\io:\Qlbar\cong \bC$). Thus we may view $\chi$ and $\phi$ as maps
\begin{equation}\label{hecke algebra iso}
\begin{split}
    \chi: K_0(\cH^k)\to \bZ[W],\\
    \phi: K_0(\cH)\to H.
\end{split}
\end{equation}

\begin{lemma}
    The maps $\chi, \ch$ and $\phi$ above are ring homomorphisms. 
\end{lemma}
\begin{proof}
    For $\phi$, it follows from the functoriality of the sheaf-to-function correspondence with respect to $\ot$, $*$-pullback and $!$ pushforward.

    For $\ch$, using that the classes of $\D_w(\l)$ (where $(\l)$ means twisting the Weil structure by a scalar $\l\in \Qlbar^\times$) generate the abelian group $K_0(\cH^\mix)$ for various $w\in W$ and $\l\in \Qlbar^\times$, it suffices to check that
    \begin{equation}\label{ch std}
        \ch(\D_{w_1}\star \D_{w_2})=\ch(\D_{w_1})\ch(\D_{w_2})
    \end{equation}
    for $w_1,w_2\in W$. Using a reduced word for $w_1$, it suffices to prove \eqref{ch std} for $w_1=s$ a simple reflection, which by using \Cref{p:ch Hecke}, is equivalent to
    \begin{equation}\label{ch sw}
        \ch(\D_{s}\star \D_{w_2})=(-v)^{-1-l(w_2)}\sfT_{s}\sfT_{w_2}.
    \end{equation}
    If $l(sw_2)=l(w_2)+1$, we have $\D_s\star\D_{w_2}\cong \D_{sw_2}$, hence both sides of of \eqref{ch sw} are equal to $(-v)^{-l(sw_2)}\sfT_{sw_2}$.

    If $l(sw_2)=l(w_2)-1$, we write $w_2=sw$, so that $\D_{w_2}\cong \D_s\star\D_w$. The left side of \eqref{ch sw} is then $\ch(\D_s\star\D_s\star \D_w)$. Using the relation $[\IC_s]=[\D_s]-[\D_1(\half)]$ in  $K_0(\cH^\mix)$, and the obvious isomorphism $\IC_s\star\IC_s\cong \IC_s\langle{1\rangle}\oplus \IC_s\langle{-1\rangle}$, we see that $[\D_s\star\D_s]=[\D_s(\half)]-[\D_s(-\half
    )]+[\D_1]\in K_0(\cH^\mix)$. This implies
    \begin{equation*}
        \ch(\D_s\star\D_s\star \D_w)=(v^{-1}-v)\ch(\D_s\star\D_w)+\ch(\D_w).
    \end{equation*}
    Using the case already proved (i.e., \eqref{ch sw} holds for $\ch(\D_s\star\D_w)$), we see that the right side above is $(-v)^{-l(sw)}(v^{-1}-v)\sfT_{sw}+(-v)^{-l(w)}\sfT_w$, which by a direct calculation is equal to the right side of \eqref{ch std}. 
    
    To show that $\chi$ is a ring homomorphism, we use the lower square of the diagram \eqref{diag cat}. Since the natural map $K_0(\cH^{\mix})\to K_0(\cH^{k})$ is a surjective ring homomorphism, and $\ch$ is known to be a ring homomorphism, so is $\chi$.

\end{proof}

Combined with  \Cref{Lem. ch isom}, we conclude:
\begin{cor}\label{Cor. Hk ring isom}
\begin{enumerate}
    \item The map 
    \begin{equation*}
        \ch^\Tate: K_0(\cH^\Tate)\to \sfH
    \end{equation*}
    is an isomorphism of $R$-algebras.
 
    \item The map 
    \begin{equation*}
        \chi: K_0(\cH^k)\to \bZ[W]
    \end{equation*}
    is a ring isomorphism.
\end{enumerate} 
\end{cor}
Finally,  we note that $C'_w=(-1)^{-l(w)} \ch(\IC_w)$ defined according to \Cref{define Kazhan-Lustig} is the Kazhdan-Lusztig basis of $\sfH$ defined in \cite[Equation 1.1.c]{MR0560412}.

\subsection{Geometrization of Hecke modules from spherical varieties}\label{Sec. Hk mod}
In the set up of \Cref{Sec. Cat gen}, assume $G$ is a split reductive group over $\bF_q$, and $B\subset G$ is a Borel subgroup. Suppose the $B$-action on $X$ extends to a $G$-action. When $X$ is irreducible, the finiteness of the number of $B$-orbits implies that $X$ is a spherical $G$-variety. The assumption \eqref{orbit assumption} is still in effect.

We introduce the various categories 
\begin{equation*}
\cM^{\Tate}=D^{\Tate}_B(X)\hookrightarrow \cM^{\mix}=D^{\mix}_B(X)\hookrightarrow
\cM=D^{b}_B(X)\to\cM^k=D^{b}_{B_k}(X_k).
\end{equation*}
We also denote the various spaces $\sfM\coloneqq R[\un{B\bs X}]$ and $M\coloneqq \Qlbar[\un{B\bs X}]=\bC[\un{B\bs X}]=\bC[B(\bF_q)\backslash X(\bF_q)]$ (using the isomorphism $\io:\Qlbar\cong \bC$).

Let $A$ be a commutative ring with an element $r\in A$, we denote by 
$\sfM_{v=r,A}\coloneqq \sfM\otimes_{R} A$ the extension of scalar along the specialization map $\bZ[v,v^{-1}]\rightarrow A$ (sending $v$ to $r\in A$). Note that $\sfM_{v=r,A}$ is naturally a $\sfH_{v=r,A}\coloneqq \sfH\otimes_{R} A$-module. We identify $M=\sfM_{v=\sqrt{q},\bC}=\sfM_{v=\sqrt{q},\Qlbar}$ as $H=\sfH_{v=\sqrt{q},\bC}=\sfH_{v=\sqrt{q},\Qlbar}$ modules (using the isomorphism $\io:\Qlbar\cong \bC$).

Now $\cH, \cH^{\mix}$ and  $\cH^k$ act on $\cM, \cM^{\mix}$ and $\cM^k$ respectively by convolution by the following formula: if $\sE\in \cH$ and $\sF\in\cM$, their convolution $\sE\star\sF\in \cM$ is 
\begin{equation*}
    \sE\star\sF=a_!(p_G^*\sE\ot p_X^*\sF).
\end{equation*}
where the maps are as follows
\begin{equation*}
    \xymatrix{& B\bs G\times^B X\ar[rr]^{a:(g,x)\mapsto gx}\ar[dl]_{p_G}\ar[dr]^{p_X} && B\bs X\\
    B\bs G/B & & B\bs X}
\end{equation*}

We recall some notions of the spherical varieties from \cite[\S 3 and \S 4]{MR1324631}. Let $s\in W$ be a simple reflection, and $P_s$ be the parabolic subgroup of $G$ containing $B$ whose Levi quotient has roots $\pm\a_s$. 
Let $\widetilde a_s: P_s\times^B X\to X$ is the action map. Let $\a\in I$, and $D_{x_\a}=\widetilde a_s^{-1}(x_\a)$, which is isomorphic to $P_s/B\cong \bP^1$. Now $D_{x_\a}$ has an action of $P_{s,x_\a}=\Stab_{P_s}(x_\a)$ with finitely many orbits. The action of $P_{s,x_\a}$ on $D_{x_\a}$ factors through a quotient $Q_{x_\a}\subset \Aut(D_{x_\a})\cong \PGL_2$, in which $Q_{x_\a}$ is a spherical subgroup. According to the possibilities of $Q_{x_\a}$ we call $(s,X_\a)$
    \begin{enumerate}
        \item[\textbf {Type (G)}:] if $Q_{x_\a}=\PGL_2$. In this case  $D_{x_\a}$ is a single $Q_{x_\a}$-orbit, and $\widetilde a_s^{-1}(X_\a)$ is a single $B$-orbit.
        \item[\textbf {Type (U)}:] if $Q_{x_\a}$ is contained in a Borel subgroup of $\PGL_2$ and contains its unipotent radical. Then $D_{x_\a}$ decomposes into two $Q_{x_\a}$-orbits $D^\circ_{x_\a}\cup \{\infty_{x}\}$, where $D^\circ_{x_\a}\cong \bA^1$. Accordingly,  $\widetilde a_s^{-1}(X_\a)$  contains two $B$-orbits. 
        
        \item[\textbf {Type (T)}:] if $Q_{x_\a}$ is a maximal torus of $\PGL_2$. Then $D_{x_\a}$ decomposes into three $Q_{x_\a}$-orbits $D^\circ_{x_\a}\cup \{\infty_{x_\a}\}\cup \{0_{x_\a}\}$, where $D^\circ_{x_\a}\cong \bG_m$. Accordingly,  $\widetilde a_s^{-1}(X_\a)$  contains three $B$-orbits, one open and two closed.       
    \end{enumerate}
    
\begin{remark} The so-called type (N) case cannot occur under our assumption \eqref{orbit assumption}:  if $Q_{x_\a}$ was the normalizer of a maximal torus of $\PGL_2$, then the stabilizer of $y$ in the open $Q_{x_\a}$-orbit of $D_{x_\a}$ would be disconnected, which would imply that $\Stab_B(y)$ was disconnected. 
\end{remark}

\begin{defn}
    When $(s,X_\a)$ is of type (U) or (T), any orbit $X_{\a'}$ that intersects $D_{x_\a}$ (equivalently $X_{\a'}\subset P_s\cdot 
    X_\a$) is called a {\em $s$-companion} of $X_\a$. By convention, $X_\a$ is its own $s$-companion. We use $\b\sim_s\a$ to denote that $X_\b$ is an $s$-companion of $X_\a$. 
\end{defn}

We say  $(s,X_\a)$ is of type (U+) (resp.  (T+)) if $(s,X_\a)$ is of type (U) (resp. (T)) and $X_\a$ has larger dimension than its $s$-companions; otherwise we say $(s,X_\a)$ is of type (U-) (resp.  (T-)).
    
\begin{prop}\label{Prop ICs preserves Tate}
    The full subcategory $\cM^\Tate\subset\cM$ is stable under the convolution by $\cH^\Tate$. In particular, $K_0(\cM^\Tate)$ is a module over $K_0(\cH^\Tate)$. Via the isomorphism $\ch^\Tate: K_0(\cM^\Tate)\cong \sfM$ (by \Cref{Lem. ch isom}), $\sfM$ has a structure of an $\sfH$-module.
    
    Explicitly, the action of $\sfT_s+1\in \sfH$ (for any $s\in S$) on the standard basis of $\sfM$ is given by
    \begin{equation*}
        (\sfT_s+1)\one_\a=
        \begin{cases}
            (v^2+1)\one_\a & \mbox{$(s,X_\a)$ is of type (G),}\\
            v^2\sum_{\b\sim_s \a}\one_\b & \mbox{$(s,X_\a)$ is of type (U+),}\\

            (v^2-1)\sum_{\b\sim_s \a}\one_\b & \mbox{$(s,X_\a)$ is of type (T+),}\\
            \sum_{\b\sim_s \a}\one_\b   &\mbox{$(s,X_\a)$ is of type (U-) or (T-),}\\
        \end{cases}
    \end{equation*}
\end{prop} 

\begin{proof} 
    
    Since $\cH^\Tate$ is generated by Tate twists of $\IC_w$,  for the first statement it suffices to show that $\IC_w\star(-)$ preserves $\cM^\Tate$. Since $\IC_w$ is a direct summand of a successive convolution of $\IC_s$ for simple reflections $s$ \footnote{This is true in $\cH^k$ by the decomposition theorem. For the purpose of checking $\IC_w\star \sF$ is Tate, one can work after Frobenius semisimplification, in which case $\IC_w$ becomes a direct summand of a successive extension of $\IC_s$. The same direct summand statement is indeed true in $\cH^\mix$ without Frobenius semisimplification by \cite{MR3003920}, although it is not necessarily to appeal to this stronger statement here.}, it suffices to show that $\IC_s\star(-)$ preserves $\cM^\Tate$.

    Let $a_s: B\bs  X\to P_s\bs X$ be the projection. The fiber $a_s^{-1}(x_\a)$ is $D_{x_\a}$.
    Since $\IC_s\langle -1\rangle$ is the constant sheaf on $B\bs P_s/B$, we see that for $\sF\in \cM$, \begin{equation}\label{conv ICs}
        \IC_s\star \sF\langle -1\rangle\cong  a_s^* a_{s !}\sF.
    \end{equation}
    By proper base change for the map $a_s$, the stalk of $\IC_s\star \sF$ at $x_\a$ is    
    \begin{equation}\label{stalk ICs F}
        (\IC_s\star \sF)_{x_\a}\cong R\Gamma(D_{x_\a}, \sF|_{D_{x_\a}})\langle 1 \rangle.
    \end{equation}
    We discuss each type of $(s,X_\a)$.
    For type (G), $\sF|_{D_{x_\a}}$ is geometrically constant, we have
    \begin{equation}\label{type G stalk}
        H^*(D_{x_\a}, \sF|_{D_{x_\a}})\cong \sF_{x_\a}\otimes H^*(\bP^1,\Qlbar).
    \end{equation} 
    So if $\sF$ is Tate along $X_\a$, so is  $H^*(D_{x_\a}, \sF|_{D_{x_\a}})$, hence the same is true for $\IC_s\star\sF$ along $X_\a$ by \eqref{stalk ICs F}. 

    For type (U),(T), we assume $(s,X_\a)$ is of type (U+), (T+), and let $X_{\a'}$ (or $X_{\a'}$ and $X_{\a''}$ in type (T)) be its $s$-companions. Let $D^\circ_{x_\a}=D_{x_\a}\cap X_\a$ and $D^\bullet_{x_\a}$ be its complement. 
    In this case
    $R\Gamma(D_{x_\a}, \sF|_{D_{x_\a}})$ fits into a distinguished triangle
    \begin{equation}\label{RG tri}
        \xymatrix{R\Gamma_c(D^\circ_{x_\a}, \sF|_{D^\circ_{x_\a}}) \ar[r] & R\Gamma(D_{x_\a}, \sF|_{D_{x_\a}})\ar[r] & \sF|_{D^\bullet_{x_\a}}\ar[r] & }
    \end{equation}
    Now if $\sF$ is Tate, the third term above is Tate because $D^\bullet_{x_\a}$ consists of either one or two points defined over $\bF_q$.  The first term above is $H^*_c(D^\circ_{x_\a})\ot \sF|_{D^\circ_{x_\a}}$ which is also Tate since $D^\circ_{x_\a}\cong \bA^1$ or $\Gm$ over $\bF_q$.  We conclude that $\IC_s\star \sF$ is Tate along any orbit $X_\a$.

    The formula for the $\sfT_s+1$-action are obtained as follows. 
    Apply \eqref{conv ICs} to $\sF=j_{\a!}\Qlbar$, and we need to express the class of $a_s^*a_{s !}j_{\a!}\Qlbar$ in terms of a linear combination of $j_{\b!}\Qlbar$, i.e., we need to compute the stalks $(a_s^*a_{s !}j_{\a!}\Qlbar)|_{x_\b}$. This amounts to computing $H^*_c(D_{x_\b}\cap X_\a)$, which can be done case by case.
\end{proof}

The next statement follows directly from the construction. 
\begin{lemma}\label{Lem. Hecke mod specialization}
    The specialization maps $\sfM=R[\un{B\bs X}]\xr{v=\sqrt q} M$ and $\sfM=R[\un{B\bs X}]\xr{v=1} \bZ[\un{B\bs X}]$ are compatible with the $\sfH$-actions, which acts on $M$ and on $\bZ[\un{B\bs X}]$ via the specializations maps for Hecke algebras $\sfH\xr{v=\sqrt q}H$ and $\sfH\xr{v=1}\bZ[W]$.
\end{lemma}

There is a natural action of $H=\bC[B(\bF_q)\backslash G(\bF_q)/B(\bF_q)$ on $M=\bC[B(\bF_q)\backslash X(\bF_q)]$ by the
convolution defined by \Cref{hecke function covolution}. 

\begin{lemma}\label{lem two hecke action}
 The $H$ action on $M$ via applying \Cref{Lem. Hecke mod specialization} to the specialization $v=\sqrt{q}$ and via \Cref{hecke function covolution} coincide using the isomorphism $\io:\Qlbar\cong \bC$.
\end{lemma}
\begin{proof}
  This follows from the functoriality of the sheaf-to-function correspondence with respect to  $\ot$, $*$-pullback and $!$ pushforward. 
\end{proof}

Applying \Cref{Prop ICs preserves Tate} and \Cref{Lem. Hecke mod specialization} to the specialization $v=1$, we recover the calculations of Knop \cite[\S 5]{MR1324631}.
\begin{cor}\label{Cor. W action sph orb}
    The action of $\bZ[W]\cong K_0(\cH^k)$ on $\bZ[\un{B\bs X}]\cong K_0(\cM^k)$ is given on simple reflections by
    \begin{equation*}
        s\cdot \one_\a=\begin{cases}
            \one_\a & \mbox{$(s,X_\a)$ is of type (G),} \\
            \one_{\a'} & \mbox{$(s,X_\a)$ is of type (U), $\a\ne \a'\sim_s \a$}, \\
            
            -\one_{\a}, &
            \mbox{$(s,X_\a)$ is of type (T+)}, \\
            \one_{\a'}+\one_{\a''}, &
            \mbox{$(s,X_\a)$ is of type (T-), $\{\a, \a',\a''\}$ are $s$-companions of $\a$},\\
        \end{cases}
    \end{equation*}
\end{cor}

\begin{defn} Let $s\in S$ and $\a\in I$. We say $X_\a$ is {\em $s$-vertical} if  $(s,X_\a)$ is of type (G), (U+) or (T+). Let $\cD(\a)=\{s\in S|(s,X_\a) \mbox{ is $s$-vertical}\}$.
\end{defn}

\begin{lemma}\label{Lem. conv ICs}
    Let $s\in S$ and $\a\in I$. 
    \begin{enumerate}
        \item If $s\in \cD(\a)$, then
        \begin{equation*}
            \IC_s\star\IC_\a\cong \IC_\a\langle 1\rangle\oplus\IC_\a\langle -1\rangle.
        \end{equation*}
        \item If $s\notin \cD(\a)$, 
        let $X_{\a'}$ be the unique $s$-companion of $X_\a$ such that $\a'>\a$. 
        Then there is a canonical decomposition in $\cM$
        \begin{equation}\label{ICs decomp}
        \IC_s\star\IC_\a\cong\IC_{\a'}\oplus\left(\bigoplus_{\b\in I, \b\le \a, s\in \cD(\b)}\IC_\b\ot H^{-d_\b-1}(\IC_\a|_{x_\b})(-\frac{d_\b+1}{2})\right).
        \end{equation}
    \end{enumerate}
\end{lemma}
\begin{proof} We use \eqref{conv ICs} to compute $\IC_s\star \IC_{\a}$. If $s\in \cD(\a)$, $\IC_\a\cong a_s^*\sF$ for a simple perverse sheaf $\sF$, therefore $\IC_s\star \IC_\a\cong a_s^*a_{s*}a_s^*\sF\langle 1\rangle\cong a_s^*(\sF\oplus \sF\langle -2\rangle)\langle 1\rangle\cong a_s^*\sF\langle 1\rangle\oplus \sF\langle -1\rangle$, using that $a_s$ is a $\bP^1$-fibration.

Now assume $s\notin\cD(\a)$. We first show that $\IC_s\star \IC_\a$ is perverse. 

Let $\b\in I$. When $(s,X_\b)$ is of type (G) (in particular, $\b\ne\a$), we easily see from \eqref{stalk ICs F} that
\begin{equation*}
    (\IC_s\star\IC_\a)|_{x_\b}\cong \IC_\a|_{x_\b}\ot H^*(\bP^1)\langle 1\rangle\cong \IC_\a|_{x_\b}\langle 1\rangle\oplus \IC_\a|_{x_\b}\langle-1\rangle.
\end{equation*}
Since $\IC_\a|_{x_\b}$ lies in degrees $\le -d_\b-1$,  $(\IC_s\star\IC_\a)|_{x_\b}$ lies in degrees $\le -d_\b$, and we have a canonical isomorphism
\begin{equation}\label{type G b top}
    H^{-d_\b}(\IC_s\star\IC_\a)|_{x_\b}\cong H^{-d_\b-1}(\IC_\a|_{x_\b} )(1/2).
\end{equation}
If $\b=\a'$ then $D_{x_{\a'}}\cap \ov X_\a$ consists of a single $\bF_q$-point. This implies
\begin{equation}\label{ICs conv ICa big orbit}
    (\IC_s\star\IC_\a)|_{x_{\a'}}\cong \Qlbar\langle d_{\a'}\rangle.
\end{equation}
Next, consider the case $\b\ne \a'$ and $(s,X_\b)$ is of type (U+) or (T+). By the distinguished triangle \eqref{RG tri} we have
\begin{equation*}
    \xymatrix{\IC_\a|_{x_\b}\ot H^*_c(D^\circ_{x_\b})\langle 1\rangle\ar[r] & (\IC_s\star \IC_\a)|_{x_\b} \ar[r] & \bigoplus_{\b'\sim_s\b, s\notin \cD(\b')}\IC_\a|_{x_{\b'}}\langle 1\rangle\ar[r] & }.
\end{equation*}
Using that $D^\circ_{x_\b}\cong \bA^1$ or $\Gm$, the first term above are in degrees $\le -d_\b$, and each summand in the third term is in degrees $\le -d_{\b'}-2<-d_\b$ because $d_{\b'}=d_\b-1$.  From this we see that the highest possible degree of $(\IC_s\star \IC_\a)|_{x_\b}$ is $-d_\b$, and we have a canonical isomorphism
\begin{equation}\label{type U b top}
    H^{-d_\b}(\IC_s\star\IC_\a)|_{x_\b}\cong H^{-d_\b-1}(\IC_\a|_{x_\b})(1/2).
\end{equation}
When $(s,X_\b)$ is of type (U-) or (T-), let $\b'$ be the unique $s$-companion of $\b$ such that $\b'>\b$.  The argument above shows that $(\IC_s\star \IC_\a)|_{x_\b}$ are in degrees $\le -d_{\b'}=-d_{\b}-1$. These degree estimates show that $\IC_s\star\IC_\a\in {}^pD^{\le 0}(X)$. By the Verdier self-duality of $\IC_s\star \IC_\a$, we conclude it is perverse. 

As a perverse sheaf, $\IC_s\star \IC_\a$ has a canonical decomposition according to the support. The multiplicity spaces that appear in \eqref{ICs decomp} follows from \eqref{type G b top}, \eqref{ICs conv ICa big orbit} and \eqref{type U b top}. 
\end{proof}

\begin{cor}
We have 
\[
\dim \Supp(\IC_w \star \IC_\alpha) \leq l(w)+d_\alpha 
\] 
for every $w\in W$ and $\alpha\in I$. 
\end{cor}
\begin{proof}
This follows from \Cref{Lem. conv ICs} by induction on $l(w)$. 
\end{proof}

\begin{defn}\label{def. I0}
We define the following subset of $I=\un{B\bs X}$:
\begin{equation*}
    I_0=\{\a\in I \mid \text{for all $s\in S$, $(s,X_\a)$ is of type (G), (U-), or (T-)}\}.
\end{equation*}
\end{defn}
$I_0$ consists of the minimal dimensional orbits in the equivalence classes generated by the relation of companions under simple reflections.

We show $\{\IC_\b|\b\in I_0\}$ is a set of generating objects for $\cM$ under the $\cH$-action in the following sense:
\begin{lemma}\label{Lem. heart gen}
\begin{enumerate}
    \item For any $\a\in I$, $\IC_\a$ is a direct summand of $\IC_w\star \IC_\b$ for some $w\in W$ and $\b\in I_0$ 
    such that $d_\a-d_\b=l(w)$.
    \item As an $\sfH$-module, $\sfM$ is generated by $\{C'_\b|\b\in I_0\}$. 
\end{enumerate}
\end{lemma}
\begin{proof}
\begin{enumerate}[wide]
  \item 
  We prove by induction on $d_\alpha$. 
If $X_\a$ is an orbit having with minimal dimension, then $\a\in I_0$ and there is nothing to prove. 
Now assume that the 
statement is proved for all $\a'$ such that $d_{\a'}<d_\a$, we show that it is true for $\a$ as follows. It suffices to consider the case when $\a\notin I_0$ for otherwise there is nothing to show. 
Since $\a\notin I_0$, there exists $s\in S$ such that $(s,X_\a)$ is of type (U+), (T+) or (N+). Let $\a'\ne \a$ be an $s$-companion of $\a$, then $d_{\a'}=d_\a-1$. Apply \Cref{Lem. conv ICs}~(2) (the role of $\a$ and $\a'$ are switched now), we conclude that $\IC_\a$ is a direct summand of $\IC_s\star\IC_{\a'}$. By induction hypothesis, $\IC_{\a'}$ is a direct summand of $\IC_u\star \IC_\b$ for some $u\in W$ and $\b\in I_0$ such that $l(u)=d_\a'-d_\b$. Therefore $\IC_\a$ is a direct summand of $\IC_s\star\IC_u\star\IC_{\b}$. The support of $\IC_s\star\IC_u\star\IC_{\b}$ contains $X_\alpha$ and has dimension $\le l(s)+l(u)+d_\b=1+d_{\a'}=d_\a$, so $l(s)+l(u) + d_\beta = d_\alpha = (d'_\alpha -d_\beta) + 1$. Hence $l(su)=1+l(u)=d_\alpha - d_\beta$.
Note that $\IC_{\alpha}$ cannot occur as a direct summand of $\IC_{u'}\star \IC_\beta$ for $l(u') \leq  l(u)$, since in that case the support dimension would be strictly less than $d_\alpha$.
Therefore $\IC_{\alpha}$ must be a direct summand of $\IC_{su}\star \IC_\b$, 
since $\IC_{s}\star \IC_{u}$ is a direct sum of $\IC_{u'}$ with $l(u')\leq l(u)$ except one summand $\IC_{su}$. 
\item
    By induction on $d_\alpha$ and part (1), it is clear that, for each $\alpha\in I$,  $C'_\alpha$ can be expressed as linear combinations of $C'_w C'_\beta$ where $w\in W$ and $\beta\in I_0$. 
\end{enumerate}
\end{proof}

\begin{prop}\label{Prop ICs preserves purity and parity} Suppose for all $s\in S$ and $\a\in I$,  $(s,X_\a)$ is of type (G) or (U). Let $w\in W$ and $\sF\in \cM$.
    \begin{enumerate}
        \item If $\sF$ is $*$-pure of weight zero, then the same is true for $\IC_w\star\sF$.
        \item If $\sF$ is $*$-parity with parity $\e$, then $\IC_w\star\sF$ is $*$-parity with parity $\e+l(w)$.
    \end{enumerate}
\end{prop}
\begin{proof}
    As in the proof of \Cref{Prop ICs preserves Tate}, it suffices to treat the case $w=s\in S$. 
    Using notation from the proof of \Cref{Prop ICs preserves Tate}, we compute the stalk $\IC_s\star \sF$ at $x_\a$. If $(s,X_\a)$ is of type (G), \eqref{type G stalk} implies the purity and parity of $\IC_s\star\sF$ along $X_\a$.

    If $(s,X_\a)$ is of type (U), we may assume it is of type (U+) for  $\IC_s\star\sF$ is constant along $P_s X_\a$.  The distinguished triangle \eqref{RG tri} has first term equal to $\sF|_{x_\a}\ot H^*_c(\bA^1)\cong \sF|_{x_\a}[-2](-1)$, and third term $\sF|_{x_{\a'}}$, where $\a\ne \a'\sim_s \a$. We conclude that $\IC_s\star \sF$ has the required purity and parity properties along $X_\a$.  This completes the argument.
\end{proof}

By \Cref{Lem. heart gen}, \Cref{Prop ICs preserves Tate} and \Cref{Prop ICs preserves purity and parity}, we deduce the following.
\begin{cor}\label{Cor. I0 implies all I}
\begin{enumerate}
    \item If $\IC_\a$ is Tate for all $\a\in I_0$, the same is true for all $\a\in I$.
    \item Suppose for all $s\in S$ and $\a\in I$,  $(s,X_\a)$ is of type (G) or (U). If $\IC_\a$ is $*$-pure of weight zero for all $\a\in I_0$, the same is true for all $\a\in I$.
    \item Suppose for all $s\in S$ and $\a\in I$,  $(s,X_\a)$ is of type (G) or (U). If $\IC_\a$ is $*$-parity of partiy $d_\a$ for all $\a\in I_0$, the same is true for all $\a\in I$.
\end{enumerate}
\end{cor}

\subsection{Kazhdan-Lusztig polynomials,  $W$-graph and cells attached to $X$}\label{ss:KLWCell}

We introduce the analog of Kazhdan-Lusztig polynomials in this context. For $\a,\b\in I$, let 
\begin{equation*}
    P_{\b,\a}(t)=\sum_{i\in \bZ}(-1)^ i\dim H^{-d_{\a}+i}(j_\b^*\IC_\a) t^i=(-t)^{d_\a} \sum_{i\in \bZ}\dim (-1)^ i H^{i}(j_\b^*\IC_\a) t^i \in \bZ[t].
\end{equation*}
\trivial[h]{If not, use 
\[
   P_{\b,\a}(t)=\sum_{i\in \bZ}\dim H^{-d_{\a}+i}(j_\b^*\IC_\a) t^i=t^{d_\a} \sum_{i\in \bZ}\dim H^{i}(j_\b^*\IC_\a) t^i \in \bZ_{\ge0}[t].
\]
}
From the definition of IC sheaves, we see that $P_{\b,\a}=0$ unless $\b\le \a$, $P_{\a,\a}=1$ and if $\b<\a$, then $P_{\b,\a}$ has degree at most $d_\a-d_\b-1$. Note that if $\overline X_\a$ is smooth, then $\IC_\a=\underline{\Ql}_{\overline X_\a}\braket{d_\alpha}$ and we  have 
\begin{equation}\label{Kazhan-Lustig smooth}
    P_{\b,\a}(t)=
    \begin{cases}
        1 \quad \mbox{if $\b\le \a$}\\
        0 \quad  \mbox{otherwise}.
    \end{cases}
\end{equation}

Recall that we defined the Kazhdan-Lusztig basis $ C'_\a=(-1)^{-d_\a} \ch(\IC_\a)$ of $\sfM$ in \Cref{define Kazhan-Lustig}. The following is immediate from the definitions.
\begin{lemma}\label{Lem. KL poly ch} If $\IC_\a$ is $*$-pure of weight zero, then
\begin{equation*}
    C'_\a=v^{-d_\a}\sum_{\b\le \a}P_{\b,\a}(v)\one_{\b}\in \sfM.
\end{equation*}
\trivial[h]{If not, use 
\begin{equation*}
    C'_\a=v^{-d_\a}\sum_{\b\le \a}P_{\b,\a}(-v)\one_{\b}\in \sfM.
\end{equation*}
}
If $\IC_\a$ is $*$-parity with parity $d_\a$, then $P_{\b,\a}(v)$ has only even degrees terms, and hence $P_{\b,\a}(v)\in \bZ_{\geq 0}[t]$.
\end{lemma}

    
The notion of a $W$-graph (undirected graph with labels on vertices and edges) is introduced by Kazhdan and Lusztig in \cite[\S 1]{MR0560412}. We define a slightly variation of a $W$-graph $\bG(X)$ (in that it is directed) attached to the spherical $G$-variety $X$ as follows. 

\begin{cons}\label{cons:W graph}
The vertex set of $\bG(X)$ we take to be $I$, and each vertex $\a\in I$ is labeled by the set $\cD(\a)\subset S$. There is an edge  $e=(\a\rightarrow \b)$ if the following conditions are satisfied
\begin{itemize}
   \item $\cD(\b)\not\subset\cD(\a)$;
   \item (Degree condition) Either $\beta<\alpha$, in which case $\deg P_{\beta,\alpha}(t)=d_\alpha-d_\beta-1$, or $\beta>\alpha$, in which case $\deg P_{\alpha,\beta}(t)=d_\beta-d_\alpha-1$. In both cases we define $\mu(e)$ to be the leading coefficient of the relevant polynomial $P_{\beta,\alpha}(t)$ or $P_{\alpha,\beta}(t)$. Note that if $d_\alpha=d_\beta+1$ in the first case (or $d_\beta=d_\alpha+1$ in the second), the degree condition is automatic
\end{itemize}
We denote the set of directed edges by $E$, and denote the labeled directed graph thus defined by
\begin{equation*}
    \bG(X)=(I, \{\cD(\a)\}_{\a\in I}, E, \{\mu(e)\}_{e\in E}).
\end{equation*}
\end{cons}
\trivial[h]{
If $\b< \a$ with $d_\a=d_\b+1$. Then  $e=(\a\leftrightarrow \b)\in E$. If $\overline {\cO_\a}$ is normal, then $\mu(e)=1$. 
}

The following result says essentially that $\bG(X)$ is a $W$-graph in the sense of Kazhdan-Lusztig, although we use the basis $\{C'_\a\}$ for the $\sfH$-module $R[I]$ which differs from theirs.

\begin{prop}\label{Cor. Hk action on KL basis} 
Suppose $\IC_\a$ is $*$-pure of weight zero for all $\a\in I$. Then the action of $\sfH$ on $R[\un{B\bs X}]$ is determined by the formulas
    \begin{equation*}
        \sfT_s\cdot C'_\a=\begin{cases}
            v^2C'_\a & s\in \cD(\a)\\
            -C'_\a+v\sum_{e=(\a\rightarrow \b)\in E, s\in \cD(\b)}\mu(e)C'_\b, & s\notin \cD(\a).
        \end{cases}
    \end{equation*}
\end{prop}
\begin{proof}
This is a direct consequence of \Cref{Lem. conv ICs}, with one extra point to verify: if there is an edge $e=(\alpha\to\beta)$ with $\alpha<\beta$ and if $s\in \cD(\beta)\setminus \cD(\alpha)$, then $\beta\sim_s \alpha$ and $\mu(e)=1$.

We first consider the case that $(s,X_\alpha)$ is of type $U^-$. Let $\gamma\sim_s\alpha$ be the $s$–companion of $\alpha$ with $s\in \cD(\gamma)$. Then $d_\gamma=d_\alpha+1$. 
By \Cref{Lem. conv ICs}(1),
\begin{equation}\label{action 1}
  \sfT_s\cdot C'_\beta = v^2\, C'_\beta .
\end{equation}
 By \Cref{Prop ICs preserves Tate},
\begin{equation}\label{action 2}
  \sfT_s\, \one_\alpha = \one_\gamma,\qquad
  \sfT_s\, \one_\gamma = v^2\, \one_\alpha + (v^2-1)\,\one_\gamma .
\end{equation}
Substitute the expansion of $C'_\beta$ from \Cref{Lem. KL poly ch} into \eqref{action 1} and compare the $\one_\alpha$– and $\one_\gamma$–coefficients using \eqref{action 2}. Since $\one_\gamma$ appears on the left by \Cref{action 2}, it must also appear on the right, hence $\gamma\leq \beta$. Then comparing the $\one_\alpha$–coefficients give
\[
  P_{\gamma,\beta}(v)=P_{\alpha,\beta}(v).
\]
\trivial[h]{From Zhiwei, this identity simply follows from the fact that $\IC_\beta$ is constant along the fibre of $X/B \rightarrow X/P_s$. Therefore, taking the stalk at $\alpha$ and $\gamma$ are the same.
}
If $\gamma\neq \beta$, then 
\[
\deg P_{\alpha,\beta}(t)=\deg P_{\gamma,\beta}(t)\le d_\beta-d_\gamma-1 < d_\beta-d_\alpha-1,
\]
which contradicts the degree constraint recalled above. Hence $\gamma=\beta$ and $P_{\alpha,\beta}(v)=P_{\beta,\beta}(v)=1$. In particular $\mu(e)=1$.

The case when $(s,X_\alpha)$ is of type $T^-$ follows by a similar argument. 
\trivial[h]{
Next we consider case when $(s,X_\a)$ is of type $T^-$. Let $\a'\sim_s \a$ be the companion of $\a$ such that $(s,X_{\a'})$ is also of type $T^-$. Then we have 
\[
 \sfT_s \one_\a= \one_{\gamma}+\one_{\a'},   \qquad  \sfT_s \one_{\a'}= \one_{\gamma}+\one_{\a} \qquad  \sfT_s \one_\gamma= (v^2-2)\one_\gamma+ (v^2-1)(\one_\a+ \one_{\a'}). 
\]
$\one_{\gamma}$ and $\one_{\a'}$ appears on LHS, therefore, it will appears on RHS. This implies that $\gamma<\b$ and $\a'<\b$. The coefficient of $\one_\a$ on LHS is  $v^{-d_\a} \cdot (v^2-1) \cdot P_{\gamma,\b}(v) +v^{-d_\a} \cdot P_{\a',\b}(v) $ and on RHS is  $v^{-d_\a} \cdot v^2 \cdot P_{\a, \b}(v) $. The coefficient of $\one_{\a'}$ on LHS is  $v^{-d_\a} \cdot (v^2-1) \cdot P_{\gamma,\b}(v) +v^{-d_\a} \cdot P_{\a,\b}(v) $ and on RHS is  $v^{-d_\a} \cdot v^2 \cdot P_{\a', \b}(v) $. We conclude from these that $P_{\gamma,\b}(v) =P_{\a', \b}(v)=P_{\a,\b}(v)$.
}
\end{proof}

We now introduce cells in the context of the spherical variety $X$. The set $I$ has a partial order: $\b\preccurlyeq \a$  if there is 
directed path from $\a$ to $\b$ in $\bG(X)$. It induces an equivalence relation: $\a\sim \b$ if $\b\preccurlyeq \a$ and $\a\preccurlyeq \b$.

\begin{defn}\label{def:cell} A cell in $I$ is an equivalence class under the equivalence relation defined above.    
\end{defn}

For each cell $c\subset I$ and $\a\in I$, we write $\a\preccurlyeq c$ to mean there exists $\a'\in c$ such that $\a\preccurlyeq \a'$. For a cell $c\subset I$, Let $R[I]_{\preccurlyeq c}\subset R[I]$ be the (free) $R$-submodule spanned by $C'_\a$ for $\a\preccurlyeq c$. By \Cref{Cor. Hk action on KL basis}, $R[I]_{\preccurlyeq c}$ is an $\sfH$-submodule of $R[I]$. In particular, $R[I]_c:=R[I]_{\preccurlyeq c}/R[I]_{\prec c}$ is an $\sfH$-module. It is the $\sfH$-module attached to the full subgraph $\bG(X)_c\subset \bG(X)$ by restricting the vertex set to $c$.

\subsection{Characteristic cycles and Springer action}\label{ss:cc gen}
 In this subsection, we work over $\bC$ and sheaves are constructible sheaves in the analytic topology with $\bQ$-coefficients. Let $G$ be a connected reductive group over $\bC$. Let $B\subset G$ be a Borel subgroup with Lie algebra $\frb$.  Let $X$ be a $G$-variety with finite many $B$-orbits $\{X_\a|\a\in I\}$. We do not assume $X$ is smooth. However, we assume $X$ is embedded into a smooth $G$-variety $Y$.

We consider the $B$-equivariant derived category $D^b_B(X,\bQ)$ and viewing it as a full subcategory of $D^b_B(Y,\bQ)$. We use parallel notations in \Cref{Sec. Cat gen}. In particular, for each $\a\in I$, we have $\Delta_\a\in D^b_B(X,\bQ)$ and $\IC_\alpha\in D^b_B(X,\bQ)$. 
Each object $\sF\in D^b_B(X,\bQ)$ has a characteristic cycle $\CC(\sF)$, which is a $\bZ$-linear combination of irreducible components of the Lagrangian $\L_{\hSXY}$ defined in \Cref{LhSXY}. This defines a map of abelian groups
\begin{equation}\label{CC gen}
    \CC: K_0(D^b_B(X,\bQ))\to \hBM{\TOP}{\L_{\hSXY},\bZ}.
\end{equation}
Here $\hBM{\TOP}{-,\bZ}$ denotes top dimensional Borel-Moore homology group with $\bZ$-coefficients. It is a free abelian group with a basis given by the classes of the closure of $T^*_{X_\a}Y$, which we denote by $\L_{\a}$.

\begin{lemma}\label{lem:CC isom}
    The map $\CC$ is an isomorphism.
\end{lemma}
\begin{proof}
For $\D_\a\in D^b_B(X)$, $\CC(\D_\a)=[\L_\a]+\mbox{(linear combination of $[\L_\b]$ for $\b<\a$)}$. Therefore the image of the basis $\{[\D_\a]\}_{\a\in I}$ of $K_0(D^b_B(X))$ under $\CC$ is upper triangular with respect to the basis $\{[\L_\a]\}_{\a\in I}$ with coefficients $1$ on the diagonal. Therefore $\CC$ is an isomorphism.
\end{proof}

Tensoring \Cref{CC gen} by $\bC$ we get an isomorphism 
\begin{equation}\label{CC gen Q}
    \CC_{\bC}: K_0(D^b_B(X))\otimes \bC\to \hBM{\TOP}{\L_{\hSXY},\bC}\cong   \hBM{\TOP}{\St_{\hSXY},\bC}.
\end{equation}
By \Cref{Lem. Hecke mod specialization} and \Cref{W-action on relative hom}, both sides of the map \eqref{CC gen Q} carry $W$-actions. The main result of this subsection is the following.
\begin{prop}\label{Prop. CC W eq}
The map $\CC_\bC$ in \eqref{CC gen} is $W$-equivariant up to tensoring with the sign representation of $W$. In other words, for any $\sF\in D^b_B(X)$ and $w\in W$ we have
\begin{equation}\label{CC W sgn}
    \CC_\bC(w\cdot [\sF])=(-1)^{l(w)}w\cdot \CC_\bC(\sF).
\end{equation}
\end{prop}
\begin{proof}
It suffices to check that \eqref{CC W sgn} for $w=s$ a simple reflection. Under the isomorphism $\chi: K_0(D^b_B(G/B))\isom \bZ[W]$ (see \Cref{hecke algebra iso}), $i_{s!}\Qlbar$ gets mapped to $s$, and hence $\IC_s$ gets mapped to $-s-1\in \bZ[W]$. We have $\CC_\bC((-s-1)\cdot [\sF])=\CC_\bC(\IC_s\star \sF)$. Therefore \eqref{CC W sgn} for $w=s$ is equivalent to
    \begin{equation}\label{CC F s}
        \CC_\bC(\IC_s\star\sF)=(s-1)\cdot \CC_\bC(\sF).
    \end{equation}

    Let $P_s\subset G$ be the parabolic subgroup containing $B$ with the only negative root $-\a_s$. Let $\g: B\bs Y\to P_s\bs Y$ be the projection, which is smooth and proper.  By \eqref{conv ICs}, we have
    \begin{equation*}
        \IC_s\star(-)\cong \g^*\g_*[1]: D_B^b(Y)\to D_B^b(Y).
    \end{equation*}
    Consider the Lagrangian correspondence between the cotangent stacks induced by the map $\g$:
    \begin{equation}\label{Lag corr gamma}
        \xymatrix{T^*(B\bs Y)=B\bs \mu^{-1}(\frb^\bot) & B\bs \mu^{-1}(\frp_s^\bot)\ar[l]\ar[r] & T^*(P_s\bs Y)=P_s\bs \mu^{-1}(\frp_s^\bot)}
    \end{equation}
    Let $\L_s=\mu^{-1}_X(\frp_s^\bot)=\mu^{-1}(\frp_s^\bot)\cap (T^*Y)|_X$. We restrict the diagram \eqref{Lag corr gamma} over $X$ and undo the $P_s$-quotient to obtain the correspondence
    \begin{equation*}
        \xymatrix{P_s\times^B \L & P_s\times^B \L_s\ar[l]_-{\io}\ar[r]^-{\b} & \L_s}
    \end{equation*}
    Here $\io$ is induced by the inclusion $\L_s\hookrightarrow \L$, and $\b$ is the action map of $P_s$ on $\L_s$, which is a $\bP^1$-fibration. 
    
    For any $\sF\in D^b_{B}(X,\bQ)$, we consider its $\CC_\bC(\sF)$ as a class in $\hBM{\TOP}{\L,\bC}$, which we identify with $\hBM{\TOP}{P_s\times^B \L,\bC}$ and denote the resulting class by $\wt\CC_\bC(\sF)\in \hBM{\TOP}{P_s\times^B \L,\bC}$.
    
    By the behavior of $\CC_\bC$ under smooth pullback and proper pushforward, see \cite[Definition~9.3.3, Proposition 9.4.2, 9.4.3]{KS}, we have
    \begin{equation}\label{CC conv ICs}
        \wt\CC_\bC(\IC_s\star \sF)=-\wt\CC_\bC(\g^*\g_*\sF)=\io_*\b^!\b_*\io^!\wt\CC_\bC(\sF)\in \hBM{\TOP}{P_s\times^B\L,\bC}.
    \end{equation}
    Here $i_s^!: \hBM{*}{P_s\times^B \L,\bC}\to \hBM{*-2}{P_s\times^B\L_s,\bC}$ is the Gysin pullback along the regular codimension one embedding $\io$, whose definition we will review in \Cref{sss:Gysin i}. The map $\b^!: \hBM{*}{\L_s,\bC}\to \hBM{*+2}{P_s\times^B\L_s,\bC}$ is the usual pullback of cycles along the smooth map $\b$ of relative dimension $1$.

    On the other hand, let $\frn_s$ be the nilpotent radical of $\frp_s$, and let
    \begin{equation*}    \Xi_s:=\mu^{-1}_X(\frn_s^\bot)\subset (T^*Y)|_X,
    \end{equation*}
    which carries an action of $P_s$. Let $L_s$ be the reductive quotient of $P_s$ with Lie algebra $\frl_s$, and let $\ov B_s\subset L_s$ be the image of $B$, with Lie algebra $\ov\frb_s$. Consider the following a diagram in which all squares are Cartesian.
    \begin{equation*}
        \xymatrix{[B\bs \L]\ar[d]\ar[r] & [B\bs \frb^\bot]\ar[r]\ar[d] & [\ov B_s\bs \ov\frb_s^\bot]\ar[d]\\
        [P_s\bs \Xi_s]\ar[d]\ar[r] & [P_s\bs \frn^\bot_s]\ar[r]\ar[d] & [L_s\bs \frl^*_s]\\
        G\bs (T^*Y)|_X\ar[r]^-{\mu_X} & [G\bs \frg^*]
        }
    \end{equation*}
    The composite of the upper squares give a Cartesian diagram
    \begin{equation}\label{Lam s Cart}
        \xymatrix{P_s\times^B \L\ar[r]\ar[d]^{\Pi'_s} & \wt\cN_{L_s}\ar[d]^{\pi'_{L_s}}\\
        \Xi_s\ar[r] & \frl^*_s}
    \end{equation}
    The $s$-action on $\hBM{\TOP}{\L,\bC}\cong \hBM{\TOP}{P_s\times^B\L\,\bC}$ comes from the Springer action on $\pi'_{L_s *}\DD_{[L_s\bs \wt\cN_{L_s}]}$ by the same procedure as \Cref{cons:W on hBM}. We can replace $L_s$ by its adjoint group, which is isomorphic to $\PGL_2$. Then \eqref{Lam s Cart} fits into the situation of \Cref{Sec. Hotta} (see diagram \eqref{Hotta setup Cart}). By \Cref{p:Hotta}, the action of $s-1$ on $\hBM{*}{P_s\times^B\L,\bC}$ is given by the composition
    \begin{equation*}
        \io_*\b^!\b_*\io^!\in \End(\hBM{*}{P_s\times^B\L,\bC}).
    \end{equation*}
    This in particular applies to the top homology. Comparing with \eqref{CC conv ICs}, we see that \eqref{CC F s} holds.
\end{proof}

\begin{remark}\label{r:CC W eq fullO}
    In \Cref{sec Hecke type I} we will need a slight variant of \Cref{Prop. CC W eq} where $G$ is the full orthogonal group $\O(V_1)$ with $\dim V_1=2m$, hence disconnected. The Weyl group $W_1\cong W_1^\circ\rtimes\langle t_m\rangle$ where $W_1^\circ$ is the Weyl group of $G^\circ=\SO(V_1)$ and $t_m$ is an involution acting on $W^\circ$ by Coxeter group automorphism (see \Cref{ss:moment cone L2} for details). The action of $t_m$ on both $K_0(D_B^b(X,\bQ))$ and on $\hBM{\TOP}{\Lambda_{\sX},\bC}$ are induced by an involution of the stack $B\bs X$, hence
    \begin{equation}
        \CC_\bC(t_m\cdot [\sF])=t_m\cdot \CC_\bC(\sF). \quad\forall \sF\in D_B^b(X,\bQ).
    \end{equation}
    If we use the length function defined in \Cref{length H1}, then combined with \Cref{Prop. CC W eq} we get
    \begin{equation*}
        \CC_\bC(w\cdot [\sF])=(-1)^{l(w)}w\cdot \CC_\bC(\sF). \quad\forall w\in W, \sF\in D_B^b(X,\bQ).
    \end{equation*}
\end{remark}

Combining \Cref{Thm decomposition of SthS}, \Cref{lem relevant decomoposition into pieces},\Cref{lem:CC isom} and \Cref{Prop. CC W eq}, we deduce the main theorem of this section. 
\begin{thm}\label{thm cc W equivariant}
Let $X$ be a $G$-variety with finite many $B$-orbits. We assume $X$ is embedded into a smooth $G$-variety $Y$. Then there is a $W$-equivariant decomposition 
  \[
   \left(K_0(D^b_B(X))\otimes \bC
\right)\otimes \sgn_W \cong \bigoplus_{(\cO,\cL)\in  \widehat \cR_G(X)} E_{\cO,\cL}\otimes V_{\cO,\cL}(X).
  \]
  where $\sgn_W$ is the sign representation of $W$. 
\end{thm}

\subsection{Proof of the spherical module decomposition}\label{subsection spherical proof}

The goal of this subsection is to prove \Cref{thm:intro spherical}.

We retain the notation from \Cref{subsection spherical hecke modules}. In particular, let $G$ be a split connected reductive group over $\bF_q$, and let $X$ be a smooth spherical $G$-variety over $\bF_q$ satisfying the orbit assumption in \Cref{intro stab}. We assume that both $G$ and $X$ admit integral models such that their base changes to $\bC$ are well-defined. Let $W$ denote the abstract Weyl group of $G$.

In addition, we impose the following compatibility condition assumption:
\begin{equation}\label{orbit assumption bFq bC}
\begin{minipage}{0.85\textwidth}
There exists a well-behaved bijection between the orbit sets $\underline{B(k)\backslash X(k)}$ and $\underline{B(\bC)\backslash X(\bC)}$, compatible with the types assigned to each orbit as defined in \Cref{Sec. Hk mod}, for every simple reflection $s \in W$.
\end{minipage}
\end{equation}

Our constructions of the Hecke category $\cH^k$ (see \Cref{Geo hecke algebra}) and the spherical Hecke module category $\cM^k$ (see \Cref{Sec. Hk mod}) extend naturally to the complex setting, where $k$ is replaced by $\bC$ and we consider constructible sheaves in the analytic topology with $\bQ$-coefficients. Superscripts $\bC$ are used to indicate the counterparts over $\bC$ of the objects defined over $k$.

As in \Cref{Cor. Hk ring isom}, the Euler characteristic map induces an isomorphism:
\begin{equation}\label{chiC H}
    \chi^{\bC}_{\cH}: K_0(\cH^{\bC}) \xrightarrow{\sim} \bZ[W].
\end{equation}

Similarly, by \Cref{Lem. ch isom}, there is an isomorphism
\[
\chi^{\bC}_{\cM}: K_0(\cM^{\bC}) \xrightarrow{\sim} \bZ[\underline{B(\bC) \backslash X(\bC)}],
\]
compatible with the $W$-actions on both sides via \eqref{chiC H}.

Using assumption \Cref{orbit assumption bFq bC}, we identify $\underline{B(\bC) \backslash X(\bC)}$ with $\underline{B(k) \backslash X(k)}$, yielding an isomorphism of abelian groups:
\begin{equation}\label{comp isom sph}
   \Comp_{\cM} \coloneqq (\chi^{\bC}_{\cM})^{-1} \circ \chi^k_{\cM}: K_0(\cM^k) \xrightarrow{\sim} K_0(\cM^{\bC}).
\end{equation}

\begin{lemma}\label{Lem. compare Fq and C sph}
The isomorphism $\Comp_{\cM}$ is $W$-equivariant.
\end{lemma}

\begin{proof}
The $W$-actions on $K_0(\cM^k)$ and $K_0(\cM^\bC)$ are both induced via \Cref{Cor. W action sph orb}. The result follows immediately from the compatibility assumption \Cref{orbit assumption bFq bC}.
\end{proof}

\begin{proof}[Proof of \Cref{thm:intro spherical}]
The proof is summarized by the commutative diagram \Cref{M diagram}.
By \Cref{Lem. Hecke mod specialization} and \Cref{lem two hecke action}, we have the isomorphism
\begin{equation}\label{Special sqrtq hecke spherical}
  \sfM_{v = \sqrt q,\bC} \cong M=\bC[B(\bF_q)\backslash X(\bF_q)]  
\end{equation}
as $H$-modules (via a fixed isomorphism $\iota:\Qlbar \cong \bC$).

Meanwhile, combining \Cref{Lem. ch isom}, \Cref{Lem. Hecke mod specialization}, \Cref{thm cc W equivariant} and \Cref{Lem. compare Fq and C sph}, we obtain a $W$-equivariant decomposition:
\begin{equation}\label{M decomp bC}
\sfM_{v = 1, \bC} \otimes \sgn_W \cong \bigoplus_{(\cO, \cL) \in \widehat \cR_{G_\bC}(X_\bC)} E_{\cO, \cL} \otimes V_{\cO, \cL}(X_\bC).
\end{equation}

Lusztig \cite[Theorem 3.1]{Lu1981BC} constructed a canonical $\bQ[v, v^{-1}]$-algebra homomorphism
\begin{equation}\label{Lusztig homomorphism}
    \lambda_W: \sfH \otimes_R \bQ[v, v^{-1}] \to \bQ[v, v^{-1}][W],
\end{equation}
whose specialization at $v = 1$ is the identity, and at $v = \sqrt q$ is an isomorphism. We use this to identify $H = \sfH_{v = \sqrt q, \bC}$ with $\bC[W]$, defining the $H$-module structure on $E(q)$ for any $W$-representation $E$ over $\bC$.

Applying \Cref{TitsLusztig} to the $\bC[v, v^{-1}]$-algebra $\sfH\otimes_{R} \bC[v,v^{-1}]$, the Lusztig homomorphism \Cref{Lusztig homomorphism} (tensoring over $\bC[v,v^{-1}])$:
\[
\lambda_{W,\bC}: \sfH\otimes_{R} \bC[v,v^{-1}] \to \bC[v, v^{-1}][W],
\]
and the $\sfH\otimes_{R} \bC[v,v^{-1}]$-module $\sfM\otimes_{R} \bC[v,v^{-1}]$, together with specializations $v = 1$ and $v = \sqrt q$,
we deduce that $\sfM_{v = \sqrt q, \bC}$ is isomorphic to $\sfM_{v = 1, \bC}$ as a $H\cong \bC[W]$-module. Hence, combining with \eqref{Special sqrtq hecke spherical} and \eqref{M decomp bC}, we obtain the desired $H$-module decomposition:
\[
M \otimes \sgn_{H} \cong \bigoplus_{(\cO, \cL) \in \widehat \cR_{G_\bC}(X_\bC)} E_{\cO, \cL}(q) \otimes V_{\cO, \cL}(X_\bC),
\]
completing the proof.
\end{proof}

\Cref{thm:intro spherical} also holds without assuming smoothness.
\begin{cor}\label{Cor thm spherical general}
 Let $G$ be a split reductive group over $\bF_q$ and let $X$ be a $G$-variety over $\bF_q$ with finitely many $B$-orbits satisfying assumptions \Cref{orbit assumption bFq bC} and \Cref{intro stab}. Then there is a canonical isomorphism of $H$-modules:
\[
\bC[B(\bF_q) \backslash X(\bF_q)]  \otimes \sgn_{H} \cong 
\bigoplus_{(\cO, \cL) \in \widehat \cR_{G_\bC}(X_\bC)} E_{\cO, \cL}(q) \otimes V_{\cO, \cL}(X_\bC).
\]
\end{cor}

\begin{proof}
The proof follows by decomposing $X$ into a disjoint union $X = \bigsqcup X_i$ of $G$-stable strata, each of which is a homogeneous spherical $G$-variety, and applying \Cref{thm:intro spherical} to each stratum $X_i$.
\end{proof}

\section{Hecke bimodules from type II theta correspondence}\label{ex: GL theta}

In this section, we apply the general framework developed in \Cref{Sec. geom gen} to study the spherical Hecke module arising from the type II theta correspondence over a finite field.

\subsection{Type II theta correspondence over finite fields}\label{sec type II theta finite}

We work over the finite field $F = \bF_q$ throughout this subsection. Let $L_1$ and $L_2$ be finite-dimensional vector spaces over $\bF_q$ of dimensions $m$ and $n$, respectively. Set $\ov G_1 = \GL(L_1)$ and $\ov G_2 = \GL(L_2)$. The group $\ov G_1 \times \ov G_2$ acts naturally on the affine space $\ov \bL_1 \coloneqq \Hom(L_1, L_2)$. The induced action of $\ov G_1(\bF_q) \times \ov G_2(\bF_q)$ on the function space $\bC[\ov \bL_1(\bF_q)]$ gives rise to the Weil representation $\omega_{\ov \bL_1}$, defined by
\begin{equation}\label{Weil type II bL1}
    (\omega_{\ov \bL_1}(g_1, g_2) f)(T) = f(g_2^{-1} T g_1), \quad g_1 \in \ov G_1(\bF_q),\ g_2 \in \ov G_2(\bF_q),\ f \in \bC[\ov \bL_1(\bF_q)],\ T \in \ov \bL_1(\bF_q).
\end{equation}
The decomposition of $\omega_{\ov \bL_1}$ defines the type II theta correspondence between $\Irr(\ov G_1(\bF_q))$ and $\Irr(\ov G_2(\bF_q))$.

Similarly, $\ov G_1 \times \ov G_2$ acts on the affine space $\ov \bL_2 \coloneqq \Hom(L_2, L_1)$, inducing another model $(\bC[\ov \bL_2(\bF_q)], \omega_{\ov \bL_2})$  of the same Weil representation, given by
\begin{equation}\label{Weil type II bL2}
    (\omega_{\ov \bL_2}(g_1, g_2) f)(T) = f(g_1^{-1} T g_2), \quad g_1 \in \ov G_1(\bF_q),\ g_2 \in \ov G_2(\bF_q),\ f \in \bC[\ov \bL_2(\bF_q)],\ T \in \ov \bL_2(\bF_q).
\end{equation}

Fix a nontrivial additive character $\psi: \bF_q \to \bC^\times$. The two models $\omega_{\ov \bL_1}$ and $\omega_{\ov \bL_2}$ are intertwined by the Fourier transform
\begin{equation}\label{Fourier transform function}
    \FT^{\bF_q}_{\ov \bL_1}: \bC[\ov \bL_1(\bF_q)] \longrightarrow \bC[\ov \bL_2(\bF_q)],
\end{equation}
defined by
\begin{equation}\label{Equa. Partial Fourier}
    \FT^{\bF_q}_{\ov \bL_1}(f)(T_2) = q^{-\frac{1}{2} \dim \ov \bL_1} \sum_{T_1 \in \ov \bL_1(\bF_q)} f(T_1)\, \psi(\langle T_1, T_2 \rangle), \quad f \in \bC[\ov \bL_1(\bF_q)],\ T_2 \in \ov \bL_2(\bF_q),
\end{equation}
where $\langle T_1, T_2 \rangle$ is the trace pairing as in \Cref{bL1 dual complex}.

For $i = 1, 2$, fix a Borel subgroup $\ov B_i \subset \ov G_i$, and let $\ov H_i=\bC[\ov B_i(\bF_q)\bs \ov G_i(\bF_q)/\ov B_i(\bF_q)]$ denote the corresponding Iwahori–Matsumoto Hecke algebra. Let
\begin{equation*}
    \ov M_i \coloneqq \omega_{\ov \bL_i}^{\ov B_1(\bF_q) \times \ov B_2(\bF_q)} = \bC\left[\underline{\ov B_1(\bF_q) \times \ov B_2(\bF_q) \backslash \ov \bL_i(\bF_q)}\right],
\end{equation*}
as modules over $\ov H_1 \otimes \ov H_2$. The Fourier transform $\FT^{\bF_q}_{\ov \bL_1}$ induces an isomorphism of $\ov H_1 \otimes \ov H_2$-modules between $\ov M_1$ and $\ov M_2$. We denote the common isomorphism class by $\ov M$, and call it the \emph{oscillator bimodule} for the pair $(\ov G_1, \ov G_2)$.

Recall from \Cref{relevant pair type II} the set $\cR(\cP(m)\times \cP(n))$ of relevant pairs associated to $\ov \bL$ and the number $m_{\gamma_1,\gamma_2}$ for each $(\gamma_1,\gamma_2)\in \cR(\cP(m)\times \cP(n))$. The following theorem will be proved in \Cref{Geometrization of the oscillator bimodule type II}.

\begin{thm}\label{Theta type II Hecke module}
There is a canonical isomorphism of $\ov H_1 \otimes \ov H_2$-modules:
\begin{equation}\label{deformation 1}
   \ov M \otimes \sgn_{\ov H_1 \otimes \ov H_2} 
\cong \bigoplus_{(\gamma_1, \gamma_2) \in \cR(\cP(m)\times \cP(n))} m_{\gamma_1,\gamma_2} E_{\cO_{\gamma_1}}(q) \boxtimes E_{\cO_{\gamma_2}}(q),
\end{equation}
where $\sgn_{\ov H_1 \otimes \ov H_2}$ denotes the sign character of the Hecke algebra $\ov H_1 \otimes \ov H_2$.
\end{thm}

\begin{remark}
This recovers \cite[Theorem 5.5]{AMR}.
\end{remark}

\subsection{Geometrization of the oscillator bimodule for $\ov G_1 \times \ov G_2$}\label{Geometrization of the oscillator bimodule type II}

In this subsection, we apply the general results from \Cref{Sec. geom gen} to geometrize the oscillator bimodules $\ov M_i$ for $i = 1,2$.

We continue with the notation from \Cref{sec type II theta finite}. Let $\ov W_i$ denote the Weyl group of $\ov G_i$, $\bsfH_i$ the corresponding generic Hecke algebra, and $\ov H_i$ the Iwahori–Hecke algebra. Fix standard bases $\{e_i\}_{i=1}^m$ and $\{f_j\}_{j=1}^n$ for $L_1$ and $L_2$ such that the complete flags
\[
\Span\{e_1,\dots,e_i\}\quad \text{and} \quad \Span\{f_1,\dots,f_j\} 
\]
are stabilized by the Borel subgroups $\ov B_1$ and $\ov B_2$, respectively. For $1 \le i \le m-1$, let $s_i \in \ov W_1$ correspond to the simple reflection swapping $e_i$ and $e_{i+1}$; similarly, for $1 \le j \le n-1$, let $s_j' \in \ov W_2$ correspond to the simple reflection swapping $f_j$ and $f_{j+1}$. Then $\ov W_1 \cong \bS_m$ and $\ov W_2 \cong \bS_n$ are symmetric groups, with sets of simple reflections
\[
\ov \sfS_1 = \{s_1, \dots, s_{m-1}\}, \qquad \ov \sfS_2 = \{s_1', \dots, s_{n-1}'\}.
\]
The generic Hecke algebra $\bsfH_i$ corresponds to the Coxeter system $(\ov W_i, \ov \sfS_i)$ (cf. \Cref{Geo hecke algebra}).

Let $\bcH_i^\Tate$, $\bcH_i^\mix$, and $\bcH_i^k$ denote various versions of Hecke categories for $\ov G_i$. By \Cref{Cor. Hk ring isom}, we have canonical isomorphisms
\[
\ch^\Tate: K_0(\bcH_i^\Tate) \xrightarrow{\sim} \bsfH_i, \qquad \chi: K_0(\bcH_i^k) \xrightarrow{\sim} \bZ[\ov W_i].
\]

We now turn to the geometry of the $\ov G_1 \times \ov G_2$-action on $\ov \bL_1$. All results below admit parallel formulations for the $\ov G_1 \times \ov G_2$-action $\ov \bL_2$.
\begin{defn}\label{def partial matching}
\begin{enumerate}
    \item A \emph{partial matching} between the sets $\{1,\dots,m\}$ and $\{1,\dots,n\}$ is a triple $\sigma = (I, J, \mu)$, where $I \subseteq \{1,\dots,m\}$, $J \subseteq \{1,\dots,n\}$, and $\mu: I \to J$ is a bijection.
    \item Let $\PM(m,n)$ denote the set of all partial matchings. For $0 \le k \le \min(m,n)$, let $\PM(m,n,k) \subset \PM(m,n)$ be the subset of matchings with $|I| = |J| = k$, so that
    \[
    \PM(m,n) = \bigsqcup_{k=0}^{\min(m,n)} \PM(m,n,k).
    \]
\end{enumerate}
\end{defn}

The group $\ov W_1 \times \ov W_2 \cong \bS_m \times \bS_n$ acts on $\PM(m,n)$ by transport of structure:
\begin{equation}\label{action Sm Sn on PM}
(w, w') * (I, J, \mu) := (w(I), w'(J), w'|_J \circ \mu \circ w^{-1}|_{w(I)}).
\end{equation}

Given a matching $\sigma = (I, J, \mu) \in \PM(m,n)$, define a point $x_\sigma \in \ov \bL_1(\bF_q)$ by:
\begin{equation}\label{def:Tsigma}
x_\sigma(e_i) := \begin{cases}
f_{\mu(i)} & \text{if } i \in I, \\
0 & \text{otherwise}.
\end{cases}
\end{equation}
Let $\cO_\sigma \subset \ov \bL_1$ be the geometric $\ov B_1 \times \ov B_2$-orbit of $x_\sigma$. Write $\mathbf{1}_\sigma$ for the characteristic function of $\cO_\sigma(\bF_q)$. We also denote $d_\s=\dim \cO_\s$. 

\begin{lemma}\label{PM vs orbits}
\begin{enumerate}
    \item The map $\sigma \mapsto \cO_\sigma$ defines a bijection between $\PM(m,n)$ and the set of geometric $\ov B_1 \times \ov B_2$-orbits in $\ov \bL_1$. In particular, $\ov \bL_1$ is a spherical $\ov G_1 \times \ov G_2$-variety.
    
    \item The assumption \eqref{orbit assumption} holds for this action. Consequently, each $\cO_\sigma(\bF_q)$ is a single $\ov B_1(\bF_q) \times \ov B_2(\bF_q)$-orbit, and the functions $\{\mathbf{1}_\sigma\}_{\sigma \in \PM(m,n)}$ form a $\bC$-basis of $\ov M_1$.
\end{enumerate}
\end{lemma}
\begin{proof}
Part $(1)$ is proved in \cite[Proposition 15.27]{MR2110098}. Next we prove part $(2)$. The proof for part $(1)$ also works over $\bF_q$. It shows that $\ov \bL_1(\bF_q)$ is the union of $\ov B_1 (\bF_q)\times \ov B_2(\bF_q)$-orbits of $x_\s$, for $\s\in \PM(m,n)$. Therefore each $\cO_\s(\bF_q)$ is a single $\ov B_1 (\bF_q)\times \ov B_2(\bF_q)$-orbit. This implies the connectivity of $S_\s:=\Stab_{\ov B_1\times \ov B_2}(x_\s)$ because the set of $\ov B_1 (\bF_q)\times \ov B_2(\bF_q)$-orbits on $\cO_\s(\bF_q)$ is parametrized by the Galois cohomology $H^1(\bF_q, S_{\s})\cong H^1(\bF_q, \pi_0(S_{\s}))$, which would be nontrivial if $\pi_0(S_\s)$ was nontrivial. This proves part $(2)$. 
\trivial[h]{
(1)By identifying $\ov \bL_1$ with the space of matrices $M_{n \times m}$ using the chosen bases, the $\ov B_i$ actions reduce to upper-triangular 
column (sweeping right) and row (sweeping up) operations. Gaussian elimination shows that each $\ov\bF_q$-point of $\ov \bL_1$ can be brought under such operations to a matrix with entries $\{0,1\}$ such that each column and row has at most one $1$, i.e., a matrix corresponding to a partial matching.  This shows that 
\[
\bigcup_{\s\in \PM(m,n)}\cO_\s=\ov \bL_1.
\]
Notice that such sweeping preserves the ranks of northeast $p\times q$ submatrix for each $1\leq p\leq m, 1\leq q\leq n$. This proves the disjointness of the above decomposition. 
}
\end{proof}


By the above lemma, results from \Cref{Sec. Hk mod} are applicable.
Applying the general framework from \Cref{Sec. Hk mod} to this setting, we obtain the following categories:
\begin{equation*}
\bcM_1^\Tate = D^{\Tate}_{\ov B_1 \times \ov B_2}(\ov \bL_1) \hookrightarrow 
\bcM_1^\mix = D^{\mix}_{\ov B_1 \times \ov B_2}(\ov \bL_1) \hookrightarrow 
\bcM_1 = D^{b}_{\ov B_1 \times \ov B_2}(\ov \bL_1) \to 
\bcM_1^k = D^{b}_{\ov B_{1,k} \times \ov B_{2,k}}(\ov \bL_1).
\end{equation*}
The category $\bcM_1$ carries commuting actions of $\bcH_1$ and $\bcH_2$; similar statements hold for their mixed and $k$ variants. Let $\bsfM_1 \coloneqq R[\un{\ov B_1 \times \ov B_2 \backslash \ov \bL_1}]=R[\PM(m,n)]$ as a free $R$-module.
The following is a special case of \Cref{Lem. ch isom} and \Cref{Prop ICs preserves Tate}, together with the functoriality of the diagram \eqref{diag cat}:

\begin{cor}\label{Cor. K0 M1 type II}
\begin{enumerate}
    \item The full subcategory $\bcM_1^{\Tate} \subset \bcM_1$ is stable under the commuting convolution actions of $\bcH_1^\Tate$ and $\bcH_2^\Tate$. In particular, $K_0(\bcM_1^{\Tate})$ admits a natural structure as a $\bsfH_1 \otimes_R \bsfH_2$-module.
    
    \item The map $\ch^\Tate: K_0(\bcM_1^\Tate) \to \bsfM_1$ is an $R$-module isomorphism, equipping $\bsfM_1$ with a natural structure of a $\bsfH_1 \otimes_R \bsfH_2$-module.
    
    \item Specializing the $\bsfH_1 \otimes_R \bsfH_2$-action on $\bsfM_1$ at $v=\sqrt{q}$, we recover the $\ov H_1\ot \ov H_2$-action on $\ov M_1$ defined in \Cref{sec type II theta finite} (via a fixed isomorphism $\iota:\Qlbar \cong \bC$). 

    \item Specializing the $\bsfH_1 \otimes_R \bsfH_2$-action on $\bsfM_1$ at $v=1$ defines an action of $\ov W_1\times \ov W_2$ on $\ov \sfM_{1,v=1}=\ZZ[\PM(m,n)]$. 
    
\end{enumerate}
\end{cor}

For each $\sigma = (I, J, \mu) \in \PM(m,n)$, we define 
\begin{equation}\label{extended sigma}
    \sigma^{\dagger}: \{1,\cdots, m\} \longrightarrow \{0,1,2,\cdots, n\}
\end{equation}
by 
\[
\sigma^{\dagger}(i)=\begin{cases}
    \sigma(i) \quad & \mbox{if $i\in I$}\\
    0 \quad & \mbox{if $i \notin I$}.
\end{cases}
\]
We also define 
\begin{equation}\label{extended sigma 2}
    \sigma_{\dagger}: \{1,\cdots, n\} \longrightarrow \{1,2,\cdots, m,\infty\}
\end{equation}
by 
\[
\sigma_{\dagger}(i)=\begin{cases}
    \sigma^{-1}(i) \quad & \mbox{if $i\in J$}\\
    \infty & \mbox{if $i \notin J$}.
\end{cases}
\]
We define orders  $<_{\PM}$ on the set $\{0, 1,2,\cdots, n\}$, $\{1,2,\cdots, m,\infty\}$ by 
\[
0<_{\PM}<1<_{\PM} 2<\cdots<_{\PM} n,\quad 1<_{\PM} 2<\cdots<_{\PM} m <_{\PM} \infty.
\]
The following lemma computes the type $(s,\cO_\sigma)$ for each  simple reflection $s$ in $\ov W_1 \times \ov W_2$ and the orbit $\cO_\sigma$ with $\sigma\in \PM(m,n)$. 

\begin{lemma}\label{Lem. type G or U type II}
Let $s\in \ov \sfS_1\sqcup \ov \sfS_2$ be a simple reflection in $\ov W_1 \times \ov W_2$, and let $\sigma = (I, J, \mu) \in \PM(m,n)$. Then the pair $(s, \cO_\sigma)$ is of type (G) or (U). More precisely, 
\begin{enumerate}
    \item If  $s = s_i$ for $1 \leq i \leq  m-1$, then $(s, \cO_\sigma)$ is of
    \[
    \begin{cases}
        \mbox{type $G$}\quad &\mbox{if $\sigma^\dagger(i)=\sigma^\dagger(i+1)$}\\
         \mbox{type $U^-$}\quad &\mbox{if $\sigma^\dagger(i)<_{\PM}\sigma^\dagger(i+1)$}\\
        \mbox{type $U^+$}\quad &\mbox{if $\sigma^\dagger(i)>_{\PM}\sigma^\dagger(i+1)$}.
    \end{cases}
    \] 
    \item If  $s = s_i'$ for $1 \leq i \leq  n-1$, then $(s, \cO_\sigma)$ is of
    \[
    \begin{cases}
        \mbox{type $G$}\quad &\mbox{if $\sigma_\dagger(i)=\sigma_\dagger(i+1)$}\\
         \mbox{type $U^-$}\quad &\mbox{if $\sigma_\dagger(i)<_{\PM}\sigma_\dagger(i+1)$}\\
        \mbox{type $U^+$}\quad &\mbox{if $\sigma_\dagger(i)>_{\PM}\sigma_\dagger(i+1)$}.
    \end{cases}
    \]
\end{enumerate}
In case $U$, the $s$-companion of $\cO_\sigma$ is $\cO_{s * \sigma}$ as defined in \eqref{action Sm Sn on PM}. 
\end{lemma}
\begin{proof}
We prove the case when $s=s_i$ for some $1\leq j\leq  m-1$. The other case are similar.  We follows the notations in \Cref{Sec. Hk mod}. Let $\ov P_{1,s_i}$ be the parabolic subgroup of $\ov G_1$ stabilizing the partial flag
    \[
    0\subseteq \langle e_1\rangle \cdots \langle e_1,\cdots, e_{i-1}\rangle \subseteq \langle e_1,\cdots,e_{i-1},e_{i},e_{i+1}\rangle\subseteq \cdots \subseteq \langle e_1,\cdots, e_n\rangle.
    \]
Let $V_{s_i}=\langle e_{i},e_{i+1}\rangle$ and $\bP(V_{s_i})$ be the projective lines in $V_{s_i}$. There is a natural projection 
\begin{equation}\label{p1siprojection}
   \ov P_{1,s_i}\longrightarrow \PGL(V_{s_i})=\Aut(\bP(V_{s_i})). 
\end{equation}
Let $P_{s_i}=P_{1,s_i} \times \barB_2$ and $\widetilde a_{s_i}: P_{s_i}\times^{\barB_1\times B_2} \bL_1\to \bL_1$ be the action map. For each $\s\in \PM(m,n)$, we consider the point $x_\s\in \cO_\s(\bF_q)$ defined in \eqref{def:Tsigma}. Then $D_{x_\s}\coloneqq \widetilde a_{s_i}^{-1}(x_\s) \cong P_{s_i}/\barB_1\times \barB_2\cong \bP(V_{s_i})$. We have 
\begin{equation*}
    P_{s_i,x_\sigma}\coloneqq \Stab_{P_{s_i}}(x_\s)=\{(g_1,b_2)\in P_{1,s_i} \times \barB_2| b_2 x_\sigma g_1^{-1}=x_\sigma\}.
\end{equation*}
The action of $P_{s_i,x_\sigma}$ on $D_{x_\s}$ factors through a quotient $Q_{x_\s}\subset \Aut(D_{x_\s})\cong \PGL(V_{s_i})$. More precisely, let 
\[
\pi_{s_i,\s}: P_{s_i}=P_{1,s_i} \times \barB_2 \longrightarrow \GL(V_{s_i})\times \GL(x_{\s}(V_{s_i})).
\]
be the natural projection map and let $\ov P_{s_i,x_\sigma}=\pi_{s_i,\s}( P_{s_i,x_\sigma})$. Then we have 
\begin{equation}\label{Psixs}
    \ov P_{s_i,x_\sigma}=\{(g_1,b_2)\in\GL(V_{s_i})\times \GL(x_\s(V_{s_i}))| b_2 x_\sigma|_{V_{s_i}} g_1^{-1}=x_\sigma|_{V_{s_i}}\} 
\end{equation}
Let $\pr_1: \GL(V_{s_i})\times \GL(T_{\s}(V_{s_i}))\longrightarrow \PGL(V_{s_i})$ be the natural projection. We have $Q_{x_\s}=\pr_1(\ov P_{s_i,x_\sigma})$. 
\begin{enumerate}
    \item If $\sigma^\dagger(i)=\sigma^\dagger(i+1)$, then $\{i, i+1\} \cap I = \emptyset$. We have $x_\s|_{V_{s_i}}=0$ and  $Q_{x_\s}\cong \PGL(V_{s_i})$; 
    \item If $\sigma^\dagger(i)<\sigma^\dagger(i+1)$, then either $i\notin I$ and $i+1\in I$ or $(i,i+1)\in I$ and $\mu(i)<\mu(i+1)$. Let $(g_1,b_2)\in \ov P_{s_i,x_\sigma}$
    \begin{itemize}
        \item If $i\notin I$ and $i+1\in I$, then we see from \Cref{Psixs} that $g_1$ stabilizes $\ker(x_\s|_{V_{s_i}})=\langle e_i\rangle$; 
        \item If $(i,i+1)\in I$ and $\mu(i)<\mu(i+1)$, then we see from \Cref{Psixs} that $g_1$ stabilizes $(x_\s|_{V_{s_i}})^{-1}(f_{\mu(i)})=\langle e_i\rangle$. 
    \end{itemize}
In both cases, one easily see that $Q_{x_\s}$ is the Borel subgroup of $\PGL(V_{s_i})$ stabilizing $\langle e_i\rangle$. Therefore, we know that $(s, \cO_\sigma)$ is of type $U^-$. 
    \item If $\sigma^\dagger(i)>\sigma^\dagger(i+1)$, then either $i\in I$ and $ j+1\notin I$ or $(i,i+1)\in I$ and $\mu(i)>\mu(i+1)$. The computations is similar to case $(2)$ and we can see that $Q_{x_\s}$ is the Borel subgroup of $\PGL(V_{s_i})$ stabilizing $\langle e_{i+1}\rangle$. Therefore, we know that $(s, \cO_\sigma)$ is of type $U^+$. 
\end{enumerate}
This finishes the proof. 
\end{proof}

\begin{remark}\label{remark assupmtion orbit type II 2}
The proof of \Cref{Lem. type G or U type II} also holds when $\bF_q$ is replaced by any field $F$. Thus, the action of $\ov B_1 \times \ov B_2$ on $\ov \bL_1$ satisfies the assumption \eqref{orbit assumption bFq bC}.
\end{remark}

\begin{proof}[Proof of \Cref{Theta type II Hecke module}]
We prove \Cref{Theta type II Hecke module} for $\ov M_1$; the argument for $\ov M_2$ is entirely parallel. By \Cref{PM vs orbits}(2) and \Cref{remark assupmtion orbit type II 2}, we know that the action of $\ov G_1 \times \ov G_2$ on $\ov \bL_1$ satisfies the assumptions \eqref{orbit assumption} and \eqref{orbit assumption bFq bC}. 

Applying \Cref{thm:intro spherical} and \Cref{lem relevant decomoposition into pieces}, we obtain a canonical isomorphism of $\ov H_1 \otimes \ov H_2$-modules:
\begin{equation}\label{thm theta type II step 1}
    \ov M_1 \otimes \sgn_{\ov H_1 \otimes \ov H_2} \cong 
    \bigoplus_{(\cO_1, \cO_2) \in \cR_{\ov G_1 \times \ov G_2}(\ov \bL)}
    E_{\cO_1}(q) \otimes E_{\cO_2}(q) \otimes V_{\cO_1 \times \cO_2}(\ov \bL).
\end{equation}

On the other hand, \Cref{thm:Htop type II}, \Cref{thm:Htop type II combi} provides a refinement of this decomposition:
\begin{equation}\label{hbM bL refinement}
    \bigoplus_{(\cO_1, \cO_2) \in  \cR_{\ov G_1 \times \ov G_2}(\ov \bL)}
    E_{\cO_1} \otimes E_{\cO_2} \otimes V_{\cO_1 \times \cO_2}(\ov \bL)
  \cong \bigoplus_{(\gamma_1, \gamma_2) \in \cR(\cP(m)\times \cP(n))} m_{\gamma_1,\gamma_2} E_{\cO_{\gamma_1}} \boxtimes E_{\cO_{\gamma_2}}.
\end{equation}

Combining \eqref{thm theta type II step 1} with \eqref{hbM bL refinement}, we obtain the desired isomorphism, completing the proof of \Cref{Theta type II Hecke module}.
\end{proof}

\subsection{Another description of the oscillator bimodule}\label{another desc type II}
In this subsection, we give another description of the deformation $q \mapsto 1$ of $\ov M_1$ as a $\ov W_1 \times \ov W_2$-representation.

\begin{lemma}\label{Knop action type II}
The action of $\ov W_1 \times \ov W_2$ on $\bZ[\PM(m,n)]$ via \Cref{Cor. K0 M1 type II} (4) is given by
\[
w \cdot \one_\sigma = \one_{w * \sigma} \quad \text{for all } w\in \ov W_1\times \ov W_2, \sigma \in \PM(m,n).
\]
\end{lemma}

\begin{proof}
It suffices to check the case where $w=s$ is a simple reflection,  in which case the statement follows from \Cref{Cor. W action sph orb} and \Cref{Lem. type G or U type II}.
\end{proof}

For each integer $0 \le i \le \min(m,n)$, the subset $\PM(m,n,i)\subset \PM(m,n)$ is a single $\bS_m \times \bS_n$-orbit. We define a special partial matching $\overline \s_i \in \PM(m,n,i)$ by
\begin{equation}\label{def sigmai}
   \overline \s_i = \left( I = \{m - i + 1, \ldots, m\} \xrightarrow{\mu_i} \{1, 2, \ldots, i\} = J \right),
\end{equation}
where $\mu_i$ is the unique order-preserving bijection between $I$ and $J$. The stabilizer of $\overline \s_i$ under $\bS_m\times \bS_n$ is the subgroup 
\begin{equation*}
    \bS_{m-i} \times \Delta(\bS_i) \times \bS_{n-i}\subseteq (\bS_{m-i} \times \bS_i )\times (\bS_i \times \bS_{n-i})\subseteq \bS_m\times \bS_n.
\end{equation*}
Here $\D(\bS_i)$ denotes the diagonally embedded copy of $\bS_i$.


In particular, we get an 
isomorphism of $\ov W_1 \times \ov W_2 = \bS_m \times \bS_n$-modules
\begin{equation}\label{degen Hecke type II}
    \bC[\PM(m,n)] \cong   
\bigoplus_{i=0}^{\min(m,n)} \Ind_{\bS_{m-i} \times \Delta(\bS_i) \times \bS_{n-i}}^{\bS_m \times \bS_n} (\mathbf{1}),
\end{equation}
where $\mathbf{1}$ denotes the trivial representation.


Combining \Cref{Knop action type II}, \Cref{Theta type II Hecke module} and \eqref{degen Hecke type II}, we obtain:

\begin{cor} There is an isomorphism of $\ov W_1 \times \ov W_2 = \bS_m \times \bS_n$-modules:
\[
\bigoplus_{i=0}^{\min(m,n)} \Ind_{\bS_{m-i} \times \Delta(\bS_i) \times \bS_{n-i}}^{\bS_m \times \bS_n} (\mathbf{1})\cong \bigoplus_{(\gamma_1, \gamma_2) \in \cR(\cP(m)\times \cP(n))} m_{\gamma_1,\gamma_2} E_{\cO_{\gamma_1}}(q) \boxtimes E_{\cO_{\gamma_2}}(q).
\]
\end{cor}
\begin{remark}
It would be interesting to give a direct proof of this identity based on the explicit Springer correspondence for general linear groups.
\end{remark}



\subsection{Generators as $\bcH_1 \otimes \bcH_2$-module}
We now describe a set of generating objects in $\bcM_1$ under the action of $\bcH_1 \otimes \bcH_2$, and draw some geometric consequences.  

Let $\PM(m,n)_0 \subset \PM(m,n)$ denote the subset of partial matchings $\s$ such that for any simple reflection $s \in \Delta_{\ov W_1} \cup \Delta_{\ov W_2}$, the pair $(s, \cO_\s)$ is either of type (G) or type (U$-$), as defined in \Cref{Sec. Hk mod}. This is consistent with the notation $I_0$ introduced in \Cref{def. I0}.

\begin{lemma}\label{Lem. I0 for PM type II}
We have $\PM(m,n)_0 = \{\overline \s_i \mid 0 \le i \le \min(m,n)\}$.
\end{lemma}
\begin{proof}
    This follows from \Cref{Lem. type G or U type II}. 
\end{proof}

The following is a direct calculation. 
\begin{lemma}\label{Lem. linear IC type II}
For each $0 \le i \le \min(m,n)$, let $\ov \bL_1^i \subset \ov \bL_1$ be the linear subspace consisting of maps $T: L_1 \to L_2$ satisfying:
\begin{itemize}
    \item $T(e_1) = \cdots = T(e_{m-i}) = 0$;
    \item $T(e_{m-i+a}) \in \langle f_1, \ldots, f_a \rangle$ for $1 \le a \le i$.
\end{itemize}
Then $\ov {\cO_{\overline \s_i}} = \ov \bL_1^i$ and $d_{\overline \s_i}=\dim\ov \bL_1^i=\frac{i(i+1)}{2}$. Moreover, the Kazhdan–Lusztig basis element corresponding to $\s_i$ satisfies
\begin{equation}\label{C' for bar sigma i}
    C'_{\overline \s_i} = v^{-d_{\overline \s_i}} \sum_{\s \le \overline \s_i} \mathbf{1}_\s=v^{-\frac{i(i+1)}{2}} \sum_{\s \le \overline \s_i} \mathbf{1}_\s.
\end{equation}
\end{lemma}

\begin{cor}[of \Cref{Lem. heart gen}]\label{Lem. type II gene}
The $\bsfH_1 \otimes \bsfH_2$-module $\bsfM_1$ is generated by $\{C'_{\overline \s_i} \mid 0 \le i \le \min(m,n)\}$. In particular, the $\bZ[\ov W_1 \times \ov W_2]$-module $\bsfM_{v=1,\bZ}=\bZ[\PM(m,n)]$ is generated by $\{\chi(\IC_{\overline \s_i}) \mid 0 \le i \le \min(m,n)\}$.
\end{cor}

\begin{cor}\label{Cor. M1 IC typeII} 
For any $\s\in \PM(m,n)$,  $\IC_\s$ is Tate, $*$-pure of weight zero, and $*$-parity with parity $l(\s)$.
\end{cor}
\begin{proof} By \Cref{Cor. I0 implies all I}, it suffices to check that $\IC_{\overline \s_i}$ has the stated properties, for $0\le i\le \min\{m,n\}$, which is clear by \Cref{Lem. linear IC type II} because $\IC_{\overline \s_i}$ is the constant sheaf $\Qlbar\langle l(\overline \s_i)\rangle$ on the linear subspace  $ \ov \bL_1^i$ .
\end{proof}

We have the Kazhdan-Lusztig polynomials $P_{\s',\s}\in\bZ[t]$ in our situation, as defined in general in \Cref{Sec. Hk mod}. By \Cref{Cor. M1 IC typeII} and \Cref{Lem. KL poly ch}, $P_{\s',\s}$ has only even degree terms and  $P_{\s',\s}\in\bZ_{\geq 0} [t]$.
\begin{lemma}\label{Lem. KL indep p type II} 
For any $\s\in \PM(m,n)$, the element $C'_\s\in  \bsfM_1$ is independent of $\bF_q$. In particular, for any $\s, \s'\in \PM(m,n)$, the Kazhdan-Lusztig polynomial $P_{\s',\s}$ is independent of $\bF_q$. 
\end{lemma}
\begin{proof}
We prove the statement by induction on $d_\s$. The case $d_\s=0$ only occurs when  $\cO_\s=\{0\}$, for which the statement is obvious. Suppose the statement is true for orbits of dimension $<d_\s$. If there exists $s\in \ov \sfS_1\sqcup \ov \sfS_2$ such that $(s,\cO_\s)$ is of type (U+), then let $\s'=s*\s$ be the $s$-companion of $\s$. By \Cref{Lem. conv ICs}(2), $C'_\s$ can be expressed as an $R$-linear combination of $(\sfT_s+1)C'_{\s'}$ together with $C'_\t$ where $d_\t<d_\s$. By inductive hypothesis, both $C'_\t$ and its coefficient (which involves weight polynomials of stalks of $\IC_{\s'}$) are independent of $\bF_q$. By \Cref{Prop ICs preserves Tate}, the $\bsfH_1\ot\bsfH_2$-action on $\bsfM_1$ is also independent of $\bF_q$, hence so is  $(\sfT_s+1)C'_{\s'}$. Altogether we conclude that $C'_{\s}$ is independent of $\bF_q$.

If for all $s\in \ov \sfS_1\sqcup \ov \sfS_2$,  $(s,\cO_\s)$ is either of type (G) or type (U-), then $\s\in \PM(m,n)_0$, hence  $\s=\overline \s_i$ by \Cref{Lem. I0 for PM type II}. In this case, \eqref{C' for bar sigma i} shows that $C'_{\overline \s_i}$ is independent of $\bF_q$. 
\end{proof}

\subsection{A $W$-graph example for type II theta correspondence}\label{section W-graph}
In this section, we computing explicitly the $W$-graph associated with the spherical Hecke module arising from the \emph{type II} theta correspondence in the case when $m=n=2$.

{\bf The orbits.}
Recall that the $\ov B_1\times \ov B_2$–orbits in $\bL_1=\Hom(L_1,L_2)$ are parametrized by the set $\PM(2,2)$. There are $7$ orbits and we simply the notation as follows:
\[
\varnothing,\quad
2\mapsto 1,\quad
2\mapsto 2,\quad
1\mapsto 1,\quad
1\mapsto 2,\quad
\begin{array}{c}1\mapsto 1\\[-2pt]2\mapsto 2\end{array},
\quad
\begin{array}{c}1\mapsto 2\\[-2pt]2\mapsto 1\end{array}.
\]

{\bf Bruhat order.}
The Hasse diagram for the  Bruhat order is shown below, where we label each vertex with its corresponding dimension $d_\sigma$.

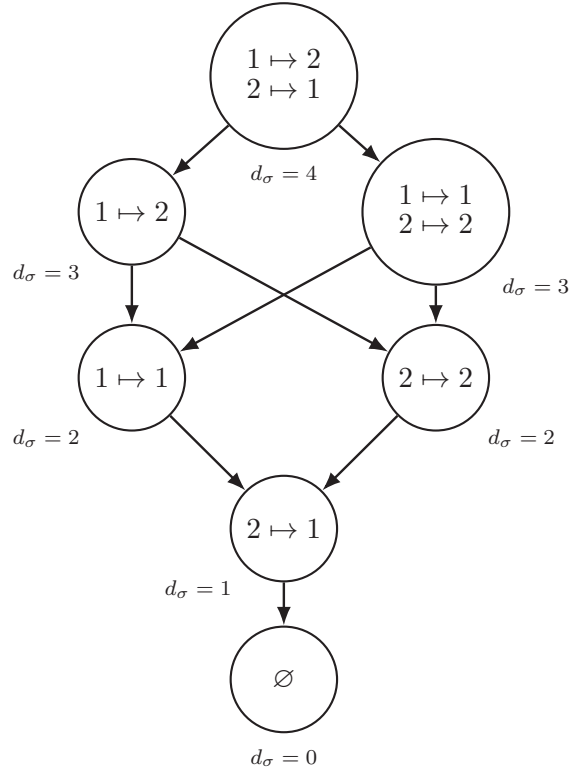
\begin{center}
\begin{tikzpicture}[node distance=18mm]
\WNode[below]{btop}{(0,0)}
  {$\begin{array}{c}1\mapsto 2\\[-2pt]2\mapsto 1\end{array}$}{$d_\s=4$}

\WNode[below left]{bL}{(-2,-1.8)}{$1\mapsto 2$}{$d_\s=3$}
\WNode[below right]{bR}{( 2,-1.8)}
  {$\begin{array}{c}1\mapsto 1\\[-2pt]2\mapsto 2\end{array}$}{$d_\s=3$}

\WNode[below left]{bLL}{(-2,-4)}{$1\mapsto 1$}{$d_\s=2$}
\WNode[below right]{bRR}{( 2,-4)}{$2\mapsto 2$}{$d_\s=2$}

\WNode[below left]{bmid}{(0,-6)}{$2\mapsto 1$}{$d_\s=1$}

\WNode[below]{bbot}{(0,-8.0)}{$\varnothing$}{$d_\s=0$}

\WEdge{btop}{bL}{0}
\WEdge{btop}{bR}{0}
\WEdge{bL}{bLL}{0}
\WEdge{bR}{bRR}{0}
\WEdge{bL}{bRR}{0}
\WEdge{bR}{bLL}{0}
\WEdge{bLL}{bmid}{0}
\WEdge{bRR}{bmid}{0}
\WEdge{bmid}{bbot}{0}
\end{tikzpicture}
\end{center}

{\bf IC sheaves and Kazhdan–Lusztig polynomials.}
\begin{enumerate}
    \item  If $\sigma\neq 1\mapsto 2$, the orbit closure $\overline\cO_\sigma$ is smooth (in fact a linear subspace in $\Hom(L_1,L_2)$). We have
    \[
      \IC_\sigma=\underline{\Ql}_{\overline{\cO}_\sigma}\langle d_\sigma\rangle
    \]
    and $P_{\sigma',\sigma}(t)=1$ for all $\sigma'<\sigma$. 
    \item If $\sigma= 1\mapsto 2$, then
    \[\overline{\cO}_\sigma=\{T\in\Hom( L_1,L_2)\mid\rank(T)=1\}.
    \]
    This is the determinantal variety of $2\times2$ matrices of rank $1$, a
    three–dimensional cone whose unique singular point is the origin
    $\{0\}$.  There is a small resolution
    \[
      \widetilde{\cO}_\sigma
      =\Bigl\{(T,L)\ \Big|\ L\subset L_1, \dim L=1,\ T\in\Hom( L_1, L_2),
          \ L\subset\ker T  
          \Bigr\}
    \]
    with projection
    \[
      \pi:\widetilde{\cO}_\sigma\longrightarrow\overline{\cO}_\sigma,\qquad (T,L)\longmapsto T, 
    \]
   and the exceptional fiber over the singular point is
    \[
      \pi^{-1}(0)\cong \bP(L_1)\cong\bP^1 .
    \]
Since $\pi$ is a small resolution, we have
$\IC_\sigma\simeq\pi_*\underline{\Ql}_{\widetilde{\cO}_\sigma}\langle 3\rangle$. A direct computation using the fiber over the singular point gives:
\[
\cH^i\bigl(j_{\emptyset}^*\IC_\sigma\bigr)
\cong H^{i+3}\bigl(\bP^1,\Ql\bigr)
=
\begin{cases}
\Ql & i=-3,-1,\\[3pt]
0 & \text{otherwise}.
\end{cases}
\]
Hence
\[
P_{\emptyset,\sigma}(t)=1+t^2,
\qquad
P_{\sigma',\sigma}(t)=1\quad (\sigma'\neq\emptyset,\ \sigma'<\sigma).
\]
\end{enumerate}

{\bf The $W$-graph.}
Combining the Bruhat graph, the above Kazhdan–Lusztig polynomials, and \Cref{Lem. type G or U type II}, we obtain the $W$-graph for the type II theta correspondence in the case $m=n=2$. All edge labels $\mu(e)$ are equal to $1$ and are omitted. 
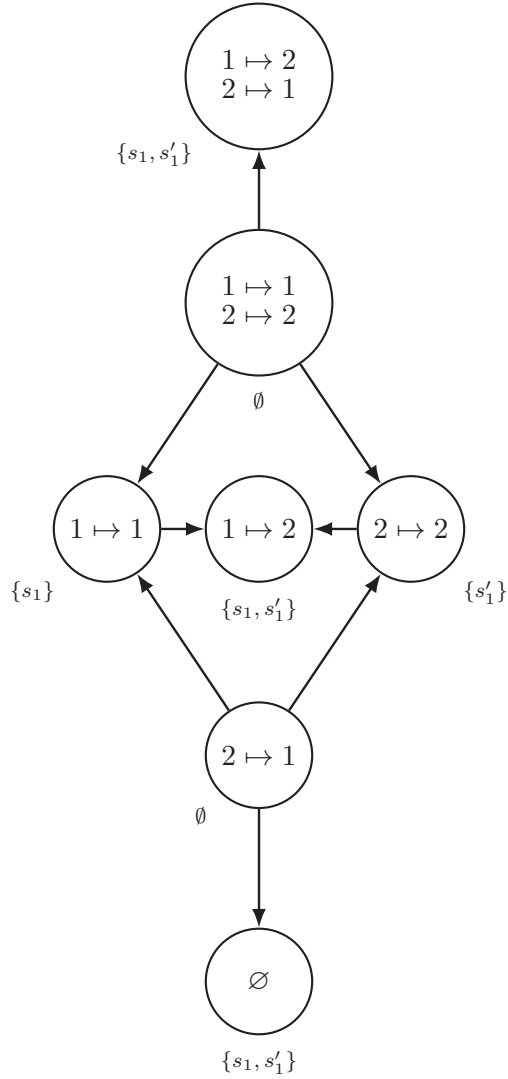
\begin{center}
\begin{tikzpicture}[node distance=18mm]
\WNode[below left]{btop}{(0,0)}
  {$\begin{array}{c}1\mapsto 2\\[-2pt]2\mapsto 1\end{array}$}{$\{s_1,s'_1\}$}

\WNode[below]{bR}{(0,-3)}
  {$\begin{array}{c}1\mapsto 1\\[-2pt]2\mapsto 2\end{array}$}{$\emptyset$}
  
\WNode[below left]{bLL}{(-2,-6)}{$1\mapsto 1$}{$\{s_1\}$}
\WNode[below right]{bL}{(2,-6)}{$2\mapsto 2$}{$\{s'_1\}$}
\WNode[below]{bRR}{(0,-6)}{$1\mapsto 2$}{$\{s_1,s'_1\}$}

\WNode[below left]{bmid}{(0,-9)}{$2\mapsto 1$}{$\emptyset$}

\WNode[below]{bbot}{(0,-12.0)}{$\varnothing$}{$\{s_1,s'_1\}$}

\WEdge{bR}{btop}{0}

\WEdge{bR}{bLL}{0}
\WEdge{bR}{bL}{0}

\WEdge{bLL}{bRR}{0}
\WEdge{bL}{bRR}{0}

\WEdge{bmid}{bLL}{0}
\WEdge{bmid}{bL}{0}

\WEdge{bmid}{bbot}{0}
\end{tikzpicture}
\end{center}
In this case, each vertex is a cell and there are seven cells.

\section{Hecke bimodules from type I theta correspondence}\label{sec Hecke type I}

In this section, we apply the general framework developed in \Cref{Sec. geom gen} to study the Hecke module arising from the type I theta correspondence for even ortho-symplectic dual pair.

\subsection{The type I theta correspondence over finite fields}\label{Sec type I finite field} 
We retain the notation in \Cref{theta correspondence finite field} and \Cref{Section Geometrization of oscillator bimodule}. 
In particular, $F=\bF_q$ is a finite field of characteristic not equal to $2$. In this section, we recall the Schr\"odinger models of the Weil representation $\omega$ of $G_1(\bF_q)\times G_2(\bF_q)$ and describe the partial Fourier transform which intertwines two different models. 

Let $V_1=L_1\oplus L_1^\vee$ be a polarization of $V_1$ such that $L_1$ is stable under $B_1$. We shall write an element in $G_1$ as a block matrix relative to this polarization. Let $P_1=M_1U_1$ be the Siegel parabolic subgroup of $G_1$ stabilizing $L_1$, where $M_1$ is the Levi component of $P_1$ stabilizing $L_1^\vee$ and $U_1$ is the unipotent radical of $P_1$. We have 
\[
M_1=\{m(a)|a\in \GL(L_1)\}, \quad U_1=\{u(b)|b\in \Herm (L_1^\vee,L_1)\}
\]
where
\begin{align*}
m(a)=\left(\begin{array}{cc}{a} & {} \\ {} &  (a^*)^{-1} \end{array}\right), \quad 
u(b)=\left(\begin{array}{cc}{1} & {b} \\ {} & 1\end{array}\right)
\end{align*}
and 
\[
\Herm(L_1^\vee,L_1)=\{b\in \Hom(L_1^\vee, L_1)|b^*=-b\}.
\]
Here, the elements $a^*\in \GL(L_1^\vee), b^*\in \Hom(L_1^\vee, L_1)$ are defined by requiring that 
\[
\begin{split}
    \langle a v, v'\rangle_{V_1}= \langle v, a^* v'\rangle_{V_1},\\
    \langle b v', v''\rangle_{V_1}=\langle v', b^* v''\rangle_{V_1}
\end{split}
\]
for $v\in L_1$ and $v',v''\in L_1^\vee$. 

Identifying $\bV$ with $\Hom(V_1, V_2)$ via the form on $V_1$. We have a polarization
\begin{equation}\label{pol bV}
    \bV=\Hom(L_1^\vee, V_2)\oplus \Hom(L_1,V_2)
\end{equation}
Let 
\begin{equation*}
    \bL_1=\Hom(L_1,V_2).
\end{equation*}
The polarization \eqref{pol bV} provides a Schr\"odinger model $\omega_{\bL_1}$ of the Weil representation realized on 
$\bC[\bL_1(\bF_q)]$. For $f\in \bC[\bL_1(\bF_q)]$ and $T\in \bL_1(\bF_q)$, the action of $P_1(\bF_q)\times G_2(\bF_q) $ is given as follows:
\begin{equation}\label{mix model equation}
\begin{aligned}
\omega(g_2) f (T) =&  f(g_2^{-1} T)\quad  &  g_2 \in G_2,\\
\omega(m(a)) f(T) =&   f( T  a) &  a \in \GL(L_1),\\
\omega(u(b)) f(T)=&  \half \psi(\langle T, T b\rangle_{\VV}) f(T) & b\in \Herm(L_1^\vee,L_1). \\
\end{aligned}
\end{equation}
We denote by $\overline{B}_1$ the image of $B$ under the projection map $P_1\rightarrow M_1\cong \GL(L_1)$. Then $\overline{B}_1$ is a Borel subgroup of $\GL(L_1)$ and we have an exact sequence 
\begin{equation}\label{exact B_1}
  1\longrightarrow U_1 \longrightarrow B_1 \longrightarrow \overline{B}_1\longrightarrow 1. 
\end{equation}
Recall that we defined the moment cone $\cNL{1}$ in \eqref{moment cone 1}. The following lemma can be easily deduced from \Cref{mix model equation}
\begin{lemma}\label{Lem. N_L invariant}
\begin{enumerate}
    \item The restriction map from $\bL_1$ to $\cN_{\bL_1}$ gives an isomorphism
\begin{equation}\label{NL1}
    \omega_{\bL_1}^{U_1(\bF_q)}\cong \bC[\cNL{1}(\bF_q)].
\end{equation}
The action of $\GL(L_1)(\bF_q)\times G_2(\bF_q)$ on $\omega^{U_1(\bF_q)}$ is transferred to the natural geometric action of $\GL(L_1)(\bF_q)\times G_2(\bF_q)$ on $\mathbb C[\cNL{1}(\bF_q)]$. 
\item \Cref{NL1} induce a canonical isomorphism of $\ov H_1\otimes H_2$-modules 
\begin{equation}\label{oscillator Bi-module Schrodinger model 1}
   M_1\coloneqq \omega_{\bL_1}^{B_1(\bF_q)\times B_2(\bF_q)}\cong \mathbb C[\cN_{\bL_1}(\bF_q)]^{\overline{B}_1(\bF_q)\times B_2(\bF_q)}= \mathbb C[(\overline{B}_1\times B_2)(\bF_q)\backslash \cN_{\bL_1}(\bF_q)].
\end{equation}
\end{enumerate}
\end{lemma}
 For each $0\leq k\leq \min\set{m,n}$, we define 
\[
\cNL{1}^k=\{T\in  \cNL{1}|\rank T=k\}.
\]
Then we have a stratification
\[
 \cNL{1}=\bigsqcup_{k=0}^{\min\set{m,n}}  \cNL{1}^k,
\]
and each $ \cNL{1}^k$ is a single $\GL(L_1)\times G_2$-orbit. 

Now we switch the role of $V_1$ and $V_2$. Let $V_2=L_2\oplus L_2^\vee$ be a polarization of $V_2$ stable under $B_2$ and $P_2=M_2U_2$ be the Siegel parabolic subgroup of $G_2$ stabilizing $L_2$, where $M_2$ is the Levi component of $P_2$ stabilizing $L_2^\vee$ and $U_2$ is the unipotent radical of $P_2$. 

This polarization provides another Schr\"odinger model $\omega_{\bL_2}$ of the Weil representation realized on $\bL_2\coloneqq \bC[\Hom(L_2, V_1)]$ and we can write the action of $G_1\times P_2$ similar to \Cref{mix model equation}. Recall that we defined another moment cone $\cNL{2}$ in \Cref{moment cone 2}. In parallel to \Cref{Lem. N_L invariant}, there is a canonical isomorphism of $H_1\otimes \ov H_2$-modules 
\begin{equation}\label{oscillator Bi-module Schrodinger model 2}
   M_2\coloneqq \omega_{\bL_2}^{B_1(\bF_q)\times B_2(\bF_q)}\cong \mathbb C[ \cNL{2}(\bF_q)]^{B_1(\bF_q)\times \overline{B}_2(\bF_q)}=  \mathbb C[(B_1\times \ov B_2)(\bF_q)\backslash \cN_{\bL_2}(\bF_q)]. 
\end{equation}

Write
\begin{eqnarray*}
    \ov \bL_1 \coloneqq  \Hom(L_1,L_2), \quad \ov \bL_2 := \Hom(L_2,L_1),\\
    \bL^b \coloneqq \Hom(L_1,L_2^\vee)=\Hom(L_2,L_1^\vee).
\end{eqnarray*}
where $\Hom(L_1,L_2^\vee)$ and $\Hom(L_2,L_1^\vee)$ are both naturally identified with $L_1^\vee \otimes L_2^\vee$. Then we have
\begin{equation}\label{L1L2 decomp}
    \bL_1=\Hom(L_1,V_2)= \ov \bL_1 \oplus \bL^b, \quad \bL_2=\Hom(L_2,V_1)=\ov \bL_2  \oplus \bL^b.
\end{equation}

Fix a nontrivial additive character $\psi: \bF_q\rightarrow \bC^\times$.
The two models $\omega_{\bL_1}$ and $\omega_{\bL_2}$ are intertwined by the partial Fourier transform  
\begin{equation}\label{partial FT function}
   \CM_{\bL_1}^{\bF_q}: \bC[\bL_1(\bF_q)]\longrightarrow  \bC[\bL_2(\bF_q)],
\end{equation}
defined by 
\begin{equation}\label{Equa. Partial Fourier}
 \CM_{\bL_1}^{\bF_q}(f)(T_2,T)=q^{-\half \dim \ov \bL_1} \sum_{T_1\in \ov \bL_1(\bF_q)} f\left( (T_1,T)\right) \psi(\langle T_1, T_2\rangle), f\in \bC[\bL_1(\bF_q)], T_2\in \ov \bL_2(\bF_q), T\in \bL^b(\bF_q).
\end{equation} 
where $\langle T_1,T_2 \rangle$ is the canonical trace pairing. The partial Fourier transform $\CM_{\bL_1}^{\bF_q}$ induce an isomorphism of $H_1\otimes H_2$-modules between $M_1$ and $M_2$. We denote the common isomorphism class by $M$.

\subsection{Geometrization of the oscillator bimodule}\label{section type I ocilator}

In this section we apply the general results from \Cref{Sec. geom gen} to give a geometrization of the oscillator bimodule $M$. 
\subsubsection{Sheaves on the moment cone}

We retain the notation from \Cref{Sec type I finite field}. For $i = 1,2$, let $W_i$, $\sfH_i$, and $H_i$ denote the Weyl group, the generic Hecke algebra, and the Iwahori–Hecke algebra associated to $G_i$, respectively. Similarly, let $\ov W_i$, $\bsfH_i$, and $\ov H_i$ denote the corresponding objects for $\ov G_i$, whose structure is described in \Cref{Geometrization of the oscillator bimodule type II}. In this subsection, we describe $W_i$, $\sfH_i$, and $H_i$ explicitly. 

We fix a basis $\{e_1, \dots, e_m\}$ of $L_1$ and the dual basis $\{e_{-1}, \dots, e_{-m}\}$ of $L_1^\vee$, such that the complete flag
\[
0 \subset \langle e_1 \rangle \subset \cdots \subset \langle e_1, \dots, e_m \rangle \subset \langle e_1, \dots, e_m, e_{-m} \rangle \subset \cdots \subset \langle e_1, \dots, e_m, e_{-m}, \dots, e_{-1} \rangle
\]
is stabilized by the Borel subgroup $B_1$. Let $T_1$ be the maximal torus in $B_1$ fixing the line $\langle e_i\rangle$ for $i=\pm 1,\cdots,\pm m$.

Likewise, fix a basis $\{f_1, \dots, f_n\}$ of $L_2$ and dual basis $\{f_{-1}, \dots, f_{-n}\}$ of $L_2^\vee$ such that the complete flag
\[
0 \subset \langle f_1 \rangle \subset \cdots \subset \langle f_1, \dots, f_n \rangle \subset \langle f_1, \dots, f_n, f_{-n} \rangle \subset \cdots \subset \langle f_1, \dots, f_n, f_{-n}, \dots, f_{-1} \rangle
\]
is stabilized by $B_2$. Let $T_2$ be the maximal torus in $B_2$ fixing the line $\langle f_i\rangle$ for $i=\pm 1,\cdots,\pm n$.

For $1 \le i \le m-1$, let $s_i \in W_1$ denote the image in $W_1 \cong N_{G_1}(T_1)/T_1$ of an element in $N_{G_1}(T_1)$ that permutes $e_i$ and $e_{i+1}$ (and hence also $e_{-i}$ and $e_{-(i+1)}$). Likewise, let $t_m \in W_1$ denote the image of the element that permutes $e_m$ and $e_{-m}$. Then $W_1 \cong \bW_m$ is a Coxeter group of type $BC_m$, with simple reflections $\sfS_1 = \{s_1, \dots, s_{m-1}, t_m\}$.

For $1 \le i \le n-1$, let $s_i' \in W_2$ denote the image in $W_2 \cong N_{G_2}(T_2)/T_2$ of an element in $N_{G_2}(T_2)$ that permutes $f_i$ and $f_{i+1}$ (and hence also $f_{-i}$ and $f_{-(i+1)}$), and let $t_n' \in W_2$ denote the image of the element that swaps $f_n$ and $f_{-n}$. Then $W_2 \cong \bW_n$ is of type $BC_n$ with simple reflections $\sfS_2 = \{s_1', \dots, s_{n-1}', t_n'\}$. 

Note that the Coxeter systems $(\ov W_i, \ov \sfS_i)$ embed naturally as Coxeter subsystems of $(W_i, \sfS_i)$ for $i = 1,2$.
The generic Hecke algebra $\sfH_2$ is associated to the Coxeter system $(W_2, \sfS_2)$; see \Cref{Geo hecke algebra}. The generic Hecke algebra $\sfH_1$ requires additional care, as $G_1$ is disconnected and the corresponding Hecke algebra carries unequal parameters. Define a length function $l : W_1 \to \bZ_{\geq 0}$ by 
\begin{equation}\label{length H1}
 l(w) := \dim \cO_w \subseteq \cB_1,   
\end{equation}
where $\cO_w$ is the $B_1$-orbit in the flag variety $\cB_1 := G_1 / B_1$ corresponding to $w \in W_1$. Then
\[
l(s_1) = \cdots = l(s_{m-1}) = 1, \quad l(t_m) = 0.
\]

The generic Hecke algebra $\sfH_1$ is the unique associative $\bZ[v^{\pm 1}]$-algebra with basis $\{\sfT_w \mid w \in W_1\}$ satisfying the relations:
\begin{enumerate}[label=(\alph*),wide=0pt]
    \item $(\sfT_{s_i} + 1)(\sfT_{s_i} - v^2) = 0$, for $1 \le i \le m-1$,
    \item $(\sfT_{t_m} + 1)(\sfT_{t_m} - 1) = 0$,
    \item $\sfT_{w_1} \sfT_{w_2} = \sfT_{w_1 w_2}$ whenever $l(w_1 w_2) = l(w_1) + l(w_2)$.
\end{enumerate}
We have 
\[
 H_1\cong \sfH_{1,v=\sqrt{q},\bC},\quad  \bC[W_1]\cong  \sfH_{1,v=1,\bC}. 
\]
The sign character $\sgn$ of $\sfH_1$ is defined by
\begin{equation}\label{sgn orthogonal group}
    \sgn(\sfT_w) = (-1)^{\ell(w)}.
\end{equation}
The similar formula also defines the sgn character on $H_1$ and on $W_1$.
We denote by $\cH^\Tate_2, \cH^\mix_2, \cH_2$ and $\cH^k_{2}$ the various Hecke categories for the connected reductive group $G_2$. For $i=1,2$, we also denote by $\bcH^\Tate_i,\bcH^\mix_i, \bcH_i$ and $\bcH^k_{i}$ the various Hecke categories for $\ov G_i=\GL(L_i)$. Since $G_1$ is disconnected, we postpone the discussion of its Hecke category until \Cref{ss:moment cone L2}

As in \Cref{moment cone 1}, we form the moment cone $\cN_{\bL_1}$, this time as a scheme over $\bF_q$. Then $\cN_{\bL_1}$ carries an action of $\ov G_1\times G_2$.
We would now like to apply the discussion in \Cref{Sec. Hk mod} to the action of $\barB_1\times B_2$ on $X=\cN_{\bL_1}$. The following combinatorial definition classifies $\overline{B}_1\times B_2$-orbits on $\cNL{1}$.

\begin{defn}\label{def signed partial matching}
\begin{enumerate}
    \item A \em{signed partial matchings} between  $\{\pm1,\cdots, \pm m\}$ and  $\{\pm1,\cdots, \pm n\}$ is a triple $\s=(I,J,\mu)$ where 
\begin{itemize}
    \item $I$ and $J$ are two subsets of $\{\pm1,\cdots, \pm m\}$ and  $\{\pm1,\cdots, \pm n\}$ such that $I= -I, J=-J$ and $|I|=|J|$; 
    \item $\mu: I\rightarrow J$ is a bijection with $\mu(-a)=-\mu(a)$ for every $a\in I$. 
\end{itemize}
We denote by $\SPM(m,n)$ the set of signed partial matchings between $\{\pm1,\cdots, \pm m\}$ and $\{\pm1,\cdots, \pm n\}$.

\item For each $0\leq k\leq \min\set{m,n}$, we define the subset $\SPM(m,n,k)\subset \SPM(m,n)$ consisting of triples $(I,J,\mu)$ in $\SPM(m,n)$ with $|I|=|J|=2k$. 
We have 
\[
\SPM(m,n)= \bigsqcup_{k=0}^{\min\set{m,n}}\SPM(m,n,k).
\]

\item For any $(\s:I\to J)\in \SPM(m,n)$, we define $I^+=I\cap [\{1,\cdots, m\}], I^-=I\cap \{-1,\cdots,-m\}, J^+=J\cap \{1,\cdots, n\}$ and $J^-=J\cap \{-1,\cdots,-n\}$.
\end{enumerate}
\end{defn}
The natural $W_1\cong \bW_m$-action on $\{\pm1,\cdots, \pm m\}$ and $W_2\cong \bW_n$-action on $\{\pm1,\cdots, \pm n\}$ induce an action of $W_1\times W_2$ on $\SPM(m,n)$ that we denote by $*$. The formula for the $*$-action is also given by \eqref{action Sm Sn on PM}.

For each $\s=(I,J,\mu)\in \SPM(m,n)$, we define $x_\sigma\in \cNL{1}\subseteq \bL_1=\Hom(L_1,V_2)$ by 
\begin{equation}\label{def:Tsigma type I}
x_\sigma(e_i) =\begin{cases}
    f_{\mu(i)}\quad &\mbox{if $i\in I^+$}\\
    0 \quad &\mbox{if $i\in \{1,\cdots, m\}\backslash I^+$}.
\end{cases}
\end{equation}
Let $\cO_{\sigma}\subset \cNL{1}$ denote the geometric  $\overline{B}_1\times B_2$-orbit of $T_{\sigma}$.  Let $\one_{\sigma}$ be the characteristic function of $\cO_\s(\bF_q)$. We also denote $d_\s=\dim \cO_\s$.
\begin{lemma}\label{SPM vs orbits}
\begin{enumerate}
    \item The map $\sigma\mapsto \cO_{\sigma}$ gives a bijection from $\SPM(m,n)$ to the set of geometric $\overline{B}_1\times B_2$-orbits on $\cNL{1}$. 
    \item The assumption \eqref{orbit assumption} holds for the $\ov B_1 \times  B_2$-action on $\cNL{1}$. In particular, each $\cO_\s$ is a single $\overline{B}_1(\bF_q)\times B_2(\bF_q)$-orbit, and $\{\one_\sigma|\sigma\in \SPM(m,n)\}$ is a $\bC$ basis for $M_1$. 
\end{enumerate}
\end{lemma}
\begin{proof}
This follows the same idea as in \Cref{PM vs orbits} and we omit the proof here. 
\end{proof}


By the above lemma, results from \Cref{Sec. Hk mod} are applicable.
Applying the general framework from \Cref{Sec. Hk mod} to this situation, we obtain the following categories
\begin{equation*} 
\cM^\Tate_1=D^{\Tate}_{\barB_1\times B_2}(\cN_{\bL_1}) \hookrightarrow \cM^\mix_1=D^{\mix}_{\barB_1\times B_2}(\cN_{\bL_1})\hookrightarrow  \cM_1=D^{b}_{\barB_1\times B_2}(\cN_{\bL_1})\to    \cM^k_{1}=D^{b}_{\barB_{1,k}\times B_{2,k}}(\cN_{\bL_1,k}).
\end{equation*}
The category $\cM_1$ carries commuting actions of $\bcH_1$ and $\cH_2$; similar statement holds for their mixed and $k$-versions. 
Let $\sfM_1:=R[\un{\ov B_1\times B_2\bs \cN_{\bL_1}}]$. The following is the special case of  \Cref{Lem. ch isom} , \Cref{Prop ICs preserves Tate} and the functoriality of the diagram \eqref{diag cat}:
\begin{cor}\label{Cor. K0 M1}
\begin{enumerate}
    \item The full subcategory $\cM^{\Tate}_1\subset \cM_1$ is stable under the commuting convolution actions of $\bcH^\Tate_1$ and $\cH^\Tate_2$. In particular, $K_0(\cM^{\Tate}_1)$ carries a canonical structure of a $\ov\sfH_1\ot_R \sfH_2$-module. 
    \item The map $\ch^\Tate: K_0(\cM^\Tate_1)\to \sfM_1$ is an isomorphism of $R$-modules, equipping $\sfM_1$ with a natural structure of a $\ov\sfH_1\ot_R \sfH_2$-module. 
    \item  Specializing the $\bsfH_1 \otimes_R \sfH_2$-action on $\sfM_1$ at $v=\sqrt{q}$, we recover the $\ov H_1\ot H_2$-action on $M_1$ define in \Cref{Sec type I finite field} (via a fixed isomorphism $\iota:\Qlbar \cong \bC$).
    \item Specializing the $\bsfH_1 \otimes_R \sfH_2$-action on $\sfM_1$ at $v=1$ defines an action of $\ov W_1\times W_2$ on $\ov M_{1,v=1}=\ZZ[\SPM(m,n)]$. 
\end{enumerate}
\end{cor}

\begin{remark}
    The $\ov\sfH_1\ot_R \sfH_2$-module structure on $\sfM_1$ is the same as the one defined in \cite[\S 5]{MR4677077}. This is easy to check for the actions of the generators $\sfT_s, s\in \ov \sfS_1\sqcup \sfS_2$, but we do not need this fact for the rest of the paper.
\end{remark}

For each $\sigma = (I, J, \mu) \in \SPM(m,n)$, let $j_\sigma: \cO_\sigma\rightarrow \cNL{1}$ be the inclusion. We have the standard and IC sheaves 
\[
\Delta_\sigma=j_{\sigma,!}(\underline{\Ql} )\braket{d_\sigma}
\quad\mbox{and}\quad  \IC_\sigma=(j_{\sigma,!*}(\underline{\Ql} )\braket{d_\sigma}. 
\]
We define 
\begin{equation}\label{extended sigma 3}
    \sigma^{\dagger}: \{\pm 1,\cdots, \pm m\} \longrightarrow \{0,\pm 1,\pm 2,\cdots, \pm n\}
\end{equation}
by 
\[
\sigma^{\dagger}(i)=\begin{cases}
    \sigma(i) \quad & \mbox{if $i\in I$}\\
    0 \quad & \mbox{if $i \notin I$}.
\end{cases}
\]
We also define 
\begin{equation}\label{extended sigma 4}
    \sigma_{\dagger}: \{\pm 1,\cdots, \pm n\} \longrightarrow \{ \pm 1,\pm 2,\cdots, \pm m,\infty\}
\end{equation}
by 
\[
\sigma_{\dagger}(j)=\begin{cases}
    \sigma^{-1}(j) \quad & \mbox{if $j\in J$}\\
    \infty & \mbox{if $j \notin J$}.
\end{cases}
\]
We define orders  $<_{\SPM}$ on the set $\{0, \pm 1,\pm 2,\cdots, \pm n\}$, $\{\pm 1,\pm 2,\cdots, \pm m,\infty\}$ by 
\begin{equation}\label{eq order SPM}
   \begin{split}
    0<_{\SPM}<1<_{\SPM} 2<\cdots<_{\SPM} n <_{\SPM}(-n)<_{\SPM}-(n-1)<_{\SPM}\cdots<_{\SPM} -1\\
   1<_{\SPM} 2<\cdots<_{\SPM} m <\infty  <_{\SPM}(-m)<_{\SPM}-(m-1)<_{\SPM}\cdots<_{\SPM} -1
\end{split} 
\end{equation}
The following lemma computes the type $(s,\cO_\sigma)$ for each  simple reflection $s$ in $\ov W_1 \times  W_2$ and the orbit $\cO_\sigma$ with $\sigma\in \SPM(m,n)$. 
\begin{lemma}\label{Lem. type G or U type I}
Let $s\in \ov \sfS_1\sqcup \sfS_2$ be a simple reflection in $\ov W_1 \times  W_2$, and let $\sigma = (I, J, \mu) \in \SPM(m,n)$. Then the pair $(s, \cO_\sigma)$ is of type (G) or (U). More precisely, 
\begin{enumerate}
    \item If  $s = s_i$ for $1 \leq i \leq  m-1$, then $(s, \cO_\sigma)$ is of
    \[
    \begin{cases}
        \mbox{type $G$}\quad &\mbox{if $\sigma^\dagger(i)=\sigma^\dagger(i+1)$}\\
         \mbox{type $U^-$}\quad &\mbox{if $\sigma^\dagger(i)<_{\SPM}\sigma^\dagger(i+1)$}\\
        \mbox{type $U^+$}\quad &\mbox{if $\sigma^\dagger(i)>_{\SPM}\sigma^\dagger(i+1)$}.
    \end{cases}
    \] 
    \item If  $s = s_i'$ for $1 \leq i \leq  n-1$, then $(s, \cO_\sigma)$ is of
    \[
    \begin{cases}
        \mbox{type $G$}\quad &\mbox{if $\sigma_\dagger(i)=\sigma_\dagger(i+1)$}\\
         \mbox{type $U^-$}\quad &\mbox{if $\sigma_\dagger(i)<_{\SPM}\sigma_\dagger(i+1)$}\\
        \mbox{type $U^+$}\quad &\mbox{if $\sigma_\dagger(i)>_{\SPM}\sigma_\dagger(i+1)$}.
    \end{cases}
    \]
    \item  If  $s = s_n'$, then $(s, \cO_\sigma)$ is of
    \[
    \begin{cases}
        \mbox{type $G$}\quad &\mbox{if $\sigma_\dagger(n)=\infty $}\\
         \mbox{type $U^-$}\quad &\mbox{if $\sigma_\dagger(n)<_{\SPM}\infty$}\\
        \mbox{type $U^+$}\quad &\mbox{if $\sigma_\dagger(n)>_{\SPM}\infty$}.
    \end{cases}
    \]
\end{enumerate}
In case $U$, the $s$-companion of $\cO_\sigma$ is $\cO_{s * \sigma}$. 
\end{lemma}
\begin{proof}
The proof is similar to that of \Cref{Lem. type G or U type II} and is omitted here.
\end{proof}
\trivial[h]{
We write down the details for the case $\dim L_1=2$ and $\dim V_2=4$. $\barG_1\times G_2=\GL(L_1)\times \Sp(V_2)$ acting on $\cNL{1}$. Fix a basis $\langle e_1,e_2\rangle$ of $L_1$ and basis $\langle f_1,f_2,f_{-2},f_{-1}\rangle$ of $V_2$. The orbits are parametrized by $\SPM(2,2)$. It has 
\begin{enumerate}
    \item rank $0$, the zero map $\sigma_0$
    \item rank $1$, $\sigma_{1\mapsto 1}$, $\sigma_{1\mapsto 2}$, $\sigma_{1\mapsto -2}$, $\sigma_{1\mapsto -1}$,  $\sigma_{2\mapsto 1}$, $\sigma_{2\mapsto 2}$, $\sigma_{2\mapsto -2}$, $\sigma_{2\mapsto -1}$;
    \item rank $2$, $\sigma_{1\mapsto 1, 2\mapsto2}$,  $\sigma_{1\mapsto 1, 2\mapsto -2}$, $\sigma_{1\mapsto 2,2\mapsto 1}$,$\sigma_{1\mapsto 2,2\mapsto -1}$, $\sigma_{1\mapsto -2, 2\mapsto 1}$, $\sigma_{1\mapsto -2,2\mapsto -1}$, $\sigma_{1\mapsto -1, 2\mapsto2}$, $\sigma_{1\mapsto -1,2\mapsto -2}$. 
\end{enumerate}
The Weyl group $\ov W_1=\langle s_1\rangle, W_2=\langle s'_1, s_2'\rangle$. 

\begin{enumerate}
    \item The simple reflection $s_1$. The corresponding parabolic is $\barG_1\times B_2$. We compute $\Stab_{\barG_1\times B_2}(x_\sigma)$ and see how it acts on $\mathbb P(L_1)$. We write $g_1\in \barG_1, g_2\in B_2$. The action is given by $(g_1,g_2)\cdot T= g_2\cdot T\cdot g_1^{-1}$. 
    \begin{enumerate}
        \item For rank $1$-orbit: 
        \begin{itemize}
    \item $\sigma_{1\mapsto 1}$ corresponds to $T(e_1)=f_1, T(e_2)=0$. Let $(g_1,g_2)\in \Stab_{\barG_1\times B_2}(T)$, then we have 
    \[
    \begin{split}
  g_2 \cdot T \cdot g_1^{-1}(e_2)=T(e_2)=0 \Rightarrow     T\cdot g_1^{-1}(e_2)= 0 \Rightarrow g_1^{-1} (e_2)\in \Ker(T)=\langle e_2\rangle. 
    \end{split}
    \]
So $g_1$ stabilizes the line $\langle e_2\rangle$ and we conclude that $\sigma_{1\mapsto 1}$ is of type $U^+$; 
    \item Similar arguments holds for $\sigma_{1\mapsto 2}$, $\sigma_{1\mapsto -2}$, $\sigma_{1\mapsto -1}$, they are both of type $U^+$; 
    \item $\sigma_{2\mapsto -1}$ corresponds to $T(e_1)=0, T(e_2)=f_{-1}$. Let $(g_1,g_2)\in \Stab_{\barG_1\times B_2}(T)$, then we have 
    \[
    \begin{split}
  g_2 \cdot T \cdot g_1^{-1}(e_1)=T(e_1)=0 \Rightarrow     T\cdot g_1^{-1}(e_1)= 0 \Rightarrow g_1^{-1} (e_1)\in \Ker(T)=\langle e_1\rangle. 
    \end{split}
    \]
So $g_1$ stabilizes the line $\langle e_1\rangle$ and we conclude that $\sigma_{1\mapsto 1}$ is of type $U^-$; 
\end{itemize}
\zjl{To summary, if $(g_1,g_2)\in \Stab_{\barG_1\times B_2}(T)$, then we know that $g_1$ stabilizes $\ker(T)$. }
\item For rank $2$-orbit:
\begin{itemize}
    \item $\sigma_{1\mapsto 1, 2\mapsto2}$ corresponds to $T(e_1)=f_1, T(e_2)=f_2$. 
     \[
    \begin{split}
  g_2 \cdot T \cdot g_1^{-1}(e_1)=T(e_1)=f_1 \Rightarrow     T\cdot g_1^{-1}(e_1)= g_2^{-1}(f_1)=\lambda f_1 \Rightarrow g_1^{-1} (e_1)=T^{-1}(\lambda f_1)=\lambda e_1 . 
    \end{split}
    \]
So $g_1$ stabilizes the line $\langle e_1\rangle$ and we conclude that $\sigma_{1\mapsto 1, 2\mapsto2}$ is of type $U^-$;
 \item $\sigma_{1\mapsto -2, 2\mapsto -1}$ corresponds to $T(e_1)=f_{-2}, T(e_2)=f_{-1}$. 
     \[
    \begin{split}
  & g_2 \cdot T \cdot g_1^{-1}(e_1)=T(e_1)=f_{-2} \Rightarrow     T\cdot g_1^{-1}(e_1)= g_2^{-1}(f_{-2})=\lambda  f_{-2}+\mu f_2+\gamma f_1 \\
  &  \Rightarrow \mu=\gamma=0 ,\,\, g_1^{-1} (e_1)=\lambda e_1 
    \end{split}
    \]
So $g_1$ stabilizes the line $\langle e_1\rangle$ and we conclude that $\sigma_{1\mapsto -2, 2\mapsto -1}$ is of type $U^-$;
\item Similar arguments works for other orbits, we conclude that \\
Type $U^-$: $\sigma_{1\mapsto 1, 2\mapsto2}$,  $\sigma_{1\mapsto 1, 2\mapsto -2}$, $\sigma_{1\mapsto 2,2\mapsto -1}$, $\sigma_{1\mapsto -2,2\mapsto -1}$; \\
Type $U^+$: $\sigma_{1\mapsto 2,2\mapsto 1}$, $\sigma_{1\mapsto -2, 2\mapsto 1}$,$\sigma_{1\mapsto -1, 2\mapsto2}$, $\sigma_{1\mapsto -1,2\mapsto -2}$. 
\end{itemize}
\zjl{To summary: if $(g_1,g_2)\in \Stab_{\barG_1\times B_2}(T)$, then we know that $g_1$ stabilizes $T^{-1}(f)$ where $f$ is the first vector in the order $<_{\SPM}$.}
    \end{enumerate}
\item The simple reflection $s_1'$. The corresponding parabolic is $\barB_1\times P_{s_1'}$ where $P_{s_1'}$ is the Siegel parabolic of $G_2$ stabilizing $L_2$. We compute $\Stab_{ \barB_1\times P_{s_1'}}(x_\sigma)$ and see how it acts on $\mathbb P(L_2)$.
\begin{enumerate}
    \item For rank $1$ orbits: 
    \begin{itemize}
    \item $\sigma_{1\mapsto 1}$ corresponds to $T(e_1)=f_{1}, T(e_2)=0$.
         \[
    \begin{split}
  g_2 \cdot T \cdot g_1^{-1}(e_1)=T(e_1)=f_1 \Rightarrow  
  g_2 \cdot T \cdot \lambda e_1 =f_1\Rightarrow 
 g_2 \cdot \lambda f_1 =f_1 
    \end{split}
    \]
So $g_2$ stabilizes the line $\langle f_1\rangle$ and we conclude that $\sigma_{1\mapsto 1}$ is of type $U^-$; 
    \item $\sigma_{1\mapsto -2}$ corresponds to $T(e_1)=f_{-2}, T(e_2)=0$. Similarly, we get $g_2 \cdot  f_{-2} =\lambda f_{-2}$. So $g_2$ stabilizes the line $\langle f_{-2}\rangle$, which will also stabilize the line $\langle f_1\rangle $. We conclude that $\sigma_{1\mapsto -2}$ is of type $U^-$; 
     \item $\sigma_{2\mapsto -1}$ corresponds to $T(e_2)=f_{-1}, T(e_1)=0$. 
       \[
    \begin{split}
  g_2 \cdot T \cdot g_1^{-1}(e_2)=T(e_2)=f_{-1} \Rightarrow  
  g_2 \cdot T \cdot (\lambda e_2+\mu e_1) =f_{-1}\Rightarrow 
 g_2 \cdot \lambda f_{-1} =f_{-1} 
    \end{split}
    \]
     So $g_2$ stabilizes the line $\langle f_{-1}\rangle$, which will also stabilize the line $\langle f_2\rangle $. We conclude that $\sigma_{2\mapsto -1}$ is of type $U^+$; 
     \item Similar arguments works for other orbits, we conclude that \\
Type $U^-$: $\sigma_{1\mapsto 1}$,  $\sigma_{2\mapsto 1}$, $\sigma_{1\mapsto -2}$, $\sigma_{2\mapsto -2}$; \\
Type $U^+$: $\sigma_{1\mapsto 2}$, $\sigma_{2\mapsto 2}$,$\sigma_{1\mapsto -1}$, $\sigma_{2\mapsto -1}$.\\

\zjl{To summary: if $(g_1,g_2)\in \Stab_{\barG_1\times B_2}(T)$, then we know that $g_2$ stabilizes $T(e)$ where $e$ is the first vector in the order $<_{\SPM}$.}
 \end{itemize} 
\item For rank $2$-orbits; 
\end{enumerate}
\end{enumerate}
}
\begin{remark}\label{remark assupmtion orbit type I}
The proof of \Cref{Lem. type G or U type I} also holds when $\bF_q$ is replaced by any field $F$. Thus, the action of $\ov B_1 \times  B_2$ on $\cNL{1}$ satisfies the assumption \eqref{orbit assumption bFq bC}.
\end{remark}

\begin{cor}\label{Knop action type I}
The action of $\ov W_1 \times W_2$ on $\bZ[\SPM(m,n)]$ via \Cref{Cor. K0 M1} (4) is given by
\[
w \cdot \one_\sigma = \one_{w * \sigma} \quad \text{for all } w\in \ov W_1\times  W_2, \sigma \in \SPM(m,n).
\]
\end{cor}
\begin{proof}
It suffices to check the case for $w\in \ov \sfS_1\sqcup \sfS_2$ ,  in which case the statement follows from \Cref{Cor. W action sph orb} and \Cref{Lem. type G or U type I}.
\end{proof}

Next we give generating objects of $\sfM_1$ under the $\ov \sfH_1\ot_{R} \sfH_2$ action. Let $\SPM(m,n)_0\subset \SPM(m,n)$ be the subset of signed partial matchings $\s$ such that for any simple reflection $s\in \ov \sfS_1\sqcup \sfS_2$, $(s,\cO_\s)$ is either of type (G) or type (U-). This is consistent with the notation $I_0$ introduced in \Cref{def. I0}.

Let $0\le i\le \min\{m,n\}$. Consider the signed partial matching $\s_i\in \SPM(m,n,i)$ by 
\begin{equation*}
    \s_i=(I=\{\pm(m-i+1),\cdots, \pm m\}\xr{\mu_i} \{\pm 1,\pm 2,\cdots, \pm i\}=J)
\end{equation*}
where $\mu_i$ is the unique order-preserving bijection with respect to the order $<_\SPM$  defined in \Cref{eq order SPM}. 
\begin{lemma}\label{Lem. I0 for SPM}
    $\SPM(m,n)_0=\{\s_i|0\le i\le \min\{m,n\}\}$.
\end{lemma}
\begin{proof}
    This follows from \Cref{Lem. type G or U type I}.
\end{proof}

The following is a direct calculation.
\begin{lemma}[of \Cref{Lem. KL poly ch}]\label{Lem. linear IC type I}
    For each $0\le i\le \min\{m,n\}$ let  $\bL^i_1\subset \bL_1$ be the linear subspace consisting of maps $T: L_1\to V_2$ satisfying
    \begin{itemize}
        \item $T(e_1)=\cdots=T(e_{m-i})=0$;
        \item $T(e_{m-i+a})\subset \langle f_1,\cdots, f_{a}\rangle$ for $1\le a\le i$.
    \end{itemize}
    Then $\ov {\cO_{\s_{i}}}=\bL^i_1$ and $d_{\s_i}=\dim \bL^i_1=\frac{i(i+1)}{2}$. Moreover, 
    \begin{equation}\label{C' for sigma i}
        C'_{\s_i}=v^{-d_{\s_i}}\sum_{\s\le \s_i}\one_\s=v^{-\frac{i(i+1)}{2}}\sum_{\s\le \s_i}\one_\s.
    \end{equation}
\end{lemma}

\begin{cor}[of \Cref{Lem. heart gen}]
    As a $\ov \sfH_1\ot_R \sfH_2$-module, $\sfM_1$ is generated by $C'_{\s_i}$ for $0\le i\le \min\{m,n\}$. In particular, as a $\bZ[\ov W_1\times  W_2]$-module, $K_0(\cM^k_1)\cong \sfM_{1,v=1}$ is generated by $\chi(\IC_{\s_i})$ for $0\le i\le \min\{m,n\}$.
\end{cor}

\begin{cor}\label{Cor. M1 IC} For any $\s\in \SPM(m,n)$,  $\IC_\s$ is Tate, $*$-pure of weight zero, and $*$-parity with parity $l(\s)$.
\end{cor}
\begin{proof} 
The proof mirrors that of \Cref{Cor. M1 IC typeII}, with \Cref{Lem. linear IC type I} replacing \Cref{Lem. linear IC type II}.
\end{proof}
We have the Kazhdan-Lusztig polynomials $P_{\s',\s}\in\bZ[t]$ in our situation, as defined in general in \Cref{Sec. Hk mod}. By \Cref{Cor. M1 IC} and \Cref{Lem. KL poly ch}, $P_{\s',\s}$ has only even degree terms and  $P_{\s',\s}\in\bZ_{\geq 0} [t]$.
\begin{lemma}\label{Lem. KL indep p} For any $\s\in \SPM(m,n)$, the element $C'_\s\in  \sfM_1$ is independent of $\bF_q$. In particular, for any $\s, \s'\in \SPM(m,n)$, the Kazhdan-Lusztig polynomial $P_{\s',\s}$ is independent of $\bF_q$. 
\end{lemma}
\begin{proof}
The argument follows that of \Cref{Lem. KL indep p type II}, except that we use \Cref{C' for sigma i} in place of \eqref{C' for bar sigma i}.
\end{proof}

\subsubsection{Generators as $\bsfH_1\ot_{R}  \bsfH_2$-module}
Next we give generating objects of $\sfM_1$ as $\bsfH:=\bsfH_1\ot_{R} \bsfH_2$-modules. This will be needed in the passage from $\bF_q$ to $\bC$ in \Cref{Sec. Fq to C}.

Let $\SPM(m,n)_\hs\subset \SPM(m,n)$ be the subset of signed partial matchings $\s$ such that for any simple reflection $s\in  \ov \sfS_1\sqcup \ov \sfS_2$ in $\ov W_1\times \ov W_2$, $(s,\cO_\s)$ is either of type (G) or type (U-). This is consistent with the notation $I_0$ introduced in \Cref{def. I0} with respect to the $\bcH$-action on $\cM_1$.

Let $0\le i\le m, 0\le j\le n$ and $\ell:=i+j\le \min\{m, n\}$. Let $I^+_\ell\subset \{1,2,\cdots, m\}$ be the subset
\begin{equation}\label{I+ ell}
    I^+_\ell=\{m-\ell+1,\cdots, m-1, m\}.
\end{equation}
Let $I_\ell=I^+_\ell\sqcup(-I^+_\ell)$. Let $J^{+}_{i,j}\subset \{\pm1, \cdot, \pm n\}$ be the subset
\begin{equation}\label{J+ ij}
    J^{+}_{i,j}=\{1,2,\cdots, i, -n, -(n-1), \cdots, -(n-j+1)\}.
\end{equation}
Let $J_{i,j}=J^{+}_{i,j}\sqcup(-J^+_{i,j})$. Let $\mu^+_{i,j}$ 
be unique order-preserving bijection $I^+_\ell\to J^{+}_{i,j}$ with respect to the order $<_\SPM$  defined in \Cref{eq order SPM}. Let $\mu_{i,j}: I_\ell\isom J_{i,j}$ be the unique bijective extension of $\mu^+_{i,j}$ such that $\mu_{i,j}(-a)=-\mu_{i,j}(a)$. We thus get an element
\begin{equation}\label{define sij}
    \s_{i,j}=(I_\ell, J_{i,j}, \mu_{i,j})\in \SPM(m,n).
\end{equation}


\begin{lemma}\label{Lem. SPM min} We have
\begin{equation*}
    \SPM(m,n)_\hs=\{\s_{i,j}|0\le i\le m, 0\le j\le n, i+j\le \min\{m,n\}\}.
\end{equation*}
\end{lemma}
\begin{proof}
This follows from \Cref{Lem. type G or U type I}. 
\end{proof}

\begin{cor}[of \Cref{Lem. heart gen}]\label{Lem. M1 gen}
    As a $\ov\sfH=\ov \sfH_1\ot_R \ov \sfH_2$-module, $\sfM_1$ is generated by $C'_{\s_{i,j}}$ for $\s_{i,j}\in \SPM(m,n)_\hs$. In particular, as a $\bZ[\ov W_1\times \ov W_2]$-module, $K_0(\cM^k_1)\cong \sfM_{1,v=1}$ is generated by $\chi(\IC_{\s_{i,j}})$ for $\s_{i,j}\in \SPM(m,n)_\hs$.
\end{cor}

\subsubsection{Changing Schr\"odinger model}\label{ss:moment cone L2}
We now switch the roles of $V_1$ and $V_2$. Recall that $L_2 \subset V_2$ is a Lagrangian subspace, and set $\bL_2 := \Hom(L_2, V_1)$. We consider the moment cone $\cN_{\bL_2}$ under the action of $G_1 \times \ov G_2$. Since $G_1$ is disconnected, we indicate here how to adapt the general framework of \Cref{Sec. geom gen} to accommodate this setting.

Let $G_1^\circ := \SO(V_1) \subset G_1$ denote the identity component, and let $W_1^\circ \subset W_1$ denote the Weyl group of $G_1^\circ$. Then the set of simple reflections in $W_1^\circ$ is given by
\[
\sfS_1^\circ := \{s_1, s_2, \dots, s_{m-1}, s_m\},
\]
where $s_m := t_m s_{m-1} t_m$ corresponds to the element in $G_1^\circ$ that permutes $e_{m-1}$ and $e_{-m}$ (and hence also $e_m$ and $e_{-(m-1)}$). 

Let $\cH_1^\circ := D^b(B_1 \backslash G_1^\circ / B_1)$ denote the Hecke category associated to $G_1^\circ$, and similarly define the mixed and Tate versions $\cH_1^{\circ, \mix}$ and $\cH_1^{\circ, \Tate}$. For the full (disconnected) group $G_1$, we define its Hecke category as $\cH_1 := D^b(B_1 \backslash G_1 / B_1)$, along with its variants $\cH_1^{\mix}$ and $\cH_1^{\Tate}$.

There is a natural decomposition:
\[
\cH_1 \cong \cH_1^\circ \oplus \cH_1^\bullet,
\]
where $\cH_1^\bullet$ consists of sheaves supported on the non-identity component of $G_1$.

Let $B_1^+ \subset G_1$ be the normalizer of $B_1$. Then $B_1^+ / B_1 \cong G_1 / G_1^\circ \cong \{\pm 1\}$. The nontrivial double coset $B_1\bs (B^+_1-B_1)/B_1$ corresponds to the element $t_m$, and we write $\IC_{t_m}$ for the constant sheaf supported on this coset. It satisfies a canonical isomorphism
\[
\IC_{t_m} \star \IC_{t_m} \cong \IC_1,
\]
where $\IC_1$ denotes the unit object. Left and right convolution with $\IC_{t_m}$ induces equivalences:
\[
\IC_{t_m} \star (-): \cH_1^\circ \to \cH_1^\bullet, \quad \text{and} \quad (-) \star \IC_{t_m}: \cH_1^\circ \to \cH_1^\bullet.
\]
Moreover, conjugation by $\IC_{t_m}$ defines a monoidal auto-equivalence of $\cH_1^\circ$ that sends $\IC_w$ to $\IC_{t_m w t_m}$ for all $w \in W_1^\circ$. In particular, we have
\[
\IC_{t_m} \star \IC_{s_{m-1}} \star \IC_{t_m} \cong \IC_{s_m}.
\]
Similar statements hold for the mixed and Tate versions of the Hecke category.

The discussions of \Cref{Sec. geom gen} applies to the $G_1^\circ\times \ov G_2$-action on $\cN_{\bL_2}$. We get categories
\begin{equation*}
\cM^\Tate_2=D^{\Tate}_{B_1\times \barB_2}(\cN_{\bL_2}) \hookrightarrow \cM^\mix_2=D^{\mix}_{B_1\times \barB_2}(\cN_{\bL_2})
\hookrightarrow \cM_2=D^{b}_{B_1\times \barB_2}(\cN_{\bL_2})
\to\cM^k_{2}=D^{b}_{B_{1,k}\times \barB_{2,k}}(\cN_{\bL_2,k}).
\end{equation*}
Because $G_1\times\ov G_2$ acts on $\cN_{\bL_1}$,  these are module categories over $\cH
^{\Tate}_1\ot \ov\cH_2^{\Tate}$, $\cH^{\mix}_1\ot \ov\cH_2^{\mix}$, $\cH_1\ot \ov\cH_2$ and $\cH^{k}_{1}\ot \ov\cH^k_{2}$ respectively. 

The $B_1\times \ov B_2$-orbits on $\cN_{\bL_2}$ are indexed by $\SPM(n,m)$. There is a natural $*$-action of $W_1\times W_2$ on $\SPM(n,m)$ defined according to \eqref{action Sm Sn on PM}. We denote element in $\SPM(n,m)$ by the triple $\rho=(J,I,\nu)$ with $J\subseteq \{\pm 1,\cdots, \pm n\}, I\subseteq \{\pm 1,\cdots, \pm m\}$ and $\nu: J\rightarrow I$ is a bijection with $\nu(-a)=-\nu(a)$ for every $a\in J$. For each $\rho=(J,I,\nu)\in \SPM(n,m)$, we define 
\begin{equation}\label{extended sigma 5}
    \rho^{\dagger}: \{\pm 1,\cdots, \pm n\} \longrightarrow \{0,\pm 1,\pm 2,\cdots, \pm m\}
\end{equation}
by 
\[
\rho^{\dagger}(j)=\begin{cases}
    \rho(j) \quad & \mbox{if $j\in J$}\\
    0 \quad & \mbox{if $j \notin J$}.
\end{cases}
\]
We also define 
\begin{equation}\label{extended sigma 6}
    \rho_{\dagger}: \{\pm 1,\cdots, \pm m\} \longrightarrow \{ \pm 1,\pm 2,\cdots, \pm n,\infty\}
\end{equation}
by 
\[
\rho_{\dagger}(i)=\begin{cases}
    \rho^{-1}(i) \quad & \mbox{if $i\in I$}\\
    \infty & \mbox{if $i \notin I$}.
\end{cases}
\]
We define total orders $<_{\SPM}$ on the sets \( \{0, \pm 1, \pm 2, \dots, \pm m\} \) and \( \{\pm 1, \pm 2, \dots, \pm n, \infty\} \) as in \Cref{eq order SPM}. 

The following lemma computes the type $(s,\cO_\rho)$ for each  simple reflection $s$ in $ W^\circ _1 \times  \ov W_2$ and the orbit $\cO_\rho$ with $\rho\in \SPM(n,m)$.
\begin{lemma}\label{Lem. type G or U type I model 2}
Let $s\in  \sfS^\circ_1\sqcup \ov \sfS_2$ be a simple reflection in $W^\circ _1 \times \ov W_2$, and let $\rho = (J, I, \nu) \in \SPM(n,m)$. Then the pair $(s, \cO_\rho)$ is of type (G) or (U). More precisely, 
\begin{enumerate}
    \item If  $s = s'_i$ for $1 \leq i \leq  n-1$, then $(s, \cO_\rho)$ is of
    \[
    \begin{cases}
        \mbox{type $G$}\quad &\mbox{if $\rho^\dagger(i)=\rho^\dagger(i+1)$}\\
         \mbox{type $U^-$}\quad &\mbox{if $\rho^\dagger(i)<_{\SPM}\rho^\dagger(i+1)$}\\
        \mbox{type $U^+$}\quad &\mbox{if $\rho^\dagger(i)>_{\SPM}\rho^\dagger(i+1)$}.
    \end{cases}
    \] 
    \item If  $s = s_i$ for $1 \leq i \leq  m-1$, then $(s, \cO_\rho)$ is of
    \[
    \begin{cases}
        \mbox{type $G$}\quad &\mbox{if $\rho_\dagger(i)=\rho_\dagger(i+1)$}\\
         \mbox{type $U^-$}\quad &\mbox{if $\rho_\dagger(i)<_{\SPM}\rho_\dagger(i+1)$}\\
        \mbox{type $U^+$}\quad &\mbox{if $\rho_\dagger(i)>_{\SPM}\rho_\dagger(i+1)$}.
    \end{cases}
    \]
    \item  If  $s = s_m $, then $(s, \cO_\rho)$ is of
    \[
    \begin{cases}
        \mbox{type $G$}\quad &\mbox{if $\rho_\dagger(m-1)=\rho_\dagger(-m)$}\\
         \mbox{type $U^-$}\quad &\mbox{if $\rho_\dagger(m-1)<\rho_\dagger(-m)$}\\
        \mbox{type $U^+$}\quad &\mbox{if $\rho_\dagger(m-1)>\rho_\dagger(-m)$}.
    \end{cases}
    \]
\end{enumerate}
In case $U$, the $s$-companion of $\cO_\rho$ is $\cO_{s * \rho}$. 
\end{lemma}
\begin{proof}
The proof is similar to that of \Cref{Lem. type G or U type I} and is omitted here.
\end{proof}
\trivial[h]{
The flag variety for $G_1^\circ$ is better thought of as a flag 
\[
0 \subset \langle e_1 \rangle \subset \cdots \subset \langle e_1, \dots, e_{m-1} \rangle
\]
together with $ \langle e_1, \dots, e_{m-1} \rangle \subseteq \langle e_1, \dots, e_{m-1},e_{m} \rangle$ and $ \langle e_1, \dots, e_{m-1} \rangle \subseteq \langle e_1, \dots,e_{m-1}, e_{-m} \rangle$. In this case, $G_1^\circ/P_{s_{m-1}}$ corresponds to the partial flag 
\[
0 \subset \langle e_1 \rangle \subset \cdots \subset \langle e_1, \dots, e_{m-2} \rangle\subset  \langle e_1, \dots, e_{m-1},e_{m}\rangle 
\]
and 
$G_1^\circ/P_{s_{m}}$ corresponds to the partial flag 
\[
0 \subset \langle e_1 \rangle \subset \cdots \subset \langle e_1, \dots, e_{m-2} \rangle\subset \langle e_1, \dots, e_{m-1},e_{-m}\rangle 
\]
}

The action of $\IC_{t_m}$ on $\cM_2$ is induced by the involution on the stack $B_1\times \ov B_2\bs \cN_{\bL_2}$ given by the non-neutral coset of $B_1^+/B_1$. 
\begin{lemma}\label{l:involution s0}
    For $\rho \in \SPM(n,m)$ we have \begin{equation*}
        \IC_{t_m}\star \IC_\rho \cong \IC_{t_m* \rho}.
    \end{equation*}
and 
\[
\sfT_{t_m} \one_{\rho}= \one_{t_m *\rho}.
\]
\end{lemma}
\begin{proof}
Let $\dot t_m\in B^+_1-B_1$ be the element that swaps $e_m$ with $e_{-m}$ and fixes the other basis elements of $V_1$. For the orbit representatives $x_\s$ for $\rho \in \SPM(n,m)$, one readily checks that $\dot t_m\cdot x_\rho=x_{t_m * \rho}$. The conclusion follows. 
\end{proof}
The analogs of \Cref{Cor. M1 IC} and \Cref{Cor. K0 M1} remain valid in the present setting. In particular, we obtain an $\sfH_1 \otimes_R \bsfH_2$-action on $\sfM_2$, whose specialization at $v = \sqrt{q}$ recovers the $H_1 \otimes \ov H_2$-action on $M_2$ defined in \Cref{oscillator Bi-module Schrodinger model 2}, via a fixed isomorphism $\iota: \overline{\bQ}_\ell \cong \bC$. The following result is the analogue of \Cref{Knop action type I}:

\begin{cor}\label{Knop action type I change}
The specialization of the $\sfH_1 \otimes_R \bsfH_2$-action on $\sfM_2$ at $v=1$ induces an action of $W_1 \times \ov W_2$ on $\bZ[\SPM(n,m)]$. It is explicitly given by
\[
w \cdot \one_\rho = \one_{w * \rho} \quad \text{for all } w \in W_1 \times \ov W_2, \ \rho \in \SPM(n,m).
\]
\end{cor}
\begin{proof}
The statement follows from \Cref{Cor. W action sph orb}, together with \Cref{Lem. type G or U type I model 2} and \Cref{l:involution s0}.
\end{proof}

\subsection{Equivariant Fourier-Deligne transform}\label{Sec Fourier}
In this subsection, we use equivariant Fourier-Deligne transform 
to construct an equivalence of the categories 
\[
\CM_{\cM_1}: \cM_1=D^b_{\barB_1\times B_2}(\NLone) \longrightarrow D^b_{B_1\times \barB_2}(\NLtwo)=\cM_2
\]
as well as its mixed, Tate and non-mixed variants. This will enable us to construct the generic oscillator bimodule and the corresponding $W_1\times W_2$-graph in the next subsection. 


We first recall some basic properties of the equivariant Fourier-Deligne transform. A reference is \cite[\S 6.9]{MR4337423}. Let $G$ be an algebraic group over $F=\bF_q$ and $S$ be a $G$-scheme over $\bF_q$. Let $p: E\rightarrow S$ be a $G$-equivariant vector bundle of rank $r$ and $p': E'\rightarrow S$ be the dual $G$-equivariant vector bundle. Consider the following diagram of quotient stacks,
\begin{equation}\label{Diagram equivariant Fourier Deligne}
\begin{tikzcd}
& E/G\times_{S/G} E'/G \arrow[ld,"\pi"] \arrow[rd,"\pi'"]   \arrow[r, "m"]& \bA^1/G\\
E/G & & E'/G 
\end{tikzcd}
\end{equation}
where $m$ is the natural pairing map, $\pi,\pi'$ are the projection maps, and the action of $G$ on $\bA_1$ is trivial. Fix a  non-trivial additive character $\psi$ of $\bF_q$, giving rise to an Artin-Schreier local system $\AS_\psi$ on $\bA^1$. the equivariant Fourier-Deligne transform functor 
\[
\CM_{E}: D_{G}^b(E)\longrightarrow D_{G}^b(E')
\]
is defined by 
\[
\CM_{E} (\sF)= \pi'_{!}\left(\pi^* \sF\otimes m^*(\AS_\psi)\right)\langle r \rangle,\quad \sF\in D_{G}^b(E).
\]

The main feature of $\FT_E$ is that it is a perverse $t$-exact equivalence that also preserves pure complexes of weight $i$ for any $i$. Its inverse is given by $\FT_{E'}$ defined using $\psi^{-1}$ instead of $\psi$.

The above discussion also works for the mixed and non-mixed variants. We use 
\[
\CM^{\mix}_E: D_{G}^{\mix}(E)\longrightarrow D_{G}^{\mix}(E')
\]
and 
\[
\CM^{k}_{E}: D_{G_k}^{b}(E_k)\longrightarrow D_{G_k}^{b}(E'_k)
\]
to denote the mixed and non-mixed versions of equivariant Fourier-Deligne transform.

By \Cref{L1L2 decomp} and the trace pairing between  $\ov \bL_1$  and $\ov \bL_2$, we view $\bL_1$ and $\bL_2$ as dual vector bundles over $\bL^b$. Taking $G=\ov B_1\times \ov B_2, E=\bL_1, E'=\bL_2$ and $S=\bL^b$ in \Cref{Diagram equivariant Fourier Deligne}, we get the equivariant Fourier Deligne transform 
\begin{equation*}
    \FT_{\bL_1}: D^b_{\ov B_1\times \ov B_2}(\bL_1)\isom D^b_{\ov B_1\times \ov B_2}(\bL_2).
\end{equation*} 


\begin{remark}
    Since the action of $\ov B_1\times \ov B_2$ on $\bL_i$ contains the scaling action, the functor $\FT_{\bL_1}$ above is independent of the choice of the additive character $\psi$ on $\bF_q$. 
\end{remark}

Since $U_2$ is unipotent, it follows from  \Cref{exact B_1} and \cite[Theorem 6.6.15]{MR4337423} that the forgetful functor $D^b_{\ov B_1\times B_2}(\bL_1)\to D^b_{\ov B_1\times \ov B_2}(\bL_1)$ 
is fully faithful. Therefore we may identify $\cM_1=D^b_{\ov B_1\times B_2}(\cN_{\bL_1})$ with the full subcategory of $D^b_{\ov B_1\times \ov B_2}(\bL_1)$ consisting of objects that are $U_2$-equivariant (which is a property rather than structure) and supported on the closed subscheme $\cN_{\bL_1}$. Similarly, we may identify $\cM_2=D^b_{B_1\times \ov B_2}(\cN_{\bL_2})$ with the full subcategory of $D^b_{\ov B_1\times \ov B_2}(\bL_2)$ consisting of objects that are $U_1$-equivariant and supported on  $\cN_{\bL_2}$. 

\begin{prop}\label{Fourier-transform 2}
Let $\sF\in D^b_{\ov B_1\times \ov B_2}(\bL_1)$. Then 
\begin{enumerate}
    \item $\sF$ is $U_2$-equivariant if and only if $\FT_{\bL_1}(\sF)$ is supported on $\cN_{\bL_2}$;
    \item  $\sF$ is supported on $\cN_{\bL_1}$ if and only if $\FT_{\bL_1}(\sF) $ is $U_1$-equivariant.
\end{enumerate}
\end{prop}

We begin the proof by the following. 
\begin{lemma}\label{lem:FourierSupport}
Let $S$ be a stack and $f: W\to V$ be a morphism between two vector bundles over $S$.
Now $W$ acts on $V$ via $w\cdot_f v = f(w) + v$. 
Let $D^b_{W,f}(V)$ be the $W$-equivariant derived category of $V$ with respect to the $\cdot_f$-action.

Let $\ckff : V^\vee \to W^\vee$ is the dual map of $f$, and 
$\cN_f := \ckff^{-1}(0_S) \subset V^\vee$ where $0_S$ denotes the zero section of the vector bundle $W^\vee\to S$.

Then the equivariant Fourier-Deligne transform $\FT_V: D^b(V)\isom D^b(V^\vee)$ restricts to an equivalence of full subcategories 
\[
\FT_V: D^b_{W,f}(V) \isom D^b(\cN_f). 
\]
\end{lemma}
\begin{proof}
Consider the map $\act_f : W\times_S V \to V $ by $(w,v)\mapsto f(w) + v$ and $p : W\times_S V \to V$ by $(w,v)\mapsto v$. Since $W$ is a unipotent group scheme over $S$, 
$D^b_{W,f}(V)$ is the full subcategory of $D^b(V)$ consisting of $\sF \in D_W(V)$ such that $\act_f^* \sF \cong p^* \sF$. 

Now $\act_f$ and $p$ are both morphisms of vector bundles over $S$. Denote their dual maps by $\ckact_f$ and $\ckpp$ respectively. We have for $\ckvv\in V^\vee$,
\[
\ckpp(\ckvv) = (0,\ckvv) \AND \ckact(\ckvv) = (\ckff(\ckvv), \ckvv). 
\]

Let $\sF\in D^b(V)$ and $\sG=\FT_V(\sF)\in D^b(V^\vee)$. If $\sF\in D^b_{W,f}(V)$, then $\act^* \sF \cong p^* \sF$, hence by the well-known properties of Fourier-Deligne transform we have isomorphisms
\begin{equation*}
    \ckact_! \sG\cong \FT_{W\times_S V}(\act^* \sF)\cong \FT_{W\times_S V}(p^* \sF)\cong \ckpp_! \sG.
\end{equation*}
In particular, $\Supp( \ckact_!\sG) = \Supp (\ckpp_! \sG)$. This forces $\Supp (\sG) \subset \cN_f$ since 
$\ckpp_! \sG$ is obviously supported on $0_S \times_S \ckV$. 

Conversely, suppose $\sG \in D^b(\cN_f)$. Note that $\ckpp$ and $ \ckact $ restrict to the same map $\cN_f\to W^\vee\times_S V^\vee$. This implies $\ckact_!\sG\cong \ckpp_!\sG$. Applying similar properties of the inverse Fourier transform we obtain
\begin{equation*}
    \act^* \sF \cong   \FT^{-1}_{W\times_S V}(\ckact_!\sG) \cong \FT^{-1}_{W\times_S V}(\ckpp_!\sG) \cong p^* \sF.
\end{equation*}
This finishes the proof. 
\end{proof}

\begin{proof}[Proof of {\Cref{Fourier-transform 2}}]
We only present the proof of (1) since the proof of (2) is similar. 

Let $S:=[\ov\bB\bs \bL^{-}]$,  $V := [\ov\bB\bs \bL_1]$. Note that $U_2$ is naturally a linear subspace of $\Hom(L_2^\vee, L_2)$. Let  $W := [\ov\bB\bs(U_2\times \bL^-)]$ as a vector bundle over $S$ with fiber $U_2$ over $0$. 

Consider the map  $f: W \to \bL_1$  induced by $\Hom(L_2^\vee,L_2)\times \bL^-= \Hom(L_2^\vee,L_2)\times \Hom(L_1, L_2^\vee)\ni (w,t^-) \mapsto  (w\circ t^-,t^-) \in \bL_1^+ \times \bL_1^- = \bL_1$. 
It is easy to see that 
\[
\cN_{\bL_2} = \set{ (t^+,t^-) | t^+(t^-)^* = 0  } = 
 \set{ (t^+,t^-) | \forall w \in U_2, t^+( wt^-)= 0  } = \ckff^{-1}(0). 
\]

On the other hand, $U_2$-equivariance for an object in $D^b_{\ov\bB}(\bL_1)$ is  the same as $W$-equivariance under the translation action of $W$ on $\bL_1$ via $f$. We finish the proof by applying
\Cref{lem:FourierSupport}.
\end{proof}

\begin{cor}\label{Cor. FT equiv bcH}
Fourier transform $\FT_{\bL_1}$ restricts to an equivalence of categories
\begin{eqnarray*}
    \FT_{\cM_1}: \cM_1 \xrightarrow{\sim} \cM_2.
\end{eqnarray*}
The same is true for the mixed and $k$-versions.
Moreover, these equivalences carry canonical equivariant structures under the actions of the Hecke category $\bcH=\bcH_1\ot \bcH_2$ and its mixed and $k$-versions.


    
\end{cor}
\begin{proof}
    By \Cref{Fourier-transform 2}, for $\sF\in \cM_1$, $\sF\in \cM_1$ if and only if   $\FT_{\bL_1}(\sF)\in \cM_2$. Therefore $\FT_{\bL_1}$ restrict to an equivalence between $\cM_1$ and $\cM_2$. The equivariance under $\bcH$ follows from the functoriality of Fourier transform with respect to arbitrary base change. The argument for the mixed and $k$-versions is similar.
\end{proof}

\trivial[h]{
{\color{red} The following proof dose not work. } 
We make some preparations for the proof of \Cref{Fourier-transform 2}.  Taking $G=\barB , E= U_2 \times \bL_1$, $E'=U_2 \times  \bL_2$ and $S=U_2\times  \Hom(L_1,L_2^\vee)$ in \Cref{Fourier transform}, we get the (partial) Fourier transform 
\[
\Four_{\bL_1\times U_2}: D_{m,\overline{B}_1\times \overline{B}_2}^b(\bL_1\times U_2)\longrightarrow D_{m,\overline{B}_1\times \overline{B}_2}^b(\bL_2\times U_2). 
\]
Let 
\begin{equation}\label{Projection map}
   p_i\colon \bL_i\times U_2 /\Gm\times\barB\longrightarrow \bL_i/\Gm\times \barB, \,\,\text{for } i=1,2
\end{equation}
be the natural projection maps. 
\begin{lemma}\label{Fourier pull back}
   For any $\sF\in \DGmBB(\bL_1)$, we have a natural isomorphism 
    \[
\Four_{\bL_1\times U_2} (p_1^*(\sF))\cong p_2^* (\Four_{\bL_1}(\sF)).
    \]
\end{lemma}
\begin{proof}
  Let $R\to S$ be a morphism of $G$-variety. 
  Consider the map $p : E\times_S R \to R$. 
\[  
 \begin{tikzcd}[ampersand replacement=\&]
 E\times_S R /\Gm \ar[d,"p"] \& \ar[l,"\pr'_1"'] (E\times_S E')\times_S R /\Gm \ar[r,"\pr'_2"] \ar[d,"p"] \& E'\times_S R /\Gm^2 \ar[d,"p'"] \\ 
 E/\Gm \& \ar[l,"\pr_1"']  (E\times_S E')/\Gm^2\ar[r,"\pr_2"] \ar[d,"m"]\& E'/\Gm  \\
 \&  \bA^1/\Gm \& \\
 \end{tikzcd}
\]
  \[
    \begin{split}
      & \Four_{E\times_S R} (p^*(\sF))\\
      & = {\pr'_2}_!(\pr'^*_1 p^* (\sF)\otimes (pm)^* \Psi_R) \\
      & = {\pr'_2}_!(p^* \pr_2^*(\sF)\otimes p^* m^* \Psi_S) \\
      & \text{(Proper base change)} \\
      & = p^* {\pr_2}_!(\pr_2^*(\sF)\otimes m^*\Psi_S) \\
      & = p^* \Four_{E}(\sF). 
    \end{split}
  \]

  Here consider 
  \[
 \begin{tikzcd}[ampersand replacement=\&]
  \Gm / \Gm \times_S R \ar[r,"u'"] \ar[d,"p'"]\& \bA^1\times_S R /\Gm  \ar[d,"p"]
  \& (E\times_S E')\times_S R /\Gm \ar[d,"p"] \ar[l,"m'"]\\
  \Gm / \Gm  \ar[r,"u"] \& \bA^1/\Gm  
  \& E\times_S E'  \Gm  \ar[l,"m"]
  \\ 
 \end{tikzcd} 
  \]

There is a natural morphism 
\[
   p^* u_* \kk  \to u'_* p'^* \kk. 
\]
  
  Now 
  \[
   \Psi_R = u'_* \kk[1](1) =  u'_* p^*  \kk [1](1)
   =  p^* \Psi_S.   
  \]
  and 
  \[
  m'^*\Psi_R = m'^* p^* \Psi_S = p^* m^* \Psi_S
  \]
 Apply the above to $R = U_2\times \Hom(L_1,L_2^\vee)$ and $S = \Hom(L_1,L_2^\vee)$ yields the lemma. 
\end{proof}

\textbf{The universal geometric action}
Let $\sK := k_! (\kk_{U_2\times \bL_1^-})  \in  \DGmBB(U_2 \times \bL_1)$
where $i$ is the map 
\[
\begin{tikzcd}[ampersand replacement=\&,row sep=0em]
  k:  \quot{ U_2\times \bL_1^- }{\barB} \ar[r]\&\quot{ U_2\times \bL_1^-\times \bL_1^+}{\GmBB}\\
    (1+x,T_1^-) \ar[r,maps to] \& (1+x,T_1^-,- x T_1^-).
\end{tikzcd}
\]
Let  $\sF_i\in \DGmBB(U_2 \times \bL_i),i=1,2$, we define the convolution $\sF_1*\sF_2 \in D_{m,\overline{B}_1\times \overline{B}_2}^b(\bL_1\times U_2)$ using \Cref{Equa: convolution}. 

The action map is 
\[
\act : U_2 \times \bL_1\to \bL_1
\]

\begin{lemma}\label{action formula by convolution}
  Let $r = \dim \bL_1^+$.
    For any $\sF\in \DGmBB(\bL_1)$, there is a natural isomorphism 
    \[
 \act^*(\sF)\cong p_2^*(\sF) *_{\bL_1^+} \sK. 
\]
where $\act$ and $p_2$ are the action and projection maps from $U_2 \times \bL_1/\GmBB$ to $\bL_1/\GmBB$.
\end{lemma}
\begin{proof}
  Let $K$ be the image of $k$. 
Consider 
\[
 \begin{tikzcd}[ampersand replacement=\&]
 \quot{U_2\times \bL_1}{\GmBB}\ar[r,equal,"{m'=\id\times k}"]  \&  \ar[dl,"\add'"'] (\quot{U_2 \times \bL_1}{\GmBB})\times_{\quot{U_2\times \bL_1^-}{\barB}} \quot{K}{\GmBB} \ar[r,"\pr'_2"] \ar[d,"i'"]\&  \quot{K}{\GmBB}  \ar[d,"i"] \\
\quot{U_2\times \bL_1}{\GmBB} \ar[d,"\act"] \& \ar[l,"\add"'] (\quot{U_2 \times \bL_1}{\GmBB})\times_{\quot{U_2\times \bL_1^-}{\barB}} (\quot{U_2\times  \bL_1}{\GmBB}) \ar[r,"\pr_2"] \ar[dr,"\pr_1"]\ar[d,"p_2\circ \pr_1"] \& \quot{U_2\times\bL_1}{\GmBB}  \\
\quot{\bL_1}{\GmBB}\ar[r,equal]  \& \quot{\bL_1}{\GmBB} \&  \quot{U_2 \times \bL_1 } {\GmBB}\ar[l,"p_2"] \\
 \end{tikzcd} 
\]
Recall that $\act(u,T^+,T^-) = (T^++(u-1)T^-,T^-)$.
Here $m' : (u,T^+,T^-)\mapsto (u,T^+,T^-, (1-u)T^-)$ and $\act \add'(u,T^+,T^-,(1-u)T^-) =  \act(u,T^+ +(1-u)T^-,T^-) = (T^+,T^-)$

Then 
\[
\begin{split}
   & \add_!(\pr_1^*(p_2^*\sF) \otimes \pr_2^* \sK) \\
   & = \add_!(\pr_1^*(p_2^*\sF) \otimes \pr_2^* i_!\Ql)\\
   & \text{(Proper base change)} \\
   & = \add_!(\pr_1^*(p_2^*\sF) \otimes  i'_! \pr'^*_2 \kk)\\
   & \text{(Projection formula)} \\
   & = \add_!(i'_!(i'^* \pr_1^* p_2^*\sF \otimes \pr'^*_2\kk)) \\
   & = \add'_!(p_2\circ \pr_1\circ i')^* \sF ) \\
   & \text{(Proper base change)} \\
   & = act^* \sF \\
\end{split}
\]

\end{proof}

\textbf{The universal character action}
The restriction of the moment map $\mu_2: \VV\rightarrow  \mathfrak g_2^*$ to $\bL_2\subseteq \VV$ has image in $\Herm (L_2,L_2^\vee)$ (\zjl{Maybe change this to symmtric and wedge}). Let 
\begin{equation*}
    \begin{split}
    Q:  \quot{U_2 \times \bL_2}{\GmBB} & \longrightarrow \quot{\bA^1}{\GmBB}\\
        (1+x, T) & \mapsto  -\half \tr(\mu_2 (T) x). 
    \end{split}
\end{equation*}
where the action of $\barB_1\times\barB_2$ on $\bA_1$ is trivial and $\Gm$ action on $U_2$ is trivial.  We define 
\[
\sK' \coloneqq Q^*(\Psi)  \in \DGmBB (U_2\times\bL_2).
\]
\begin{prop}
\label{Geo to character universal}
    There is a natural isomorphism 
    \[
    \Four_{U_2 \times \bL_1} (\sK) \cong \sK'.
    \]
\end{prop}
\begin{proof}
Consider the following commutative diagram   
\[
\begin{tikzcd}[ampersand replacement=\&]
  \& \quot{\bA^1}{\GmBB} \ar[rr,equal,"\id"] \& \& \quot{\bA^1 }{\GmBB}\\
\quot{U_2 \times \bL_1^- \times \bL_2^+}{\GmBB} \ar[r,"k'"] \ar[d,"\pr'_1"] \ar[ru,"m\circ i'"]\& \quot{U_2 \times \bL_1^+ \times \bL_1^- \times \bL_2^+}{\Gm\times \GmBB} \ar[d,"\pr_1"] \ar[r,"\pr_2"] \ar[u,"m"] \& \quot{U_2\times \bL_1^-\times \bL_2^+}{\GmBB}  \ar[ru,"Q"]\\
\quot{U_2 \times \bL_1^-}{\barB} \ar[r,"k"] \& \quot{U_2 \times \bL_1^+ \times \bL_1^-}{\GmBB}\\
\end{tikzcd}
\]
The lower square and top parallelogram are Cartesian.
\trivial[]{
$m \circ i' (1+x,T_1^+,T_2^+) = (1+x,-xT_1^-,T_1^-,T_2^+) = \tr(-x T_1^- T_2^+) = Q(x, ({T_1^-}^*,T_2^+))$

$T_1^- \in \Hom(L_1,L_2^\vee) = \Hom(L_2^\vee,L_1^\vee)$

}
Here $i \colon (1+x,T_1^-)\mapsto (u,-xT_1^-,T_1^-)$. 
We have 
\[
\begin{split}
  & \Four_{U_2\times \bL_1}(\sK) \\ 
  & = {\pr_2}_!(\pr_1^*(k_!\kk)\otimes (m)^*(\Psi)) \\
  & \text{(Proper base change)} \\
  & = {\pr_2}_!(k'_!(\pr'^*_1\Ql)\otimes (m)^*(\Psi)) \\
  & \text{(Projection formula)} \\
  & = {\pr_2}_!(k'_!(\kk)\otimes (m)^*(\Psi)) \\
  & = {\pr_2}_!(k'_!((k'\circ m)^*(\Psi)) \\
  & \text{(Proper base change)} \\
  & = Q^*(\id_!(\Psi)) = Q^*(\Psi) = \sK' \\
\end{split}
\]
\end{proof}

\begin{cor}\label{Geo to character universal II}
    For any $\sF\in D_{m,\overline{B}_1\times \overline{B}_2 }^b(\bL_1)$, we have natural isomorphisms 
    \[
    \Four_{U_2\times \bL_1}( p_1^*(\sF) * \sK)  \cong \Four_{U_2\times \bL_1}( p_1^*(\sF))\otimes \sK' [-r] \cong p_2^*(\Four_{\bL_1}(\sF))\otimes \sK'[-r]. 
    \]
\end{cor}
\begin{proof}
    This follows from \Cref{Convolution vs tensor},\Cref{Fourier pull back} and \Cref{Geo to character universal}. 
\end{proof}
\begin{proof}[Proof for \Cref{Fourier-transform 2}]
    The proof of (1) and (2) are the same and we give the proof of (1) here.
    Let $r = \dim \Hom(L_1,L_2)$. 
    Let $\sF\in D_{m,\overline{B}_1\times \overline{B}_2 }^b(\bL_1)$ and $$\sF'= \Four_{\bL_1} (\sF).$$ Assume that $\sF\in D_{m,\overline{B}_1\times B_2 }^b(\bL_1)$.  
    Then by  \Cref{action formula by convolution}, we know that there is an isomorphism 
    \begin{equation}\label{Fourier-transform equation 1}
     p_2^*(\sF)\cong \act^*(\sF)\cong  p_2^*(\sF)* \sK.   
    \end{equation}
    Applying the functor $ \Four_{\bL_1\times U_2}$ to \Cref{Fourier-transform equation 1} and by \Cref{Fourier pull back} , \Cref{Geo to character universal II},  we deduce that  
    \begin{equation}\label{Fourier-transform equation 2}
          p_2^*( \sF') \otimes \sK' [-r]\cong  p_2^*( \sF').  
    \end{equation}

    Take $T \in \bL_2$. 
 Consider the following diagram 
 \[
 \begin{tikzcd}[ampersand replacement=\&]
 \quot{U_2}{\Gm\times \barB_1\times \barB_2 }\arrow[r,"i_T"] \& 
 \quot{U_2\times \bL_2}{\Gm\times \barB_1\times \barB_2}   \arrow[r,"p_2"]    \&  \quot{\bL_2}{\Gm \times \barB_1\times \barB_2}
 \end{tikzcd}
 \]
 where $i_T$ is the map $u\mapsto (u, T)$. 

 \begin{claim}
  Let $q : U_2 \rightarrow U_2/\GmBB$. 
  If $T\in \NLtwo$, then 
  $q^* i_T^*\sK'$ is the constant sheaf $\kk[-1]$, otherwise
  $q^* i_T^*\sK'$ is not isomorphic the constant. 
 \end{claim}
 \begin{proof}
 Suppose $T\in \NLtwo$. 
 Consider the following commutative diagram 
 \[
 \begin{tikzcd}[ampersand replacement=\&]
 U_2 /\BB \ar[r,"q"] \ar[d,"a'"] \& U_2/\GmBB \arrow[r,"i_T"] \ar[d, "a"]\& 
  U_2\times \bL_2/\GmBB   \arrow[d,"Q"] \\   
 0/\BB\ar[r,"q_0"] \&  0/\GmBB \ar[r,"i_{0}"] \&  \bA^1/\GmBB\\
 \end{tikzcd}
 \]
 So we conclude that $q^*i_T^* \sK' = q^* i_T^* Q^* \Psi = q^*a^* i_0^* \Psi = q^* a^* {q_0}_! \kk[1](1) = q^* q_!\kk[1](1) = \kk[-1]$.

 Suppose $T\notin \NLtwo$. 
 Consider the following commutative diagram 
 \[
 \begin{tikzcd}[ampersand replacement=\&]
 U_2 \ar[r,"q"] \ar[d,"a'"] \&
  U_2/\GmBB \arrow[r,"i_T"] \ar[d, "a"]\& 
  U_2\times \bL_2/\GmBB   \arrow[d,"Q"] \\   
 \bA^1 \ar[r,"q"] \&  
 \bA^1/\GmBB \ar[r,"i_{0}"] \&  \bA^1/\GmBB\\
 \end{tikzcd}
 \]
 Here $a$ is induced by the map $a' : (1+x)\mapsto \tr(x\mu(T))$ which is an epimorphism. 
 Now $q^* i_T^* \sK' = q^* i_T^* Q^* \Psi = a'^* q^* i_0^*\Psi$.
  Since $\Psi$ is not isomorphic to the constant sheaf, 
  $q^* i_T^* \sK'$ is not isomorphic to the constant sheaf. 
 \end{proof}

  Assume that $x\in \mathrm{Supp} (\sF')$, i.e., $\sF'_{x}\neq 0 \in D^b(\Spec(\overline F))$.

  Suppose $T\notin \cN_{\bL_2}$. 
    By \eqref{Fourier-transform equation 2}, we have 
    \[
     i_T^*p^*_2(\sF') \cong i^*_{T} (p_2^*(\sF') \otimes i^*_{T}\sK' 
    \] 
    If $i_T^* p^*_2 \sF'$ is non-zero, then it is the constant sheaf. 
    This implies $i^*_T\sK'$ is isomorphic to the constant sheaf $\kk$, and 
    leads to a contradiction. 

    This proves the ``only if'' part.

    \medskip

    Now assume that $\sF'$ supported on $\NLtwo$. 
    Consider the following diagram 
 \[
 \begin{tikzcd}[ampersand replacement=\&]
 \quot{U_2\times\NLtwo }{\GmBB} \ar[r,"p'_2"] \ar[d,"i'"] \&
  \quot{\NLtwo}{\GmBB}  \ar[d,"i"]\\   
 \quot{U_2\times \bL_2 }{\GmBB} \ar[r,"p_2"] \&  
 \quot{\bL_2}{\GmBB}  \\
 \end{tikzcd}
 \]
 By definition, there is $\sG\in \DGmBB(\NLtwo)$ such that 
 $\sF' = i_! \sG$. 

Note that $i'^*\sK' \cong \kk$. Now 
 \[
 \begin{split}
  p_2^* \sF' &= p_2^* (i_!\sG) \cong i'_! p'^*_2(\sG) \cong i'_!(p'^*_2(\sG)\otimes \kk) \\
  &\cong  i'_!(p'^*_2(\sG)\otimes i'^*\sK') \cong i'_!(p'^*_2(\sG)) \otimes \sK'= p_2^* \sF'\otimes \sK' 
 \end{split}
 \]
By \Cref{Geo to character universal II}, $\act^* \sF \cong \sF * \sK \cong p_2^*(\sF)$, which implies that $\sF$ is $U_2$-equivariant. This finished the proof of the ``if'' part.  

\end{proof}

Let $\barG:= \GL(L_1)\times \GL(L_2)$ with the fixed Borel $\barB := \barB_1\times \barB_2$. 
Let $\bcH$ be the Hecke category for $\barG$ and then $\barIC_s$ for $s\in \barW$ generats $\bcH$.  
Let $\ICone_s$ and $\ICtwo_s$  be the corresponding objects in $\bcH_1\otimes \cH_2$ and $\cH_1\otimes \bcH_2$ respectively.  

The following lemma is easy.  
\begin{lemma}
The map $\barIC_s \mapsto \ICone_s$ induces an embedding $\iota : \bcH \to \bcH_1\otimes \cH_2$.  
Moreover, 
\[
\sT \star \sF  =  \iota(\sT) \star \sF
\] 
for all $\sT\in \bcH$ and $\sF \in D_{m,\barB_1\times B_2}^b(\bL_1)$.  Similar statement holds for $\bL_2$. 
\end{lemma}
\begin{proof}
  The first claim is standard. 
  It suffices to consider the case when $\sT = \barIC_s$.
  Let $P$ and $\wtP$ be the parablic subgroup in $\sfL$ and $\GL(L_1)\times G_2$ corresponding to $s$ respectively.  

  Consider the diagram.   
\[
 \begin{tikzcd}[ampersand replacement=\&,row sep=2em]
  P \times \bL_1 \ar[r,"\wtiota"] \ar[d,"q"]\& \wtP \times \bL_1 \ar[r,"\pr_{\bL_1}"]\ar[d,"\wtqq"] \& \bL_1\\
  P \times^{\barB} \bL_1 \ar[r,"\iota"] \ar[d,"a"]\& \wtP \times^{\barB_1\times B_2} \bL_1 \ar[d,"\wtaa"] \& \\
 \bL_1 \ar[r,equal,"\id"] \& \bL_1 
 \end{tikzcd}
\]

Let $\tsG \in D_{m}^{b}(P \times^{\barB} \bL_1)$ such that $\wtqq^* \tsG = \pr_{\bL_1} \sF$. 
Now 
\[
\begin{split}
\IC_s \star \sF &=  \wtaa_! \tsG = \id^*\wtaa_!\tsG \\
& = a_! \iota^* \tsG.
\end{split}
\]
Here $q^*\iota^*\tsG = \wtiota^* \wtqq^* \tsG = \wtiota^* \pr_{\bL_1}^* \sF$. 
Hence $a_!\tsG = \barIC_s \star \sF$ and the lemma is proved. 
\end{proof}

\begin{thm}
  For any $\sS\in \DbarBB(\bL_1)$ and $\sT \in \bcH$, then 
  \[
  \Four_{\bL_1}(\sT\star \sS) = \sT \star \Four_{\bL_1}(\sS)
  \] 
\end{thm}
\begin{proof}
Take $s\in S$ let $P$ be the parabolic subgroup in $\barG$ corresponds to $s$.
Let $\pr_{\bL_i} \colon P\times \bL_i \rightarrow \bL_i$.

First, we compute $\Four_{\bL_1}(\barIC_s \star \sF)$. Consider the following diagram. 
\[
 \begin{tikzcd}[ampersand replacement=\&,row sep=2em]
 P\times \bL_1^+\times \bL_1^- \times \bL_2^+ \ar[d,"q''"] \& \ar[l,"b"'] P\times \bL_1^+\times \bL_1^- \times \bL_2^+ \ar[d,"q'"] \ar[r,"\pr'_1"]\& P\times \bL_1 \ar[d,"q"] \ar[r,"\pr_{\bL_1}"] \& \bL_1\\ 
 P/B\times \bL_1^+\times \bL_1^- \times \bL_2^+  \ar[dd,"p"] \& \ar[l,"b^B"'] P\times^B (\bL_1^+\times \bL_1^- )\times \bL_2^+ \ar[d,"a'"] \ar[r,"\pr^B_1"] \&  P\times^B \bL_1 \ar[d,"a"]\\ 
\& \bL_1^+\times \bL_1^- \times \bL_2^+ \ar[d,"\pr_2"'] \ar[r,"\pr_1"] \ar[dr,"m"] \& \bL_1\\
 \bL_1^-\times \bL_2^+ \& \ar[l,equal] \bL_1^-\times \bL_2^+ \& \bA^1\\
 \end{tikzcd}
\]

Let $\sG\in D^b_{m}(P\times^B \bL_1)$ be the sheaf such that $q^*\sG = \pr_{\bL_1}^* \sF$.  
\[
\begin{split}
\pr^*_1 (\barIC_s\star \sF) = \pr_1^* a_! \sG
= a'_! (\pr^B_1)^* \sG
\end{split}
\]
Let $\sG' := (\pr^B_1)^*(\sG)$. 
Then 
\[
q'^*(\sG') = q'^* (\pr^B_1)^* (\sG) = {\pr'_1}^* q^* \sG = {\pr'_1}^* \pr_{\bL_1}^* (\sF)   
\]

The map $b$ is defined by 
\[
b(p,x,y,z) = (p,p\cdot x, p \cdot x, z )
\]
is $B$-equivariant isomorphism and induces $b^B$ 
\[
b^B([p,x,y],z) = (pB, px, py, z).
\] Here we let $[p,x,y]$ represent the its image in 
$P\times^B (\bL_1^+\times \bL_1^-)$.
Let $\sG'' := ({b^B}^{-1})^* \sG'$. Then 
\[
\begin{split}
  \Four_{\bL_1}(\barIC_s\star \sF) 
 & = {\pr_2}_!(a'_! (\pr^B_1)^* \sG  \otimes m^*\Psi) \\
 & = (\pr_2\circ a')_! ({\pr^B_1}^* \sG  \otimes (m a')^*\Psi)) \\
 & = (\pr_2\circ a')_! (\sG'  \otimes (m a')^*\Psi)) \\
 & = p_!({b^B}^{-1})^*(\sG' \otimes (m a')^* \Psi) \\
 & = p_!(\sG'' \otimes (m a' {b^B}^{-1})^* \Psi). \\
\end{split}
\]

Let $u: P\times \bL_1^+\times \bL_1^-\times \bL_2^+$ be the map $(p,x,y,z)\mapsto (p^{-1}x,p^{-1}y)$. Then 
$u  = \pr_{\bL_1}\circ \pr'_1\circ b^{-1}$ and
\[
\begin{split}
  q''^* (\sG'') & =
q''^* ({b^B}^{-1})^*\sG' = (b^{-1})^* q'^* \sG' = (b^{-1})^*  
= (b^{-1})^* {\pr'_1}^* \pr_{\bL_1}^* \sF \\ 
& = (\pr_{\bL_1} \circ \pr'_1\circ b^{-1})^* \sF\\ 
& = u^* \sF
\end{split}
\]
Let $\wtmm  : P \times \bL_1^+\times \bL_1^-\times \bL_+^+ \to \bA^1$  be the map defined by $(p,x,y,z)\mapsto \inn{x}{z}$.
Then $\wtmm := q'' m a' (b^B)^{-1}$. 

Note that $q''^*$ is fully faithful.  
In summary, 
\begin{equation} \label{eq:FourICs}
  \Four_{\bL_1}(\barIC_s\star \sF) = p_!(\tsG'' )
\end{equation}
where $\tsG''$ is the sheaf in $D_{m}^{b}(P/B\times \bL_1^+\times \bL_1^-\times \bL_2^+)$ such that 
\begin{equation}\label{eq:tsG''}
q''^*(\tsG'') = u^*\sF \otimes \wtmm \Psi. 
\end{equation}

\medskip

We now compute $\barIC_s\star \Four_{\bL_1}(\sF)$.  
Consider the diagram
\[
 \begin{tikzcd}[ampersand replacement=\&,row sep=2em,column sep=3em]
 \& \bA^1 \& \\
  P\times \bL_1^-\times \bL_2^+ \ar[d,"q"] \ar[dr,"c"]
 \& \ar[l,"\pr_2"' ]  P \times \bL_1^+\times \bL_1^- \times \bL_2^+ \ar[r,"\pr_{\bL_1}\circ\pr_1"]\ar[u,"\wtmm"] 
 \ar[dr,"b"]
 \& \bL_1^+\times \bL_1^-  \\ 
  P \times^B (\bL_1^-\times \bL_2^+) \ar[d,"a"] \ar[dr,"c^B"] \& P\times \bL_1^- \times \bL_2^+  \ar[d,"q'"] 
  \&  P \times \bL_1^+\times \bL_1^- \times \bL_2^+  \ar[l,"\pr_2"'] \ar[d,"q''"]\\ 
\bL_1^- \times \bL_2^+ \& \ar[d,"p'"] P/B \times \bL_1^- \times \bL_2^+ \& P/B \times\bL_1^+ \times \bL_1^- \times \bL_2^+  \ar[l,"\pr'_2"'] \ar[dl,"p"]\\
\& \bL_1^-\times \b_2^+ \ar[ul,equal]\& \\
 \end{tikzcd}
\]


By base change, $\pr^*_{\bL_2} \Four_{\bL_1}(\sF) = \Four_{P\times \bL_1}(\pr^*_{\bL_1}(\sF))$ where $\Four_{P\times \bL_1}$ denote the Fourier-Deligne transform with respect to the basis $P\times \bL_1^-$.   

Let $\tsG \in D_{m}^b(P\times^B(\bL_1^-\times \bL_2^+))$ be the sheaf such that 
\[
\begin{split}
q^*\tsG & = \pr_{\bL_2}^* \Four_{\bL_1}(\sF)
= \Four_{P\times \bL_1}(\pr^*_{\bL_1}\sF) \\
&= {\pr_2}_!((\pr_{\bL_1}\circ \pr_1)^*\sF\otimes \wtmm^*\Psi) \\
\end{split}
\] 
Let $c'$ and $c^B$ are isomorphisms induced by $b$ which are given by 
\[
c(p,y,z) = (p,p\cdot y, p\cdot z) \quad \text{and} \quad 
c^B([p,y,z]) = (pB,p\cdot y, p\cdot z). 
\]
So 
\[
\begin{split}
q'^*({c^B}^{-1})^* \tsG  &  = (c^{-1})^* q^* \tsG 
=  (c^{-1})^* {\pr_2}_!((\pr_{\bL_1}\circ \pr_1)^*\sF\otimes \wtmm^*\Psi) \\
& = {\pr_2}_! (b^{-1})^*((\pr_{\bL_1}\circ \pr_1)^*\sF\otimes \wtmm^*\Psi)\\
& = {\pr_2}_! (((\pr_{\bL_1}\circ \pr_1 \circ b^{-1})^*\sF\otimes (\wtmm\circ b^{-1})^*\Psi)\\
\end{split}
\]
It is straightforward to verify that $\wtmm\circ b^{-1} = \wtmm$ and  
$ \pr_{\bL_1}\circ \pr_1 \circ b^{-1} = u$. 
By \eqref{Eq:tsG''}, we have 
\[
q'^*({c^B}^{-1})^* \tsG = {\pr_2}_!(q''^* \tsG'') = q'^* {\pr'_2}_! \tsG'',
\]
i.e. $({c^B}^{-1})^* \tsG \cong {\pr'_2}_!\tsG''$.

Now 
\[
\barIC_s\star \Four_{\bL_1}(\sF) = a_! \tsG
 = p'_! ({c^B}^{-1})^* \tsG  = p'_! {\pr'_2}_! \tsG''  =  p_! \tsG''  
\]

Comparing with \eqref{eq:FourICs}, we prove the lemma. 
\end{proof}

\mjj{Hint to pass to $\bC$}
Let $\sT_1\in \cH_1$ and $\sS\in \cM$, we define 
\begin{equation}\label{Transfer monoid action}
    \sT_1 \star\sS\coloneqq \Four_{\bL_2}^{-1} (\sT_1 \star \Four_{\bL_1}(\sS)) \in \cM.
\end{equation}
By the above discussion, this extends the monoid action of $\bcH_1$  on  $\cM$ to  $\cH_1$.

$X$ is a variety defined over $\bZ$. 

$\sF$ constructable on $X$. 

$X_\bC$ and $X_{p}$. 

Let $X$ be a $U_2$-orbit in $\bL_2$. 

Consider $R$ of finite type over $\bZ$. 

$X := (U_2)_R\to \bL_1$. 

$X_p := X\times_{\bZ} \Spec(\overline{\bF_p}) \to X$. 

$X$ is finite type over $\bZ$. 
If $\sF\in Const(\bL_1)$. 
 
$\sF_{X_p}$ is constant.  

Claim: Since $\sF$ is constructable, it implies $\sF_{X_\bC}$ is constant. 
}

%
%

\begin{lemma}\label{FT tate type I}
    The Fourier transform $\FT_{\cM_1}$ restricts to an equivalence
    \begin{equation}\label{FT Mmix}
        \FT^\Tate_{\cM_1}: \cM^\Tate_1\isom \cM_2^\Tate.
    \end{equation}
\end{lemma}
\begin{proof}    
    It suffices to show that $\FT_{\cM_1}$ sends $\cM^\Tate_1$ to $\cM^\Tate_2$, because the same argument for the inverse Fourier transform implies that $\FT^\Tate_{\cM_1}$ is an equivalence.

    Now $\cM^\Tate_1$ is generated as a triangulated category by $\IC_\s(i/2)$ for $\s\in \SPM(m,n)$ and $i\in \ZZ$, it suffices to show that $\FT_{\cM_1}(\IC_\s)\in \cM^\Tate_2$. Since $\FT_{\cM_1}$ sends simple perverse sheaves to simple perverse sheaves, $\FT_{\cM_1}(\IC_\s)$ must be of the form $\IC_{\t}\ot V$ for some $B_1\times \ov B_2$-orbit $\cO_\t$ in $\cN_{\bL_2}$, and a one-dimensional $\Frob$-module $V$. Since $\IC_\t$ is Tate by the analogue of \Cref{Cor. M1 IC} for $\cM_2$, it suffices to show that $V$ is Tate as an object of $D^b(\Spec \bF_q)$, i.e., the scalar by which $\Frob$ acts on $V$ belongs to $\{q^{i/2}; i\in \ZZ\}$. The stalk of $\IC_\t$ at $0\in \cN_{\bL_2}$ is nonzero and Tate, therefore $V$ is Tate if and only if the stalk $\FT_{\cM_1}(\IC_\s)_0$ is. 

    We reduce to showing the following slightly general statement: if $\sF\in \cM^\Tate_1$, then   $\FT_{\cM_1}(\sF)_0$ is Tate. Recall that $\FT_{\cM_1}(\sF)_0\cong R\Gamma_c(\cN_{\bL_1}, \sF)$, which is a successive extension of $R\Gamma_c(\cO_\s, j^*_\s\sF)$. Choose a point $x_\s\in \cO_\s(\bF_q)$, then $j^*_\s\sF\cong \un{\Qlbar}\ot \sF_{x_\s}$, hence 
    \begin{equation}\label{O sig sF}
        R\Gamma_c(\cO_\s, j^*_\s\sF)\cong R\Gamma_c(\cO_\s,\Qlbar)\ot \sF_{x_\s}.
    \end{equation} 
    Since $\cO_\s\cong (\ov B_1\times B_2)/\Stab(x_\s)$, we see that $\cO_\s$ is an affine space fibration over a torus, hence $R\Gamma_c(\cO_\s,\Qlbar)$ is Tate. Since $\sF_{x_\s}$ is assumed to be Tate, the right side of \eqref{O sig sF} is Tate. This implies that the left side of \eqref{O sig sF} is Tate, hence $\FT_{\cM_1}(\sF)_0\cong R\Gamma_c(\cN_{\bL_1},\sF)$ is Tate. This finishes the proof.
\end{proof}
Below we will simply use $\FT$ 
to denote various Fourier-Deligne transform, if the meaning is clear from the context.
\begin{cor} \label{FT type I generic}
Passing to Grothendieck groups, Fourier transform induces an isomorphism of $\ov \sfH= \ov\sfH_1\ot_R\ov\sfH_2$-modules
    \begin{equation}\label{Fourier generic M12}    
    \phi: \sfM_1=R[\SPM(m,n)]\xrightarrow{\ch^{-1}} K_0(\cM_1^\Tate)\xrightarrow{K_0(\FT)}K_0(\cM_2^\Tate)\xrightarrow{\ch}\sfM_2= R[\SPM(n,m)]
    \end{equation}
    and a $\ov W_1\times \ov W_2$-equivariant isomorphism 
    \begin{equation}\label{Fourier W12}    
    \phi^k: \sfM_{1,v=1}=\bZ[\SPM(m,n)]\xrightarrow{\chi^{-1}}K_0(\cM^k_1)\xrightarrow{K_0(\FT)}K_0(\cM_2^k)\xrightarrow{\chi}\sfM_{2,v=1}=\bZ[\SPM(n,m)].
    \end{equation}
Via the sheaf-to-function map, Fourier transform induces an isomorphism of $\ov H= \ov H_1\ot_{\bC} \ov H_2$-modules
    \begin{equation}\label{Fourier function M12}    
    \phi^{F}: M_1 \rightarrow M_2
    \end{equation}
    and this is the same as the Fourier transform map given in \Cref{partial FT function}  (using the isomorphism $\io:\Qlbar\cong \bC$). 
\end{cor}

For any $\s\in \SPM(m,n)$, the Fourier transform $\FT(\IC_\s)\in \cM^k_2$ is a simple perverse sheaf, hence is isomorphic to $\IC_\t$ for a unique $\t\in \SPM(n,m)$. The assignment $\s\mapsto \t$ defines a bijection
\begin{equation}\label{define Phi}
   \Phi= \Phi_{m,n}: \SPM(m,n)\isom \SPM(n,m).
\end{equation}
A priori, the bijection $\Phi$ may depend on the choice of the field $k = \ov{\bF}_q$. The following proposition shows that it does not:
\begin{prop}\label{c:Fourier indep k}
The bijection $\Phi$ is independent of $k$.
\end{prop}
As the argument is somewhat technical, we defer the proof of \Cref{c:Fourier indep k} to \Cref{subsection Fourier bijection}. In \Cref{subsection Fourier bijection}, we will also determine $\Phi$ explicitly. 
\begin{lemma}\label{l:phi IC}
    For any $\s\in \SPM(m,n)$, we have
    \begin{equation*}
        \phi(C'_\s)=C'_{\Phi(\s)}\in \sfM_2.    
    \end{equation*}
\end{lemma}
\begin{proof}
   Let $\t=\Phi(\s)$. We know $\FT(\IC_\s)$ is isomorphic to $\IC_{\t}$ over $k$, therefore $\FT(\IC_\s)\cong \IC_{\t
    }\ot L$ for some one-dimensional $\Frob$-module. Both $\FT(\IC_\s)$ and $\IC_{\t}$ are pure of weight zero, the same is true for $L$. Taking $\ch$ we conclude that $\phi(C'_\s)=\ch(\FT(\IC_\s))=\ch(\IC_\t)=C'_\t$.
\end{proof}

\subsection{Generic oscillator bimodule and  oscillator $W_1\times W_2$-graph}\label{Sec gen oscillator bimod}
\begin{cons}[Generic oscillator bimodule]
    There is a canonical extension of $\ov\sfH_1\ot_R \sfH_2$ on $\sfM_1$ to an action of $\sfH_1\ot_R \sfH_2$. Indeed, from the presentations of $\sfH_1$ and $\sfH_2$, we see that $\sfH_1\ot_R \sfH_2$ is the coproduct $\ov\sfH_1\ot_R\sfH_2$ and $\sfH_1\ot_R\ov\sfH_2$ over $\ov\sfH_1\ot_R\ov\sfH_2$, in the category of associative $R$-algebras. Therefore the $\ov\sfH_1\ot_R \sfH_2$-action on $\sfM_1$, and the transport of the $\sfH_1\ot_R\ov\sfH_2$-action on $\sfM_1$ via the isomorphism $\phi$, which agree on $\ov\sfH_1\ot_R\ov\sfH_2$, together give an action of $\sfH_1\ot_R \sfH_2$-module on $\sfM_1$.
    
    We denote the resulting $\sfH_1\ot \sfH_2$-module by $\sfM$, and call it the {\em generic oscillator bimodule}. Concretely, for $s\in \ov S_1\sqcup S_2$, $\sfT_s$ acts the way it does on $\sfM_1$. For the remaining simple reflection $t_m\in \sfS_1$, $\sfT_{t_m}$ acts on $\sfM_1$ via $\phi^{-1}\circ \sfT_{t_m}\circ \phi$. This is the same as the bimodule $\sfM$ constructed in \cite[\S 5]{MR4677077}. By the same consideration, the $\ov W_1\times W_2$-action on $\bZ[\SPM(m,n)]=K_0(\cM^k_1)$ extends to an action of $W_1\times W_2$.
\end{cons}

Using \Cref{Cor. K0 M1} and its analog for $\sfM_2$, we see that:
\begin{cor}\label{c:specialize generic bimod}
\begin{enumerate}
    \item The specialization $\sfM_{v=\sqrt q}$ is canonically isomorphic to the oscillator $H_1\ot H_2$-bimodule $M$.
    \item The specialization $\sfM_{v=1}$ is equipped with a canonical $W_1\times W_2$-action extending the $\ov W_1\times W_2$-action under the identification $\sfM_{v=1}=\sfM_{1,v=1}$, and extending the $W_1\times \ov W_2$-action under the identification $\phi^k: \sfM_{v=1}=\sfM_{1,v=1}\cong \sfM_{2,v=1}$. 
\end{enumerate}
\end{cor}

Let $\s\in \SPM(m,n)$ and $\IC_\s\in \cM^k_1$. By \Cref{l:involution s0}, we know that $\FT^{-1} \circ \IC_{t_m} \circ \FT (\IC_{\s})\in \cM_1^k$ is a simple perverse sheaf, hence is isomorphic to $\IC_\t$ for a unique $\t\in \SPM(m,n)$. The assignment $\s\mapsto \t$ defines a involution  
\begin{equation}\label{iota tm}
    \iota_m: \SPM(m,n)\rightarrow \SPM(m,n). 
\end{equation}
Using \Cref{l:involution s0} and \Cref{define Phi}, we obtain the explicit formula
\begin{equation}\label{Foueire t m iota m}
  \iota_m(\sigma)=\Phi^{-1}\bigl(t_m * \Phi(\sigma)\bigr).
\end{equation}
We caution that, in general, $\iota_m(\sigma)\neq t_m * \sigma$ for $\sigma\in \SPM(m,n)$; see \Cref{ss: type I W graph examples}.

\begin{cons}[Oscillator $W_1\times W_2$-graph]
    For the moment cone $\cN_{\bL_1}$ as a spherical variety for $\ov G_1\times G_2$,  by \Cref{cons:W graph}, we have attached a $\ov W_1\times W_2$-graph 
    \begin{equation*}
        \bG_1=(\SPM(m,n), \{\cD_1(\s)\}_{\s\in \SPM(m,n)}, E_1, \{\mu_1(e)\}_{e\in E_1}).
    \end{equation*}
    Here $E_1$ is the set of oriented edges of $\bG_1$. 
    
    For the moment cone $\cN_{\bL_2}$ as a spherical variety for $G^\circ_1\times \ov G_2$, by \Cref{cons:W graph}, we have attached a $W^\circ_1\times \ov W_2$-graph 
    \begin{equation*}
        \bG^\circ_2=(\SPM(n,m), \{\cD_2(\rho)\}_{\rho\in \SPM(n,m)}, E_2, \{\mu_2(e)\}_{e\in E_2}).
    \end{equation*}
    Note here $\cD_2(\rho)$ is a subset of $\sfS_1^\circ\sqcup \ov \sfS_2$, not of $\sfS_1\sqcup \ov \sfS_2$.

    The involution $t_m\in W_1$ induces an involution on the graph $\bG^\circ_2$: its action on the vertex set is given by the multiplication $*$, its action on $\sfS_1^\circ\sqcup \ov \sfS_2$ swaps $s_{m-1}$ and $s_m$ in $\sfS_1^\circ$ and fixes all other elements. We have
    \begin{equation}\label{involution edge}
         \mu_2(e)=\mu_2(t_m* e),\qquad e\in E_2. 
    \end{equation}
    Here we define $t_m*e= (t_m* \t \mapsto t_m * \s)$ if $e= (\t \mapsto \s)\in E_2$.

We now define a decorated directed graph $\bG^\circ$ by combining $\bG_1$ and $\bG^\circ_2$ using the Fourier bijection $\Phi$ in \eqref{define Phi} to identify their set of vertices. Namely, the vertex set of $\bG^\circ$ is still $\SPM(m,n)$. For $\s\in \SPM(m,n)$, let 
    \begin{equation*}
        \cD(\s)=\cD_1(\s)\cup \cD_2(\Phi(\s))\subset \sfS_1^\circ\sqcup \sfS_2.
    \end{equation*}
    The edge set for $\bG^\circ$ is
    \begin{equation*}
        E=E_1\cup \Phi^{-1}(E_2).
    \end{equation*}
    Here we use $\Phi$ to also denote the bijection between the set of ordered pairs in $\SPM(m,n)$ and the set of ordered pairs in $\SPM(n,m)$, so that $\Phi^{-1}(E_2)$ makes sense as a subset of ordered pairs in $\SPM(m,n)$. Finally, for $e\in E$ we define
    \begin{equation*}
        \mu(e)=\begin{cases}
        \mu_1(e), & \mbox{if } e\in E_1;\\
        \mu_2(\Phi(e)), & \mbox{if } \Phi(e)\in E_2.
        \end{cases}
    \end{equation*}
    The well-definedness of $\mu$ is verified in \Cref{p:W12 graph}.
  
 The involution $t_m\in W_1$ also induces an involution on the decorated direct graph $\bG^\circ$.  Its action on the vertex set $\SPM(m,n)$ is via \Cref{iota tm}, and action on $\sfS_1^\circ\sqcup  \sfS_2$ swaps $s_{m-1}$ and $s_m$ in $\sfS_1^\circ$ and fixes all other elements. 
By \Cref{involution edge} and the well-definedness of $\mu$, we know that for any edge $e\in E$, 
   \begin{equation}\label{s0 preserves mu}
        \mu(e)=\mu(\iota_m(e))
    \end{equation}
Here we use $\iota_m$ to also denote the bijection between the set of ordered pairs in $\SPM(m,n)$ induced by \Cref{iota tm}.
\end{cons}

\begin{prop}\label{p:W12 graph}
    The function $\mu: E\to \ZZ$ is well-defined. The action of $\sfH_1\ot_R\sfH_2$ on $\sfM$ is determined by the formulas, where $s\in \sfS_1^\circ\sqcup \sfS_2$ and $\s\in \SPM(m,n)$
    \begin{eqnarray*}
        \sfT_s\cdot C'_\s&=&\begin{cases}
            v^2C'_\s & s\in \cD(\s)\\
            -C'_\s+v\sum_{e=(\s\to \tau)\in E, s\in \cD(\tau)}\mu(e)C'_\tau, & s\notin \cD(\s).
        \end{cases}\\
  \sfT_{t_m}\cdot C'_{\s}&=&C'_{ \iota_m(\s)}
    \end{eqnarray*}
\end{prop}
\begin{proof} The formula for the $\sfT_s$-action follows from \Cref{Cor. Hk action on KL basis} applied to $\sfM_1$ and $\sfM_2$, as long as we check that $\mu$ is well-defined. 

It thus remains to check that $\mu$ is well-defined: if $e=(\s\to \t)\in E_1$ and $\Phi(e)=(\Phi(\s)\to \Phi(\t))\in E_2$, then $\mu_1(e)=\mu_2(\Phi(e))$.

For $s\in \ov \sfS_1\sqcup \ov \sfS_2$ and any $\s\in \SPM(m,n)$, $s\in \cD_1(\s)$ if and only if $\sfT_s\cdot C'_\s=v^2C'_\s$; similarly, $s\in \cD_2(\Phi(\s))$ if and only if $\sfT_s\cdot C'_{\Phi(\s)}=v^2C'_{\Phi(\s)}$. By \Cref{FT type I generic}, we know that $\cD_1(\s)$ and $\cD_2(\s)$ have the same intersection with $\ov \sfS_1\sqcup \ov \sfS_2$. In other words, $\cD(\s)\cap (\ov \sfS_1\sqcup \ov \sfS_2)=\cD_i(\s)\cap (\ov \sfS_1\sqcup \ov \sfS_2)$ for $i=1,2$. 

Now suppose $e=(\s\to \t)\in E_1$ and $\Phi(e)=(\Phi(\s)\to \Phi(\t))\in E_2$. If there exists $s\in \ov \sfS_1\sqcup \ov \sfS_2$ such that $s\in \cD(\t)-\cD(\s)$, then $v\mu_1(e)$ is the coefficient of $C'_\t$ in $\sfT_s\cdot C'_\s$, and $v\mu_2(\Phi(e))$ is the coefficient of $C'_{\Phi(\t)}$ in $\sfT_s\cdot C'_{\Phi(\s)}$. Since $\phi$ intertwines the $\sfT_s$-action and sends $C'_\s$ (resp. $C'_\t$) to $C'_{\Phi(\s)}$ (resp. $C'_{\Phi(\t)}$) by \Cref{l:phi IC}, we see that $\mu_1(e)=\mu_2(\Phi(e))$.

The remaining case is when $\cD_1(\t)-\cD_1(\s)=\{s_m\}$. In this case, we have $\cD_1(\io_m(\t))-\cD_1(\io_m(\s))=\{s_{m-1}\}$ which is in $\ov S_1$, hence 
\begin{equation}\label{comptible 1}
    \mu_1(\io_m(e))=\mu_2(\Phi(\io_m(e))=\mu_2(t_m* \Phi(e))
\end{equation}
 by the previous discussion and \Cref{Foueire t m iota m}. By \Cref{involution edge}, we also have 
 \begin{equation}\label{comptible 2}
  \mu_2(t_m* \Phi(e))=\mu_2(\Phi(e))
\end{equation}
Combining \Cref{comptible 1} and \eqref{comptible 2}, we have $\mu_1(e)=\mu_2(\Phi(e))$. 

The formula for the action of $\sfT_{t_m}$ follows from \Cref{l:involution s0}. This finishes the proof.
\end{proof}
\begin{remark}
    Using the $W_1^\circ\times W_2$-graph $\bG^\circ$, we can apply the procedure of \Cref{def:cell} to define cells for $\SPM(m,n)$, and the cell filtration for $\sfM$ as an $\sfH_1^\circ\ot_R \sfH_2$-bimodule .
\end{remark}

\subsection{A $W$-graph example for type I theta correspondence}\label{ss: type I W graph examples}
In this section, we computing explicitly the $W$-graph associated with the oscillator Hecke module arising from the \emph{type I} theta correspondence in the case when $m=2,n=1$. In this case, we have $W_1^\circ=\langle s_1,s_2\rangle$, $W_1=\langle s_1,t_2\rangle$ and $W_2=\langle s_1'\rangle$. 
\begin{enumerate}
        \item The computations for $\ov W_1\times W_2$-graph $\bG_1$ is similar to \Cref{section W-graph}. We therefore omit the details and record the resulting $W$-graph below.
    
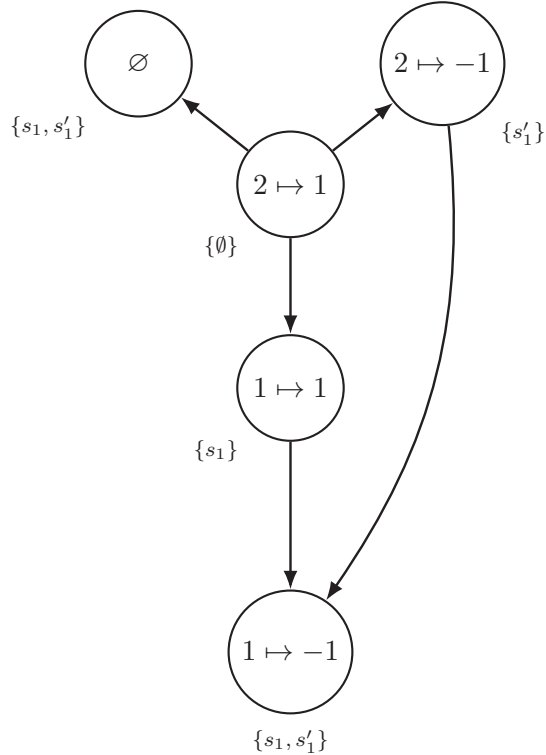
\begin{center}
\begin{tikzpicture}[node distance=18mm]
\WNode[below left]{v0}{(-2,2.8)}{$\varnothing$}{$\{s_1,s'_1\}$}
\WNode[below right]{v1}{( 2,2.8)}{$2\mapsto -1$}{$\{s'_1\}$}


\WNode[below left]{v2}{(0,1.2)}{$2\mapsto 1$}{$\{\emptyset\}$}

\WNode[below left]{v3}{(0,-1.5)}{$1\mapsto 1$}
  {$\{s_1\}$}

\WNode[below]{v4}{(0,-5)}{$1\mapsto -1$}{$\{s_1,s'_1\}$}

\WEdge{v2}{v0}{0}
\WEdge{v2}{v1}{0}
\WEdge{v2}{v3}{0}
\WEdge{v3}{v4}{0}

\WEdge{v1}{v4}{20}
\end{tikzpicture}
\end{center}
\item The $W_1^\circ\times \ov W_2=W_1^\circ$-graph $\bG_2$ is given as follows. Here we use the $\leftrightarrow$ to denote the involution induced by $t_2$ on $\bG_2$. 
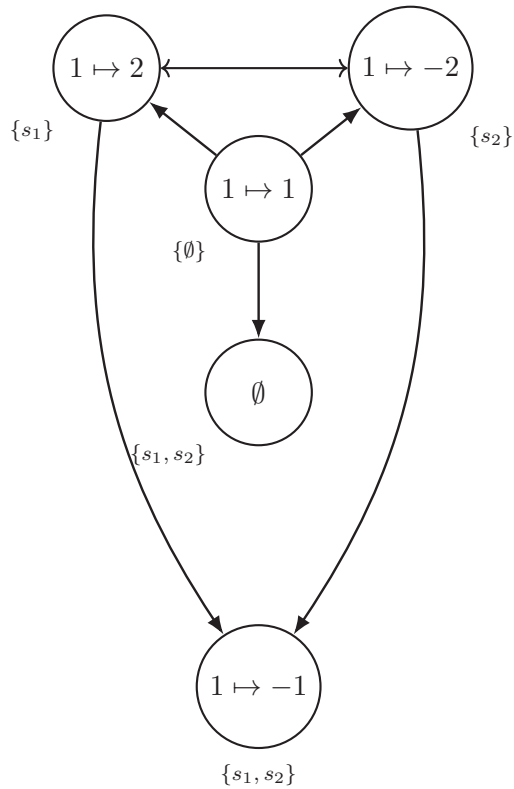
\begin{center}
\begin{tikzpicture}[node distance=20mm]

\WNode[below left]{v0}{(-2,2.8)}{$1\mapsto 2$}{$\{s_1\}$}
\WNode[below right]{v1}{( 2,2.8)}{$1\mapsto -2$}{$\{s_2\}$}

\WEdgeBoth{v0}{v1}{0}

\WNode[below left]{v2}{(0,1.2)}{$1\mapsto 1$}{$\{\emptyset\}$}

\WNode[below left]{v3}{(0,-1.5)}{$\emptyset$}
  {$\{s_1,s_2\}$}

\WNode[below]{v4}{(0,-5.4)}{$1\mapsto -1$}{$\{s_1,s_2\}$}

\WEdge{v2}{v0}{0}
\WEdge{v2}{v1}{0}
\WEdge{v2}{v3}{0}
\WEdge{v0}{v4}{-20}
\WEdge{v1}{v4}{20}
\end{tikzpicture}
\end{center}
\item By applying the results in \Cref{subsection Fourier bijection}, we compute the Fourier transform: 
    \begin{equation}\label{Fourier type I example}
    \begin{split}
        \Phi (\emptyset)& = (1\mapsto 2) \\
        \Phi(2\mapsto 1)&=  (1\mapsto 1)\\
        \Phi(1\mapsto 1)&= (\emptyset)\\
        \Phi(2\mapsto -1)&= (1\mapsto -2)\\
         \Phi(1\mapsto -1)&= (1\mapsto -1)\\
    \end{split}
    \end{equation}
    \item Combining all above, we get the $W_1^\circ\times  W_2$-graph $\bG$ as follows. We use$\leftrightarrow$ to denote the involution of $\iota_2$ on $\bG$.  
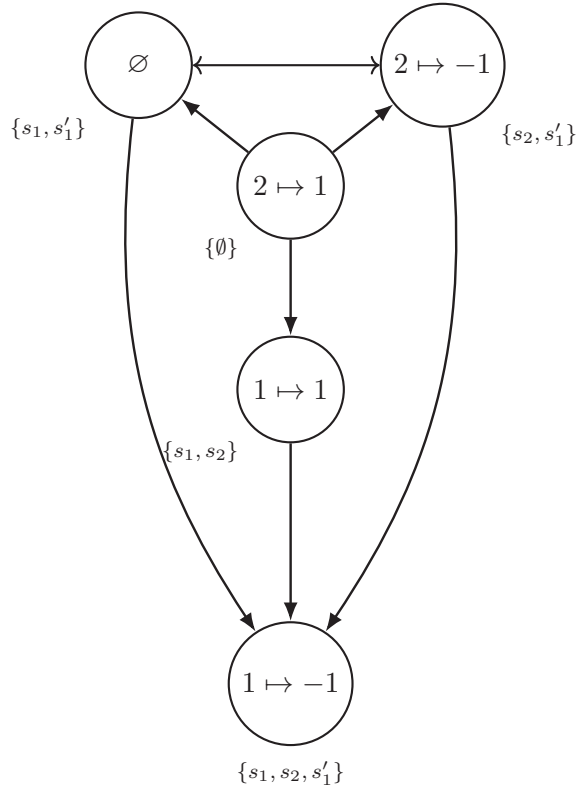
\begin{center}
\begin{tikzpicture}[node distance=20mm]

\WNode[below left]{v0}{(-2,2.8)}{$\varnothing$}{$\{s_1,s'_1\}$}
\WNode[below right]{v1}{( 2,2.8)}{$2\mapsto -1$}{$\{s_2,s'_1\}$}

\WEdgeBoth{v0}{v1}{0}

\WNode[below left]{v2}{(0,1.2)}{$2\mapsto 1$}{$\{\emptyset\}$}

\WNode[below left]{v3}{(0,-1.5)}{$1\mapsto 1$}
  {$\{s_1,s_2\}$}

\WNode[below]{v4}{(0,-5.4)}{$1\mapsto -1$}{$\{s_1,s_2,s'_1\}$}

\WEdge{v2}{v0}{0}
\WEdge{v2}{v1}{0}
\WEdge{v2}{v3}{0}
\WEdge{v3}{v4}{0}

\WEdge{v0}{v4}{-20}
\WEdge{v1}{v4}{20}
\end{tikzpicture}
\end{center}
\end{enumerate}
In this case, each vertex is a cell and there are $5$-cells. 
\trivial[h]{Some details computation. 
\begin{Exa}\label{W graph example type I}
    Let $m=2, n=1$. We have $W_1^\circ=\langle s_1,s_2\rangle$, $W_1=\langle s_1,t_2\rangle$ and $W_2=\langle s_1'\rangle$. 
\begin{enumerate}
    \item We first construct the $\ov W_1\times W_2$-graph $\bG_1$. The Bruhat order for $\SPM(2,1)$ is given as follows 
    \begin{center}
\begin{tikzpicture}[node distance=20mm]

\WNode[below]{vtop}{(0,2.8)}{$1\mapsto -1$}{$d_\s=3$}

\WNode[below left]{vL}{(-1.8,1.0)}{$1\mapsto 1$}{$d_\s=2$}
\WNode[below right]{vR}{( 1.8,1.0)}{$2\mapsto -1$}{$d_\s=2$}

\WNode[below left]{vmid}{(0,-0.6)}{$2\mapsto 1$}{$d_\s=1$}

\WNode[below]{vbot}{(0,-2.6)}{$\varnothing$}{$d_\s=0$}

\WEdge{vtop}{vL}{0}
\WEdge{vtop}{vR}{0}

\WEdge{vL}{vmid}{0}
\WEdge{vR}{vmid}{0}

\WEdge{vmid}{vbot}{0}
\end{tikzpicture}
\end{center}
One checks that
\begin{itemize}
    \item $\overline{\cO}_\sigma$ is smooth (in fact a linear subspace of $\Hom(L_1,V_2)$) unless $\sigma=1\mapsto -1$.  In the smooth cases we have
    \[
      \IC_\sigma=\underline{\Ql}_{\overline{\cO}_\sigma}\langle d_\sigma\rangle
    \]
    and $P_{\sigma',\sigma}(t)=1$ for $\sigma'<\sigma$. 
    \item If $\sigma=1\mapsto -1$, then
    \[
       \overline{\cO}_\sigma=\{T\in\Hom(L_1,V_2)\mid\rank(T)=1\}.
    \]
    This is the determinantal variety of $2\times2$ matrices of rank $1$, with a singularity at $0$. By the same computations in \Cref{Example W-graph type II}, we deduce that 
    \begin{equation}\label{Kazhan-lustig type I}
         P_{\sigma',\sigma}(t)=\begin{cases}
        1+t^2 \qquad & \mbox{if $\sigma'=\emptyset$}\\
        1 \quad & \mbox{otherwise}. 
    \end{cases} 
    \end{equation}
 \end{itemize}   
Combining the Bruhat graph, \Cref{Kazhan-lustig type I} and \Cref{Lem. type G or U type I}, we get the $\ov W_1\times W_2$-graph $\bG_1$ as follows 
\begin{center}
\begin{tikzpicture}[node distance=20mm]

\WNode[below]{vtop}{(0,2.8)}{$1\mapsto -1$}{$\{s_1,s'_1\}$}

\WNode[below left]{vL}{(-1.8,1.0)}{$1\mapsto 1$}{$\{s_1\}$}
\WNode[below right]{vR}{( 1.8,1.0)}{$2\mapsto -1$}{$\{s'_1\}$}

\WNode[below left]{vmid}{(0,-0.6)}{$2\mapsto 1$}{$\emptyset$}

\WNode[below]{vbot}{(0,-2.6)}{$\varnothing$}{$\{s_1,s'_1\}$}

\WEdge{vL}{vtop}{0}
\WEdge{vR}{vtop}{0}

\WEdge{vmid}{vL}{0}
\WEdge{vmid}{vR}{0}

\WEdge{vmid}{vbot}{0}

\end{tikzpicture}
\end{center}

    \item Next we construct the $W_1^\circ\times \ov W_2=W_1^\circ$-graph $\bG_2$. The Bruhat order for $\SPM(1,2)$ is given as follows 
    \begin{center}
\begin{tikzpicture}[node distance=20mm]
\WNode[below]{vtop}{(0,2.8)}{$1\mapsto -1$}{$d_\s=3$}

\WNode[below left]{vL}{(-1.8,1.0)}{$1\mapsto 2$}{$d_\s=2$}
\WNode[below right]{vR}{( 1.8,1.0)}{$1\mapsto -2$}{$d_\s=2$}

\WNode[below left]{vmid}{(0,-0.6)}{$1\mapsto 1$}{$d_\s=1$}

\WNode[below]{vbot}{(0,-2.6)}{$\varnothing$}{$d_\s=0$}

\WEdge{vtop}{vL}{0}
\WEdge{vtop}{vR}{0}

\WEdge{vL}{vmid}{0}
\WEdge{vR}{vmid}{0}

\WEdge{vmid}{vbot}{0}
\end{tikzpicture}
\end{center}
One checks that
\begin{itemize}
    \item $\overline{\cO}_\sigma$ is smooth (in fact a linear subspace of $\Hom(L_2,V_1)$) unless $\sigma=1\mapsto -1$.  
    \item If $\sigma=1\mapsto -1$, then $\overline{\cO}_\sigma=\cN_{\bL_2}$ is the variety of isotropic vectors in $V_1$, with a singularity at $0$. One can 
compute the Kazhdan-Lusztig polynomial using the resolution 
\[
\wt \cN_{\bL_2}=\{(T,L)| L\subseteq V_1, L\,\,\mbox{isotropic},\,\, T\in \Hom(L_1,L)\}
\] 
with
\[
\pi: \wt \cN_{\bL_2} \longrightarrow \cN_{\bL_2}, \quad (T,L)\mapsto T.
\]
The exceptional fibre over the singular point $0$ is the set of isotropic line in $V_1$, which is isomorphic to 
    \[
      \pi^{-1}(0)\cong \bP^1\times \bP_1 .
    \]
We have 
    \begin{equation}\label{Kazhan-lustig type I 2}
         P_{\sigma',\sigma}(t)=\begin{cases}
        1+t^2 \qquad & \mbox{if $\sigma'=\emptyset$}\\
        1 \quad & \mbox{otherwise}. 
    \end{cases} 
    \end{equation}
 \end{itemize}   
 Combining the Bruhat graph, \Cref{Kazhan-lustig type I 2} and \Cref{Lem. type G or U type I model 2}, we get the $W_1^\circ \times \ov W_2$-graph $\bG_2$ as follows 
\begin{center}
\begin{tikzpicture}[node distance=20mm]

\WNode[below]{vtop}{(0,2.8)}{$1\mapsto -1$}{$\{s_1,s_2\}$}

\WNode[below left]{vL}{(-1.8,1.0)}{$1\mapsto -2$}{$\{s_2\}$}
\WNode[below right]{vR}{( 1.8,1.0)}{$1\mapsto 2$}{$\{s_1\}$}

\WNode[below left]{vmid}{(0,-0.6)}{$1\mapsto 1$}{$\emptyset$}

\WNode[below]{vbot}{(0,-2.6)}{$\varnothing$}{$\{s_1,s_2\}$}

\WEdge{vL}{vtop}{0}
\WEdge{vR}{vtop}{0}

\WEdge{vmid}{vL}{0}
\WEdge{vmid}{vR}{0}

\WEdge{vmid}{vbot}{0}

\end{tikzpicture}
\end{center}
    \item The Fourier transform and bijection in \Cref{define Phi}. In this case, the orbit closure is linear subspace expect for one case, using , we get 
    \begin{equation}\label{Fourier type I example}
    \begin{split}
        \Phi (\emptyset)& = (1\mapsto 2) \\
        \Phi(2\mapsto 1)&=  (1\mapsto 1)\\
        \Phi(1\mapsto 1)&= (\emptyset)\\
        \Phi(2\mapsto -1)&= (1\mapsto -2)\\
         \Phi(1\mapsto -1)&= (1\mapsto -1)\\
    \end{split}
    \end{equation}
    The map $\iota_2$ in \Cref{iota tm} is computed as follows: 
    \begin{equation}\label{iota tm example}
    \begin{split}
        \iota_2 (\emptyset)& = (2\mapsto -1) \\
        \iota_2(2\mapsto 1)&=  (2\mapsto 1)\\
        \iota_2(1\mapsto 1)&= (1\mapsto 1)\\
        \iota_2(2\mapsto -1)&= (\emptyset)\\
         \iota_2(1\mapsto -1)&= (1\mapsto -1)\\
    \end{split}
    \end{equation}
    \item Combining all above, we get the $W_1^\circ\times  W_2$-graph $\bG$. Here we use the $\leftrightarrow$ to denote the involution of $\iota_2$ on $\bG$.  
\begin{center}
\begin{tikzpicture}[node distance=20mm]

\WNode[below left]{v0}{(-2,2.8)}{$\varnothing$}{$\{s_1,s'_1\}$}
\WNode[below right]{v1}{( 2,2.8)}{$2\mapsto -1$}{$\{s_2,s'_1\}$}

\WEdgeBoth{v0}{v1}{0}

\WNode[below left]{v2}{(0,1.2)}{$2\mapsto -1$}{$\{\emptyset\}$}

\WNode[below left]{v3}{(0,-1.5)}{$1\mapsto 1$}
  {$\{s_1,s_2\}$}

\WNode[below]{v4}{(0,-5.4)}{$1\mapsto -1$}{$\{s_1,s_2,s'_1\}$}

\WEdge{v2}{v0}{0}
\WEdge{v2}{v1}{0}
\WEdge{v2}{v3}{0}
\WEdge{v3}{v4}{0}

\WEdge{v0}{v4}{-20}
\WEdge{v1}{v4}{20}
\end{tikzpicture}
\end{center}
\end{enumerate}
\end{Exa}
}

\subsection{From  $\bF_q$ to $\bC$}\label{Sec. Fq to C}
Our previous discussions about the Hecke category $\cH^k_{i}$, the Hecke module categories $\cM^k_{i}$ ($i=1,2$) work equally well when $k$ is replaced by $\bC$. The only modification needed is to replace Fourier-Deligne transform by Fourier-Sato transform, which works for $\Gm$-monodromic sheaves on vector bundles over stacks over $\bC$. 

We use a superscript $\bC$ to indicate the counterpart over $\bC$ of what we considered over $\bF_q$. For example, the symplectic and quadratic spaces $V^{\bC}_{1}$ and $V^{\bC}_{2}$, their isometry groups $G^{\bC}_{i}$, their Hecke categories  $\cH^{\bC}_{i}$, the $\bcH_1^{\bC}\ot \cH_2^{\bC}$-module category $\cM^{\bC}_1$ (depending on the choice of a Lagrangian $L_1^{\bC}\subset V_1^{\bC}$) and the $\cH_1^{\bC}\ot \bcH_2^{\bC}$-module category $\cM^{\bC}_2$ (depending on the choice of a Lagrangian $L_2^{\bC}\subset V_2^{\bC}$). The same argument as \Cref{Cor. Hk ring isom} shows that $\chi$ gives a ring isomorphism for $i=1,2$
\begin{equation}\label{chiC Hi}
    \chi^{\bC}_{\cH_i}: K_0(\cH^{\bC}_{i})\isom \bZ[W_i].
\end{equation}

Similarly, by \Cref{Lem. ch isom}, $\chi$ gives an isomorphism:
\begin{eqnarray*}
    \chi^{\bC}_{\cM_1}: K_0(\cM^{\bC}_{1})\isom \bZ[\underline{\ov B^{\bC}_{1}\times B^{\bC}_{2}\bs \cN^{\bC}_{\bL_{1}}}].
\end{eqnarray*}
compatible with the actions of the both sides by $K_0(\bcH^\bC_1)\ot K_0(\cH^\bC_2)$ and $\bZ[\ov W_1]\ot \bZ[W_2]$ which are identified via \eqref{chiC Hi} and its analog for $\bcH_i$. Likewise, we have an isomorphism
\begin{equation*}
    \chi^{\bC}_{\cM_2}: K_0(\cM^{\bC}_{2})\isom \bZ[\underline{ B^{\bC}_{1}\times \ov B^{\bC}_{2}\bs \cN^{\bC}_{\bL_{2}}}]
\end{equation*}
compatible with the $K_0(\cH^\bC_1)\ot K_0(\bcH^\bC_2)$ and $\bZ[ W_1]\ot \bZ[\ov W_2]$ on both sides.

The same argument of \Cref{Fourier-transform 2} shows that partial Fourier-Sato transform gives an equivalence
\begin{equation*}
    \FT^\bC_{\cM_1}:\cM_1^\bC\isom  \cM_2^\bC.
\end{equation*}
Passing to Grothendieck groups it defines a $\ov W_1\times \ov W_2$-equivariant isomorphism
\begin{equation*}
    \phi^\bC: K_0(\cM_1^\bC)\isom  K_0(\cM_2^\bC).
\end{equation*}
The bijection on $\IC$ sheaves induced by $\FT^\bC_{\cM_1}$ gives a bijection
    \begin{equation*}
        \Phi^\bC: \SPM(m,n)\isom \SPM(n,m).
    \end{equation*}

By \Cref{SPM vs orbits}, both orbit sets $\underline{\ov B^{\bC}_{1}\times B^{\bC}_{2}\bs \cN^{\bC}_{\bL_{1}}}$ and $\underline{\ov B_{1,k}\times B_{2,k}\bs \cN_{\bL_{1},k}}$ are canonically identified with $\SPM(m,n)$. The same is true for $\underline{ B^{\bC}_{1}\times \ov B^{\bC}_{2}\bs \cN^{\bC}_{\bL_{2}}}$ and $\underline{B_{1,k}\times \ov B_{2,k}\bs \cN_{\bL_{2},k}}$. We therefore get a canonical isomorphism of abelian groups
\begin{equation}\label{comp isom}
   \Comp_{\cM_i}: \chi^{\bC,-1}_{\cM_i}\circ \chi^k_{\cM_i}: K_0(\cM^k_i)\cong \ZZ[\SPM(m,n)]\cong K_0(\cM^{\bC}_i).
\end{equation}

\begin{lemma}\label{Lem. compare Fq and C}
\begin{enumerate}
    \item The isomorphism $\Comp_{\cM_1}$ is $\ov W_1\times W_2$-equivariant; 
    \item The isomorphism $\Comp_{\cM_1}$ sends $\chi_{\cM_1}(\IC_\s)\in \cM^k_1$ to $\chi^\bC_{\cM_1}(\IC_\s)\in K_0(\cM_1^{\bC})$ for any $\s\in \SPM(m,n)$.
\end{enumerate}
Parallel statements also holds for $\Comp_{\cM_2}$ and $\rho\in \SPM(n,m)$. 
\end{lemma}
\begin{proof}
    (1) The actions of simple reflections both for $K_0(\cM^k_1)$ and for $K_0(\cM^\bC_1)$ are given by \Cref{Cor. W action sph orb}.

    (2) The calculations of $\IC_s\star \IC_\s$ over $k$ and over $\bC$ are both given by \Cref{Lem. conv ICs}. The rest of the argument is similar to that of \Cref{Lem. KL indep p}, by induction on $d_\s$, with initial cases being $\s=\s_i$, in which case $\chi(\IC_{\s_i})=(-1)^{d_{\s_i}}\sum_{\s\le \s_i}\one_{\s}$ holds both over $k$ and over $\bC$. 
\end{proof}

\begin{cor}\label{Cor. compare FT Fq vs C}
\begin{enumerate}
   \item $\Phi^\bC=\Phi$.
    \item The following diagram commutes:
    \begin{equation*}       \xymatrix{K_0(\cM^k_1)\ar[r]^-{\phi^{k}}\ar[d]_{\Comp_{\cM_1}} & K_0(\cM^k_2)\ar[d]^{\Comp_{\cM_2}}\\     K_0(\cM^{\bC}_1)\ar[r]^-{\phi^{\bC}} & K_0(\cM^{\bC}_2)}
    \end{equation*}
\end{enumerate}
\end{cor}
\begin{proof}
$\Phi$ is determined in \Cref{Prop FT SPM via PM}, and the discussions in \Cref{Prop FT SPM via PM} also work over $\bC$, hence $\Phi^\bC=\Phi$. The commutativity in (2) then follows from (1) together with \Cref{Lem. compare Fq and C}(2).
\end{proof}

\begin{cor}\label{c:W12 action over C}
    The $\ov W_1\times W_2$-action on $K_0(\cM^{\bC}_1)$ extends canonically to an action of $W_1\times W_2$ such that under the isomorphism $\phi^{\bC}$, it extends the $ W_1\times \ov W_2$-action on $K_0(\cM^{\bC}_2)$. Moreover, the isomorphism $\Comp_{\cM_1}: \sfM_{v=1,\bZ}\isom K_0(\cM^{\bC}_1)$ is $W_1\times W_2$-equivariant. 
\end{cor}

\subsection{Characteristic cycles and proof of the main theorem}\label{section proof main thm}

In this subsection we work over $\bC$ until before the proof of \Cref{thm: Springer for theta}. 
Retain the notations from \Cref{section Ortho-symplectic}.
By \Cref{pro Lambda}, the union of the conormals of $\ov B_1\times B_2$-orbits on $\cN_{\bL_1}$ (conormals calculated in $T^*\bL_1$) is identified with $\Lambda_{\VV}$ at least up to taking the reduced structure. 
Applying \Cref{CC gen Q} to the case of $\ov G_1\times G_2$ action on $\cN_{\bL_1}$ and $\bL_1$, we get an isomorphism of abelian groups
\begin{equation}\label{CC M1}
    \CC_\bC: \sfM_{v=1,\bC}=K_0(\cM_1^{\bC})\otimes \bC\to \hBM{\TOP}{\Lambda_{\VV},\bC}\cong \hBM{\TOP}{\St_{\bV},\bC}.
\end{equation}

In \Cref{section Ortho-symplectic}, we obtain a $W_1\times W_2$-action on $\hBM{\TOP}{\St_{\bV},\bC}$ by relative Springer theory. On the other hand, by \Cref{c:W12 action over C}, $K_0(\cM^{\bC}_1)\otimes \bC$ also carries a canonical $W_1\times W_2$-action, which is intertwined with the $W_1\times W_2$-action on $\sfM_{v=1,\bC}$ constructed in \Cref{c:specialize generic bimod} via the comparison isomorphism \eqref{comp isom}.

\begin{prop}\label{prop:CC W12 eq}
    The map $\CC_\bC$ in \eqref{CC M1} is $W_1\times W_2$-equivariant up to tensoring with the sign character of $W_1\times W_2$.
\end{prop}
\begin{proof}
    By \Cref{Prop. CC W eq}, $\CC_{\bC}$ is $\ov W_1\times W_2$-equivariant up to the sign character. Now apply the same consideration to the $B_1\times \ov B_2$-action on $\cN_{\bL_2}$, we get an isomorphism
    \begin{equation*}
        \CC_{\bL_2,\bC}: K_0(\cM_2^{\bC})\otimes \bC\isom \hBM{\TOP}{\Lambda_{\VV},\bC}\cong \hBM{\TOP}{\St_{\bV},\bC}
    \end{equation*}
    that is $W_1\times\ov W_2$-equivariant up to the sign character. Here we are using \Cref{r:CC W eq fullO} to deal with the full orthogonal group. We denote the map \eqref{CC M1} by $\CC_{\bL_1,\bC}$. By \cite[]{KS}, Fourier-Sato transform preserves characteristic cycles, therefore the diagram
    \begin{equation*}
        \xymatrix{K_0(\cM_1^{\bC})\otimes \bC\ar[d]_{\phi^{\bC}}\ar[r]^-{\CC_{\bL_1,\bC}} & \hBM{\TOP}{\Lambda_{\VV},\bC}\ar@{=}[d]\\
        K_0(\cM_2^{\bC})\otimes \bC \ar[r]^-{\CC_{\bL_2,\bC}} & \hBM{\TOP}{\Lambda_{\VV},\bC}
        }
    \end{equation*}
    is commutative. Since the $W_1$-action on $\sfM_{v=1,\bC}=\sfM_{1,v=1,\bC}$ is obtained by identifying it with $K_0(\cM_2^{\bC})\otimes \bC$ via $\phi^{\bC}$, the above diagram implies that $\CC_{\bC}=\CC_{\bL_1,\bC}$ is also $W_1$-equivariant up to the sign character. This finishes the proof. 
\end{proof}

Now we can give the proof of our main result on theta correspondence.

\begin{proof}[Proof of \Cref{thm: Springer for theta}]

Combining \Cref{thm:Htop},  \Cref{c:W12 action over C} and \Cref{prop:CC W12 eq}, we obtain a $W_1 \times W_2$-equivariant isomorphism
\begin{equation} \label{M nu1 decomp}
    \sfM_{v = 1, \bC} \otimes \sgn_{W_1\times W_2} \cong \bigoplus_{(\cO, \cO_1, \cL_1, \cO_2, \cL_2) \in \cR^\hs_{G_1,G_2}(\VV)} E_{\cO_1, \cL_1} \otimes E_{\cO_2, \cL_2}.
\end{equation}
On the other hand, by \Cref{c:specialize generic bimod}, we have an isomorphism
\[
\sfM_{v = \sqrt{q}, \bC} \cong M
\]
as $H_1 \otimes H_2$-modules (using the fixed isomorphism $\iota: \Qlbar \cong \bC$).

The remainder of the proof then follows by the same argument as in the proof of \Cref{thm:intro spherical}, by applying \Cref{TitsLusztig} to Lusztig's homomorphism
\[
\lambda_{W_1, \bC} \otimes \lambda_{W_2, \bC} : \left( \sfH_1 \otimes_R \sfH_2 \right) \otimes_R \bC[v, v^{-1}] \xrightarrow{\sim} \bC[v, v^{-1}][W_1 \times W_2],
\]
the $(\sfH_1 \otimes_R \sfH_2) \otimes_R \bC[v, v^{-1}]$-module $\sfM \otimes_R \bC[v, v^{-1}]$, and the specializations $v = 1$ and $v = \sqrt{q}$. 


\end{proof}

\appendix 
\def\Out{\mathrm{Out}}
\def\Inn{\mathrm{Inn}}
\def\Av{\mathrm{Av}}
\def\tcNO{{\widetilde{\cN}^\circ}}
\section{Springer correspondence for even orthogonal groups}\label{sec:orthogonal Springer}
The Springer correspondence for connected reductive group is generalized to disconnected reductive groups in \cite{MR3845761} and \cite{dillery2023stacky}. The goal of this appendix is to make the Springer correspondence for even orthogonal groups explicit. We begin by formulating the Springer correspondence for (possibly disconnected) reductive groups, adopting the framework in \cite{clausen2008springer}. We also prove the compatibility between the Springer correspondence for a reductive group and that for its identity component.



We work over \( \bC \) throughout this section, and . Let \( G \) be a (possibly disconnected) reductive group, with identity component \( G^\circ \). Let \( \pi_0(G) = G/G^\circ \) denote the group of connected components. To simply notation, we write $\pi_0(G)=\Gamma$. We have an exact sequence 
\begin{equation}\label{GG0}
    1\longrightarrow G^\circ \longrightarrow G \longrightarrow \Gamma \longrightarrow 1. 
\end{equation}
\trivial[h]{
This induce an map 
\[
\Gamma \longrightarrow \Aut(G^\circ)/\Inn(G^\circ)=\Out(G^\circ).
\]
By choosing a pinning \( (G^\circ, B^\circ, T^\circ, \{X_\alpha\}) \) of \( G^\circ \), we get a section of $\Out(G^\circ)$ in $\Aut(G_0)$ as pinned isomorphism $\Aut(G^\circ, B^\circ, T^\circ, \{X_\alpha\})$. This gives a group homomorphism 
\[
\Gamma=\pi_0(G)\rightarrow \Aut(G^\circ, B^\circ, T^\circ, \{X_\alpha\})
\]
In practice, this means that given for each $z\in \pi_0(G)$, one can choose a section $\widetilde{z}\in G$ 
 such that the conjugate action of $\widetilde{x}$ on $G_0$ lies in $\Aut(G^\circ, B^\circ, T^\circ, \{X_\alpha\})$. Note that the lifting $z\mapsto \widetilde{z}$ may not be a group homomorphism and the obstruction is given by
 \[
 H^2(\pi_0(G), Z(G_0)) . 
 \]
 }
We fix a Borel subgroup $B $ of $G^\circ$ and a maximal torus $T \subseteq B$ of $G^\circ$. Let $\frak b$ and $\frak t$ be the Lie algebras of $B$ and $T$, respectively, and $\frak u=[\frak b,\frak b]$ be the nil-radical of $\frak b$. 

Let \( W = N_G(T)/T \) and \( W^\circ = N_{G^\circ}(T)/T\) denote the Weyl groups of \( G \) and \( G^\circ \), respectively. 
The choice of $B$ gives a set of simple reflections in $W^\circ$, making it a Coxeter group. Since all maximal tori of $G^\circ$ are conjugate in $G^\circ$, the restriction of the surjection $G\rightarrow \Gamma$ to $N_G(T)$ is still surjective and it induce an exact sequence 
\[
1\longrightarrow W^\circ \longrightarrow W \longrightarrow \Gamma\longrightarrow 1. 
\]
One also checks that $\Gamma\cong N_{G}(B,T)/T$ and this provides a section of the above exact sequence and hence a semi-direct product decomposition
\[
W \cong W^\circ \rtimes \Gamma,
\]
where $\Gamma$ acts on $W^\circ$ by Coxeter group automorphisms.
Let \(\fgg\) be the Lie algebra of \(G^\circ\), and let \(\cN \subset \fgg\) be the nilpotent cone. The forgetful functor
\[
\For_{G^\circ}^G: D^b_G(\cN) \longrightarrow D^b_{G^\circ}(\cN)
\]
admits both left and right adjoints, denoted \(\Av_{G^\circ,!}^G\) and \(\Av_{G^\circ,*}^G\). Since \(G^\circ\) has finite index in \(G\), \cite[Cor.~6.6.3]{MR4337423} implies that
\[
\Av_{G^\circ,!}^G \cong \Av_{G^\circ,*}^G
\quad\text{and both are \(t\)-exact.}
\]
We henceforth write \(\Av_{G^\circ}^G := \Av_{G^\circ,!}^G \cong \Av_{G^\circ,*}^G\). In particular, $\Av_{G^\circ}^G$ restricts to an exact functor
\[
\Av_{G^\circ}^G: \Perv_{G^\circ}(\cN) \longrightarrow \Perv_{G}(\cN).
\]

Let \(\LSG\) and \(\LSGO\) be as in \Cref{eq local system}. Then
\[
\Irr\bigl(\Perv_G(\cN)\bigr)=\{\IC(\cO,\cL)\mid (\cO,\cL)\in\LSG\},\qquad
\Irr\bigl(\Perv_{G^\circ}(\cN)\bigr)=\{\IC(\cO,\cL)\mid (\cO,\cL)\in\LSGO\}.
\]
We say that a \(G\)-orbit \(\cO\subset\cN\) and a \(G^\circ\)-orbit \(\cO^\circ\subset\cN\) are \emph{related} if \(G\cdot \cO^\circ=\cO\); in this case we write \(\cO \overset{\mathrm{rel}}{\sim} \cO^\circ\). Then there is a natural injective homomorphism of component groups
\[
i_{\cO^\circ,\cO}: A^\circ_{\cO^\circ}\hookrightarrow A_{\cO}.
\]
For \((\cO,\cL)\in\LSG\) and \((\cO^\circ,\cL^\circ)\in\LSGO\) with \(\cO \overset{\mathrm{rel}}{\sim} \cO^\circ\), set
\[
m(\cL,\cL^\circ)
:= \dim \Hom_{A_{\cO^\circ}}\!\bigl(\cL|_{A^\circ_{\cO^\circ}},\,\cL^\circ\bigr)
 = \dim \Hom_{A_{\cO}}\!\bigl(\cL,\,\Ind_{A^\circ_{\cO^\circ}}^{A_{\cO}} \cL^\circ\bigr).
\]
The following lemma is easy to verify.
\begin{lemma}\label{Perv G and Gcirc}
\leavevmode
\begin{enumerate}
    \item For each \((\cO,\cL)\in\LSG\),
    \[
       \For_{G^\circ}^G\bigl(\IC(\cO,\cL)\bigr)
       \cong \bigoplus_{\substack{(\cO^\circ,\cL^\circ)\in\LSGO\\ \cO \overset{\mathrm{rel}}{\sim} \cO^\circ}}
          m(\cL,\cL^\circ)\,\IC(\cO^\circ,\cL^\circ).
    \]
    \item For each \((\cO^\circ,\cL^\circ)\in\LSGO\),
    \[
       \Av_{G^\circ}^G\bigl(\IC(\cO^\circ,\cL^\circ)\bigr)
       \cong \bigoplus_{\substack{(\cO,\cL)\in\LSG\\ \cO \overset{\mathrm{rel}}{\sim} \cO^\circ}}
          m(\cL,\cL^\circ)\,\IC(\cO,\cL).
    \]
\end{enumerate}
\end{lemma}
Now let $\star\in \{\emptyset, \circ\}$.
Let 
\[
\tcN^\star \coloneqq  
G^\star \times^B \frak u 
\]
along with the Springer resolution $\pi^\star :\tcN^\star \rightarrow \cN,\quad (g, u) \mapsto \Ad(g) u$. We also let 
\[
\tcG^\star \coloneqq  
G^\star \times^B \frak b
\] 
along with the \emph{Grothendieck-Springer simultaneous resolution} $\Pi^\star:\tcG^\star \rightarrow \frak g,\quad (g, b)\mapsto \Ad(g) b$.   Let $\frak g_{\rs}\subseteq \frak g$ be the set of regular semi-simple element of $\frak g$ and let $\widetilde{\frak g}^\star_{\rs}=\Pi^{-1}(\frak g_{\rs})$ and $\Pi^\star_{\rs}=\Pi|_{\widetilde{\frak g}^\star_{\rs}}$. The key players are summarized in the following commutative diagram:
\[
\begin{tikzcd}
  \widetilde{\frak g}^\star_{\rs} \ar[r] \ar[d,"\Pi^\star_{\rs}"] & \tcG^\star  \ar[d,"\Pi^\star"]  & \ar[l] \tcN^\star \ar[d,"\pi^\star"] \\
  \frak g_{\rs} \ar[r] & \frak g & \cN  \ar[l] 
\end{tikzcd}
\]

\begin{lemma}\label{Grs}
\begin{enumerate}
    \item $\widetilde{\frak g}^{\star}_{\rs}$ carries commutating left $G^\star$-action and right $W^\star$-action. The map  $\Pi^\star_{\rs}$ is $G^\star$-equivariant and is a $W^\star$-torsor. 
    \item There are a canonical $G\times W^\circ$-equivariant isomorphism
    \begin{equation*}
        \widetilde{\frak g}_{\rs}= G\times^{G^\circ} \widetilde{\frak g}^\circ_{\rs} 
    \end{equation*}
    and a canonical $G^\circ\times W$-equivariant isomorphism
    \begin{equation*}
        \widetilde{\frak g}_{\rs} \cong \widetilde{\frak g}^\circ_{\rs}\times^{W^\circ} W. 
    \end{equation*}
\end{enumerate}
\end{lemma}
\begin{proof}
The proof of (1) follows from the same argument in \cite[Lemma 8.2.3]{MR4337423} with a little modification. For the convenience of the reader, we provide a proof here. 

Let \(\mathfrak{b}_{\rs} = \mathfrak{b} \cap \mathfrak{g}_{\rs}\) and
\(\mathfrak{t}_{\rs} = \mathfrak{t} \cap \mathfrak{g}_{\rs}\).
We then have \(\widetilde{\mathfrak{g}}^\star_{\rs} = G^\star \times^{B} \mathfrak{b}_{\mathrm{rs}}\).
The inclusion map \(\mathfrak{t}_{\mathrm{rs}} \hookrightarrow \mathfrak{b}_{\mathrm{rs}}\)
induces a map
\begin{equation}
  G^\star \times^{T} \mathfrak{t}_{\mathrm{rs}}
  \longrightarrow
  G^\star  \times^{B} \mathfrak{b}_{\mathrm{rs}}
  = \widetilde{\mathfrak{g}}^\star_{\mathrm{rs}}.
  \tag{8.2.2}\label{eq:8.2.2}
\end{equation}
It is straightforward to check that this is a bijection on $\bC$-points and induces an isomorphism on tangent spaces, hence an isomorphism. 

Define a right action of \(W^\star\) on \(G^\star \times^{T} \mathfrak{t}_{\mathrm{rs}}\) as follows: for
\(w \in W^\star \), we set
\[
(g,x)\cdot w = (gw^{-1},\operatorname{Ad}(\dot{w})(x)),
\]
where $\dot w$ is any lift of $w$ in $N_{G^\star}(T)$. This action commute with the left action of $G^\star$ on \(G^\star \times^{T} \mathfrak{t}_{\mathrm{rs}}\). One also checks that this action is free, and \(\Pi^{\star}_{\mathrm{rs}}\) factors through the quotient to
induce a map
\begin{equation}
  \ov{p}:(G^\star \times^{T} \mathfrak{t}_{\mathrm{rs}})/W^\star
  \longrightarrow
  \mathfrak{g}_{\mathrm{rs}}.
  \tag{8.2.3}\label{eq:8.2.3}
\end{equation}

To finish the proof of (1), we must show that $\ov p$ is an
isomorphism of varieties.  
The map $p: G^\star\times^{T} \mathfrak{t}_{\mathrm{rs}}\to \frg_{\rs}$ is \'etale by checking on tangent spaces, hence $p$ is \'etale, and so is $\ov p$. Thus it is
enough to show that $\ov p$ is a bijection on $\bC$-points.


Note that $(G^\star \times^{T} \mathfrak{t}_{\mathrm{rs}})/W^\star=G^\star \times^{N_{G^\star}(T)} \mathfrak{t}_{\mathrm{rs}}$ , which classifies pairs $(\frh,x)$ where $x\in \frg_{\rs}$ and $\frh$ is a Cartan subalgebra of $\frg$ containing $x$. Since every 
regular semisimple $x$ lies in a unique Cartan subalgebra (namely, its centralizer), $\ov p$ is a bijection on $\bC$-points. 
This finished the proof of (1).  
The proof of (2) follows from tracking $G^\star$ and $W^\star$ action the proof in (1). 

\end{proof}

The maps $\pi^{\star}$ and $\Pi^{\star}$ are $G^{\star}$-equivariant and semi-small (see \cite[Lemma 8.1.6, Lemma 8.2.5]{MR4337423}).  Let 
\[
\cS^\star_{\cN} = \pi^\star_* \underline{\bC}_{\tcN^\star} [\dim \tcN^\star]\in \Perv_{G^\star}(\cN)\quad \cS^\star_{\frak g}\coloneqq  \Pi^\star_* \underline{\bC}_{\tcG^\star} [\dim \tcG^\star]\in \Perv_{G^\star}(\frak g)
\]
be the Springer sheaf and \emph{Grothendieck-Springer sheaf} for $G^\star$. We fix a $G^\star$-invariant bilinear form on $\frak g$ and let 
\[
\Four_{\frak g}: D^b_{G^\star}(\frak g) \rightarrow  D^b_{G^\star}(\frak g) 
\]
be the \emph{equivariant Fourier-Deligne transform}, which is a self equivalence of categories. 
\begin{lemma}\label{Fourier}
    Let $i_{\cN}: \cN\hookrightarrow \frak g$ be the inclusion map. There is a canonical isomorphism 
    \begin{equation}\label{Springer 2}
          \Four_{\frak g} (\cS^\star_{\frak g})\cong i_{\cN,*} \cS^\star_{\cN}.
    \end{equation}
\end{lemma}
\begin{proof}
See \cite[Theorem 8.2.8]{MR4337423} for the case when $\star=\circ$. The other case follows by the same proof. 
\end{proof}
\begin{lemma}\label{Springer sheave main}
\begin{enumerate}
    \item We have a canonical  isomorphism of $\bC$-algebras
    \begin{equation*}\label{Springer 1}
          \End_{\Perv_{G^\star}(\frak g)}(\cS^\star_{\frak g}) \cong \bC[W^\star]
    \end{equation*}
    and 
    \[
    \For_{G^\circ}^G \cS_{\frak g} \cong \Ind_{W^\circ}^W \cS^\circ_{\frak g}:=\cS^\circ_{\frg}\otimes_{\bC[W^\circ]} \bC[W] \in \Perv_{G^\circ\times W}(\frak g).
    \]
    \[
\Av_{G^\circ}^G \cS^\circ_{\frak g} \cong \Res_{W^\circ}^{W} \cS_{\frak g} \in \Perv_{G\times W^\circ}(\frak g).
    \]
    \item We have   
    \begin{equation}\label{Springer 3}
             \End_{\Perv_{G^\star }(\cN)}(\cS^\star_{\cN}) \cong \bC[W^\star] 
    \end{equation}
    and 
    \begin{equation}\label{Springer 4}
    \begin{split}
        \For_{G^\circ}^G \cS_{\cN} &\cong \Ind_{W^\circ}^W \cS^\circ_{\cN}=\cS^\circ_{\cN}\otimes_{\bC[W^\circ]} \bC[W] \in \Perv_{G^\circ\times W}(\cN).   \\
        \Av_{G^\circ}^G \cS^\circ_{\cN} &\cong \Res_{W^\circ}^{W} \cS_{\cN} \in \Perv_{G\times W^\circ}(\cN).
    \end{split}
    \end{equation}
\end{enumerate}
\end{lemma}
\begin{proof}
    (1) follows from the \Cref{Grs} and the fact that the map $\Pi^\star$ is small and is a $W^\star$-torsor over $\frak g_{\rs}\subseteq \frak g$ (see \cite[Lemma 8.2.5]{MR4337423}) . (2) follows from (1) and \Cref{Fourier}. 
\end{proof}

By \cite[Theorem 6.5]{MR0732546} we know that $\Perv_{G^\circ}(\cN)$ is semi-simple (hence also $\Perv_{G}(\cN)$). We define the \emph{Springer functors} using \Cref{Springer 3}:
\begin{equation}\label{Springer functor}
\begin{split}
   \Spr_{G^\star} := \cS^\star_{\cN} \otimes_{\bC[W]} - & :  \Rep(W^\star ) \longrightarrow \Perv_{G^\star}(\cN), \quad \star=\{\emptyset, \circ\}
\end{split}
\end{equation}

\begin{thm}\label{thm:Springer compatibility}
\leavevmode
\begin{enumerate}
    \item The Springer functors $\Spr_{G^\star}$ embeds $\Rep(W^\star)$ as a full subcategory of $\Perv_{G^\star}(\cN)$. In particular, they induce injective maps
    \begin{equation}\label{gener Springer map}
        \Spr_{G^\star}: \Irr(W^\star) \hookrightarrow \LSGS 
    \end{equation}
    \item The functors $\Spr_G$ and $\Spr_{G^\circ}$ satisfy the following compatibilities:
    \begin{enumerate}
        \item They intertwine the restriction functor $\Res_{W^\circ}^{W}$ and the forgetful functor $\For_{G^\circ}^G$:
        \begin{equation}\label{Springer compatibility I}
            \begin{tikzcd}
                \Rep(W) \ar[r, "\Spr_G"] \ar[d, "\Res_{W^\circ}^{W}"'] & \Perv_G(\cN) \ar[d, "\For_{G^\circ}^G"] \\
                \Rep(W^\circ) \ar[r, "\Spr_{G^\circ}"] & \Perv_{G^\circ}(\cN)
            \end{tikzcd}
        \end{equation}

        \item They intertwine the induction functor $\Ind_{W^\circ}^W$ and the averaging functor $\Av_{G^\circ}^G$:
        \begin{equation}\label{Springer compatibility II}
            \begin{tikzcd}
                \Rep(W^\circ) \ar[r, "\Spr_{G^\circ}"] \ar[d, "\Ind_{W^\circ}^W"'] & \Perv_{G^\circ}(\cN) \ar[d, "\Av_{G^\circ}^G"] \\
                \Rep(W) \ar[r, "\Spr_G"] & \Perv_G(\cN)
            \end{tikzcd}
        \end{equation}
    \end{enumerate}
\end{enumerate}
\end{thm}
\begin{proof}
(1) is a consequence of \Cref{Springer 3}. (2) is a consequence of \Cref{Springer 4} (See \cite[Proposition~A.2]{clausen2008springer}).
\end{proof}

Let $(\cO, \cL) \in \LSG$ be in the image of $\Spr_G$, and define $E_{\cO, \cL} := \Spr_G^{-1}(\cO, \cL) \in \Irr(W)$. Similarly, for $(\cO^\circ, \cL^\circ) \in \LSGO$ in the image of $\Spr_{G^\circ}$, define $E^\circ_{\cO^\circ, \cL^\circ} := \Spr_{G^\circ}^{-1}(\cO^\circ, \cL^\circ) \in \Irr(W^\circ)$.

\begin{cor}\label{cor:Springer compatibility decomposition}
\leavevmode
\begin{enumerate}
    \item For each $(\cO, \cL) \in \LSG$ in the image of $\Spr_G$, we have an isomorphism in $\Rep(W)$
    \[
    \Res_{W^\circ}^{W} E_{\cO, \cL} = \bigoplus_{\substack{(\cO^\circ, \cL^\circ) \in \LSGO \\ \cO \overset{\mathrm{rel}}{\sim} \cO^\circ}} m(\cL, \cL^\circ) \cdot E^\circ_{\cO^\circ, \cL^\circ}.
    \]

    \item For each $(\cO^\circ, \cL^\circ) \in \LSGO$ in the image of $\Spr_{G^\circ}$, we have an isomorphism in $\Rep(W^\circ)$
    \[
    \Ind_{W^\circ}^{W} E^\circ_{\cO^\circ, \cL^\circ} = \bigoplus_{\substack{(\cO, \cL) \in \LSG \\ \cO \overset{\mathrm{rel}}{\sim} \cO^\circ}} m(\cL, \cL^\circ) \cdot E_{\cO, \cL}.
    \]
\end{enumerate}
\end{cor}
\begin{proof}
    This follows immediately from \Cref{Perv G and Gcirc} and \Cref{thm:Springer compatibility}.
\end{proof}
We now make the discussion explicit for \(G=\rO(V)\) with \(\dim V=2m\), so that \(G^\circ=\SO(V)\).
In this case the Weyl groups are
\[
W \cong \bS_m \rtimes (\bZ/2\bZ)^m \quad\text{(type \(BC_m\))},\qquad
W^\circ \cong \bS_m \rtimes (\bZ/2\bZ)^{m-1} \quad\text{(type \(D_m\))}.
\]

\begin{defn}
\begin{enumerate}
\item An orthogonal partition \(\lambda\in\cP_+(2m)\) (see \Cref{defn orthogonal symplectic partition})
is called \emph{very even} if \(\odd_\lambda=\varnothing\); equivalently, all parts of \(\lambda\) are even.
Write \(\cP_+^{\even}(2m)\subseteq\cP_+(2m)\) for the subset of very even partitions.
\item For
\[
\lambda=\bigl[\underbrace{t_1,\ldots,t_1}_{z_1},\ldots,\underbrace{t_\ell,\ldots,t_\ell}_{z_\ell}\bigr]\in\cP_+(2m),
\quad t_1>\cdots>t_\ell>0,
\]
recall
\[
\cS_\lambda=\bigoplus_{t\in\odd_\lambda}\bZ/2\bZ\cdot a_t.
\]
Define the character
\[
\det_\lambda:\cS_\lambda\to\bZ/2\bZ,\qquad a_t\longmapsto 1,
\]
and set \(\cS_\lambda^+:=\ker(\det_\lambda)\).
\end{enumerate}
\end{defn}

\begin{lemma}\label{Even orthogonal and special even orbit}
\begin{enumerate}
\item \(G\)-orbits in the nilpotent cone \(\cN\) are parametrized by \(\cP_+(2m)\); denote by \(\cO_\lambda\) the orbit attached to \(\lambda\).
There is a canonical isomorphism of component groups \(A_{\cO_\lambda}\cong\cS_\lambda\).

\item If \(\lambda\in \cP_+^{\even}(2m)\), the \(G\)-orbit \(\cO_\lambda\) splits into two
\(G^\circ\)-orbits, denoted \(\cO_\lambda^{I}\) and \(\cO_\lambda^{II}\).
If \(\lambda\notin \cP_+^{\even}(2m)\), then \(\cO_\lambda\) remains a single \(G^\circ\)-orbit, which we still denote by \(\cO_\lambda\). 

\item 
If \(\lambda\in\cP_+^{\even}(2m)\), we have \(A^\circ_{\cO_{\lambda}^{*}}\cong\cS_\lambda^+=1\) for \(*\in \{I,II\}\). If \(\lambda\not \in\cP_+^{\even}(2m)\), the natural map on component groups
\(A^\circ_{\cO_{\lambda}}\hookrightarrow A_{\cO_\lambda}\) corresponds to the inclusion
\(\cS_\lambda^+\hookrightarrow\cS_\lambda\). 



\end{enumerate}
\end{lemma}
\begin{proof}
    See \cite[\S 6.1]{MR1251060}.
\end{proof}

\begin{remark}
For $\chi\in \Irr(\cS_\l)$ with the corresponding local system $\cL_\chi$ on $\cO_\l$, we denote $E_{\cO,\cL_\chi}$ simply by $E_{\cO,\chi}$. When $\cS_\l=1$, we denote $E_{\cO_\l, 1}$ simply by $E_{\cO_\l}$. Similarly we write $E^\circ_{\cO_\l,\chi}$ and $E^\circ_{\cO^I_\l}$, etc. 
\end{remark}

\begin{thm}
For \(\lambda\in\cP_+(2m)\) and \(\chi\in\Irr(\cS_\lambda)\), the restriction of Springer representations satisfies
\[
\begin{aligned}
E_{\cO_\lambda,\chi}\big|_{W^\circ} &\cong
E^\circ_{\cO_\lambda,\ \chi|_{\cS_\lambda^+}}
&&\text{if }\lambda\notin\cP_+^{\even}(2m),\\
E_{\cO_\lambda}\big|_{W^\circ} &\cong
E^\circ_{\cO_\lambda^{I}}\ \oplus\ E^\circ_{\cO_\lambda^{II}}
&&\text{if }\lambda\in\cP_+^{\even}(2m).
\end{aligned}
\]
\end{thm}
\begin{proof}
Combine \Cref{cor:Springer compatibility decomposition} with \Cref{Even orthogonal and special even orbit}. Note that \(\cS_\lambda\) is abelian, the restriction \(\chi|_{\cS_\lambda^+}\) is irreducible for every \(\chi\in\Irr(\cS_\lambda)\). 
\end{proof}

\section{Tits' deformation and Lusztig's homomorphism}\label{TitsLusztigsec}

Let $F$ be an algebraically closed field of characteristic zero. Suppose $A$ is a (commutative) domain over $F$, and $W$ is a finite group.  Let $H$ be a finite locally free (not necessarily commutative) $A$-algebra with an $A$-algebra homomorphism $\phi: H \to A[W]$.

For any homomorphism $x: A\to F$, let $H_x=H\ot_{A,x}F$, and $\phi_x:H_x\to F[W]$ be the base change of $\phi$. For any $A$-module $N$, let $N_x=N\ot_{A,x}F$ be the specialized $H_x$-module.

\begin{lemma}\label{TitsLusztig} Let $N$ be an $H$-module that is finite locally free as an $A$-module. Let $x,y: A\to F$ be two homomorphisms such that $\phi_x$ and $\phi_y$ are isomorphisms. Then the specializations $N_x$ and $N_y$ are isomorphic as $F[W]$-modules, via the isomorphisms $\phi_x$ and $\phi_y$.    
\end{lemma}
\begin{proof} By localizing $A$ we may assume $H$ is free over $A$. Since there are $x,y: A\to F$ such that $\phi_x$ and $\phi_y$ are isomorphisms, we know that there exists $U\subset \Spec A$, a Zariski dense open affine subset, where $\phi$ is an isomorphism (it is the non-vanishing locus of the determinant of $\phi$ under a basis of $H$ over $A$). Replacing $\Spec A$ by  $U$, we may assume that $\phi$ is an isomorphism. Therefore we may assume $H=A[W]$.

We need to show that,
for every irreducible character $\chi$ of $W$ over $F$,  the multiplicity  of $\chi$ in  $N_x$ is  independent of the $F$-point $x$ of $\Spec A$.

Let $e_\chi\in F[W]\subset A[W]$ be the corresponding idempotent. Let $\tr(e_\chi|N)\in A$ be the trace of $ e_\chi\in A[W]$ on $N$. Let $K$ be the fraction field of $A$, and $\ov K$ be an algebraic closure of $K$. We have $$\tr(e_\chi|N)/\deg\chi =\tr(e_\chi|N_{\ov K})/\deg \chi \in \ol K.$$
The latter is the multiplicity of $\chi_{\ol K}$ in $N_{\ol K}=N\ot_A\ov K$, and thus in $\bZ_{\ge0}$. 
So the restriction of $\tr(e_\chi|N)$ to any $F$-point $x$ of $\Spec A$,  which is the multiplicity of $\chi $ in $N_x$,
is independent of $x$. The lemma follows.
\end{proof}

\section{An explicit description of $\cRV$ and $\cRL$}\label{subsection explicit relavant}
In this appendix, we prove \Cref{prop GL orbits} and \Cref{orthosymplectic orbits}, which provide an explicit combinatorial description of $\cRV$ and $\cRL$ defined in \Cref{def relevant GL} and \Cref{def relevant orthsymplectic}.

\subsection{Nilpotent orbits of classical groups and their component groups}
We begin by reviewing some preliminaries on nilpotent orbits, following \cite{MR1251060}. Let $G$ be a connected reductive group over $\mathbb{C}$, $\mathfrak{g}$ be its Lie algebra and  $\mathcal{N}$ be the set of nilpotent elements in $\mathfrak{g}$. Define
\[
\mathcal{A}_{\text{triple}}(\frg) = \{(h, e, f) \mid h, e, f \in \mathfrak{g}, [h, e] = 2e, [h, f] = -2f, [e, f] = h\},
\]
the set of $\mathfrak{sl}_2$-triples in $\mathfrak{g}$. By the Jacobson--Morozov theorem (see \cite[\S 3.3]{MR1251060}), the map sending an $\mathfrak{sl}_2$-triple $\Gamma = (h, e, f)$ to its nilpositive element $e$ induces a bijection between the set of $G$-orbits in $\mathcal{A}_{\text{triple}}(\frg)$ and the set of $G$-orbits in $\cN$.

For $\Gamma = (h, e, f) \in \mathcal{A}_{\text{triple}}(\frg)$, let $G_\Gamma$ and $G_e$ denote the stabilizers of $\Gamma$ and $e$ in $G$, respectively. Their component groups are $A_\Gamma = \pi_0(G_\Gamma)$ and $A_e = \pi_0(G_e)$. The inclusion $G_\Gamma \hookrightarrow G_e$ induces an isomorphism $A_\Gamma \cong A_e$ \cite[Lemma 3.7.3]{MR1251060}, allowing us to identify $A_e$ with $A_\Gamma$.

\subsubsection{General linear group}
We now turn to the case of general linear groups. Let $L$ be a finite-dimensional vector space, and let $G = \GL(L)$ be the general linear group. Given an $\mathfrak{sl}_2$-triple $\Gamma = \{h, e, f\}$ in $\mathfrak{g}$, we obtain an $\SL_2(\bC)$-representation on $L$ with decomposition
\begin{equation}\label{decompso L}
    L = \bigoplus_{k=1}^{l} S_k \otimes M_k,
\end{equation}
where $S_{k}\coloneqq \Sym^{k-1}(\bC^2)$ is the $k-1$-th symmetric power of the standard representation and $M_k$ is the multiplicity space. One easily checks that this gives a bijection between the set of $G$-orbits in $\mathcal{A}_{\text{triple}}(\frg)$ and the isomorphism classes of $\SL_2(\bC)$-representations with underlining space $L$.

By \Cref{decompso L}, we can associate to each $\Gamma\in  \mathcal A_{\text{triple}}(\frg)$ a partition: 
\[
\lambda = [\underbrace{l,\cdots, l}_{m_l}, \cdots,\underbrace{1,\cdots, 1}_{m_1}]\in \cP(\dim L)
\]
where $m_k=\dim M^k$, and $\lambda$ determines the isomorphism class of $L$ as $\SL_2(\bC)$-representation. This implies that the nilpotent orbits of $G$ are in bijection with $\cP(\dim L)$. 

We also have 
\[
G_{\Gamma}\cong \prod_{k=1}^{l} \GL(M_k) \quad \mbox{and}\quad A_e\cong \pi_0(G_\Gamma)\cong 1.
\]

\subsubsection{Orthogonal and symplectic groups}
Now consider even orthogonal and symplectic groups. Let $(V_1, \langle , \rangle_{V_1})$ be a $2m$-dimensional orthogonal space and $(V_2, \langle , \rangle_{V_2})$ a $2n$-dimensional symplectic space. Let $G_i$ ($i = 1, 2$) be the corresponding orthogonal or symplectic group with Lie algebra $\frg_i$. Note that $G_1$ is not connected; let $G_1^\circ = \SO(V_1) \subset G_1$ be its neutral component. 

The Jacobson--Morozov theorem remains valid for $G_1$ with the following modification: the set of $G_1$-orbits of $\mathfrak{sl}_2$-triples in $\mathfrak{g}_1$ corresponds bijectively to the set of $G_1$-orbits in $\cN_1$.

Given an $\mathfrak{sl}_2$-triple $\Gamma = \{h, e, f\}$ in $\mathfrak{g}_i$, we obtain an $\SL_2(\bC)$-representation on $V_i$ with decomposition
\[
V_i = \bigoplus_{k=1}^{l} S_k \otimes M_k^i,
\]
where  
$M_k^i$ is the multiplicity space of $S_k$. Recall that $S_k=\Sym^{k-1}(\bC^2)$ is orthogonal (resp. symplectic) if $k$ is odd (resp. even). The  $\SL_2(\bC)$-invariant bilinear form on $S_k$ is unique up to scalar, and we fix one such form $\langle \cdot, \cdot \rangle_{S_k}$.

The restriction of $\langle \cdot, \cdot \rangle_{V_i}$ to each isotypic component $S_k \otimes M_k^i$ remains non-degenerate, inducing a non-degenerate bilinear form $\langle \cdot, \cdot \rangle_{M_k^i}$ on $M_k^i$ such that
\begin{equation}\label{decomposition Vi}
(V_i, \langle \cdot, \cdot \rangle_{V_i}) \cong \bigoplus_{k=1}^l  \left(S_k, \langle \cdot, \cdot \rangle_{S_k}\right) \otimes \left(M_k^i, \langle \cdot, \cdot \rangle_{M_k^i}\right)
\end{equation}
as formed spaces. In particular, $\langle \cdot, \cdot \rangle_{M_k^i}$ is skew-symmetric (and hence $\dim M_k^i$ is even) if $i = 1$ and $k$ is even, or if $i = 2$ and $k$ is odd. Otherwise, $\langle \cdot, \cdot \rangle_{M_k^i}$ is symmetric.


By \Cref{decomposition Vi}, we can associate to each $\Gamma\in  \mathcal A_{\text{triple}}(\frg_i)$ an orthogonal or symplectic partition (cf. \Cref{defn orthogonal symplectic partition}): 
\[
\lambda_i = [\underbrace{l,\cdots, l}_{m_l^i}, \cdots, \underbrace{1,\cdots, 1}_{m_1^i}]\in \cP_{(-1)^{i+1}}(\dim V_i)
\]
where $m_k^i = \dim M_k^i$, and $\lambda_i$ determines the isomorphism class of $(V_i, \langle \cdot, \cdot \rangle_{V_i})$ as $\SL_2(\bC)$-representation, which is in bijection with the $G_i$-orbits in $\mathcal A_{\text{triple}}(\frg_i)$. This implies that the nilpotent orbits of $G_i$ are in bijection with $\cP_{(-1)^{i+1}}(\dim V_i)$. 

Let $I_i^+$ be the set consisting of $k$ in \Cref{decomposition Vi} such that $\langle \cdot, \cdot \rangle_{M^k_i}$ is nonzero and orthogonal, and $I_i^+$ the set consisting of $k$ such that $\langle \cdot, \cdot \rangle_{M^k_i}$ is nonzero and symplectic. Then
\[
(G_i)_\Gamma \cong \prod_{k \in I_i^+} \rO(M^k_i) \times \prod_{k \in I_i^-} \Sp(M^k_i)\quad \mbox{and}\quad A_e \cong \pi_0((G_i)_\Gamma) = \prod_{k \in I_i^+} (\bZ/2\bZ) a_k.
\]

\subsection{The GL-orbits of nilpotent pairs}\label{Sec GL nilpotent}
This section is devoted to the proof of \Cref{prop GL orbits}.

Fix an $\mathfrak{sl}_2$-triple $(h, e, f)$ in $\mathfrak{sl}_2(\mathbb{C})$, and define $\widetilde{\mathrm{SL}}_2(\mathbb{C}) := \SL_2(\mathbb{C}) \rtimes \mathbb{Z}/2\mathbb{Z} = \mathrm{SL}_2(\mathbb{C}) \rtimes \langle \sigma \rangle$, with $\sigma$ acts on $\Lie(\mathrm{SL}_2(\mathbb{C}))=\mathfrak{sl}_2(\mathbb{C})$ by
\[ \sigma(h) = h, \quad \sigma(e) = -e, \quad \sigma(f) = -f. \]

Let $\{v_{-n}, v_{-n+2}, \dots, v_{n-2}, v_n\}$ be a basis for $S_{n+1}= \Sym^n(\mathbb{C}^2)$ satisfying:
\begin{itemize}
  \item $h v_i = i v_i$ for $-n \leq i \leq n$;
  \item $e v_i = v_{i+2}$ for $-n \leq i \leq n - 2$ and $e v_n = 0$.
\end{itemize}

The following lemma is straightforward to verify.

\begin{lemma}\label{lem irr tilde SL2}
The irreducible representations of $\widetilde{\mathrm{SL}}_2(\mathbb{C})$ fall into the following four types:
\begin{itemize}
  \item $S_{k+1,k} := \Sym^{2k}(\mathbb{C}^2)$ with $\sigma$ acting as $+1$ on $\langle v_{-2k}, v_{-2k+4}, \dots, v_{2k} \rangle$ and $-1$ on \\$\langle v_{-2k+2}, v_{-2k+6}, \dots, v_{2k-2} \rangle$;
  \item $S_{k,k+1} := \Sym^{2k}(\mathbb{C}^2)$ with $\sigma$ acting as $-1$ on $\langle v_{-2k}, v_{-2k+4}, \dots, v_{2k} \rangle$ and $+1$ on\\ $\langle v_{-2k+2}, v_{-2k+6}, \dots, v_{2k-2} \rangle$;
  \item $S^+_{k,k} := \Sym^{2k-1}(\mathbb{C}^2)$ with $\sigma$ acting as $+1$ on $\langle v_{-2k+1}, v_{-2k+5}, \dots, v_{2k-3} \rangle$ and $-1$ on\\ $\langle v_{-2k+3}, v_{-2k+7}, \dots, v_{2k-1} \rangle$;
  \item $S^-_{k,k} := \Sym^{2k-1}(\mathbb{C}^2)$ with $\sigma$ acting as $-1$ on $\langle v_{-2k+1}, v_{-2k+5}, \dots, v_{2k-3} \rangle$ and $+1$ on\\ $\langle v_{-2k+3}, v_{-2k+7}, \dots, v_{2k-1} \rangle$.
\end{itemize}
\end{lemma}

\begin{remark}
In \cite[\S 6]{MR694606}, these four types are described via $ab$-strings:
\begin{equation}\label{ab string}
    \underbrace{abab\cdots aba}_{2k+1}, \quad \underbrace{bab\cdots bab}_{2k+1}, \quad \underbrace{abab\cdots ab}_{2k}, \quad \underbrace{bab\cdots ba}_{2k}.
\end{equation}
We denote them as $ (k+1,k), (k,k+1), (k,k)^+$ and $ (k,k)^- $.
\end{remark}

\begin{proof}[Proof of \Cref{prop GL orbits}]
We retain the notation from \Cref{section type II}. Let $\ov L =  L_1 \oplus  L_2$ and define $\ov K = \GL( L_1 \oplus  L_2)$ with Lie algebra $\ov{\mathfrak{k}} = \mathfrak{gl}( L_1 \oplus  L_2)$. Define an inner involution $\ov \theta := \mathrm{Int}(\ov J)$ on $\ov K$ where
\[ \ov J := \begin{pmatrix} \Id_{ L_1} & 0 \\ 0 & -\Id_{ L_2} \end{pmatrix}. \]
This yields a $\mathbb{Z}/2\mathbb{Z}$-grading $\ov{\mathfrak{k}} = \ov{\mathfrak{k}}^0 \oplus \ov{\mathfrak{k}}^1$ with
\begin{equation}\label{identification ovK}
    \ov{\mathfrak{k}}^0 = \mathfrak{gl}( L_1) \oplus \mathfrak{gl}( L_2), \quad \ov{\mathfrak{k}}^1 = \Hom( L_1,  L_2) \oplus \Hom( L_2,  L_1) = \ov \bL. 
\end{equation} 

Let $\mathcal{N} \subseteq \ov{\mathfrak{k}}$ be the nilpotent cone of $\ov K$. Then we have $\mathcal{N}_{\ov \bL} = \mathcal{N} \cap \ov{\mathfrak{k}}^1$ under the identification \Cref{identification ovK}. Define the set of $\ov \theta$-stable $\mathfrak{sl}_2$-triples in $\ov{\mathfrak{k}}$ as
\[ \mathcal{A}_{\ov \theta, \text{triple}} (\ov\frk)= \{ (H, X, Y) \mid H \in \ov{\mathfrak{k}}^0, \; X, Y \in \ov{\mathfrak{k}}^1, \; [H, X] = 2X, [H, Y] = -2Y, [X, Y] = H \}. \]
This set carries a natural $\ov K^{\ov \theta} = \ov G_1 \times \ov G_2$-action. The graded Jacobson--Morosov theorem \cite[\S2.3]{MR3697026} gives a bijection between $\ov G_1 \times \ov G_2$-orbits in $\mathcal{N}_{\ov \bL}$ and those in $\mathcal{A}_{\ov \theta, \text{triple}}(\ov\frk)$, by mapping a triple $(H, X, Y)$ to its nilpositive element $X$.

Each triple $\Gamma = (H, X, Y)$ defines a $\widetilde{\mathrm{SL}}_2(\mathbb{C})$-representation on $\ov L =  L_1 \oplus  L_2$, with $h \mapsto H, e \mapsto X, f \mapsto Y, \sigma \mapsto \ov J$. One checks that this gives a bijection between $\ov G_1 \times \ov G_2$-orbits in $\mathcal{A}_{\ov \theta, \text{triple}}(\ov\frk)$ and isomorphism classes of $\widetilde{\mathrm{SL}}_2(\mathbb{C})$-representations on $\ov L$ such that $\s$ acts as $\ov J$.

Using \Cref{lem irr tilde SL2}, any $\Gamma\in \mathcal{A}_{\ov \theta, \text{triple}}(\ov\frk)$ gives a decomposition
\begin{equation}\label{Eq decomposition ovL}
\ov L \cong \bigoplus_{k\ge 0} \left( M_{k+1,k} \otimes S_{k+1,k} \oplus M_{k,k+1} \otimes S_{k,k+1} \right) \oplus \bigoplus_{k\ge 0}\left( M^+_{k,k} \otimes S^+_{k,k} \oplus M^+_{k,k} \otimes S^-_{k,k} \right),
\end{equation}
where $M^{\ep}_{*,*}$ are the multiplicity spaces, whose dimensions determine a decorated bipartition 
\begin{equation}\label{Gamma type II}
\begin{split}
   &\gamma=\\
   &\{ \underbrace{(k+1,k),\cdots,(k+1,k)}_{\dim M_{k+1,k}}, \underbrace{(k,k+1),\cdots,(k,k+1)}_{\dim M_{k,k+1}},\underbrace{(k,k)^+,\cdots,(k,k)^+}_{\dim M^+_{k,k}}, \underbrace{(k,k)^-,\cdots, (k,k)^-}_{\dim M^-_{k,k}}|k\in \bZ_{\geq 0}\}\\
   &\in \BP(m,n).
  \end{split}  
\end{equation}

Let $\cO_\gamma$ be the $\ov G_1 \times \ov G_2$-orbit in $\cN_{\ov \bL}$ corresponding to $\gamma$ as in \Cref{Gamma type II}. By \cite[\S5.3]{MR549399},
\begin{equation}\label{eq dimesion moment map type II}
\dim \cO_\gamma = \tfrac{1}{2}(\dim \mu_1(\cO_\gamma) + \dim \mu_2(\cO_\gamma) + \dim \ov \bL) - \mathcal{D}_\gamma,
\end{equation}
where 
\begin{equation}\label{Dgamma}
    \mathcal{D}_\gamma := \sum_k \dim M_{k+1,k} \cdot \dim M_{k,k+1}.
\end{equation} 
Combining with \Cref{def relevant GL}, \Cref{defn ab diagram}, we see $\cO_\gamma \in \cR_{\ov \bL}$ iff $\cD_\gamma=0$, which means that $M_{k,k+1}$ and $M_{k+1,k}$ are not both nonzero for each $k\ge0$, i.e., $\gamma \in \RBP(m,n)$.

Let $\LS^{\ep}_{*,*}, \RS^{\ep}_{*,*}\subset S^{\ep}_{*,*}$ be the $+1$ and $-1$ eigenspaces of $\sigma$ on $S^{\ep}_{*,*}$. The decomposition \Cref{Eq decomposition ovL} induce  decompositions 
\begin{equation}\label{Eq decomposition ov L1}
\begin{split}
     L_1= \bigoplus_{k\ge0} \left( M_{k+1,k}\otimes  \LS_{k+1,k}\oplus  M_{k,k+1} \otimes  \LS_{k,k+1}\right)\oplus  \bigoplus_{k\ge0} \left( M^+_{k,k}\otimes  \LS^+_{k,k} \oplus M^-_{k,k} \otimes   \LS^-_{k,k}\right)\\
     L_2= \bigoplus_{k\ge0} \left( M_{k+1,k}\otimes  \RS_{k+1,k}\oplus  M_{k,k+1} \otimes  \RS_{k,k+1}\right)\oplus  \bigoplus_{k\ge0} \left( M^+_{k,k}\otimes  \RS^+_{k,k} \oplus M^-_{k,k} \otimes   \RS^-_{k,k}\right),
\end{split}
\end{equation}
Write $X=(T_1,T_2)$ where $T_1\in \Hom( L_1, L_2)$ and $T_2\in \Hom( L_2, L_1)$. The nilpotent endomorphisms $\mu_1(X) = T_2 T_1$ of $ L_1$ and $\mu_2(X) = T_1 T_2$ of $ L_2$ preserve the above decompositions. 
Taking lengths of Jordan blocks of $\mu_i(X)$ we get partitions:
\[
\mu^\dagger_1(\gamma)= \left\{ \underbrace{k+1,\cdots,k+1 }_{\dim M_{k+1,k}}, \underbrace{k,\cdots,k }_{\dim M_{k,k+1}},\underbrace{k,\cdots,k}_{\dim M^+_{k,k}}, \underbrace{k,\cdots, k}_{\dim M^-_{k,k}}\middle|k\in \bZ_{\geq 0}\right\}\in \cP(m),
\] 
\[
\mu^\dagger_2(\gamma)= \left\{ \underbrace{k,\cdots,k }_{\dim M_{k+1,k}}, \underbrace{k+1,\cdots,k+1 }_{\dim M_{k,k+1}},\underbrace{k,\cdots,k}_{\dim M^+_{k,k}}, \underbrace{k,\cdots, k}_{\dim M^-_{k,k}}\middle|k\in \bZ_{\geq 0}\right\}\in \cP(n).
\]
Hence $\mu(X)$ maps to the pair $(\cO_1, \cO_2)$ with $\cO_1$ and $\cO_2$ the orbits corresponding to these partitions. 
\end{proof}

\subsection{Ortho-symplectic orbits and their component groups}\label{Sec Orth-symplectic}
This section is devoted to the proof of \Cref{orthosymplectic orbits}.

Fix a primitive fourth root of unity $\zeta\in \bC$. Define an automorphism $\theta$ of $\widetilde{\SL}_2(\mathbb{C})$ by
\[
\theta(h) = h, \quad \theta(e) = \zeta e, \quad \theta(f) = -\zeta f, \quad \theta(\sigma) = \sigma.
\]
\begin{defn}\label{theta twisted pairing}
    Let $L_1$ and $L_2$ be representations of $\widetilde{\SL}_2(\mathbb{C})$. A non-degenerate bilinear pairing
\[
\langle\cdot, \cdot\rangle_{L_1, L_2} \colon L_1 \otimes L_2 \to \mathbb{C}
\]
is said to be \emph{$\theta$-twisted invariant} if
\begin{equation}\label{form theta twisted}
\langle \theta(g)(v_1), g(v_2) \rangle_{L_1, L_2} = \langle v_1, v_2 \rangle_{L_1, L_2} \quad \text{for all } g \in \widetilde{\SL}_2(\mathbb{C}), \ v_1 \in L_1, \ v_2 \in L_2.
\end{equation}
When $L_1 = L_2 = L$, we simply the notation $\langle \cdot, \cdot \rangle_{L_1,L_2}$ by $\langle \cdot, \cdot \rangle_L$ and refer to it as a $\theta$-twisted invariant form on $L$. The form $\langle \cdot, \cdot \rangle_L$ is said to be 
\begin{enumerate}
    \item \emph{$\sigma$-sign} $+1$ if
\[
\langle v, w \rangle_L = \langle \sigma(w), v \rangle_L \quad \text{for all } v, w \in L,
\]
\item \emph{$\sigma$-sign} $-1$ if
\[
\langle v, w \rangle_L = -\langle \sigma(w), v \rangle_L \quad \text{for all } v, w \in L.
\]
\end{enumerate}
\end{defn}
\trivial[h]{
Such a form defines a $\widetilde{\SL}_2(\mathbb{C})$-equivariant isomorphism
\[
\begin{aligned}
L_1^\theta &\cong L_2^\vee, \\
v &\mapsto \langle v, \cdot \rangle,
\end{aligned}
\]
where $L_1^\theta$ denotes the $\theta$-twist of $L_1$.
}

Recall from \Cref{lem irr tilde SL2} that the irreducible representations of $\widetilde{\SL}_2(\mathbb{C})$ are of the following four types:
\[
S_{k+1,k}, \quad S_{k,k+1}, \quad S^+_{k,k}, \quad S^-_{k,k} \quad \text{for } k \geq 0.
\]

\begin{lemma}\label{form tilde sl2 representation}
\begin{enumerate}
    \item There exist $\theta$-twisted invariant forms $\langle\cdot, \cdot\rangle_{S_{k+1,k}}$ and $\langle\cdot, \cdot\rangle_{S_{k,k+1}}$ on $S_{k+1,k}$ and $S_{k,k+1}$, respectively. These forms are unique up to scalar and have $\sigma$ signs $(-1)^k$ and $(-1)^{k-1}$. 
    \item There exist $\theta$-twisted invariant pairings
    \[
    \langle \cdot, \cdot \rangle_{S^+_{k,k}, S^-_{k,k}} \quad \text{and} \quad \langle \cdot, \cdot \rangle_{S^-_{k,k}, S^+_{k,k}},
    \]
    unique up to scalar, and related by the identity
    \[
    \langle v, w \rangle_{S^+_{k,k}, S^-_{k,k}} = (-1)^k \langle \sigma(w), v \rangle_{S^-_{k,k}, S^+_{k,k}} \quad \forall v \in S^+_{k,k}, \ w \in S^-_{k,k}.
    \]
    Define a $\theta$-twisted invariant form on
    \[
    S_{k,k} \coloneqq S^+_{k,k} \oplus S^-_{k,k}
    \]
    by
    \[
    \langle v_1 + w_1, v_2 + w_2 \rangle_{S_{k,k}} = \langle v_1, w_2 \rangle_{S^+_{k,k}, S^-_{k,k}} + \langle w_1, v_2 \rangle_{S^-_{k,k}, S^+_{k,k}}.
    \]
    This form has $\sigma$ sign $(-1)^k$. 
\end{enumerate}
\end{lemma}

\begin{proof}
Uniqueness up to scalar follows from Schur's lemma. The form $\langle \cdot, \cdot \rangle_{S_{k+1,k}}$ is given by 
\[
\langle v_{-2k}, v_{2k} \rangle = 1, \quad \langle v_{-2k+2}, v_{2k-2} \rangle = -\zeta, \quad \ldots, \quad \langle v_{2k}, v_{-2k} \rangle = (-1)^k.
\]
The others are similar. 
\end{proof}

\begin{proof}[Proof of \Cref{orthosymplectic orbits}]
We retain the notations from \Cref{section Ortho-symplectic}. Let $V = V_1 \oplus V_2$, and define a bilinear form $\langle\, ,\,\rangle_V$ on $V$ by
\[
\langle v_1 + v_2, w_1 + w_2 \rangle_V = \langle v_1, w_1 \rangle_{V_1} + \langle v_2, w_2 \rangle_{V_2}, \quad \text{for all } v_i, w_i \in V_i.
\]
This form satisfies the identity
\begin{equation}\label{J sign}
 \langle v, w \rangle_V = \langle J(w), v \rangle_V, \quad \forall v, w \in V,
\end{equation}
where
\begin{equation}\label{Defin of J}
J := 
\begin{pmatrix}
\Id_{V_1} & 0 \\
0 & -\Id_{V_2}
\end{pmatrix}.
\end{equation}
Let $K = \GL(V)$ with Lie algebra $\frak k = \mathfrak{gl}(V)$. Define an automorphism $\theta: K \to K$ by the condition
\[
\langle \theta(g)(v), g(w) \rangle_V = \langle v, w \rangle_V \quad \text{for all } v, w \in V,\ g \in K.
\]
We denote by the same symbol $\theta$ the induced automorphism on the Lie algebra $\frak k$.

One directly checks that $\theta^2 = \mathrm{Int}(J)$ hence $\theta^4 = \Id$. We write elements in $K$ and $\frak k$ as block matrices with respect to the decomposition $V = V_1 \oplus V_2$. One easily finds
\[
\begin{aligned}
K^{\theta} =
\left\{
\begin{pmatrix}
    g_1 & 0 \\
    0 & g_2
\end{pmatrix}
\,\bigg\rvert\,
g_1 \in \rO(V_1), g_2 \in \Sp(V_2)
\right\}
= G_1\times G_2, \\
 K^{\theta^2} =
\left\{
\begin{pmatrix}
    g_1 & 0 \\
    0 & g_2
\end{pmatrix}
\,\bigg\rvert\,
g_1 \in \GL(V_1), g_2 \in \GL(V_2)
\right\}
= \GL(V_1) \times \GL(V_2).
\end{aligned}
\]
We obtain a $\mathbb{Z}/4\mathbb{Z}$-grading of $\frak k$:
\[
\frak k = \frak k^0 \oplus \frak k^1 \oplus \frak k^2 \oplus \frak k^3, \quad \frak k^i := \left\{x \in \frak{k} \mid \theta(x) = \zeta^i x\right\},
\]
with  
\begin{equation}\label{identification K}
\begin{aligned}
\frak k^0 &= \left\{
\begin{pmatrix}
    x_1 & 0 \\
    0 & x_2
\end{pmatrix}
\,\bigg\rvert\,
x_1 \in \frak{o}(V_1), x_2 \in \frak{sp}(V_2)
\right\}
= \frak{o}(V_1) \oplus \frak{sp}(V_2), \\
\frak k^0 \oplus \frak k^2 &= \left\{
\begin{pmatrix}
    x_1 & 0 \\
    0 & x_2
\end{pmatrix}
\,\bigg\rvert\,
x_1 \in \frak{gl}(V_1), x_2 \in \frak{gl}(V_2)
\right\}
= \frak{gl}(V_1) \oplus \frak{gl}(V_2), \\
\frak k^1 &= \left\{
\begin{pmatrix}
    0 & \zeta B^* \\
     B & 0
\end{pmatrix}
\,\bigg\rvert\,
B \in \Hom(V_1, V_2)
\right\}
= \Hom(V_1, V_2), \\
\frak k^3 &= \left\{
\begin{pmatrix}
    0 & -\zeta B^* \\
    B & 0
\end{pmatrix}
\,\bigg\rvert\,
B \in \Hom(V_1, V_2)
\right\}
= \Hom(V_1, V_2).
\end{aligned}
\end{equation}
where for $T \in \Hom(V_i, V_j)$, we define $T^* \in \Hom(V_j, V_i)$ by
\[
\langle Tu, v \rangle_{V_j} = \langle u, T^*v \rangle_{V_i}, \quad \forall u \in V_i, \ v \in V_j.
\]

Let $\mathcal{N} \subset \mathfrak{k}$ be the nilpotent cone. Then under the identification above, we have $\mathcal{N}_\VV = \mathcal{N} \cap \mathfrak{k}^1$.

Define the set of $\theta$-stable $\mathfrak{sl}_2$-triples in $\mathfrak{k}$ by
\[
\mathcal{A}_{\theta, \text{triple}}(\frk) = \{(H, X, Y) \mid H \in \frak k^0, X \in \frak k^1, Y \in \frak k^3, [H, X] = 2X, [H, Y] = -2Y, [X, Y] = H\}.
\]
This set carries a natural action of $G_1 \times G_2$. The graded Jacobson–Morozov theorem (cf. \cite[\S 2.3]{MR3697026}) provides a bijection between the set of $G_1 \times G_2$-orbits in $\mathcal{A}_{\theta, \mathrm{triple}}(\frk)$ and in $\mathcal{N}_\VV$ via $(H, X, Y) \mapsto X$.

Each triple $\Gamma = (H, X, Y)$ defines a representation of $\widetilde{\SL}_2(\mathbb{C})$ on $V$ by setting
\[
h \mapsto H, \quad e \mapsto X, \quad f \mapsto Y, \quad \sigma \mapsto J.
\]
One checks that the bilinear form $\langle \cdot, \cdot \rangle_V$ is $\theta$-twisted invariant in the sense of \Cref{form theta twisted} and of $\sigma$-sign $+1$ (see \Cref{J sign}). This defines a bijection between the $G_1 \times G_2$-orbits on $\mathcal{A}_{\theta, \mathrm{triple}}(\frk)$ and the isomorphism classes of representations of $\widetilde{\SL}_2(\mathbb{C})$ on the formed space $(V, \langle\cdot, \cdot\rangle_V)$ satisfying:
\begin{enumerate}
    \item $\sigma$ acts as $J$, and $\langle\cdot, \cdot\rangle_V$ has $\sigma$-sign $+1$;
    \item $\langle\cdot, \cdot\rangle_V$ is $\theta$-twisted invariant.
\end{enumerate}
\trivial[h]{The isomorphism is required to preserve the form $\langle \,,\,\rangle_V$.}
By \Cref{lem irr tilde SL2}, the triple $\Gamma = (H, X, Y) \in \mathcal{A}_{\theta, \mathrm{triple}}(\frk)$ induces a decomposition of $V$ as a $\widetilde{\SL}_2(\mathbb{C})$-representation:
\begin{equation}\label{Eq decomposition V}
V \cong \bigoplus_{k\ge0} \left( M_{k+1,k} \otimes S_{k+1,k} \oplus M_{k,k+1} \otimes S_{k,k+1} \right) \oplus \bigoplus_{k\ge0} \left( M^+_{k,k} \otimes S^+_{k,k} \oplus M^-_{k,k} \otimes S^-_{k,k} \right),
\end{equation}
where $M^{\ep}_{*,*}$ are multiplicity spaces. By \Cref{form tilde sl2 representation}, the restriction of $\langle\cdot, \cdot\rangle_V$ to each isotypic component in \eqref{Eq decomposition V} gives:
\begin{itemize}
    \item A non-degenerate bilinear form $\langle \cdot, \cdot \rangle_{M_{k+1,k}}$ on $M_{k+1,k}$ of sign $(-1)^k$. In particular, this form is:
    \begin{itemize}
        \item skew-symmetric if $k$ is even (hence $\dim M_{k+1,k}$ is even),
        \item symmetric if $k$ is odd;
    \end{itemize}
    
    \item A non-degenerate bilinear form $\langle \cdot, \cdot \rangle_{M_{k,k+1}}$ on $M_{k,k+1}$ of sign $(-1)^{k-1}$. In particular:
    \begin{itemize}
        \item skew-symmetric if $k$ is odd (hence $\dim M_{k,k+1}$ is even),
        \item symmetric if $k$ is even;
    \end{itemize}
    
    \item Perfect pairings between $M^+_{k,k}$ and $M^-_{k,k}$, hence a non-degenerate bilinear form $\langle \cdot, \cdot \rangle_{M^+_{k,k} \oplus M^-_{k,k}}$ of sign $(-1)^k$, implying $\dim M^+_{k,k} = \dim M^-_{k,k}$.
\end{itemize}
\trivial[h]{
We first give a pairing between $\Hom(\AS_{k,k},V)$ and $ \Hom(\BS_{k,k},V)$ using the pairing of $\AS_{k,k}$ and $\BS_{k,k}$, and the self-pairing of $V$. The show that this pairing is $\widetilde{\frak{sl}}_2$-invariant. Restricted this pairing to the invariant subspace $\Hom_{\widetilde{\frak{sl}}_2}(\AS_{k,k},V)=\AM_{k,k}$ and  $\Hom_{\widetilde{\frak{sl}}_2}(\BS_{k,k},V)=\BM_{k,k}$ is still non-degenerate. 
\[
J(x\otimes y)= \sigma(x)\otimes y,\quad   x\in \AS_{k,k+1}, y\in \AM_{k,k+1}
\]
\[
\langle x\otimes y, u\otimes v\rangle_{V} =\langle x,u\rangle_{\AS_{k,k+1}} \langle y, u\rangle_{\AM_{k,k+1}}. 
\]
\[
\langle x\otimes y, u\otimes v\rangle_{V}= \langle J(u\otimes v), x\otimes y\rangle_{V}
\]
\[
\langle x,u\rangle_{\AS_{k,k+1}}= (-1)^k\langle \sigma(u),x\rangle_{\AS_{k,k+1}}
\]
\[
 \langle y, u\rangle_{\AM_{k,k+1}}=  (-1)^k \langle (u), y\rangle_{\AM_{k,k+1}}
\]
}
such that the isomorphism \Cref{Eq decomposition V} is an isomorphism of formed spaces. This determines an ortho-symplectic partition
\begin{equation}\label{gamma type I}
\begin{split}
\gamma = \Big\{ 
&\underbrace{(k+1,k), \dots, (k+1,k)}_{\dim M_{k+1,k}},\ 
\underbrace{(k,k+1), \dots, (k,k+1)}_{\dim M_{k,k+1}},\\
&\underbrace{(k,k)^+, \dots, (k,k)^+}_{\dim M^+_{k,k}},\
\underbrace{(k,k)^-, \dots, (k,k)^-}_{\dim M^-_{k,k}} 
\ \Big|\ k \in \mathbb{Z}_{\geq 0}
\Big\} \in \OSP(2m, 2n).
\end{split}
\end{equation}

For $\Gamma = (H, X, Y) \in \mathcal{A}_{\theta, \mathrm{triple}}(\frk)$, let $\Stab_{G_1 \times G_2}(\Gamma)$ and $\Stab_{G_1 \times G_2}(X)$ denote the stabilizers of $\Gamma$ and $X$ in $G_1 \times G_2$, respectively. Their component groups are denoted
\[
A_\Gamma := \pi_0(\Stab_{G_1 \times G_2}(\Gamma)), \quad A_X := \pi_0(\Stab_{G_1 \times G_2}(X)).
\]
The inclusion $\Stab_{G_1 \times G_2}(\Gamma) \hookrightarrow \Stab_{G_1 \times G_2}(X)$ induces an isomorphism $A_\Gamma \cong A_X$, and we henceforth identify $A_X$ with $A_\Gamma$.

We identify $\GL(M^+_{k,k})$ with $\GL(M^-_{k,k})$ via the perfect pairings. 
Then,
\begin{equation}\label{Stab G1G2 Gamma}
\Stab_{G_1 \times G_2}(\Gamma) \cong 
\prod_{(x,y) \in \oddeven_\gamma} \rO(M_{x,y}) 
\times \prod_{(x,y) \in \evenodd_\gamma} \Sp(M_{x,y}) 
\times \prod_{(x,y)^+ \in \eveneven_\gamma \sqcup \oddodd_\gamma} \GL(M^+_{x,y}),
\end{equation}
and
\begin{equation}\label{eq component group gamma}
A_X = A_\Gamma \cong 
\prod_{(x,y)\in \oddeven_\gamma} \mathbb{Z}/2\mathbb{Z} \cdot a_{(x,y)} 
= \mathcal{S}_\gamma.
\end{equation}

Let $\cO_\gamma$ be the $G_1 \times G_2$-orbit in $\cN_\VV$ corresponding to $\gamma$ from \Cref{gamma type I}. Then \cite[\S 7.1]{MR694606} shows:
\begin{equation}\label{eq dimesion moment map type II}
\dim \cO_\gamma = \tfrac{1}{2} \left( \dim \mu_1(\cO_\gamma) + \dim \mu_2(\cO_\gamma) + \dim \VV - \mathcal{D}_\gamma \right),
\end{equation}
where $\mathcal{D}_\gamma$ is defined in \Cref{Dgamma}. Combining this with \Cref{def relevant orthsymplectic} and \Cref{defn orth-symplectic partitions}, we conclude that $\cO_\gamma \in \mathcal{R}_\VV$ if and only if $\cD_\gamma=0$, i.e., $\gamma \in \ROSP(2m, 2n)$.

The decomposition in \Cref{Eq decomposition V} induces a decomposition of the formed spaces $(V_1,\langle\,,\rangle_{V_1})$ and $(V_2,\langle\,,\rangle_{V_2})$ as in \Cref{Eq decomposition ov L1}:
\begin{equation}\label{Eq decomposition V1}
\begin{split}
    V_1&= \bigoplus_{k\ge0} \left( M_{k+1,k}\otimes  \LS_{k+1,k}\oplus  M_{k,k+1} \otimes  \LS_{k,k+1}\right)\oplus  \bigoplus_{k\ge0} \left( M^+_{k,k}\otimes  \LS^+_{k,k} \oplus M^-_{k,k} \otimes   \LS^-_{k,k}\right)\\
    V_2&= \bigoplus_{k\ge0} \left( M_{k+1,k}\otimes  \RS_{k+1,k}\oplus  M_{k,k+1} \otimes  \RS_{k,k+1}\right)\oplus  \bigoplus_{k\ge0} \left( M^+_{k,k}\otimes  \RS^+_{k,k} \oplus M^-_{k,k} \otimes   \RS^-_{k,k}\right).
\end{split}
\end{equation}
Identify $X$ with $T\in \Hom(V_1,V_2)$, the nilpotent endomorphisms $\mu_1(X) = T^*T$ of $V_1$ and $\mu_2(X) = T T^*$ of $V_2$ preserve the above decompositions. Taking lengths of Jordan blocks, we see that $\mu_1(X)$ and $\mu_2(X)$ correspond to orthogonal and symplectic partitions:

\[
\mu_1^\dagger(\gamma) = \left\{ 
\underbrace{k+1, \dots, k+1}_{\dim M_{k+1,k}},\
\underbrace{k, \dots, k}_{\dim M_{k,k+1}},\
\underbrace{k, \dots, k}_{\dim M^+_{k,k}},\
\underbrace{k, \dots, k}_{\dim M^-_{k,k}} 
\ \middle|\ k \in \mathbb{Z}_{\geq 0}
\right\} \in \cP^+(2m),
\]
\[
\mu_2^\dagger(\gamma) = \left\{
\underbrace{k, \dots, k}_{\dim M_{k+1,k}},\
\underbrace{k+1, \dots, k+1}_{\dim M_{k,k+1}},\
\underbrace{k, \dots, k}_{\dim M^+_{k,k}},\
\underbrace{k, \dots, k}_{\dim M^-_{k,k}} 
\ \middle|\ k \in \mathbb{Z}_{\geq 0}
\right\} \in \cP^-(2n).
\]
Hence $\mu(\cO_\g)\subset \cO_1\times \cO_2$ with $\cO_1$ and $\cO_2$ the nilpotent orbits corresponding to these partitions.

For $i=1,2$, we extend $\mu_i(X)$ to an $\frak{sl}_2$-triple $\Gamma_i$ in $\frak g_i$ persevering the decomposition \Cref{Eq decomposition V1}. Then we have:
\[
\begin{split}
C_{G_1}(\Gamma_1) &\cong 
\prod_{k \mbox{ odd} 
} \rO\left(M_{k,k+1} \oplus M_{k,k-1} \oplus M^+_{k,k} \oplus M^-_{k,k} \right) \\
&\quad \times \prod_{k \mbox{ even} 
} \Sp\left(M_{k,k+1} \oplus M_{k,k-1} \oplus M^+_{k,k} \oplus M^-_{k,k} \right), \\
C_{G_2}(\Gamma_2) &\cong 
\prod_{k \mbox{ even}
} \rO\left(M_{k+1,k} \oplus M_{k-1,k} \oplus M^+_{k,k} \oplus M^-_{k,k} \right) \\
&\quad \times \prod_{k \mbox{ odd}
} \Sp\left(M_{k+1,k} \oplus M_{k-1,k} \oplus M^+_{k,k} \oplus M^-_{k,k} \right).
\end{split}
\]


Under these isomorphisms and \Cref{Stab G1G2 Gamma}, the maps $C_{G_1}(\Gamma_1) \rightarrow \Stab_{G_1 \times G_2}(\Gamma)$ and $C_{G_2}(\Gamma_2) \rightarrow \Stab_{G_1 \times G_2}(\Gamma)$ are given by the inclusions. For example, the factor $\GL(M^+_{k,k})\cong \GL(M^-_{k,k})$ in \Cref{Stab G1G2 Gamma} maps into $\rO(M^+_{k,k}\oplus M^-_{k,k})$ or $\Sp(M^+_{k,k}\oplus M^-_{k,k})$ as the Siegel-Levi subgroup stabilizing the exhibited polarizations. These inclusions induce the morphism $\Delta^\dagger$ on the component groups defined in \Cref{map component group}. 
\trivial[h]{
 We extend $X_1\coloneqq  X^*X$ to be a $\frak{sl}_2$-triple $\gamma_1=\{ H_1, X_1, Y_1\}$ in $\frak{g}_1$ by defining the action of $H_1$ on $S^+_{k+1,k},  S^+_{k,k+1}, \overrightarrow{S^+_{k,k}}$ and $\overrightarrow{S^+_{k,k}}$ as follows 
\begin{itemize}
    \item We have $S^+_{k+1,k}=\langle v_{-2k}, v_{-2k+4},\cdots , v_{2k}\rangle$ and $X_1(v_{2i})=v_{2i+4}$ for $-k\leq i\leq k-2$ and $X_1(v_{2k})=0$. Define $H_1(v_{2i})= i v_{2i}$ for $i=-k,-k+2,\cdots, k$;
 \item $S^+_{k,k+1}=\langle v_{-2k+2}, v_{-2k+6},\cdots , v_{2k-2}\rangle$ and $X_1(v_{2i})=v_{2i+4}$ for $-k+1\leq i\leq k-3$ and $X_1(v_{2k-2})=0$. Define $H_1(v_{2i})= i v_{2i}$ for $ -k+1,-k+3,\cdots k-1$;
 \item $\overrightarrow{S^+_{k,k}}=\langle v_{-2k+1}, v_{-2k+5},\cdots , v_{2k-3}\rangle$ and $X_1(v_{2i-1})=v_{2i+3}$ for $-k+1\leq i\leq k-3$ and $X_1(v_{2k-3})=0$. Define $H_1(v_{2i-1})= i v_{2i-1}$ for $ -k+1, -k+3\cdots , k-1$;
  \item $\overleftarrow{S^+_{k,k}}=\langle v_{-2k+3}, v_{-2k+7},\cdots , v_{2k-1}\rangle$ and $X_1(v_{2i+1})=v_{2i+5}$ for $-k+1\leq i\leq k-3$ and $X_1(v_{2k-1})=0$. Define $H_1(v_{2i+1})= i v_{2i+1}$ for $ -k+1,-k+3,\cdots  k-1$. 
\end{itemize}
The action of $Y_1$ is uniquely determined by the action of $H_1$ and $X_1$. 
}
\end{proof}

\section{Fourier bijection of orbits}\label{subsection Fourier bijection}
The purpose of this section is to determine the bijection $$\Phi = \Phi_{m,n}: \SPM(m,n)\isom \SPM(n,m)$$ in \Cref{define Phi}, and show that $\Phi$ is independent of the base field $k$, as promised in \Cref{c:Fourier indep k}. 
\subsection{Fourier bijection vs conormal bijection}\label{sec: Fourier bijection vs conormal bijection}

We work over $\bC$ throughout this section. Let $V$ be a finite-dimensional $\bC$-vector space equipped with a linear action of a {\em connected} algebraic group $H$ that contains a dilation and has finitely many orbits on $V$. We first introduce an \emph{conormal bijection} between the set of $H$-orbits in $V$ and in $V^*$  by identifying the closures of their conormal bundles. We then prove that this conormal-bundle bijection agrees with the \emph{Fourier bijection} induced by the Fourier--Sato transform on simple perverse sheaves. The results of this section will be used in \Cref{App FT Type I}.


\begin{lemma}\label{l:fin orb}
    The group $H$ has finitely many orbits on $V$ if and only if it has finitely many orbits on $V^*$.
\end{lemma}
\begin{proof}

    Let $\mu: V\times V^*\to \frh^*$ be the moment map for the Hamiltonian action of $H$ on $T^*V=V\times V^*$. Then the finiteness of $H$-orbits in $V$ and $V^*$ are both equivalent to that the zero fiber of $\mu$ is a Lagrangian in $V\times V^*$.
\end{proof}

We say that the $H$-action on $V$ {\em  contains dilation} if $H$ contains a subgroup isomorphic to $\Gm$ that acts on $V$ with a single nonzero weight. Clearly if the $H$-action on $V$ contains dilation, the same holds for its action on $V^*$. In particular, all $H$-equivariant sheaves on $V$ and $V^*$ are monodromic under the dilation $\Gm$-actions, and the Fourier-Sato transform is defined, giving an equivalence
\begin{equation*}
    \FT_{H,V}: D_H(V)\isom D_{H}(V^*).
\end{equation*}

\subsubsection{Conormal bijection}\label{sss:conormal bij}
Now we assume $H$ has finitely many orbits on $V$ (hence on $V^*$ as well, by Lemma \ref{l:fin orb}). Let $I$ be the set of $H$-orbits in $V$, and $I^*$ be the set of $H$-orbits in $V^*$. For each $\s\in I$, let $O_\s\subset V$ be the corresponding $H$-orbit. Let $\L^\circ_\s:=T^*_{O_\s}V\subset V\times V^*$ be the conormal bundle of $O_\s$. Similarly, for $\r\in I^*$ we have the conormal bundle $\L^\circ_\r:=T^*_{O_\r}V^*\subset V\times V^*$.

Let $\mu: V\times V^*\to \frh^*$ be the moment map for the Hamiltonian action of $H$ on $T^*V=V\times V^*$. The zero fiber $\mu^{-1}(0)$ is thus identified with the union of conormals in two ways
\begin{equation*}
   \bigsqcup_{\s\in I}\L^\circ_\s= \mu^{-1}(0)=\bigsqcup_{\r\in I^*}\L^\circ_\rho.
\end{equation*}
Thus we get canonical bijections between the set of  irreducible components of $\mu^{-1}(0)$ with both $I$ and $I^*$
\begin{equation}\label{eq conoral bijection}
    I\bij \Irr(\mu^{-1}(0))\bij I^*.
\end{equation}
The composition is a ``conormal bijection'':
\begin{equation*}
    \Phi_{H,V}: I\isom I^*.
\end{equation*}
This bijection is characterized as follows: for $\s\in I$ and any point $v\in O_\s$, there is an open dense subset of the conormal fiber $\L^\circ_\s|_v=(\frg\cdot v)^\bot\subset V^*$ that lies in an orbit $O_\r$. Then $\Phi_{H,V}(\s)=\r$.

\begin{lemma}\label{l:conn stab}
    Suppose $V$ has finitely many $H$-orbits and all $H$-stabilizers are connected, then the same is true for the $H$-action on $V^*$.
\end{lemma}
\begin{proof}
    The t-exact equivalence $\FT_{H,V}$ implies that the number of simple perverse sheaves up to isomorphism is the same for $D_H(V)$ and for $D_H(V^*)$, which we denote by $N$. By assumption, $N=\#I$. By the bijection $\Phi_{H,V}$, we have $\#I=\#I^*$, hence the number of simple perverse sheaves up to isomorphism in $D_H(V^*)$ is equal to $\#I^*$, showing that each orbit only carries the trivial $H$-equivariant local system, i.e., stabilizers of $H$ on $V^*$ are all connected.
\end{proof}

Let $\IC_\s$ denote the intersection complex of the closure of $O_\s$, for $\s\in I$ or $I^*$. If the assumptions of Lemma \ref{l:conn stab} are satisfied, then
\begin{equation*}
    \{\IC_{\s}\}_{\s\in I} \mbox{ and }\{\IC_{\r}\}_{\r\in I^*}
\end{equation*}
are the set of isomorphism classes of simple perverse sheaves in $D_H(V)$ and $D_H(V^*)$ respectively. The equivalence $\FT_{H,V}$ gives a bijection between them, hence a ``Fourier bijection'' between $I$ and $I^*$.
The next proposition says that the Fourier bijection is the same as $\Phi_{H,V}$.

\begin{prop}[Fourier bijection=conormal bijection]\label{p:conormal bij Fourier bij}
    Suppose the action of $H$ on $V$ contains dilation, has finitely many orbits and all stabilizers are connected. Then for any $\s\in I$ we have an isomorphism of perverse sheaves in $D_H(V^*)$
    \begin{equation*}
        \FT_{H,V}(\IC_\s)\cong \IC_{\Phi_{H,V}(\s)}.
    \end{equation*}
\end{prop}
\begin{proof}
    We denote $\Phi_{H,V}$ simply by $\Phi$.
    Let $\Phi': I\isom I^*$ be the Fourier bijection, i.e., $\FT_{H,V}(\IC_\s)\cong \IC_{\Phi'(\s)}$. Since Fourier-Sato transform preserves singular support, we see that
    \begin{equation*}
        \L^\circ_{\s}\subset SS(\IC_\s)=SS(\IC_{\Phi'(\s)}).
    \end{equation*}
    Let $p_{V^*}: V\times V^*\to V^*$ be the projection. Then
    \begin{equation}\label{pV* inclusion}
        p_{V^*}(\L^\circ_{\s})\subset p_{V^*}(SS(\IC_{\Phi'(\s)}))=\ov O_{\Phi'(\s)}.
    \end{equation}
    By the characterization of $\Phi_{H,V}$ in Section \ref{sss:conormal bij}, $\L^\circ_{\s}$ contains a Zariski dense subset of $\L^\circ_{\Phi(\s)}$, hence $p_{V^*}(\L^\circ_{\s})$ contains the orbit $ O_{\Phi(\s)}$. By \eqref{pV* inclusion}, we get $O_{\Phi(\s)}\subset \ov O_{\Phi'(\s)}$. In particular, for any $\s\in I$
    \begin{equation}\label{dim ineq}
        \dim O_{\Phi(\s)}\le \dim O_{\Phi'(\s)}
    \end{equation}
    and equality holds only when $\Phi(\s)=\Phi'(\s)$.
    Now adding up the inequalities \eqref{dim ineq} for all $\s\in I$, using that both $\Phi$ and $\Phi'$ are bijections, we get
    \begin{equation*}
        \sum_{\r\in I^*}\dim O_\r=\sum_{\s\in I}\dim O_{\Phi(\s)}\le \sum_{\s\in I}\dim O_{\Phi'(\s)}=\sum_{\r\in I^*}\dim O_\r.
    \end{equation*}
    This forces equality in \eqref{dim ineq} to hold for all $\s\in I$, which implies that $\Phi(\s)=\Phi'(\s)$ for all $\s\in I$.
\end{proof}

\subsection{Fourier bijection for type II theta correspondence}\label{App FT Type I} 


The goal of this subsection is to determine the Fourier bijection of orbits in the setting of type II theta correspondence. This will be achieved in \Cref{c:Fourier bij type II} by identifying it with known combinatorics of mazes introduced in \cite{MR4865134}.

We retain the notation in \Cref{ex: GL theta}.
Throughout this subsection—until the final paragraph—we work over the finite field $F = \bF_q$. In particular, we consider the $\ov\bL_1=\Hom(L_1, L_2)$ under the $\ov B_1\times \ov B_2$-action. Switching the roles of $ L_1$ and $ L_2$ in the discussion of \Cref{Geometrization of the oscillator bimodule type II} and \Cref{another desc type II}, we get categories 
\begin{equation*}
\bcM_2^\Tate = D^{\Tate}_{\ov B_1 \times \ov B_2}(\ov \bL_2) \hookrightarrow 
\bcM_2^\mix = D^{\mix}_{\ov B_1 \times \ov B_2}(\ov \bL_2) \hookrightarrow 
\bcM_2 = D^{b}_{\ov B_1 \times \ov B_2}(\ov \bL_2) \to 
\bcM_2^k = D^{b}_{\ov B_{1,k} \times \ov B_{2,k}}(\ov \bL_2)
\end{equation*}
with the commuting Hecke category actions. The $\ov B_1 \times \ov B_2$-orbits on $\ov \bL_2$ are in natural bijection with $\PM(n,m)$. Passing to Grothendieck groups we obtain a $\ov\sfH_1\ot_R \ov\sfH_2$-module $\ov\sfM_2$ that specializes to the $\ov H_1\ot \ov H_2$-module $\ov M_2$ at $v=\sqrt{q}$, and to a $\ov W_1\times \ov W_2$-module $\ZZ[\PM(n,m)]$ at $v=1$.

In this subsection, we study the equivariant Fourier-Deligne transform
\begin{equation}\label{FT type II}
    \CM_{\ov \bL_1}: \bcM_1= D^{b}_{\ov B_1 \times \ov B_2}(\ov \bL_1)  \longrightarrow \bcM_2= D^{b}_{\ov B_1 \times \ov B_2}(\ov \bL_2)
\end{equation}
as well as its mixed, Tate and non-mixed variants.


\begin{remark}
    Since the action of $\ov B_1 \times \ov B_2$ on $\ov\bL_i$ contains the scaling action, the functor $\FT_{\bL_1}$ above is independent of the choice of the additive character $\psi$ on $\bF_q$. 
\end{remark}

\begin{lemma}\label{FT tate type II}
    The Fourier transform $\FT_{\ov \bL_1}$ restricts to an equivalence
    \begin{equation}\label{FT Mmix type II}
        \FT^\Tate_{\ov\bL_1}: \bcM^\Tate_1\isom \bcM_2^\Tate.
    \end{equation}
\end{lemma}
\begin{proof}  
The proof is analogous to that of \Cref{FT tate type I}, and will be omitted.
\end{proof}

Below we will simply use $\FT$ to denote various Fourier-Deligne transform, if the meaning is clear from the context.
\begin{cor} \label{FT type II generic}
Passing to Grothendieck groups, Fourier transform induces an isomorphism of $\ov\sfH:=\ov\sfH_1\ot_R\ov\sfH_2$-modules
    \begin{equation}\label{Fourier generic M12 type II}    
    \psi: \ov\sfM_1=R[\PM(m,n)]\xrightarrow{\ch^{-1}} K_0(\ol\cM_1^\Tate)\xrightarrow{K_0(\FT)}K_0(\ov\cM_2^\Tate)\xrightarrow{\ch}\ov\sfM_2= R[\PM(n,m)]
    \end{equation}
    and a $\ov W_1\times \ov W_2$-equivariant isomorphism 
    \begin{equation}\label{Fourier W12 type II}    
    \psi^k: \ov\sfM_{1,v=1}=\bZ[\PM(m,n)]\xrightarrow{\chi^{-1}}K_0(\ov\cM^k_1)\xrightarrow{K_0(\FT)}K_0(\ov\cM_2^k)\xrightarrow{\chi}\ov\sfM_{2,v=1}=\bZ[\PM(n,m)].
    \end{equation}
Via the sheaf-to-function map, Fourier transform induces an isomorphism of $\ov H:=\ov H_1\ot_{\bC} \ov H_2$-modules
    \begin{equation}\label{Fourier function M12 type II}    
    \psi^{F}: \ov M_1 \rightarrow \ov M_2
    \end{equation}
    and this is the same as the Fourier transform map given in \Cref{Fourier transform function}  (using the isomorphism $\io:\Qlbar\cong \bC$).  
\end{cor}

For any $\s\in \PM(m,n)$, the Fourier transform $\FT(\IC_\s)\in \ov\cM^k_2$ is a simple perverse sheaf, hence is isomorphic to $\IC_\t$ for a unique $\t\in \PM(n,m)$. The assignment $\s\mapsto \t$ defines a bijection
\begin{equation}\label{define Phi type II}
   \Psi= \Psi_{m,n}: \PM(m,n)\isom \PM(n,m).
\end{equation}

A priori, the bijection $\Psi$ may depend on the choice of the field $k = \ov{\bF}_q$. The following proposition shows that it does not:

\begin{prop}\label{c:Fourier indep k type II}
The bijection $\Psi$ is independent of $k$.
\end{prop}
We needed some preparations to prove \Cref{c:Fourier indep k type II}.
\def\PP#1#2{\mathrm{PP}(#1,#2)}
\begin{defn}
Let $\PP{m}{n} := \Set{(I,J) | I \subset \{ 1,\cdots, m\}, J\subset \{ 1,\cdots, n\} \text{ and } |I|=|J|}$ denote the pairs of subsets of $\{ 1,\cdots, m\}$ and $\{ 1,\cdots, n\}$ with same size.  
\end{defn}

For any $\mu = (I,J)\in \PP{m}{n}$ with $I = \{i_1\leq  i_2\leq  \ldots\leq i_k\}$ and $J = \{j_1\leq j_2\leq \ldots\leq j_k\}$, we define 
 $\ov \bL_1^\mu\subset \ov \bL_1$ be the linear subspace consisting of maps $T: L_1 \to L_2$ satisfying the following condition for all $1 \leq \lambda \leq m$:
 \begin{itemize}
     \item $T(e_{\lambda})=0$ if $\lambda< i_1$;
     \item $T(e_{\lambda})\in \langle f_1,f_2,\cdots, f_{j_l}\rangle\,\, \mbox{with} \,\,,  i_l\leq \lambda < i_{l+1}$
 \end{itemize}
Under the identification $\bL_1 = \Hom(L_1, L_2) \cong \Mat_{n \times m}$ via the fixed bases, the subspace $\ov \bL_1^\mu$ corresponds to a block matrix configuration. See \Cref{fig:pp} for an example, where the shaded region indicates the allowed matrix entries, with its corners marked by filled circles.
\begin{figure}[h]
\[\left(\raisebox{-6em}{
\begin{tikzpicture}

\coordinate (p1) at (0,5);
\coordinate (p2) at (2,5);
\coordinate (p3) at (2,3);
\coordinate (p4) at (4,3);
\coordinate (p5) at (4,2);
\coordinate (p6) at (6,2);
\coordinate (p7) at (6,0);
\coordinate (p8) at (8,0);



\draw[thick] (p1) -- (p2) -- (p3) -- (p4) -- (p5) -- (p6) -- (p7) -- (p8);

\draw[thick] ($(p2)-(0.2,0.2)$) circle (0.1);
\draw[thick,fill=black] ($(p3)+(0.2,0.2)$) circle (0.1);
\draw[thick] ($(p4)-(0.2,0.2)$) circle (0.1);
\draw[thick,fill=black] ($(p5)+(0.2,0.2)$) circle (0.1);
\draw[thick] ($(p6)-(0.2,0.2)$) circle (0.1);
\draw[thick,fill=black] ($(p7)+(0.2,0.2)$) circle (0.1);

\node at (1.1,4.7) {$(q_1,p_1)$};
\node at (2.9,3.3) {$(j_1,i_1)$};
\node at (3.1,2.7) {$(q_2,p_2)$};
\node at (4.9,2.3) {$(j_2,i_2)$};
\node at (5.1,1.7) {$(q_3,p_3)$};
\node at (6.9,.3) {$(j_3,i_3)$};

\node at (1.5,2) {$\ov \bL_2^{\nu}$};
\node at (6.5,3.5) {$\ov \bL_1^{\mu}$};

\begin{scope}
\clip (p2)--(p3)--(p4)--(p5)--(p6)--(p7)--(p8)--(8,5);
\fill[gray,opacity=0.2] (0,5) rectangle (8,0);
\end{scope}

\begin{scope}
\clip (p1)--(p2)--(p3)--(p4)--(p5)--(p6)--(p7)--(0,0);
\fill[pattern={Lines[angle=-25,distance={20pt/sqrt(2)}]}, opacity=0.5] (0,5) rectangle (8,0);
\end{scope}
\end{tikzpicture}
}
\right)
\]
\caption{The space $\ov \bL_1^{\mu}$ and $\ov \bL_2^{\nu}$}    
\label{fig:pp}
\end{figure}
Since $\ov \bL_1^\mu$ is an irreducible subvariety, there exists a unique $\sigma \in \PM(m,n)$ such that
\[
\ov \bL_1^\mu = \overline{\cO_\sigma},
\]
where $\cO_\sigma$ denotes the $\ov B_1 \times \ov B_2$-orbit corresponding to the partial matching $\sigma$ (see \Cref{PM vs orbits}). The following lemma gives an algorithm of the correspondence $\mu \mapsto \sigma$. In what follows, we identify a partial matching $\sigma$ with its graph, viewed as a subset of $\{1,\dots,m\} \times \{1,\dots,n\}$.
\begin{lemma}\label{PPmn subset}
Given $\mu = (I,J)\in \PP{m}{n}$ with $I = \{i_1\leq  i_2\leq  \ldots\leq i_k\}$ and $J = \{j_1\leq j_2\leq \ldots\leq j_k\}$, the partial matching $\sigma\in \PM(m,n)$ such that 
$V_\mu = \overline{\cO_\sigma}$ can be determined using the following algorithm:   

    \begin{enumerate}[label=\arabic*.]
        \item Initialize an empty partial matching $\sigma$ and an empty stack $S$ (in the sense of data type).
        \item For each position $p$ from $1$ to $k$:
        \begin{enumerate}[label=\alph*.]
            \item Push the integers $j_{p-1}+1, \cdots, j_{p}-1, j_{p}$ into the stack $S$ one by one (where $j_0 = 0$ for the first iteration).
            \item For each integer $x$ from $i_p$ to $i_{p+1}-1$ (where $i_{k+1} = m+1$ for the last iteration):
            \begin{enumerate}
                \item[$\bullet$] If the stack is not empty, pop an element $y$ from the stack and add the pair $(x, y)$ to the partial matching $\sigma$.
                \item[$\bullet$] If the stack is empty, break out of the inner loop and proceed to the next position.
            \end{enumerate}
        \end{enumerate}
        \item Return the constructed partial matching $\sigma$.
    \end{enumerate}
  Note that $\sigma$ will be empty if $\mu$ is empty, as determined by the algorithm.
\end{lemma}
\begin{proof}
    We leave the proof to the reader.
\end{proof}
We use \Cref{PPmn subset} to identify $\PP{m}{n}$ as a subset of $\PM(m,n)$. By switching the roles of $L_1$ and $L_2$, we similarly obtain $\PP{n}{m} := \Set{(Q,P) | Q \subset \{ 1,\cdots, n\}, P\subset \{ 1,\cdots, m\} \text{ and } |Q|=|P|}$, which parametrizes certain linear subspaces in $\ov \bL_2$, and may be viewed as a subset of $\PM(n,m)$.
\begin{defn}
   For each subset \( X \subset \{1,\dotsc, m\} \), we define the \emph{cycle right shift} and \emph{cycle left shift} by 
\[
X_{+1} := \{ (x+1) \bmod m \mid x \in X \}, \qquad 
X_{-1} := \{ (x-1) \bmod m \mid x \in X \},
\]
where we adopt the convention that \( m+1 \equiv 1 \) and \( 0 \equiv m \) modulo \( m \), so that elements are viewed cyclically. That is, the set \( \{1, \dotsc, m\} \) is treated as a circle with wrap-around indexing. Similar notation applies for subsets of \( \{1, \dotsc, n\} \).
\end{defn}
For instance, if \( m = 5 \) and \( X = \{1,2,5\} \), then
 \[
 X_{+1} = \{2, 3,1\}, \qquad X_{-1} = \{5,1,4\}.
 \] 

\trivial[h]{
Suppose $V$ is a vector space with basis $\set{e_1, e_2,\cdots, e_n}$. 
For $I=\set{i_1 < i_2 <\cdots<i_l}\subset \{ 1,\cdots, n\}$, 
define grading 
\[
\Gr^{(I]}_{k}(V):= \braket{e_a | i_{k} <   a \leq  i_{k+1}}
\]
and 
\[
\Gr^{[I)}_{k}(V):= \braket{e_a | i_{k}\leq   a < i_{k+1}}
\]
where $i_0 =0$ and $i_{l+1} = \infty$ by convention. 

For $I\subset \{ 1,\cdots, n\}$, 
define
\[
I^+ :=  (I+1)\cap \{ 1,\cdots, n\} 
\]
and 
\[
a^{++} := (\set{1}\cup (I+1)) \cap \{ 1,\cdots, n\}. 
\]
Then, if $n\notin I$ (i.e. $i_l \neq n$),
\[
\Gr^{[I^+)}(V)_k =
\Gr^{(I]}(V)_k 
\]
and 
\[
\Gr^{[I^{++})}(V)_{k+1} =
\Gr^{(I]}(V)_k. 
\]
If $n\in I$, 
then 
the above equation still holds since 
$\Gr^{(I]}_l = 0$.

Also define
\[
I^- :=  (I-1)\cap \{ 1,\cdots, n\} 
\]
and
\[
I^{--} :=  ((I-1)\cup \set{n})\cap \{ 1,\cdots, n\} 
\]

Suppose $1\notin I$. Then 
\[
\Gr^{(I^-]}(V)_k =
\Gr^{(I^{--}]}(V)_k= 
\Gr^{[I)}(V)_k
\]
Suppose $1\in I$, 
Then 
\[
\Gr^{(I^-]}(V)_k  = \Gr^{(I^{--}]}(V)_k =
\Gr^{[I)}(V)_{k+1}
\]

Consider the space 
\[
\bL_{I,J}:= \sum_{j < i} \Hom(\Gr^{[I)}_i(L_1),\Gr^{(J]}_j(L_2)). 
\]
Then 
\[
\begin{split}
(\bL_{I,J})^\perp &=
\sum_{j\geq i}\Hom(\Gr^{(J]}_j(L_2),\Gr^{[I)}_i(L_1)) \\
\end{split}
\]
Suppose $1\in I$ and $n\in J$, 
then 
\[
\begin{split}
RHS &=  
\sum_{i\leq j}\Hom(\Gr^{[J^+)}_j(L_2),\Gr^{(I^-]}_{i-1}(L_1)) \\
&=\sum_{i< j}\Hom(\Gr^{[J^+)}_j(L_2),\Gr^{(I^-]}_{i}(L_1))
= \bL_{Q,P}. 
\end{split}
\]
Suppose $1\in I$ and $n\notin J$, 
then 
\[
\begin{split}
RHS &=  
\sum_{i\leq j}\Hom(\Gr^{[J^+)}_j(L_2),\Gr^{(I^{--}]}_{i-1}(L_1)) \\
&=\sum_{i< j}\Hom(\Gr^{[J^+)}_j(L_2),\Gr^{(I^{--}]}_{i}(L_1))
= \bL_{Q,P}. 
\end{split}
\]
Suppose $1\notin I$ and 
$n\in J$, 
then 
\[
\begin{split}
RHS &=  
\sum_{i\leq j}\Hom(\Gr^{[J^{++})}_{j+1}(L_2),\Gr^{(I^-]}_{i}(L_1)) \\
&=\sum_{i< j}\Hom(\Gr^{[J^{++})}_{j}(L_2),\Gr^{(I^-]}_{i}(L_1))
= \bL_{Q,P}. 
\end{split}
\]
Suppose $1\notin I$ and 
$n\notin J$, 
then 
\[
\begin{split}
RHS &=  
\sum_{i\leq j}\Hom(\Gr^{[J^{++})}_{j+1}(L_2),\Gr^{(I^-]}_{i}(L_1)) \\
&=\sum_{i< j}\Hom(\Gr^{[J^{++})}_{j}(L_2),\Gr^{(I^{--}]}_{i}(L_1))
= \bL_{Q,P}. 
\end{split}
\]
}

\def\Psibu{\Psi^\bullet}
\begin{lemma}\label{Lem. type A Fourier linear} 
The bijection $\Psi_{m,n}$ restricts to a bijection
\begin{equation*}
    \Psibu=\Psibu_{m,n}: \PP{m}{n}\isom \PP{n}{m}.
\end{equation*}
More precisely, if $\mu=(I,J)\in \PP{m}{n}$, then $\nu\coloneqq \Psibu_{m,n}(\mu)=(Q,P)$ with
\[
(Q,P)=
\begin{cases}
   ( J_{+1}\cup \{1\}, I_{-1}\cup \{m\} ) \quad &\mbox{if $1\notin I, n\notin J$} \\
    ( J_{+1}\backslash \{1\} , I_{-1}\backslash \{m\})\quad & \mbox{if $1\in I, n\in J$}\\
     ( J_{+1}, I_{-1}) \quad & \mbox{otherwise}.
\end{cases}
\]
In particular, $\Psibu_{m,n}$ is independent of the choice of $\bF_q$.
\end{lemma}
\begin{proof}
Since $\overline{\cO_\sigma} = \ov \bL_1^\mu$ is smooth, we have that $\IC_{\cO_\sigma}$ is the constant sheaf (up to shift and twist) on $\ov \bL_1^\mu$. Its Fourier transform is then the constant sheaf (up to shift and twist) supported on the orthogonal complement ${\ov \bL_1^\mu}^\perp \subset \ov \bL_2$. One verifies that ${\ov \bL_1^\mu}^\perp$ coincides with $\ov \bL_2^\nu$ for $\nu = (Q, P) \in \PP{n}{m}$ as described above. See \Cref{fig:pp} for an illustration, where the matrix blocks of $\ov \bL_1^\mu$ are represented by shaded squares, and those of $\ov \bL_2^\nu$ are indicated by dashed lines with corners marked by hollow circles.
\end{proof}
\trivial[h]{
Recall that for each $0\leq i\leq \min\{m,n\}$, we defined $\overline\s_i\in \PM(m,n)$ in \Cref{def sigmai}, and the elements $\{C'_{\overline\s_i} \mid 0 \leq i \leq \min(m,n)\}$ generate $\ov\sfM_1$ as an $\ov\sfH$-module
\begin{cor}\label{Cor. type A Fourier linear}
For each $0 \leq i \leq \min\{m,n\}$, we have 
\[
 \Psi(\overline\s_i)=
 \begin{cases}
  (\{1,\cdots, i+1\}, \{m-i,\cdots, m\},\nu)   &\qquad \mbox{if $i<\min\{m,n\}$}\\
    (\{1,\cdots, n\}, \{m-n,\cdots, m-1\},\nu)  & \qquad \mbox{if $i=n<m$}\\
    (\{2,\cdots, m+1\}, \{1,\cdots, m\},\nu)  &\qquad \mbox{if $i=m<n$}\\
    (\{2,\cdots, n-1\}, \{1,\cdots, m-1\},\nu)  & \qquad \mbox{if $i=m=n$}. 
 \end{cases}
\]
 In all cases, $\nu$ is the unique order preserving bijection. 
\end{cor}
\begin{proof}
Since $\overline{\sigma}_i\in \PP{m}{n}$, the claim follows from
\Cref{Lem. type A Fourier linear} together with \Cref{PPmn subset}.
\end{proof}
}
\begin{proof}[Proof of \Cref{c:Fourier indep k type II}]
Recall that for each \(0\le i\le \min\{m,n\}\) we defined
\(\overline{\sigma}_i\in \PM(m,n)\) in \Cref{def sigmai}. A direct check shows
\(\overline{\sigma}_i\in \PP{m}{n}\). Hence, by \Cref{Lem. type A Fourier linear},
\(\Psi(\overline{\sigma}_i)\) is independent of \(k\). Moreover,
by \Cref{Lem. type II gene}, the objects
\(\{\IC_{\overline{\sigma}_i}\mid 0\le i\le \min(m,n)\}\) generate
\(\overline{\cM}_1^k\) under the \(\overline{\cH}^k\)-action. It follows that
\(\Psi\) is independent of \(k\). 
\end{proof}

\subsubsection{Over $\bC$}
From now on, we works over $\bC$. Our previous discussions about the Hecke module categories $\bcM_i^k$ ($i=1,2$) work equally well when $k$ is replaced by $\bC$. The only modification needed is to replace Fourier-Deligne transform by Fourier-Sato transform, which works for $\Gm$-monodromic sheaves, and gives an equivalence
\begin{equation*}
    \FT^\bC_{\ov\bL_1}:\bcM_1^\bC\isom  \bcM_2^\bC.
\end{equation*}
and it induce a bijection 
\[
\Psi^\bC: \PM(m,n) \isom \PM(n,m). 
\]
Similar to \Cref{Cor. compare FT Fq vs C}, we have 
\begin{equation}\label{PsiC vs Psi}
  \Psi^\bC=\Psi. 
\end{equation}
The action of $\ov B^\bC=\ov B^\bC_1\times \ov B^\bC_2$ on $\ov\bL_1^\bC$ 
fits into the framework in \Cref{sec: Fourier bijection vs conormal bijection}. Therefore, \Cref{p:conormal bij Fourier bij} implies that $\Psi^\bC$ equals to the conormal bijection defined using \Cref{eq conoral bijection}. The conormal bijection is described explicitly in \cite[Lemma 4.9.1]{MR4865134} using the combinatorics of rook placements and mazes \cite[Definition 4.2.1]{MR4865134}. The set $\PM(m,n)$ can be identified with the set of $m\times n$ rook placements in an obvious way: for $(I,J,\mu)\in \PM(m,n)$, put a hook at entry $(i,\mu(i))$ for each $i\in I$. 

Summarizing, we obtain:
\begin{cor}\label{c:Fourier bij type II}
    For $k$ either a finite field or $\bC$, the Fourier bijection $\Psi_{m,n}: \PM(m,n) \isom \PM(n,m)$ is the bijection given by the involution $\iota$ in \cite[Definition 4.2.1]{MR4865134}, after identifying both $\PM(m,n)$ and $\PM(n,m)$ first with $m\times n$ rook placements and with $m\times n$ mazes via \cite[Lemma 4.2.2]{MR4865134}. 
\end{cor}

\subsection{Fourier bijection for type I theta correspondence} 
We will give a formula for the bijection $\Phi$ defined in \Cref{define Phi} for the partial Fourier transform appearing in the type I theta correspondence, based on the knowledge of the Fourier bijection $\Psi$ studied in \Cref{App FT Type I}. This is \Cref{Prop FT SPM via PM}, which in particular implies that $\Phi$ is independent of the base field $k$.
 
It will be convenient to allow a bit more flexibility in the notation. If $\Sigma_1$ and $\Sigma_2$ are two linearly ordered finite sets of cardinality $m$ and $n$ respectively, we define $\PM(\Sigma_1,\Sigma_2)$ to be the set of partial matchings between subsets of $\Sigma_1$ and subsets of $\Sigma_2$. Of course we have a canonical bijection $\PM(\Sigma_1,\Sigma_2)\cong \PM(m,n)$. The bijection $\Psi_{m,n}$ induces a bijection
\begin{equation*}
    \Psi_{\Sigma_1,\Sigma_2}: \PM(\Sigma_1,\Sigma_2)\isom \PM(\Sigma_2,\Sigma_1).
\end{equation*}

Recall that for $\s=(I,J,\mu)\in \SPM(m,n)$, we defined $I^+=I\cap \{ 1,\cdots, m\}$ in \Cref{def signed partial matching}. We consider the decomposition
\begin{equation}\label{decomp I J}
    I^+=I^+_1\sqcup I^+_2, \quad \mu(I^+)=J^+_1\sqcup J^-_2
\end{equation}
where $I^+_1=\{i\in I^+|\mu(i)>0\}, I^+_2=\{i\in I^+|\mu(i)<0\}$, $J^+_1=\mu(I^+_1)$ and $J^-_2=\mu(I^+_2)$. Let $I^-_i=-I^+_i$ and $J^-_i=-J^+_i$ for $i=1,2$.  Let $I_i=I^+_i\sqcup I^-_i$ and $J_i=J^+_i\sqcup J^-_i$. Then $\mu$ restricts to bijections
\begin{eqnarray*}
    \mu_i: I_i\isom J_i, \quad i=1,2
\end{eqnarray*}
and
\begin{equation*}
    \mu^+_1: I^+_1\isom J^+_1, \quad \mu^+_2: I^+_2\isom J^-_2.
\end{equation*}

\begin{prop}\label{Prop FT SPM via PM}
Let $\s=(I,J,\mu)\in \SPM(m,n)$. Using the decomposition \eqref{decomp I J}, let
\begin{equation*}
    \Sigma_1=\{1,2,\cdots, m\}-I^+_2, \quad
\Sigma_2=\{1,2,\cdots, n\}-J^+_2.
\end{equation*}
Let $\s^+_1=(I_1^+, J_1^+, \mu_1^+)\in \PM(\Sig_1,\Sig_2)$. Write
\begin{equation*}
    \Psi_{\Sigma_1, \Sigma_2}(\s^+_1)=(Q^+,P^+,\nu^+)\in \PM(\Sigma_2, \Sigma_1).
\end{equation*}
Let $Q=Q^+\sqcup(-Q^+), P=P^+\sqcup(-P^+)$, and $\nu:Q\isom P$   the unique extension of $\nu^+$ such that $\nu(-x)=-\nu^+(x)$ for $x\in Q^+$. Then 
\begin{equation*}
    \Phi(I,J,\mu)=(Q\sqcup J_2, P\sqcup I_2, \nu\sqcup\mu_2^{-1}).
\end{equation*}
\end{prop}
\begin{proof}
Let $\tau := \Phi(\s) \in \SPM(n,m)$. We write $\tau = (J',I', \mu')$ and define the decomposition
\begin{equation*}
    J'=J'_1\sqcup J'_2, \quad I'=I'_1\sqcup I'_2, \mu'=\mu'_1\sqcup \mu'_2
\end{equation*}
according to the same rule of \eqref{decomp I J}. Similarly we use superscripts $\pm$ to denote the restrictions to positive or negative elements, e.g., $\mu'^+_2: J'^+_2\isom I'^-_2$.

Retain the notations in \Cref{Sec type I finite field}. 
Let $p_i: \bL_i\to \bL^b$ and $q_i: \bL_i\to \ov \bL_i$ be the projections, $i=1,2$. Since $\FT(\IC_\s)\cong \IC_\t$ is the Fourier transform along the fibers of $p_1$, we have $p_1(\ov \cO_\s)=p_1\Supp(\IC_\s)=p_2\Supp(\cO_\t)=p_2(\ov \cO_\t)$. This implies $p_1(\cO_\s)=p_2(\cO_\t)$. Recall the point $x_\s\in \cO_\s$ defined in \eqref{def:Tsigma}. Then $p_1(x_\s)\in \Hom(L_1,L_2^\vee)$ is the map
\begin{equation*}
    p_1(x_\s)(e_i)=\begin{cases}
        f_{\mu(i)} & \mbox{if $i\in I^+$ and $\mu(i)<0$}\\
        0 & \mbox{otherwise.}
    \end{cases}
\end{equation*}
Define $x_\t\in \cO_\t $ in a similar way using the same standard basis, then $p_2(x_\t)\in \Hom(L_2,L_1^\vee)$ is the map 
\begin{equation*}
    p_2(x_\t)(f_j)=\begin{cases}
        e_{\mu'(j)} & \mbox{if $j\in J'^+$ and $\mu'(j)<0$}\\
        0 & \mbox{otherwise.}
        \end{cases}
\end{equation*}
Since $p_1(\cO_\s)=p_2(\cO_\t)$, $p_1(x_\s)$ is in the same $\ov\bB$-orbit as $p_2(x_\t)$. This forces $p_2(x_\t)$ to be the transpose of $p_1(x_\s)$ , which implies 
\begin{equation*}
    J'^+_2=J^+_2, \quad I'^-_2=I^-_2, \quad \mu'^+_2=(\mu^-_2)^{-1}.
\end{equation*}

It remains to determine $(J'^+_1, I'^+_1, \mu'^+_1)\in \PM(n,m)$. For this we only need to determine $\IC_\tau |_{p_2^{-1}(p_2(T_\t))}$, which is the Fourier transform of $\IC_\s|_{p_1^{-1}(p_1(x_\s))}$.

Let 
\begin{equation*}
    Z_\s :=q_1(\ov\cO_\s\cap p_1^{-1}(p_1(x_\s)))\subset \Hom(L_1,L_2).
\end{equation*}
Since $p_1(x_\s)$ is in the open $\ov\bB$-orbit of $p_1(\ov \cO_\s)$, 
   $ q_{1*}(\IC_\s|_{p_1^{-1}(p_1(x_\s))})$ is $  \IC(Z_\s)$ up to shift and twist.

Let us first determine $\IC(Z_\s)$, up to shift and twist. We need some notations.
For a subset $\Sig\subset \{1,2,\cdots, m\}$, let $L_1(\Sig)\subset L_1$ be the span of $\{e_i|i\in \Sig\}$ and $L_1^*(\Sigma) \subset L_1^*$ be the span of ${f_i|i\in \Sigma}$. Similarly define $L_2(\Sig)\subset L_2$ and $L_2^*(\Sigma)\subset L_2^*$ for $\Sig\subset\{1,2,\cdots, n\}$. If $T\in \cN_{\bL_1}\cap p_1^{-1}(p_1(x_\s))$, then the fact that $T(L_1)$ is isotropic implies that we have $T(L_1(\Sig_1))\subset L_2(\Sig_2)$. Thus the image of $\cN_{\bL_1}\cap p_1^{-1}(p_1(x_\s))\xrightarrow{q_1} \Hom(L_1,L_2)$ lies in the linear subspace
\begin{equation*}
    \bL^+_1(\Sig_1,\Sig_2):=\{T\in \Hom(L_1,L_2)|T(L_1(\Sig_1))\subset L_2(\Sig_2)\}\xrightarrow{\ i_1 \ }\bL_1 = \Hom(L_1,L_2).
\end{equation*}
where $i_1$ is the natural inclusion. 
Furthermore, we have a naturl linear surjection
\begin{equation*}
    \pi_1: \bL^+_1(\Sig_1,\Sig_2)\to \Hom(L_1(\Sig_1), L_2(\Sig_2))\times \Hom(L_1(I^+_2), L_2(J^+_2)).
\end{equation*}
\trivial[h]{
    Note that the fiber of $\pi_1$ is $\Hom(L_1(I^+_2),L_2(\Sigma_2))$. 
    Moreover, the proejction of $B_1\times B_2 \cdot x_\sigma$ to $\Hom(L_1(I^+_2),L_2(\Sigma_2))$ can take any value because of the $B_2$-action. 
    it can take any value. 
}
The image of $T$ is the pair,  whose first component is $T|_{L_1(\Sig_1)}$ and second component is the induced map between the quotients $L_1(I^+_2)\cong L_1/L_1(\Sig_1)$ and $ L_2(J^+_2)\cong L_2/L_2(\Sig_2)$. 

Let $\ov B_i(\Sig_i)=\ov B_i\cap \GL(L_i(\Sig_i))$ for $i = 1,2$. Then $\s^+_1=(I_1^+, J_1^+, \mu_1^+)\in \PM(\Sig_1,\Sig_2)$ defines a $\ov B_1(\Sig_1)\times \ov B_2(\Sig_2)$-orbit $\cO_{\s_1^+}$ on $\Hom(L_1(\Sig_1), L_2(\Sig_2))$.
For an isomorphism $\gimel : L_1(I_2^+)\to L_2(J_2^-)$, we define  $\Herm_\gimel$ be the subspace of $\Hom(L_1(I^+_2),L_2(J_2^+))$ given by 
\[
\begin{split}
&\Herm_\gimel
:= \Set{ X  | \forall v_2, v'_2\in L_2(J_2^-),  \inn{X \gimel^{-1} v_2}{v'_2}_{V_2}  + \inn{v_2}{X \gimel^{-1} v'_2}_{V_2}=0}, 
\end{split}
\]
i.e. $X \in \Herm(L_1(I^+_2), L_2(J^+_2))$ if and only if $X \gimel^{-1}$ is
``skew-symmetric'' 
under the perfect pairing between $L_2^*(J_2^+)$ and  $L_2(J_2^+)$. 
\trivial[h]{
The map $X\to \Hom(L_2,L_2^*)$ is called skew-symmetric if 
\[
\inn{Tv}{w} + \inn{v}{Tw} = 0,
\]
i.e. $T^* + T = 0$ where 
$T^*$ is defined by $ 
\inn{T^*v}{w}=\inn{v}{Tw}$ where $v,w\in L_2$ and $\inn{}{}$ denote the natural paring.

Let us make more concentrate definition of 
$\Herm_{\gimel}$. 
We take the standard basis of $L_1(I_2^+)$ and
$L_2(J_2^+)$.
Now $\gimel$ is a permutation matrix.  
Then the equation become
Assume $V_2$ is quadratic.  then the equation became 
\[
X \gimel^{-1} + (X \gimel^{-1})^T = 0.
\]
i.e.  
\[
 X =  - (\gimel^{-1})^T X^T \gimel.
\]
We know that $\dim \Herm_{\gimel} = d(d-1)/2$ where $d = |I_1^+|$.

Similarly, the 
equation for
$\Herm'_{\gimel}$ is given by 
\[
Y (\gimel^T)^{-1} - (Y (\gimel^T)^{-1})^T = 0.
\]  
i.e.
\[
Y = \gimel^{-1} Y^T \gimel^T.
\]
We know that $\Herm'_{\gimel^T} = d(d+1)/2$.  

Now 
\[
\tr(XY) = \tr(- (\gimel^{-1})^T X^T \gimel \gimel^{-1} Y^T \gimel^T) = 
\tr(-X^T Y^T) = \tr(-YX) = -\tr(XY). 
\]
This implies $\tr(XY) = 0$. 

By dimension counting, we conclude that 
$\Herm_{\gimel}$ and $\Herm'_{\gimel^T}$ are orthogonal complements for each other. 
}

One checks that
\begin{equation}\label{Os}
    q_1(\cO_\s\cap p_1^{-1}(p_1(x_\s)))=i_1(\pi_1^{-1}(\cO_{\s^+_1}\times\Herm_{p_1(x_\sigma)})).
\end{equation}
Moreover a similar formula holds $\sigma$ is replaced by any other $\sigma'$ such that $p_1(T_{\sigma'})=p_1(x_\sigma)$, which implies that 
\begin{equation*}
    Z_\s=i_1(\pi_1^{-1}(\ov\cO_{\s^+_1}\times\Herm_{p_1(x_\sigma)}).
\end{equation*}
Therefore, up to shift and twist,
\begin{equation}\label{ICZs}
    \IC(Z_\s)\cong i_{1*}\pi_1^*(\IC(\cO_{\sigma^+_1})\boxtimes \iota_{1*}\Qlbar)
\end{equation}
where $$\iota_1: \Herm_{p_1(x_\sigma)}\hookrightarrow\Hom(L_1(I^+_2), L_2(J^+_2))$$ is the inclusion.

Having determined $\IC(Z_\s)$, now we compute
$\FT_{\bL_1^+}(\IC(Z_\s))$.
Consider the linear inclusion
\begin{equation*}
    i_2: \bL^+_2(\Sig_2,\Sig_1):=\{T'\in \Hom(L_2,L_1)|T'(L_2(\Sig_2))\subset L_1(\Sig_1)\}\rightarrow \bL_2^+=\Hom(L_2,L_1)
\end{equation*}
and the linear surjection
\begin{equation*}
    \pi_2: \bL^+_2(\Sig_2,\Sig_1)\to \Hom(L_2(\Sig_2), L_1(\Sig_1))\times \Hom(L_2(J^+_2), L_1(I^+_2)).
\end{equation*}
Then we have the following Cartesian diagram 
\[
\begin{tikzcd}[ampersand replacement=\&]
\bL_2^+(\Sigma_2,\Sigma_1) \ar[r,"i_2"]\ar[d,"\pi_2"]
    \& \Hom(L_2,L_1) \ar[d,"i'_1"]  \\
\Hom(L_2(\Sigma_2),L_1(\Sigma_1))\oplus \Hom(L_2(J^+_2),L_1(I^+_2)) \ar[r,"\pi'_1"]\&
\bL_1^+(\Sigma_1,\Sigma_2)^* 
\end{tikzcd}
\]
where $\pi'_1$ and $i'_1$ denote the transpose of $\pi_1$ and $i_1$ respectively.

Let $\tau^+_1:=\Psi_{\Sigma_1,\Sigma_2}(\sigma^+_1)=(Q^+,P^+, \nu^+)$, i.e., $\FT(\IC(\cO_{\sigma^+_1}))\cong \IC(\cO_{\t^+_1})$
where $\cO_{\tau^+_1}$ denote the $\barB_1(\Sigma_1)\times \barB_2(\Sigma_2)$-orbit in $\Hom(L_2(\Sigma_2), L_1(\Sigma_1))$ determined by the partial matching $\tau^+_1$.  

By \eqref{ICZs} and by the compatibility of the Fourier transform under linear maps, we have up to shift and twist
\begin{equation}\label{FT ICZs}
    \begin{split}
    &\FT_{\bL_1^+}(\IC(Z_\s))\cong
    {\pi'_1}^* {i'_{1}}_*(\FT(\IC_{\s^+_1})\boxtimes\FT(\iota_{1*}\Qlbar)) \\
  &\cong  {i_{2}}_*\pi_2^*(\FT(\IC(\cO_{\s^+_1}))\boxtimes\FT(\iota_{1*}\Qlbar)) \\
  &\cong {i_{2}}_*\pi_2^*( \IC(\cO_{\tau^+_1})\boxtimes\FT(\iota_{1*}\Qlbar)).
    \end{split}
\end{equation}
\trivial[h]{
\[
\begin{split}
\FT_{\bL^+_1}(\IC(Z_\sigma)) 
&\cong \FT_{\bL^+_1}(i_{1!}\pi_1^*(\IC_{\s^+_1}\boxtimes \iota_{1*}\Qlbar)) \\
&= i_1'^* {\pi'_1}_!(\FT(\IC_{\sigma^+_1})\boxtimes \FT(\iota_{1*}{\Qlbar}))\\
&=i_1'^* {\pi'_1}_*(\FT(\IC_{\sigma^+_1})\boxtimes \FT(\iota_{1*}{\Qlbar}))
\end{split}
\]
Here $i'_1$ and $\pi'_1$ are the transpose of $i_1$ and $\pi_1$. 

Now we consider the diagram involves $i'_1$ and $\pi'_2$ 
To ease notation write $A_i=L_i(\Sigma_i)$, $B_1 = L_1(I^+_2)$ and $B_2 = L_2(J^+_2)$. 
Then the transpose of the diagram 
\[
\begin{tikzcd}[ampersand replacement=\&]
\Hom(A_1,A_2)\oplus \Hom(B_1,A_2)\oplus \Hom(B_1,B_2)\ar[r,"i_1"] \ar[d,"\pi_1"]\&
\Hom(L_1,L_2)\\
\Hom(A_1,A_2)\oplus \Hom(B_1,B_2) \& 
\end{tikzcd}
\] 
gives the following Cartesian diagram.
\[
\begin{tikzcd}[ampersand replacement=\&]
\Hom(A_2,A_1)\oplus \Hom(B_2,A_1)\oplus \Hom(B_2,B_1) \ar[r,"i_2"]\ar[d,"\pi_2"]
    \& \Hom(L_2,L_1) \ar[d,"i'_1"]  \\
\Hom(A_2,A_1)\oplus \Hom(B_2,B_1) \ar[r,"\pi'_1"]\&
\Hom(A_2,A_1)\oplus \Hom(A_2,B_1)\oplus \Hom(B_2,B_1) 
\end{tikzcd}
\] 
Since $\pi'_1$ is a closed embedding (proper), we get
\[
i_1'^* {\pi'_1}_*(\FT(\IC_{\sigma^+_1})\boxtimes \FT(\iota_{1*}{\Qlbar})) \cong i_{2*}\pi_2^* (\FT(\IC_{\sigma^+_1})\boxtimes \FT(\iota_{1*}{\Qlbar})).\] 

Note that $p_2(T_\tau)$ is the transpose of $p_1(x_\sigma)$
}

We now compute $\FT(\io_{1*}\Qlbar)$.  
For an isomorphism $\gimel : L_2(J_2^+) \to L_1(I_2^+)$ define 
\[
\Herm'_{\gimel}
:= \Set{ X  | \forall v_1, v'_1\in L_1(J_2^+),  \inn{X \gimel^{-1} v_1}{v'_1}_{V_1}  + \inn{v_2}{X \gimel^{-1} v'_1}_{V_1}=0}. 
\] and 
\[
\iota_2: \Herm'_{p_2(T_\tau)}, L_1(I^+_2))\hookrightarrow\Hom(L_2(J^+_2), L_1(I^+_2))
\] be the inclusion.
It is straightforward to check that  and  $\Herm_{p_1(x_\sigma)}$ and $\Herm'_{p_2(T_\tau)}$ are mutual orthogonal complements under the trace pairing between $\Hom(L_1(I^+_2),L_2(J^+_2))$ and $\Hom(L_2(J^+_2), L_1(I^+_2))$. Therefore, up to shift and twist, we have $\FT(\io_{1*}\Qlbar)\cong \io_{2*}\Qlbar$. By \eqref{FT ICZs}, this implies up to shift and twist 
\begin{equation}\label{eq:FT ICZs}
    \FT_{\bL_1^+}(\IC(Z_\s))={i_2}_*(\pi^*_2(\IC(\cO_{\tau^+_1})\boxtimes {\iota_1}_* \Qlbar)).
\end{equation}
By the same analysis for $\IC_\sigma$, the equation \eqref{eq:FT ICZs} implies that $\tau$ is determined by the procedure outlined in the proposition's statement.
\end{proof}

\begin{proof}[Proof of \Cref{c:Fourier indep k}]
This follows from \Cref{Prop FT SPM via PM} and \Cref{c:Fourier indep k type II}. 
\end{proof}

\section{Hotta's local formula revisited}\label{Sec. Hotta}

The goal of this appendix is to give a modern proof of Hotta's formula \cite{Hotta} for the Springer action of a simple reflection.  A slight improvement from Hotta's original proof is that it now works over any base field in the context of \'etale sheaves, and it has a uniform formulation on all homology classes, without distinguishing vertical or horizontal cases.
In the rest of the section, let $\kk$ be the coefficient system such that $\mathrm{char}\, \kk \neq p$ or $\mathrm{char} \, \kk =0$. 
\subsection{The setup}
Let $\cN$ denote the nilcone of the Lie algebra $\sl_2$ over an arbitrary base field, with the adjoint action of $H=\PGL_2$. Let $\nu: \wt\cN\to \cN$ be the Springer resolution of $\cN$. Consider a Cartesian diagram of stacks
\begin{equation}\label{Hotta setup Cart}
    \xymatrix{X \ar[r]^-{\wt f} \ar[d]^{\pi} & [\wt\cN/H]\ar[d]^{\nu}\\
    Y\ar[r]^-{f} & [\cN/H]}
\end{equation}
Let $Y_0\subset Y$ be the preimage of the zero orbit $\{0\}/H$ under $f$. Let $X_0=\pi^{-1}(Y_0)$ with closed embedding $i:X_0\hookrightarrow X$, and a $\bP^1$-fibration $\pi_0:=\pi|_{X_0}: X_0\to Y_0$. We view $X_0$ as a correspondence between $X$ and $Y_0$
\begin{equation}\label{corr X0}
    \xymatrix{Y_0& \ar[l]_-{\pi_0} X_0 \ar@{^{(}->}[r]^-{i}& X}
\end{equation}

We have the Springer action of $S_2=\{1,s\}$ on the complex $\nu_{*}\bD_{[\wt\cN/H]}\in D([\cN/H])$. Here the dualizing complex $\bD_{[\wt\cN/H]}$ of $[\wt\cN/H]$ is isomorphic to the constant sheaf $\kk[-2]$.
\begin{conv}\label{con:Springer}
Our convention is that $s$ acts trivially on the constant sheaf summand of $\nu_{*}\DD_{[\wt\cN/H]}$ and acts by $-1$ on the skyscraper summand.
\end{conv}

By proper base change, we get an action of $s$ on the complex
\begin{equation*}
f^!\nu_{*}\bD_{[\wt\cN/H]}\cong \pi_*\wt f^!\bD_{[\wt\cN/H]}\cong \pi_* \bD_X\in D(Y).
\end{equation*}
Taking global sections we get an action of $s$ on $\hBM{*}{X}$. The goal is to describe this action more explicitly.

\subsection{The description of the action of $s$}

\begin{prop}[c.f. \cite{Hotta}]\label{p:Hotta}
    The action of $s-1$ on $\hBM{*}{X}$ is given by the composition
    \begin{equation}\label{Hotta}
        \hBM{*}{X}\xr{i^!}\hBM{*-2}{X_0}(1)\xr{\pi_{0*}}\hBM{*-2}{Y_0}(1)\xr{\pi_0^!}\hBM{*}{X_0}\xr{i_*}\hBM{*}{X}.
    \end{equation}
    Here $i^!$ will be defined in \Cref{sss:Gysin i}; $\pi_0^!$ denotes the smooth pullback on Borel-Moore homology along $\pi_0$; $i_*$ and $\pi_{0*}$ are the pushforwards on Borel-Moore homology along proper maps.

\end{prop}

\subsubsection{Gysin pullback along $i$}\label{sss:Gysin i}

Since $[\wt\cN/H]=[\bA^1/\Gm]\times \bB\bG_a$, the map $\wt f$ induces a map
\begin{equation*}
    \ph: X\to [\bA^1/\Gm].
\end{equation*}
It is clear that $X_0=\ph^{-1}(\{0\}/\Gm)$. Since we will be applying the setup to cases where $X$  is possibly reducible, $X_0$ may not be a divisor of $X$: it may contain irreducible components of $X$.

We note the Cartesian diagram
\begin{equation*}
    \xymatrix{X_0\ar[d]^{\wt f_0} \ar[r]^-{i} & X\ar[d]^{\ph}   \\
     [\{0\}/\Gm] \ar[r]^-{\io} & [\bA^1/\Gm] }
\end{equation*}
We have a canonical map obtained from the proper base change by adjunction
\begin{equation}\label{adj to bc}
    i^*\bD_X\cong i^*\wt f^!\bD_{[\wt\cN/H]} \cong i^* \varphi^! \bD_{[\bA^1/\bG_m]}\to \wt{f}^!_0\io^*\bD_{[\bA^1/\Gm]}.
\end{equation}
Since $\io$ is a regular embedding of codimension $1$, we have a canonical isomorphism  $\io^*\bD_{[\bA^1/\Gm]}\cong \kk_{[0/\Gm]}\cong \bD_{[0/\Gm]}[2](1)$. We therefore get
\begin{equation*}
    f^!_0\,\iota^*\bD_{[\bA^1/\Gm]}\cong \wt{f}^!_0\bD_{[0/\Gm]}[2](1)\cong \bD_{X_0}[2](1).
\end{equation*}
Composing with \eqref{adj to bc} we get a canonical map $i^*\bD_X\to \bD_{X_0}[2](1)$, or equivalently
\begin{equation}\label{pre Gysin}
    \bD_X\to i_*\bD_{X_0}[2](1).
\end{equation}
Taking global sections gives a map
\begin{equation*}
    i^!: \hBM{*}{X}\to \hBM{*-2}{X_0}(1).
\end{equation*}
This is the homological version of the Gysin pullback of algebraic cycles.


\subsubsection{Vertical case} Consider an  irreducible subvariety $C\subset X_0$. It is called {\em $s$-vertical} if $C=\pi_0^{-1}(\pi_0(C))$. 

\begin{cor}
    For $C\subset X_0$ an $s$-vertical irreducible subvariety of dimension $n$, with cycle class $[C]\in \hBM{2n}{X}(-n)$,
    we have
    \begin{equation*}
        s[C]=-[C].
    \end{equation*}
\end{cor}
\begin{proof}
Replacing $X$ by $C$ and $Y$ by $\pi_0(C)$, 
we reduce to the case where $C=X=X_0$, and we would like to show that $s[X]=-[X]$.

    The map $i^!: \hBM{*}{X}\to \hBM{*-2}{X_0}(1)$ in this case is the cap product with the Chern class of the line bundle $\cL$ given by the composition 
\begin{equation*}
    X\xr{\ph}[\bA^1/\Gm]\to \bB\Gm.
\end{equation*}
The line bundle is the pullback of the line bundle $\cO(\bP^1)$ on $\wt\cN\cong T^*\bP^1$.  In particular, the restriction of $\cL$ to each fiber of $\pi_0$ is isomorphic to $\om_{\bP^1}\cong \cO(-2)$. Therefore the composition
\begin{equation*}
    \pi_{0*}i^!: \hBM{*}{X}\to \hBM{*-2}{Y}(1)
\end{equation*}
sends $[X]$ to $-2[Y]$, which gets sent to $-2[X]$ by the smooth pullback $\pi_0^!$. Hence \Cref{p:Hotta} implies that $(s-1)[X]=-2[X]$, therefore $s[X]=-[X]$. 
\end{proof}

\subsection{Proof of \Cref{p:Hotta}}

\subsubsection{Cohomological correspondences}
We will use the formalism of cohomological correspondences in \cite[III,\S3]{SGA5}. More details can be found in \cite{SGA5} and \cite[Appendix A]{Yun}. Consider a correspondence between $X_1$ and $X_2$ over $S$
\begin{equation}\label{C corr}
    \xymatrix{ & C \ar[dl]_-{p_2}\ar[dr]^-{p_1}\\
    X_2 \ar[dr]_-{\pi_2} &  & X_1\ar[dl]^-{\pi_1}\\
    &S}
\end{equation}
A cohomological correspondence supported on $C$ between the dualizing complexes $\DD_{X_1}$ and $\DD_{X_2}$ is a map 
\begin{equation}\label{corr coho}
    \z: p_1^*\DD_{X_1}\to p_2^!\DD_{X_2}\cong \DD_C.
\end{equation}
We denote the vector space of such maps to be
\begin{equation*}
    \Corr_C(\DD_{X_1}, \DD_{X_2}).
\end{equation*}
Replacing $\DD_{X_2}$ by $\DD_{X_2}[n]$ (or a Tate twist of it), we similarly define $\Corr_C(\DD_{X_1}, \DD_{X_2}[n])$. Denote
\begin{equation*}
    \Corr^*_C(\DD_{X_1}, \DD_{X_2})=\bigoplus_{n\in\ZZ}\Corr_C(\DD_{X_1}, \DD_{X_2}[n]).
\end{equation*}
We call $C$ is {\em target-proper} if $p_2$ is proper.  Now we assume that $p_2$ is target proper. Applying $p_{2*}$ to $\z$ and using adjunctions, we get a map
\begin{equation*}
    \pi_2^*\pi_{1*}\DD_{X_1}\to p_{2*}p_1^*\DD_{X_1}\xr{p_{2*}\z} p_{2*}p_2^!\DD_{X_2}=p_{2!}\pr_{2}^!\DD_{X_2}\to \DD_{X_2}.
\end{equation*}
By adjunction this gives a map
\begin{equation*}
    \z_{\#}: \pi_{1*}\DD_{X_1}\to \pi_{2*}\DD_{X_2}.
\end{equation*}
Taking global sections we get a map $\hBM{*}{X_1}\to \hBM{*}{X_2}$. Of course the same constructions works if we replace $\DD_{X_2}$ by a shift $\DD_{X_2}[n]$. This construction gives a linear map
\begin{equation*}
    (-)_{\#}: \Corr^*_C(\DD_{X_1}, \DD_{X_2})\to \Ext^*_S(\pi_{1*}\DD_{X_1}, \pi_{2*}\DD_{X_2}).
\end{equation*}

When $X_1$ is smooth of dimension $d_1$, the fundamental class of $X_1$ gives an isomorphism $\DD_{X_1}\cong \kk_{X_1}[2d_1](d_1)$. This gives an identification 
\begin{equation}\label{eq:CorrHBM}
\Corr_C(\DD_{X_1},\DD_{X_2}[n]) \cong \hBM{2d_1-n}{C}(-d_1).
\end{equation}
In fact,
$\z\in\Corr_C(\DD_{X_1}, \DD_{X_2}[n])$ is the same datum as a map $\kk_{C}[2d_1](d_1) = p_1^* \kk_{X_1}[2d_1](d_1)\to p_2^! \DD_{X_2}[n] = \DD_C[n]$, or the same datum as a class $\z\in \hBM{2d_1-n}{C}(-d_1)$. We thus get a map
\begin{equation*}
    (-)_{\#}: \hBM{2d_1-*}{C}(-d_1)\to \Ext^*_S(\pi_{1*}\DD_{X_1}, \pi_{2*}\DD_{X_2}).
\end{equation*}

Given correspondences $C_1$ between $X_1$ and $X_2$, and $C_2$ between $X_2$ and $X_3$, we have the composition $C=C_2\times_{X_2}C_1$, as a correspondence between $X_1$ and $X_3$. For $\z_1\in\Corr^*_{C_1}(\DD_{X_1},\DD_{X_2})$ and $\Corr^{*}_{C_2}(\DD_{X_2},\DD_{X_3})$, one defines their composition with degrees added
\begin{equation*}
    \z=\z_2\circ\z_1\in \Corr^*_{C}(\DD_{X_1},\DD_{X_3}).
\end{equation*}
See \cite[A.2]{Yun}. When $C_1$ and $C_2$ are both target-proper, so is $C$, and we have (see \cite[\S5.2]{SGA5:FL}) 
\begin{equation*}
    (\z_2\circ\z_1)_{\#}=(\z_2)_\#\circ(\z_1)_\#\in \Ext^*_S(\pi_{1*}\DD_{X_1},\pi_{3*}\DD_{X_3}).
\end{equation*}

\subsubsection{Pullback} If we are given a map $b: S'\to S$, we can base change all spaces in \eqref{C corr} from $S$ to $S'$ to obtained a correspondence $C'=C\times_S S'$ between $X'_1=X_1\times_S S'$ and $X'_2=X_2\times_S S'$. We have a map of correspondences where both squares are Cartesian
\begin{equation*}
    \xymatrix{X'_2 \ar[d]^{f_2} & C'\ar[d]^{c} \ar[l]_-{p'_2}\ar[r]^-{p'_1} & X'_1\ar[d]^{f_1}\\
    X_2 & \ar[l]_-{p_2}  C \ar[r]^-{p_1} & X_1}
\end{equation*}

Then we have a pullback map
\begin{equation*}
    c^!: \Corr^*_{C}(\DD_{X_1}, \DD_{X_2})\to \Corr^*_{C'}(\DD_{X'_1}, \DD_{X'_2})
\end{equation*}
defined by sending $\z: p_1^*\DD_{X_1}\to \DD_C[n]$ to the composition
\begin{equation*}
    p'^*_1\DD_{X'_1}=p_1'^*f_1^!\DD_{X_1}\to c^!p_1^*\DD_{X_1}\xr{c^!\z} c^!\DD_C[n]=\DD_{C'}[n].
\end{equation*}
\trivial[h]{
The map $p_1'^*f_1^!\DD_{X_1}\to c^!p_1^*\DD_{X_1}$ is another type of base change. See \cite[Prop 3.1.9 (iii)]{KS}
}
Here the first map is obtained by adjunction from the proper base change isomorphism attached to the right Cartesian square above (see \cite[Proposition~3.1.9~(iii)]{KS}). A simple diagram chase shows that the following diagram is commutative
\begin{equation}\label{pullback coho corr}
    \xymatrix{\Corr^*_C(\DD_{X_1}, \DD_{X_2}) \ar[r]^{c^!}\ar[d]^{(-)_{\#}} & \Corr^*_{C'}(\DD_{X'_1}, \DD_{X'_2})\ar[d]^{(-)_{\#}}\\
    \Ext^*_S(\pi_{1*}\DD_{X_1}, \pi_{2*}\DD_{X_2})\ar[r]^-{b^!} & \Ext^*_{S'}(\pi'_{1*}\DD_{X'_1}, \pi'_{2*}\DD_{X'_2})}
\end{equation}
Here $b^!$ sends $\ph: \pi_{1*}\DD_{X_1}\to \pi_{2*}\DD_{X_2}[n]$ to the composition
\begin{equation*}
\pi'_{1*}\DD_{X'_1}=\pi'_{1*}f_1^!\DD_{X_1}\cong b^!\pi_{1*}\DD_{X_1}\to b^!\pi_{2*}\DD_{X_2}[n]\cong \pi'_{2*}f_2^!\DD_{X_2}[n]=\pi'_{1*}\DD_{X'_1}[n].
\end{equation*}

\subsubsection{}
Consider the correspondence over $S=[\cN/H]$ 
\begin{equation}\label{corr P1}
    \xymatrix{[\pt/H] & \ar[l]_-{p}\ar[r]^-{z}[\bP^1/H] & [\wt\cN/H]=[T^*\bP^1/H]}
\end{equation}
where $z$ is the inclusion of the zero section, $\pt$ to $\cN$ is the inclusion of the zero element, and its transpose
\begin{equation}\label{P1 transp}
    \xymatrix{[\wt\cN/H]=[T^*\bP^1/H] & \ar[l]_-{z}\ar[r]^-{p}[\bP^1/H] & [\pt/H]}
\end{equation}
Since $[\wt\cN/H]$ is smooth of dimension $-1$, \eqref{eq:CorrHBM} gives 
\begin{equation*}
    \Corr_{[\bP^1/H]}(\DD_{[\wt\cN/H]}, \DD_{[\pt/H]}[2](1))\cong \hBM{-4}{[\bP^1/H]}(-2).
\end{equation*}
We take $\z=[\bP^1/H]\in \hBM{-4}{[\bP^1/H]}(-2)$ to be the fundamental class. Recall the projection $\nu: [\wt\cN/H]\to [\cN/H]$, and let $a: [\pt/H]\to [\cN/H]$ be the inclusion of the zero orbit. Then 
\begin{equation*}
    \z_{\#}: \nu_*\DD_{[\wt\cN/H]}\to a_*\DD_{[\pt/H]}[2](1).
\end{equation*}
The same class $[\bP^1/H]$ but viewed as a cohomological correspondence supported on the transpose diagram \eqref{P1 transp} defines $\z'\in \Corr_{[\bP^1/H]}(\DD_{[\pt/H]}[2](1), \DD_{[\wt\cN/H]})$ and
\begin{equation*}
    \z'_{\#}: a_*\DD_{[\pt/H]}[2](1)\to \nu_*\DD_{[\wt\cN/H]}.
\end{equation*}

\begin{lemma}\label{l:sl2 Spr}
    The Springer action of $s-1$ on $\nu_*\DD_{[\wt\cN/H]}$ is the composition $\z'_{\#} \circ \z_{\#}$. 
\end{lemma}
\begin{proof}
    The map induced by forgetting equivariance and taking global sections
    \begin{equation*}
        \End(\nu_*\DD_{[\wt\cN/H]})\to \End(\nu_*\DD_{[\wt\cN]})\to \End(\hBM{*}{\wt\cN})
    \end{equation*}
    is injective: the left side is $\kk\oplus \kk$, with basis given by the  projectors to the constant and skyscraper summands of $\nu_*\DD_{[\wt\cN/H]}$, and it is easy to check that their induced maps on $\hBM{*}{\wt\cN}$ are linearly independent. Therefore it suffices to check that $s-1$ and $\z'_{\#} \circ \z_{\#}$ have the same effect on $\hBM{*}{\wt\cN}$ after taking global sections. 
   Under \Cref{con:Springer}, we known that $s$ acts on the top homology $\hBM{4}{\wt\cN}$ by $1$ and on $\hBM{2}{\wt\cN}=\kk [\bP^1]$ by $-1$. Thus it suffices to show that
    \begin{equation}\label{s-1 sl2}
        \mbox{$\z'_{\#}\circ \z_{\#}$ acts by $0$ on $\hBM{4}{[\wt\cN]}$ and by $-2$ on $\hBM{2}{\wt\cN}$.}
    \end{equation}
    By Poincar\'e duality,  $\z_{\#}: \hBM{i}{\wt\cN}\to \hBM{i-2}{\bP^1}(1)$ can be identified with the restriction map $\cohog{4-i}{\wt\cN}\to \cohog{4-i}{\bP^1}$ (up to Tate twists). From this we see $\z_{\#}$ is zero on $\hBM{4}{\wt\cN}$, and sends $[\bP^1]\in \cohog{2}{\wt\cN}(1)$ to $-2\in \cohog{2}{\bP^1}(1)$ because  the self-intersection number of $\bP^1$ in $\wt\cN=T^*\bP^1$ is $-2$. From this \eqref{s-1 sl2} follows easily.
\end{proof}

\subsubsection{Finish of the proof of  \Cref{p:Hotta}}
We view the correspondence \eqref{corr X0} as over $Y$, then it is obtained from the correspondence \eqref{corr P1} by base change along $f: Y\to [\cN/H]$. The same holds for their transposes. We have the pullbacks $f^!\z\in\Corr_{X_0}(\DD_X, \DD_{Y_0}[2](1))$, and $f^!\z'\in \Corr_{X_0}(\DD_{Y_0}[2](1), \DD_{X})$. By \Cref{l:sl2 Spr}, the action of $s-1$ on $\nu_*\DD_{[\wt\cN/H]}$ is given by $\z'_{\#}\circ \z_{\#}$. Therefore by the diagram \eqref{pullback coho corr}, the action of $s-1$ on $f_*\DD_{X}$ is given by the composition $(f^!\z')_{\#}\circ (f^!\z)_{\#}$.

On the other hand, letting $\wt a: Y_0\hookrightarrow Y$ be the inclusion, the composition \eqref{Hotta} is the effect on global sections of the composition
\begin{equation*}
    \pi_*\DD_X\xr{i^!}(\wt a\pi_{0})_{*}\DD_{X_0}[2](1)\xr{\pi_{0*}}\wt a_*\DD_{Y_0}[2](1)\xr{\pi_0^!}(\wt a\pi_{0})_{*}\DD_{X_0}\xr{i_*}\pi_{*}\DD_{X}.
\end{equation*}
Here $i^!$ is defined by applying $f_*$ to the map \eqref{pre Gysin}. 

It remains to show that 
\begin{eqnarray}
\label{fz}    (f^!\z)_{\#}=\pi_{0*}\circ i^!: \pi_*\DD_X\to \wt a_*\DD_{Y_0}[2](1)\\
\label{fz'}    (f^!\z')_{\#}=i_{*}\circ \pi_{0}^!: \wt a_*\DD_{Y_0}[2](1)\to \pi_*\DD_X. 
\end{eqnarray}
Both identities follow from the definitions by standard diagram-chasing. 

\bibliographystyle{alpha}
\bibliography{Thetacell}


\clearpage
\ 
\end{document}